\documentclass[11pt,twoside,a4paper,final]{report}

\usepackage[DM,newLogo]{uaThesis}

\def\ThesisYear{2012}

\usepackage[english]{babel}
\usepackage{hyperref}
\usepackage{amsmath}
\usepackage{amssymb}
\usepackage{amsthm}
\usepackage{graphicx}
\usepackage{makeidx}

\usepackage{fancyhdr}

\usepackage{pagesInBib}

\addto\captionsenglish{}

\newtheorem{theorem}{Theorem}[section]
\newtheorem{lemma}[theorem]{Lemma}
\newtheorem{corollary}[theorem]{Corollary}
\newtheorem{proposition}[theorem]{Proposition}
\newtheorem{ex}[theorem]{Example}

\newtheorem{definition}[theorem]{Definition}
\newtheorem{example}[theorem]{Example}

\newtheorem{remark}[theorem]{Remark}

\numberwithin{equation}{section}

\def\I{\mathtt{i}}

\def\EXP#1.{e^{2\pi\I #1}}

\def\topfigrule{\kern 7.8pt \hrule width\textwidth\kern -8.2pt\relax}
\def\dblfigrule{\kern 7.8pt \hrule width\textwidth\kern -8.2pt\relax}
\def\botfigrule{\kern -7.8pt \hrule width\textwidth\kern 8.2pt\relax}

\makeindex

\begin{document}

\fancyhf{}

\fancyhead[LO]{\slshape \rightmark}
\fancyhead[RE]{\slshape \leftmark}
\fancyfoot[C]{\thepage}


\TitlePage
\HEADER{\BAR\FIG{\includegraphics[height=60mm]{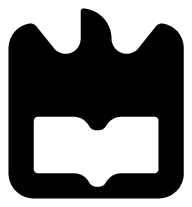}}}{\ThesisYear}
\TITLE{Artur Miguel C. \newline Brito da Cruz}{Symmetric Quantum Calculus}
\EndTitlePage


\titlepage\ \endtitlepage


\TitlePage
\HEADER{}{\ThesisYear}
\TITLE{Artur Miguel \newline C. Brito da Cruz}{C\'alculo Qu\^antico Sim\'etrico}
\vspace*{15mm}
\TEXT{}{Tese de doutoramento apresentada \`a Universidade de Aveiro para cumprimento dos requisitos
necess\'arios \`a obten\c c\~ao do grau de Doutor em Matem\'atica,
Programa Doutoral em Matem\'{a}tica e Aplica\c{c}\~{o}es --- PDMA 2008-2012 --
da Universidade de Aveiro e Universidade do Minho,
rea\-li\-za\-da sob a orienta\c c\~ao
cient\'\i fica do Doutor Delfim Fernando Marado Torres, Professor Associado com Agrega\c{c}\~{a}o
do Departamento de Matem\'atica da Universidade de Aveiro, e co-orienta\c c\~ao da Doutora Nat\'alia da Costa Martins,
Professora Auxiliar do Depar\-ta\-men\-to de Matem\'atica da Universidade de Aveiro.}
\EndTitlePage


\titlepage\ \endtitlepage


\TitlePage
\vspace*{55mm}
\TEXT{\textbf{o j\'uri~/~the jury\newline}}{}
\TEXT{presidente~/~president}{\textbf{Prof. Doutor Ant\'{o}nio Manuel Melo de Sousa Pereira}\newline {\small
        Professor Catedr\'atico da Universidade de Aveiro (por delega\c c\~ao da Reitoria da
        Universidade de Aveiro)}}
\vspace*{5mm}
\TEXT{vogais~/~examiners committee}{\textbf{Prof. Doutora Lisa Maria de Freitas Santos}\newline {\small
        Professora Associada com Agrega\c{c}\~ao da Escola de Ci\'{e}ncias da Universidade do Minho}}
\vspace*{5mm}
\TEXT{}{\textbf{Prof. Doutor Delfim Fernando Marado Torres}\newline {\small
        Professor Associado com Agrega\c{c}\~ao da Universidade de Aveiro (Orientador)}}
\vspace*{5mm}
\TEXT{}{\textbf{Prof. Doutor Jos\'{e} Carlos Soares Petronilho}\newline {\small
        Professor Associado da Faculdade de Ci\^{e}ncias e Tecnologia da Universidade de Coimbra}}
\vspace*{5mm}
\TEXT{}{\textbf{Prof. Doutor Jos\'{e} Lu\'{\i}s dos Santos Cardoso}\newline {\small
        Professor Associado da Escola de Ci\^{e}ncias e Tecnologias da Universidade de
        Tr\'{a}s- -Os-Montes e Alto Douro}}
\vspace*{5mm}
\TEXT{}{\textbf{Prof. Doutora Nat\'{a}lia da Costa Martins}\newline {\small
        Professora Auxiliar da Universidade de Aveiro (Coorientadora)}}
\vspace*{5mm}
\TEXT{}{\textbf{Prof. Doutor Ricardo Miguel Moreira de Almeida}\newline {\small
        Professor Auxiliar da Universidade de Aveiro}}
\EndTitlePage


\titlepage\ \endtitlepage


\TitlePage
\vspace*{55mm}
\TEXT{\textbf{agradecimentos~/\newline acknowledgements}}{}
\TEXT{}{Uma tese de Doutoramento \'{e} um processo solit\'{a}rio a que um aluno est\'{a} destinado.
S\~{a}o quatro anos de trabalho que s\~{a}o mais pass\'{\i}veis
de suportar gra\c{c}as ao apoio de v\'{a}rias pessoas e institui\c{c}\~{o}es.\\
Assim, e antes dos demais, gostaria de agradecer aos meus orientadores,
Professora Doutora Nat\'{a}lia Martins e Professor Doutor Delfim F. M. Tor\-res,
pelo apoio, pela partilha de saber e por estimularem o meu interesse pela Matem\'{a}tica.\\
Estou igualmente grato aos meus colegas e aos meus Professores do Programa Doutoral
pelo constante incentivo e pela boa disposi\c{c}\~{a}o que me transmitiram durante estes anos.\\
Gostaria de agradecer \`{a} FCT (``Fundac\~ao para a Ci\^encia e a Tecnologia'')
pelo apoio financeiro atribu\'{\i}do atrav\'{e}s da bolsa de Doutoramento
com a refer\^{e}ncia SFRH/BD/33634/2009.
Agrade\c{c}o \`{a} Escola Superior de Tec\-no\-lo\-gia de Set\'{u}bal,
nomeadamente aos meus colegas do Departamento de Matem\'{a}tica, pela sua disponibilidade. \\
Por \'{u}ltimo, mas sempre em primeiro lugar, agrade\c{c}o \`{a} minha fam\'{i}lia.}
\vfill

\TEXT{}{
\hspace*{-0.35cm} \includegraphics [scale=0.3]{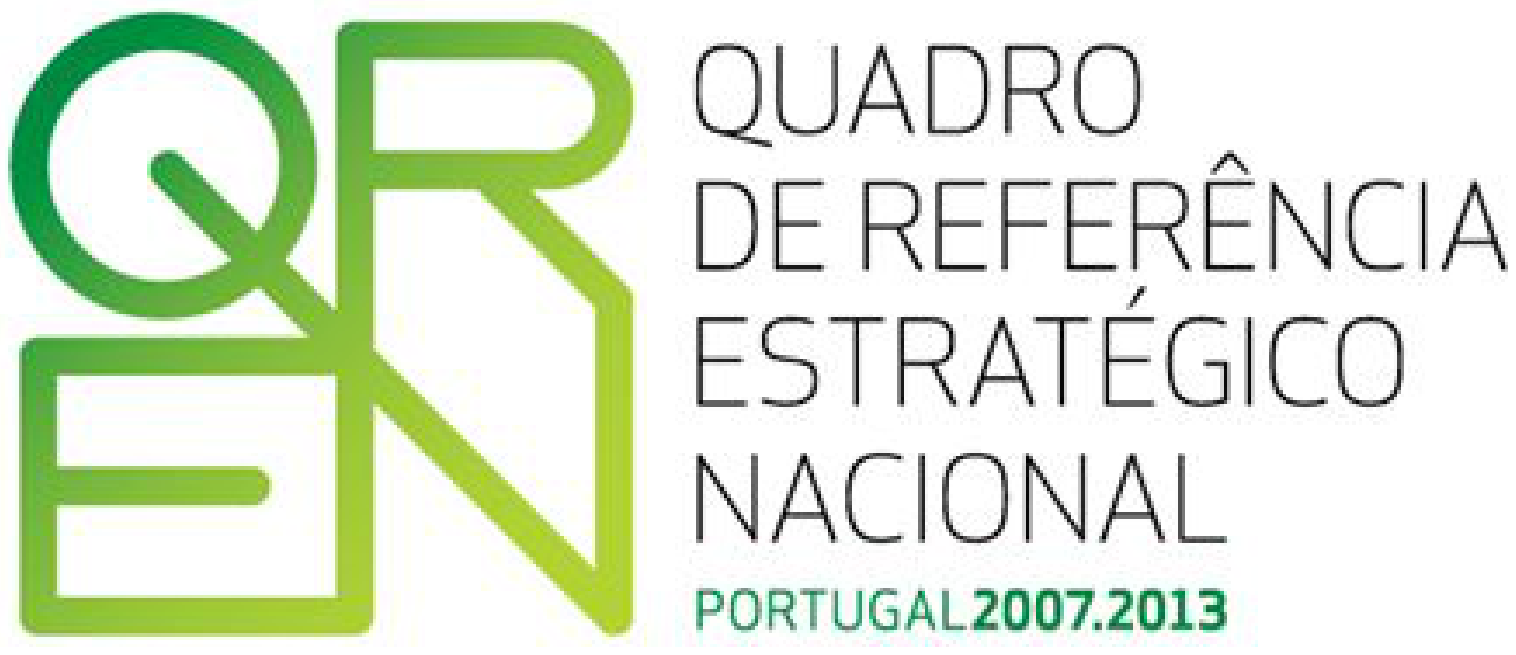}\includegraphics [scale=0.32]{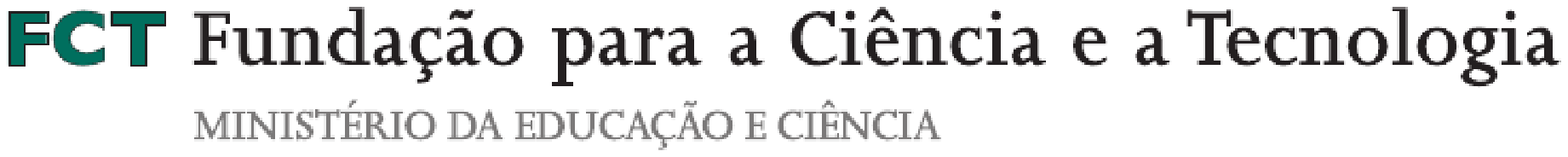}\\
\hspace*{+1.60cm} \includegraphics [scale=0.3]{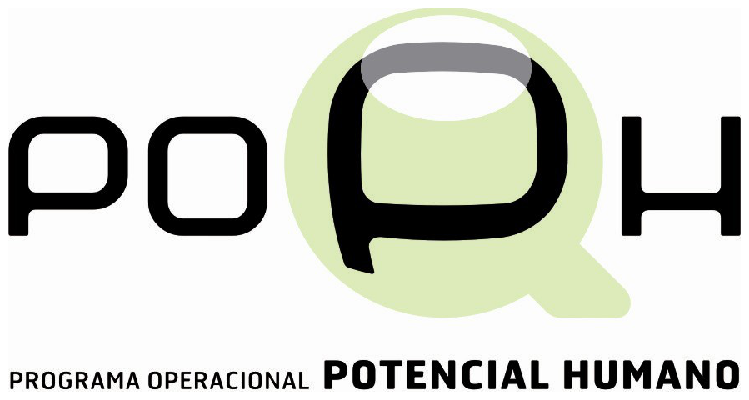} \qquad \hspace*{+3.60cm} \includegraphics [scale=0.5]{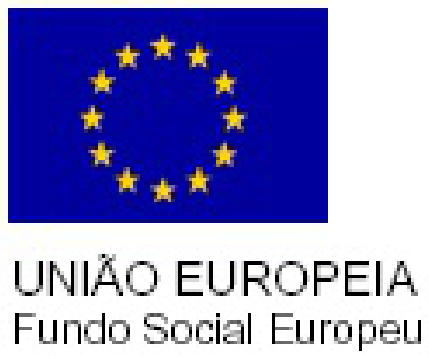}
}

\EndTitlePage


\titlepage\ \endtitlepage


\TitlePage
\vspace*{55mm}
\TEXT{\textbf{Resumo}}{Generalizamos o c\'{a}lculo Hahn variacional para problemas
do c\'{a}lculo das varia\c{c}\~{o}es que envolvem derivadas de ordem superior.
Estudamos o c\'{a}lculo qu\^{a}ntico sim\'{e}trico, nomeadamente o c\'{a}lculo qu\^{a}ntico $\alpha,\beta$-sim\'{e}trico,
$q$-sim\'{e}trico e Hahn-sim\'{e}trico. Introduzimos o c\'{a}lculo qu\^{a}ntico sim\'{e}trico variacional
e deduzimos equa\c{c}\~{o}es do tipo Euler--Lagrange para o c\'{a}lculo $q$-sim\'{e}trico e Hahn sim\'{e}trico.
Definimos a derivada sim\'{e}trica em escalas temporais e deduzimos algumas das suas propriedades.
Finalmente, introduzimos e estudamos o integral diamond que generaliza o integral diamond-$\alpha$
das escalas temporais.\\ \\ \\}

\TEXT{\textbf{Palavras-chave}}{C\'{a}lculo qu\^{a}ntico, c\'{a}lculo das varia\c{c}\~{oes},
condi\c{c}\~{o}es necess\'{a}rias do tipo Euler-–Lagrange, derivada sim\'{e}trica,
integral diamond, desigualdades integrais, escalas temporais.\\ \\ \\}

\TEXT{\textbf{2010 Mathematics Subject Classification:}}{\\34N05, 39A12, 39A13, 49K05, 49K15.}


\EndTitlePage
\titlepage\ \endtitlepage


\TitlePage
\vspace*{55mm}
\TEXT{\textbf{Abstract}}{We generalize the Hahn variational calculus
by studying problems of the calculus of variations with higher-order derivatives.
The symmetric quantum calculus is studied, namely
the $\alpha,\beta$-symmetric, the $q$-symmetric,
and the Hahn symmetric quantum calculus. We introduce the symmetric
quantum variational calculus and an Euler--Lagrange type
equation for the $q$-symmetric and Hahn's symmetric quantum calculus is proved.
We define a symmetric derivative on time scales and derive some of its properties.
Finally, we introduce and study the diamond integral,
which is a refined version of the diamond-$\alpha$ integral on time scales.\\ \\ \\}

\TEXT{\textbf{Keywords}}{Quantum calculus, calculus of variations,
necessary optimality conditions of Euler-–Lagrange type,
symmetric derivative, diamond integral, integral inequalities, time scales.\\ \\ \\}

\TEXT{\textbf{2010 Mathematics Subject Classification:}}{\\34N05, 39A12, 39A13, 49K05, 49K15.}
\EndTitlePage


\titlepage\ \endtitlepage


\pagenumbering{roman}
\tableofcontents

\cleardoublepage
\pagenumbering{arabic}


\chapter*{Introduction}\markboth{INTRODUCTION}{}

Four years ago I started the Ph.D. Doctoral Programme in Mathematics and
Applications offered jointly by University of Aveiro and University of
Minho.\ In the first year we, students, had several one-semester courses in
distinct fields of mathematics. Many interesting subjects were discussed in
these courses and one day, one of this discussions led my Advisor, Professor
Delfim F. M. Torres, to propose me that the subject of my Ph.D. thesis
should be symmetric quantum calculus and/or symmetric calculus on time scales.
Gladly, I accepted the offer and here we present the result of our work.

In classical calculus, the symmetric derivative is defined by
\[
f^{s}\left( x\right) =\lim_{h\rightarrow 0}\frac{f\left( x+h\right) -f\left(
x-h\right) }{2h}.
\]
The definition of the symmetric derivative emerged from the necessity to
extend the notion of differentiability to points that do not have classical
derivative and its applications are important in many problems, namely in
the study of trigonometric series. Its study began in the mid-19th century
and since then, according to Thomson \cite{Thomson}, many important
figures in analysis contributed for this topic. Thomson also highlights the
importance of authors such Charzy\'{n}ski, Denjoy, Kijtchine, Marcinkiewicz,
Sierpinski and Zygmund (see, e.g.,
\cite{Charzynski,Denjoy,Khintchine,Marcinkiewicz,Sierpinski,Zygmund}).
Thomson refers that the study of the symmetric properties of real functions
is as relevant as ever. In the last decades we have seen many new
material produced in this topic (see, e.g.,
\cite{Aull,Belna:2,Belna,Belna:3,Buczolich,Evans:3,Evans:2,Evans,Foran,Freiling,Larson,Larson:2,Mazurkiewicz}),
in spite the difficulty of obtaining new results.

However, the idea of my advisor was not to develop new results for the classical
symmetric calculus, but to introduce the symmetric quantum
calculus. The fusion of these two subjects, symmetric calculus and quantum calculus,
seemed natural to him since he is an active researcher
on the time scale theory (some approaches of quantum
calculus can be seen as a particular case of time scale calculus) and knew
the absence of the notion of the symmetric derivative on time scale calculus.

Quantum from the Latin word ``quantus'' (in Portuguese ``quantos'') literally
means how much. Usually we associate the term quantum to the minimum amount
of any measure or entity. In mathematics, the quantum calculus refers to the
``calculus without limits''. Usually, the quantum calculus is identified with
 the $q$-calculus and the $h$-calculus and both of them were studied by
Kac and Cheung in their book \cite{Kac}.

In 1750 Euler proved the pentagonal number theorem which was the first example
of a q-series and, in some sense, he introduced the $q$-calculus.
The $q$-derivative was (re)introduced by Jackson \cite{Jackson:old} and for
$q\in \left] 0,1\right[$, the $q$-derivative of a function
$f:\overline{q^\mathbb{Z}}\rightarrow \mathbb{R}$ is defined by
\[
D_{q}\left[ f\right] \left( t\right) :=\frac{f\left( qt\right)
- f\left(t\right) }{\left( q-1\right) t}, \ \ t\neq 0.
\]
Today, and according to Ernst \cite{Ernst,Ernst:1meio}, the
majority of scientists who use $q$-calculus are physicists,
and he cites Jet Wimp \cite{Wimp}:
\begin{quote}
``The field has expanded explosively, due to the fact that applications of
basic hypergeometric series to the diverse subjects of combinatorics,
quantum theory, number theory, statistical mechanics, are constantly being
uncovered. The subject of q-hypergeometric series is the frog-prince of
mathematics, regal inside, warty outside.''
\end{quote}

For $h>0$, the $h$-derivative of a function
$f : h \mathbb{Z} \rightarrow \mathbb{R}$ is defined by
\[
D_{h}\left[ f\right] \left( t\right)
:=\frac{f\left( t+h\right) -f\left(t\right)}{h}
\]
and is also known as finite difference operator. Taylor's ``Methods
Incrementorum'' \cite{Taylor} is considered the first reference of the
$h$-calculus or the calculus of finite differences but it is Jacob Stirling
\cite{Stirling} who is considered the founder of the $h$-calculus. In 1755,
Leonhard Euler \cite{Eulero} introduced the symbol $\Delta $ for
differences. Some of the more important works about $h$-differences are from
authors like Boole \cite{Boole:old}, Markoff \cite{Markoff},
Whittaker and Ronbison \cite{Whittaker}, N\"{o}rlund \cite{Norlund}, Milne-Thomson
\cite{Milne} or Jordan \cite{Jordan} (we ordered them historically).
Note that the finite difference operator is a discretization
of the classical derivative and an immediate application of the
$h$-derivative (but also the $q$-derivative) is in numerical analysis,
especially in numerical differential equations, which aim at the numerical
solution of ordinary and partial differential equations. For a
deeper understanding of quantum calculus and its history we refer the reader
to \cite{Ernst,Ernst:2,Gasper,Kac,Koekoek}.

Another type of quantum calculus is the Hahn's quantum calculus which can be
seen as a generalization of both $q$-calculus and $h$-calculus.
For $q\in \left] 0,1\right[$ and $\omega \geqslant0$,
Hahn's difference operator is defined by
\[
D_{q,\omega }\left[ f\right] :=\frac{f\left( qt+\omega \right) - f\left(
t\right) }{\left( q-1\right) t+\omega }, \ \ t\neq\frac{\omega}{1-q},
\]
where $f$ is a function defined on a real interval
that contains $\omega_{0}:=\displaystyle\frac{\omega}{1-q}$.
Although Hahn defined this operator in 1949, only in 2009 Aldwoah
\cite{Aldwoah,Aldwoah:2} constructed its inverse operator. This
construction also provided us tools to generalize both $q$-calculus and the
$h$-calculus to functions defined in real intervals instead of the sets
$\overline{q^{\mathbb{Z}}}$ and $h \mathbb{Z}$. We remark here that the term
``quantum calculus'' is applied in the literature (and in this thesis)
in different contexts: is used to denote the $h$-calculus and $q$-calculus,
but also to denote the Hahn calculus (that is also a ``calculus without limits''
but in this theory the domain of each function is a real interval
instead of the quantum sets $\overline{q^{\mathbb{Z}}}$ or $h \mathbb{Z}$).

In a first step towards the development of the symmetric quantum calculus we
decided to study an application of Hahn's quantum calculus. Since both my
advisors mainly work in Calculus of Variations and Optimal Control, we
decided to study the higher-order Hahn quantum variational calculus. The
calculus of variations is concerned with the problem of extremising
functionals and in Chapter~\ref{Calculus of Variations} the reader
can find a brief introduction to this subject. For a study of the
calculus of variations within the quantum calculus, we suggest the articles
\cite{Bangerezako,Bangerezako:2,Malinowska:3,Martins:2}.

In this work, we have tried to develop the calculus of variations
in the context for symmetric quantum calculus.
But soon we understood that this is not always possible,
since for example for the $\alpha,\beta$-symmetric quantum calculus
(see Chapter~\ref{A Symmetric Quantum Calculus}) we do not have (yet?)
a Fundamental Theorem of Integral Calculus and/or an Integration by Parts
formula. For the symmetric $q$-calculus and Hahn's symmetric quantum calculus,
we were able to introduce and develop the respective calculus of variations.

Another subject that was of our interest was the symmetric calculus on time
scales. In 1988, Hilger \cite{Hilger} introduced the theory of time scales
which is a theory that was created in order to unify and to extend discrete
and continuous analysis into a single theory. A time scale is a nonempty closed
subset of $\mathbb{R}$ and both $q$-calculus and $h$-calculus can be seen as a particular case
of time scale calculus. Since then, many works arose and
we highlight \cite{Agarwal:2,Agarwal,Bohner,Bohner:2} for the time scale calculus and
\cite{Almeida,Bartosiewicz,Bartosiewicz05,Bohner:3,Bohner:4,Ferreira:3,Ferreira,Ferreira:2,Hilscher,Hilscher:2,
Malinowska35,Malinowska,Malinowska:4,Malinowska:34,Malinowska45,Martins,Martins15,Martins:25,Martins:4,Torres}
for the calculus of variations on time scales. We have made progresses in this field
and we present our results in Chapter~\ref{The Symmetric Calculus on Time Scales}. We were not successful to
define a symmetric integral in the time scale context, mainly because there is not a symmetric
integral for the classical case. However, we were able to generalize the
diamond-$\alpha $ integral which, in some sense, is an attempt to define a
symmetric integral on time scales.

We divided this thesis in two parts. In the first part we present some
preliminaries about time scale calculus, quantum calculus and calculus of variations. The second
part is divided in six chapters where we present the original work and our
conclusions. Specifically, in Chapter~\ref{Higher-order Hahn's Quantum Variational Calculus}
we develop the higher-order Hahn Quantum Variational Calculus proving the Euler--Lagrange
equation for the higher-order problem of the calculus of variations within the Hahn quantum calculus.
In Chapter~\ref{A Symmetric Quantum Calculus} we develop a Symmetric Quantum Calculus:
in Section~\ref{hs:sec:3.3} we present some mean value theorems for the symmetric calculus
and in Section~\ref{hi:sec:ineq} we prove H\"{o}lder's, Cauchy-Schwarz's and Minkowski's inequalities
in the setting of the $\alpha,\beta$-symmetric calculus. In Chapter~\ref{The $q$-Symmetric Variational Calculus}
we develop the $q$-Symmetric Calculus presenting a necessary optimality condition
and a sufficient optimality condition for variational problems involving the $q$-symmetric derivative.
In Chapter~\ref{Hahn's Symmetric Quantum Variational Calculus} we generalize
for the Hahn Symmetric Quantum Variational Calculus the results obtained in the preceding chapter.
We also prove, in Section~\ref{qhs:L}, that Leitmann's Direct Method can be applied
in Hahn's symmetric quantum variational calculus. In Chapter~\ref{The Symmetric Calculus on Time Scales}
we introduce the symmetric derivative on time scales and prove some properties of this new derivative.
In Section~\ref{ts:sec:int} we introduce the diamond integral and deduce some results for this new integral.

In Chapter~\ref{Conclusions}, we write our conclusions
and some possible directions for future work.


\clearpage{\thispagestyle{empty}\cleardoublepage}


\part{Synthesis}
\label{Synthesis}


\clearpage{\thispagestyle{empty}\cleardoublepage}


\chapter{Time Scale Calculus}
\label{ts:sec:pre}

The theory of time scales was born in 1988 with the Ph.D. thesis of Stefan Hilger,
done under the supervision of Bernd Aulbach \cite{Hilger}. The aim of this theory
was to unify various definitions and results from the theories of discrete and
continuous dynamical systems, and to extend such theories
to more general classes of dynamical systems.

The calculus of time scales is nowadays an area of great interest of many mathematicians;
this can be showed by the numerous papers published in this field.

For a general introduction to the theory of time scales we refer the reader
to the excellent books \cite{Bohner,Bohner:2}. Here we only give those definitions
and results needed in this thesis.

As usual, $\mathbb{R,Z}$ and $\mathbb{N}$ denote, respectively,
the set of real, integer and natural numbers.

A\ nonempty closed subset of $\mathbb{R}$ is called a \emph{time scale}\index{Time scale}
and is denoted by $\mathbb{T}$. Thus $\mathbb{R,Z}$ and $\mathbb{N}$ are trivial examples
of time scales. Other examples of time scales are: $\left[-2,5\right]\cup\mathbb{N}$,
$h\mathbb{Z}:=\{hz: z\in\mathbb{Z}\}$ for some $h>0$,
$\overline{q^{\mathbb{Z}}}:=\{q^z:z\in\mathbb{Z}\}\cup{0}$ for some $q\neq 1$ and the Cantor set.

We assume that a time scale has the topology inherited
from $\mathbb{R}$ with the standard topology.

We consider two jump operators:
the \emph{forward jump operator}\index{Forward jump operator}
$\sigma : \mathbb{T}\rightarrow \mathbb{T}$, defined by
\begin{equation*}
\sigma \left( t\right) :=\inf \left \{ s\in \mathbb{T}:s>t\right \},
\end{equation*}
and the \emph{backward jump operator}\index{Backward jump operator}
$\rho :\mathbb{T} \rightarrow \mathbb{T}$ defined by
\begin{equation*}
\rho \left( t\right) :=\sup \left \{ s\in \mathbb{T}:s<t\right \}
\end{equation*}
with $\inf \emptyset =\sup \mathbb{T}$ (i.e., $\sigma \left( M\right) =M$ if
$\mathbb{T}$ has a maximum $M$) and $\sup \emptyset =\inf \mathbb{T}$ (i.e.,
$\rho \left( m\right) =m$ if $\mathbb{T}$ has a minimum $m$). A point $t\in
\mathbb{T}$ is said to be \emph{right-dense}\index{Right-dense point},
\emph{right-scattered}\index{Right-scattered point},
\emph{left-dense}\index{Left-dense point} and
\emph{left-scattered}\index{Left-scattered point}
if $\sigma \left( t\right) =t$, $\sigma \left( t\right) >t$,
$\rho \left( t\right) =t$ and $\rho \left( t\right) <t$, respectively.
A point $t\in \mathbb{T}$ is called \emph{dense}\index{Dense point}
if it is right and left dense.
The \emph{backward graininess function}\index{Backward graininess function}
$\nu:\mathbb{T}\rightarrow\left[0,+\infty\right[$ is defined by
$$
\nu \left( t\right):=t-\rho\left(t\right), \ \ t\in\mathbb{T}.
$$
The \emph{forward graininess function}\index{Forward graininess function}
$\mu:\mathbb{T}\rightarrow\left[0,+\infty\right[$ is defined by
$$
\mu \left( t\right):=\sigma\left(t\right)-t, \ \ t\in\mathbb{T}.
$$
It is clear that when $\mathbb{T}=\mathbb{R}$ one has
$\sigma\left(t\right)=t=\rho\left(t\right)$
and $\nu\left(t\right)=\mu\left(t\right)=0$ for any $t\in\mathbb{T}$.
When $\mathbb{T}=h\mathbb{Z}$ (for some $h>0$), then $\sigma\left(t\right)=t+h$,
$\rho\left(t\right)=t-h$ and $\nu\left(t\right)=\mu\left(t\right)=h$.

In order to introduce the definition of the delta derivative and the nabla derivative
we define the sets $\mathbb{T}^{\kappa}$ and $\mathbb{T}_{\kappa}$ by
\begin{equation*}
\begin{array}{l}
\mathbb{T}^{\kappa}:=\left \{
\begin{array}{l}
\mathbb{T}\backslash \left \{ t_{1}\right \} \text{ ,\ if }\mathbb{T}\text{
has a left-scattered maximum }t_{1} \\
\\
\mathbb{T}\text{ , otherwise}
\end{array}
\right.  \\
\\
\mathbb{T}_{\kappa}:=\left \{
\begin{array}{l}
\mathbb{T}\backslash \left \{ t_{2}\right \} \text{ ,\ if }\mathbb{T}\text{
has a right-scattered minimum }t_{2} \\
\\
\mathbb{T}\text{ , otherwise.}
\end{array}
\right .  \\
\\
\end{array}
\end{equation*}

We say that a function $f:\mathbb{T} \rightarrow \mathbb{R}$
is delta differentiable at $t\in \mathbb{T}^{\kappa}$ if there exists a
number $f^{\Delta }\left( t\right) $ such that, for all $\varepsilon >0$,
there exists a neighborhood $U$ of $t$ such that
\begin{equation*}
\left \vert f\left( \sigma \left( t\right) \right) -f\left( s\right)
-f^{\Delta }\left( t\right) \left( \sigma \left( t\right) -s\right) \right
\vert \leqslant \varepsilon \left \vert \sigma \left( t\right) -s\right
\vert ,
\end{equation*}
for all $s\in U$. We call $f^{\Delta }\left( t\right)$
the \emph{delta derivative}\index{Delta derivative}
of $f$ at $t$ and we say that $f$ is delta differentiable
if $f$ is delta differentiable for all $t\in \mathbb{T}^{\kappa}$.

We say that a function $f:\mathbb{T} \rightarrow \mathbb{R}$
is nabla differentiable at $t\in \mathbb{T}_{\kappa}$ if there exists a
number $f^{\nabla }\left( t\right) $ such that, for all $\varepsilon >0$,
there exists a neighborhood $V$ of $t$ such that
\begin{equation*}
\left \vert f\left( \rho \left( t\right) \right) -f\left( s\right)
-f^{\nabla }\left( t\right) \left( \rho \left( t\right) -s\right) \right
\vert \leqslant \varepsilon \left \vert \rho \left( t\right) -s\right \vert ,
\end{equation*}
for all $s\in V$. We call $f^{\nabla }\left( t\right)$
the \emph{nabla derivative}\index{Nabla derivative}
of $f$ at $t$ and we say that $f$ is nabla differentiable
if $f$ is nabla differentiable for all $t\in \mathbb{T}_{\kappa}$.

In order to simplify expressions, we denote
$f\circ\sigma$ by $f^{\sigma}$ and $f\circ\rho$ by $f^{\rho}$.

The delta and nabla derivatives verify the following properties.

\begin{theorem}[\protect cf. \cite{Bohner,Bohner:2}]
Let $f:\mathbb{T} \rightarrow \mathbb{R}$ be a function.
Then we have the following:
\begin{enumerate}
\item If $f$ is delta differentiable at $t\in\mathbb{T}^{\kappa}$
or nabla differentiable at $t\in\mathbb{T}_{\kappa}$,
then $f$ is continuous on $t$;

\item If $f$ is continuous at $t\in \mathbb{T}^{\kappa}$ and $t$ is
right-scattered, then $f$ is delta differentiable at $t$ with
\begin{equation*}
f^{\Delta }\left( t\right) =\frac{f^{\sigma}\left( t\right)
-f\left( t\right) }{\sigma \left( t\right)-t};
\end{equation*}

\item $f$ is delta differentiable at a right-dense point $t\in \mathbb{T}^{\kappa}$
if the limit
\begin{equation*}
\lim_{s\rightarrow t}\frac{f\left( t\right) -f\left( s\right) }{t-s}
\end{equation*}
exists and in that case
\begin{equation*}
f^{\Delta }\left( t\right) =\lim_{s\rightarrow t}\frac{f\left( t\right)
-f\left( s\right) }{t-s};
\end{equation*}

\item If $f$ is delta differentiable at $t\in\mathbb{T}^{\kappa}$, then
\begin{equation*}
f^{\sigma }\left( t\right) =f\left( t\right)+\mu\left( t\right)f^\Delta\left(t\right);
\end{equation*}

\item If $f$ is continuous at $t\in \mathbb{T}_{{k}}$ and $t$ is
left-scattered, then $f$ is nabla differentiable at $t$ with
\begin{equation*}
f^{\nabla }\left( t\right) =\frac{f^{\rho}\left( t\right)
-f\left( t\right) }{\rho \left( t\right) -t};
\end{equation*}

\item $f$ is nabla differentiable at a left-dense
point $t\in \mathbb{T}_{{k}}$ if the limit
\begin{equation*}
\lim_{s\rightarrow t}\frac{f\left( t\right) -f\left( s\right) }{t-s}
\end{equation*}
exists and in that case
\begin{equation*}
f^{\nabla }\left( t\right) =\lim_{s\rightarrow t}\frac{f\left( t\right)
-f\left( s\right) }{t-s};
\end{equation*}

\item If $f$ is nabla differentiable at $t\in\mathbb{T}_{\kappa}$, then
\begin{equation*}
f^{\rho }\left( t\right) =f\left( t\right)-\nu\left( t\right)f^\nabla\left(t\right).
\end{equation*}
\end{enumerate}
\end{theorem}

\begin{remark}
\begin{enumerate}
\item If $\mathbb{T=
\mathbb{R}
}$, then the delta and nabla derivative are the usual derivative.

\item If $\mathbb{T=}\mathbb{
\mathbb{Z},
}$ then the delta derivative is the
\emph{forward difference operator}\index{Forward difference operator} defined by
\begin{equation*}
f^{\Delta }\left( t\right) =f\left( t+1\right) -f\left( t\right)
\end{equation*}
and the nabla derivative is the
\emph{backward difference operator}\index{Backward difference operator} defined by
\begin{equation*}
f^{\nabla }\left( t\right) =f\left( t\right) -f\left( t-1\right) .
\end{equation*}

\item For any time scale $\mathbb{T}$, if $f$ is constant, then
$
f^\Delta=f^\nabla\equiv0;
$
if $f\left(t\right)=kt$ for some constant $k$, then
$
f^\Delta=f^\nabla\equiv k.
$
\end{enumerate}
\end{remark}

The delta derivative satisfies the following properties.

\begin{theorem}[\protect cf. \cite{Bohner,Bohner:2}]
Let $f,g:\mathbb{T} \rightarrow \mathbb{R}$
be delta differentiable functions. Let $t\in \mathbb{T}^{\kappa}$
and $\lambda \in \mathbb{R}$. Then
\begin{enumerate}
\item The function $f+g$ is delta differentiable with
$$
\left( f+g\right) ^{\Delta }\left( t\right)  =f^{\Delta }\left( t\right)
+g^{\Delta }\left( t\right)   ;$$

\item The function $\lambda f$ is delta differentiable with
$$
\left( \lambda f\right) ^{\Delta }\left( t\right)  =\lambda f^{\Delta
}\left( t\right) ;
$$

\item The function $fg$ is delta differentiable with
$$
\left( fg\right) ^{\Delta }\left( t\right)  =f^{\Delta }\left( t\right)
g\left( t\right) +f^{\sigma }\left( t\right) g^{\Delta }\left( t\right)
;$$

\item The function $f/g$ is delta differentiable with
\begin{equation*}
\left( \frac{f}{g}\right) ^{\Delta }\left( t\right) =\frac{f^{\Delta }\left(
t\right) g\left( t\right) -f\left( t\right)
g^{\Delta }\left( t\right)}{g\left( t\right) g^{\sigma }\left( t\right)}
\end{equation*}
provided that $g\left( t\right) g^{\sigma }\left( t\right) \neq 0$.
\end{enumerate}
\end{theorem}

The nabla derivative satisfies the following properties.
\begin{theorem}[\protect cf. \cite{Bohner,Bohner:2}]
Let $f,g:\mathbb{T} \rightarrow \mathbb{R}$
be nabla differentiable functions. Let $t\in \mathbb{T}_{\kappa}$
and $\lambda \in \mathbb{R}$. Then
\begin{enumerate}
\item The function $f+g$ is nabla differentiable with
$$
\left( f+g\right) ^{\nabla }\left( t\right)  =f^{\nabla }\left( t\right)
+g^{\nabla }\left( t\right) ;
$$

\item The function $\lambda f$ is nabla differentiable with
$$
\left( \lambda f\right) ^{\nabla }\left( t\right)  =\lambda f^{\nabla
}\left( t\right) ;
$$

\item The function $fg$ is nabla differentiable with
$$
\left( fg\right) ^{\nabla }\left( t\right) =f^{\nabla }\left( t\right)
g\left( t\right) +f^{\rho }\left( t\right) g^{\nabla }\left( t\right);
$$

\item The function $f/g$ is nabla differentiable with
$$
\left( \frac{f}{g}\right)^{\nabla }\left( t\right) =\frac{f^{\nabla }\left(
t\right) g\left( t\right) -f\left( t\right)
g^{\nabla }\left( t\right) }{g\left( t\right) g^{\rho }\left( t\right) }
$$
provided that $g\left( t\right) g^{\rho }\left( t\right) \neq 0$.
\end{enumerate}
\end{theorem}

In this chapter all the intervals are time scales intervals, that is, for
$a,b \in \mathbb{T}$, with $a<b$,
$$
\left[ a,b\right]_{\mathbb{T}}
$$
denotes the set $\{t\in\mathbb{T}: a \leqslant t \leqslant b\}$;
open intervals and half-open intervals are defined accordingly.

\begin{theorem}[\protect Mean value theorem for the delta derivative
(cf. \cite{Bohner:2})]\index{Mean value theorem for the delta derivative}
Let $f$ be a continuous function on $\left[ a,b\right] _{\mathbb{T}}$ that
is delta differentiable on $\left[ a,b\right[ _{\mathbb{T}}$. Then there
exist $\xi ,\tau \in \left[ a,b\right[ _{\mathbb{T}}$ such that
\begin{equation*}
f^{\Delta }\left( \tau \right) \leqslant \frac{f\left( b\right) -f\left(
a\right) }{b-a}\leqslant f^{\Delta }\left( \xi \right) .
\end{equation*}
\end{theorem}

\begin{corollary}[\protect Rolle's theorem for the delta derivative (cf. \cite{Bohner:2})]
\label{ts:pre valor medio}\index{Rolle's theorem for the delta derivative}
Let $f$ be a continuous function on $\left[ a,b\right] _{\mathbb{T}}$ that
is delta differentiable on $\left[ a,b\right[ _{\mathbb{T}}$ and satisfies
\begin{equation*}
f\left( a\right) =f\left( b\right).
\end{equation*}
Then there exist $\xi ,\tau \in \left[ a,b\right[ _{\mathbb{T}}$ such that
\begin{equation*}
f^{\Delta }\left( \tau \right) \leqslant 0\leqslant f^{\Delta }\left( \xi
\right) .
\end{equation*}
\end{corollary}

Another useful result is the following consequence of the Mean Value Theorem.

\begin{corollary}[cf. \cite{Bohner:2}]
Let $f$ be a continuous function on $\left[ a,b\right] _{\mathbb{T}}$ that
is delta differentiable on $\left[ a,b\right[ _{\mathbb{T}}$. If $f^\Delta >0$,
$f^\Delta <0$, $f^\Delta \geqslant0$ or $f^\Delta \leqslant0$, then $f$
is increasing, decreasing, nondecreasing or nonincreasing
on $\left[a,b\right]_{\mathbb{T}}$, respectively.
\end{corollary}

Note that similar mean values Theorems can be enunciate for the nabla derivative.

As usually expected when we generalize some theory, we can lose some nice properties.
This situation happens with the chain rule on time scale calculus.
The chain rule as we know it in classical calculus is not valid for a general time scale.
For a simple example, let $\mathbb{T}=\mathbb{Z}$ and $f,g:\mathbb{T}\rightarrow\mathbb{R}$
be such that $f\left(t\right)=t^2$ and $g\left(t\right)=3t$. It is simple to verify that
$f^\Delta\left(t\right)=t+\sigma\left(t\right)=2t+1$ and $g^\Delta\left(t\right)=3$, and hence
$$
\left(f\circ g\right)^\Delta\left(t\right)
=18t+9\neq f^\Delta\left(g\left(t\right)\right)
\cdot g^\Delta\left(t\right) =18t+3.
$$
However, there are some special chain rules in the context of time scale calculus.

\begin{theorem}[cf. \cite{Bohner}]
Assume that $\nu:\mathbb{T}\rightarrow\mathbb{R}$ is strictly increasing
and $\tilde{\mathbb{T}}:=\nu\left(\mathbb{T}\right)$ is a time scale.
Let $f:\mathbb{\tilde{T}}\rightarrow\mathbb{R}$. If $\nu^\Delta\left(t\right)$
and $f^{\tilde{\Delta}}\left(\nu\left(t\right)\right)$
exist for $t\in\mathbb{T}^{\kappa}$, then
$$
\left(f\circ\nu\right)^\Delta\left(t\right)
=\left(f^{\tilde{\Delta}}\circ\nu\right)\left(t\right)
\cdot\nu^\Delta\left(t\right)
$$
(where $\tilde{\Delta}$ denotes the delta derivative
with respect to the time scale $\tilde{\mathbb{T}}$).
\end{theorem}
For other special chain rules on time scales
we refer the reader to \cite{Bohner,Bohner:2}.

A function $f:\mathbb{T} \rightarrow \mathbb{R}$
is called \emph{rd-continuous}\index{Rd-continuous function}
if it is continuous at all right-dense points and if
its left-sided limits exist and are finite at all left-dense points. We denote the set of
all rd-continuous functions on $\mathbb{T}$ by $C_{rd}(\mathbb{T},\mathbb{R})$ or simply by $C_{rd}$.
Analogously, a function $f:\mathbb{T} \rightarrow \mathbb{R}$ is called
\emph{ld-continuous}\index{Ld-continuous function} if it is continuous at all left-dense points and
if its right-sided limits exist and are finite at all right-dense points. We denote the set of all
ld-continuous functions on $\mathbb{T}$ by $C_{ld}(\mathbb{T},\mathbb{R})$ or simply by $C_{ld}$.
The following results concerning rd-continuity and ld-continuity are useful.

\begin{theorem}[cf. \cite{Bohner}]
Let  $\mathbb{T}$ be a time scale and $f:\mathbb{T}\rightarrow\mathbb{R}$ a given function.
\begin{enumerate}

\item If $f$ is continuous, then $f$ is rd-continuous and ld-continuous;

\item The forward jump operator, $\sigma$, is rd-continuous
and the backward jump operator, $\rho$, is ld-continuous;

\item If $f$ is rd-continuous, then $f^\sigma$ is also rd-continuous;
if $f$ is ld-continuous, then $f^\rho$ is also ld-continuous;

\item If $\mathbb{T}=\mathbb{R}$, then $f$ is continuous if, and only if,
$f$ is ld-continuous and if, and only if, $f$ is rd-continuous;

\item If $\mathbb{T}=h\mathbb{Z}$ (for some $h>0$),
then $f$ is rd-continuous and ld-continuous.
\end{enumerate}
\end{theorem}

\emph{Delta derivatives of higher-order}\index{Delta derivatives of higher-order}
are defined in the standard way: for $n\in\mathbb{N}$, we define the
$n^{\text{th}}$-delta derivative of $f$ to be the function
$f^{\Delta^n}:\mathbb{T}^{k^n}\rightarrow\mathbb{R}$, defined by
$f^{\Delta^n}=\left(f^{\Delta^{n-1}}\right)^\Delta$ provided $f^{\Delta^{n-1}}$
is delta differentiable on $\mathbb{T}^{k^n}:=\left(\mathbb{T}^{k^{n-1}}\right)^k$.
Analogously, we can define the
\emph{nabla derivatives of higher-order}\index{Nabla derivatives of higher-order}.

The set of all delta differentiable functions with rd-continuous delta derivatives
is denoted by $C^1_{rd}$ or $C^1_{rd}\left(\mathbb{T},\mathbb{R}\right)$. In general,
for a fixed $n\in\mathbb{N}$, we say that $f\in C^n_{rd}$ if an only if $f^\Delta\in C^{n-1}_{rd}$,
where $C^0_{rd}=C_{rd}$. Similarly, for a fixed $n\in\mathbb{N}$, we say that $f\in C^n_{ld}$
if an only if $f^\nabla\in C^{n-1}_{ld}$, where $C^0_{ld}=C_{ld}$

A function $F:\mathbb{T}\rightarrow \mathbb{R}$ is said to be a
\emph{delta antiderivative}\index{Delta antiderivative} of
$f:\mathbb{T}\rightarrow \mathbb{R}$, provided
\begin{equation*}
F^{\Delta }\left( t\right) =f\left( t\right)
\end{equation*}
for all $t\in \mathbb{T}^{\kappa}$. For all $a,b\in\mathbb{T}$, $a<b$,
we define the \emph{delta integral}\index{Delta integral} of $f$
from $a$ to $b$ (or on $\left[a,b\right]_{\mathbb{T}}$) by
\begin{equation*}
\int_{a}^{b}f\left( t\right) \Delta t=F\left( b\right) -F\left( a\right).
\end{equation*}

\begin{theorem}[\protect cf. \cite{Bohner,Bohner:2}]
Every rd-continuous function $f:\mathbb{T}\rightarrow\mathbb{R}$
has a delta antiderivative. In particular,
if $t_{0}\in\mathbb{T}$, then $F$ defined by
$$
F\left(t\right):=\int_{t_{0}}^{t}f\left(\tau\right)\Delta\tau
\ \    \text{for}   \ \ t\in\mathbb{T}
$$
is a delta antiderivative of $f$.
\end{theorem}

The delta integral satisfies the following property:
\begin{equation*}
\int_{t}^{\sigma \left( t\right) }f\left( \tau \right) \Delta \tau =\left(
\sigma \left( t\right) -t\right) f\left( t\right) .
\end{equation*}

A function $G:\mathbb{T} \rightarrow \mathbb{R}$
is said to be a \emph{nabla antiderivative}\index{Nabla antiderivative}
of $g:\mathbb{T} \rightarrow \mathbb{R}$ provided
\begin{equation*}
G^{\nabla }\left( t\right) =g\left( t\right)
\end{equation*}
for all $t\in \mathbb{T}_{{k}}$. For all $a,b\in\mathbb{T}$, $a<b$,
we define the \emph{nabla integral} \index{Nabla integral}of $g$ from
$a$ to $b$ (or on $\left[a,b\right]_{\mathbb{T}}$) by
\begin{equation*}
\int_{a}^{b}g\left( t\right) \nabla t=G\left( b\right) -G\left( a\right).
\end{equation*}

\begin{theorem}[\protect cf. \cite{Bohner,Bohner:2}]
Every ld-continuous function $f:\mathbb{T}\rightarrow\mathbb{R}$
has a nabla antiderivative. In particular,
if $t_{0}\in\mathbb{T}$, then $F$ defined by
$$
F\left(t\right):=\int_{t_{0}}^{t}f\left(\tau\right)\nabla\tau
\ \    \text{for}   \ \ t\in\mathbb{T}
$$
is a nabla antiderivative of $f$.
\end{theorem}

The nabla integral satisfies the following property:
\begin{equation*}
\int_{\rho \left( t\right) }^{t}g\left( \tau \right) \nabla \tau
=\left(t-\rho \left( t\right) \right) g\left( t\right) .
\end{equation*}

Next, we review the properties of the delta integral.

\begin{theorem}[\protect cf. \cite{Bohner,Bohner:2}]
\label{ts:props integral delta}
Let $f,g:\mathbb{T} \rightarrow \mathbb{R}$ be delta integrable on
$\left[ a,b\right]_{\mathbb{T}}$.
Let $c\in \left[ a,b\right]_{\mathbb{T}}$
and $\lambda \in \mathbb{R}$. Then
\begin{enumerate}
\item $\displaystyle \int_{a}^{a}f\left( t\right) \Delta t=0$;

\item $\displaystyle \int_{a}^{b}f\left( t\right) \Delta
t=\int_{a}^{c}f\left( t\right) \Delta t+\int_{c}^{b}f\left( t\right)
\Delta t $;

\item $\displaystyle \int_{a}^{b}f\left( t\right) \Delta t
=-\displaystyle \int_{b}^{a}f\left( t\right) \Delta t$;

\item $f+g$ is delta integrable on $\left[ a,b\right]_{\mathbb{T}} $ and
\begin{equation*}
\int_{a}^{b}\left( f+g\right) \left( t\right) \Delta t=\int_{a}^{b}f\left(
t\right) \Delta t+\int_{a}^{b}g\left( t\right) \Delta t\text{;}
\end{equation*}

\item $\lambda f$ is delta integrable on $\left[ a,b\right]_{\mathbb{T}} $ and
\begin{equation*}
\int_{a}^{b}\lambda f\left( t\right) \Delta t=\lambda \int_{a}^{b}f\left(
t\right) \Delta t\text{;}
\end{equation*}

\item $fg$ is delta integrable on $\left[ a,b\right]_{\mathbb{T}} $;

\item For $p>0$, $|f|^{p}$ is delta integrable on $\left[ a,b\right]_{\mathbb{T}}$;

\item  If $f$ and $g$ are delta differentiable, then
$$
\displaystyle \int_{a}^{b}f^{\sigma}\left( t\right) g^{\Delta}\left(t\right) \Delta t
=f\left(t\right)g\left(t\right)\bigg{|}^{b}_{a}
-\int_{a}^{b}f^{\Delta}\left( t\right) g\left(t\right) \Delta t;
$$

\item If $f$ and $g$ are delta differentiable, then
$$
\displaystyle \int_{a}^{b}f\left( t\right) g^{\Delta}\left(t\right) \Delta t
=f\left(t\right)g\left(t\right)\bigg{|}^{b}_{a}
-\int_{a}^{b}f^{\Delta}\left( t\right) g^{\sigma}\left(t\right) \Delta t;
$$

\item If $f\left( t\right) \geqslant 0$
on $\left[ a,b\right]_{\mathbb{T}}$, then
\begin{equation*}
\int_{a}^{b}f\left( t\right) \Delta t\geqslant 0;
\end{equation*}

\item If $\left \vert f\left( t\right) \right \vert \leqslant
g\left(t\right) $ on $\left[ a,b\right]_{\mathbb{T}} $, then
\begin{equation*}
\left \vert \int_{a}^{b}f\left( t\right) \Delta t\right \vert
\leqslant \int_{a}^{b}g\left( t\right) \Delta t.
\end{equation*}
\end{enumerate}
\end{theorem}

The formulas 8. and 9. in Theorem~\ref{ts:props integral delta}
are called \emph{integration by parts}\index{Delta integration by parts} formulas.
Analogously, the nabla integral satisfies the corresponding properties.

\begin{remark}
\begin{enumerate}
\item If $\mathbb{T}=\mathbb{R}$, then
\begin{equation*}
\int_{a}^{b}f\left( t\right) \Delta t
= \int_{a}^{b}f\left( t\right) \nabla t
=\int_{a}^{b}f\left( t\right) dt
\end{equation*} where the last integral is the usual Riemman integral;

\item If $\mathbb{T}=h\mathbb{Z}$ for some $h>0$, and $a,b\in\mathbb{T}$, $a<b$, then
\begin{equation*}
\int_{a}^{b}f\left( t\right) \Delta t
=\sum_{k=\frac{a}{h}}^{\frac{b}{h}-1}hf\left( kh\right)
\end{equation*}
and
\begin{equation*}
\int_{a}^{b}f\left( t\right) \nabla t
=\sum_{k=\frac{a}{h}+1}^{\frac{b}{h}}hf\left( kh\right);
\end{equation*}

\item If $a<b$ and $\left[ a,b\right]_{\mathbb{T}}$
consists of only isolated points, then
\begin{equation*}
\int_{a}^{b}f\left( t\right) \Delta t=\sum_{t\in \left[ a,b\right[_{\mathbb{T}}}
\mu \left( t\right)  f\left( t\right)
\end{equation*}
and
\begin{equation*}
\int_{a}^{b}f\left( t\right) \nabla t
=\sum_{t\in \left] a,b\right]_{\mathbb{T}}}
\nu \left( t\right)  f\left( t\right) .
\end{equation*}
\end{enumerate}
\end{remark}

The calculus on time scales using the delta derivative and the delta integral
is usually known by delta-calculus; the calculus done with the nabla derivative
and the nabla integral is known by nabla-calculus. The first developments on the
time scale theory was done essentially using the delta-calculus. However,
for some applications, in particular to solve problems of the Calculus of Variations
and Control Theory in economics \cite{Attici}, is often more convenient
to work backwards in time, that is, using the nabla-calculus.

Recently, Caputo provided a duality technique \cite{Caputo} which allows to obtain
nabla results on time scales from the delta theory and vice versa
(see also \cite{Girejko,Malinowska45,Martins15}).

Another approach to the theory of time scale is the diamond-$\alpha$ calculus,
that uses the notion of the diamond-$\alpha$ derivative. To introduce this notion
(as defined in \cite{Sheng:2,Sheng:3,Sheng}) let $t,s\in \mathbb{T}$ and define
$\mathbb{T}_k^k:=\mathbb{T}_k\cap\mathbb{T}^k$, $\mu _{ts}:=\sigma \left( t\right)
-s$ and $\eta _{ts}:=\rho \left( t\right) -s$. We say that a function
$f:\mathbb{T}\rightarrow\mathbb{R}$ is \emph{diamond-$\alpha$
differentiable}\index{Diamond-$\alpha $ derivative}
on $t\in\mathbb{T}_{\kappa}^{\kappa}$ if there exists a number
$f^{\diamondsuit _{\alpha }}\left(t\right)$
such that, for all $\varepsilon >0$, there exists
a neighborhood $U$ of $t$ such that, for all $s\in U$,
\begin{equation*}
\left \vert \alpha \left[ f^{\sigma }\left( t\right) -f\left( s\right) \right]
\eta _{ts}+\left( 1-\alpha \right) \left[ f^{\rho }\left( t\right) -f\left(
s\right) \right] \mu _{ts}-f^{\diamondsuit_{\alpha }}\left(t \right)
\mu_{ts}\eta_{ts}\right \vert \leqslant \varepsilon \left \vert \mu _{ts}\eta _{ts}\right \vert .
\end{equation*}
A function $f$ is said to be diamond-$\alpha $ differentiable provided
$f^{\diamondsuit _{\alpha }}\left(t\right)$ exists for all
$t\in \mathbb{T}_{\kappa}^{\kappa}$.

\begin{theorem}[\cite{Sheng:2}]
Let $0\leqslant \alpha \leqslant 1$ and let $f$ be both nabla and delta
differentiable on $t\in \mathbb{T}_{\kappa}^{\kappa}$. Then $f$ is diamond-$\alpha$
differentiable at $t$ and
\begin{equation}
f^{\diamondsuit _{\alpha }}\left( t\right) =\alpha f^{\Delta }\left(
t\right) +\left( 1-\alpha \right) f^{\nabla }\left( t\right) \text{.}
\label{ts:2}
\end{equation}
\end{theorem}

\begin{remark}
If $\alpha =1$, then the diamond-$\alpha $ derivative reduces to
the delta derivative; if $\alpha=0$, the diamond-$\alpha$
derivative coincides with the nabla derivative.
\end{remark}

Note that equality \eqref{ts:2} is given as the definition
of the diamond-$\alpha$ derivative in \cite{Sheng}.

\begin{theorem}[cf. \protect \cite{Sheng}]
Let $f,g:\mathbb{T}\rightarrow \mathbb{R}$
be diamond-$\alpha $ differentiable functions.
Let $t\in \mathbb{T}_{\kappa}^{\kappa}$ and
$\lambda \in \mathbb{R}$. Then

\begin{enumerate}
\item The function $f+g$ is diamond-$\alpha $ differentiable with
\begin{equation*}
\left( f+g\right) ^{\diamondsuit _{\alpha }}\left( t\right)
=f^{\diamondsuit_{\alpha }}\left( t\right)
+ g^{\diamondsuit _{\alpha }}\left( t\right) ;
\end{equation*}

\item The function $\lambda f$ is diamond-$\alpha$ differentiable with
\begin{equation*}
\left( \lambda f\right) ^{\diamondsuit _{\alpha }}\left( t\right)
=\lambda f^{\diamondsuit _{\alpha }}\left( t\right) ;
\end{equation*}

\item The function $fg$ is diamond-$\alpha$ differentiable with
\begin{equation*}
\left( fg\right) ^{\diamondsuit _{\alpha }}\left( t\right)
=f^{\diamondsuit_{\alpha }}\left( t\right) g\left( t\right) +\alpha f^{\sigma }\left(
t\right) g^{\Delta }\left( t\right) +\left( 1-\alpha \right) f^{\rho }\left(
t\right) g^{\nabla }\left( t\right) .
\end{equation*}
\end{enumerate}
\end{theorem}

Let $a,b\in \mathbb{T}$, $a<b$, $h:\mathbb{T} \rightarrow \mathbb{R}$
and $\alpha \in \left[ 0,1 \right]$. The
\emph{diamond-$\alpha$ integral}\index{Diamond-$\alpha$ integral}
of $h$ from $a$ to $b$ (or on $[a,b]_\mathbb{T}$) is defined by
\begin{equation*}
\int_{a}^{b}h\left( t\right) \diamondsuit _{\alpha }t=\alpha
\int_{a}^{b}h\left( t\right) \Delta t+\left( 1-\alpha \right)
\int_{a}^{b}h\left( t\right) \nabla t\text{, \ }
\end{equation*}
provided that $h$ is delta integrable
and nabla integrable on $\left[a,b\right]_\mathbb{T}$.

\begin{theorem}[\protect cf. \cite{Malinowska:2}]\label{ts:props integral diamond}
Let $f,g:\mathbb{T} \rightarrow \mathbb{R}$ be diamond-$\alpha$ integrable on
$\left[ a,b\right]_{\mathbb{T}}$. Let $c\in \left[ a,b\right]_{\mathbb{T}}$
and $\lambda \in \mathbb{R}$. Then
\begin{enumerate}
\item $\displaystyle \int_{a}^{a}f\left( t\right) \diamondsuit _{\alpha } t=0$;

\item $\displaystyle \int_{a}^{b}f\left( t\right) \diamondsuit _{\alpha } t
=\int_{a}^{c}f\left( t\right) \diamondsuit_{\alpha } t
+\int_{c}^{b}f\left( t\right) \diamondsuit _{\alpha } t$;

\item $\displaystyle \int_{a}^{b}f\left( t\right)\diamondsuit _{\alpha } t
=-\displaystyle \int_{b}^{a}f\left( t\right) \diamondsuit _{\alpha } t$;

\item $f+g$ is diamond-$\alpha$ integrable on $\left[ a,b\right]_{\mathbb{T}} $ and
\begin{equation*}
\int_{a}^{b}\left( f+g\right) \left( t\right) \diamondsuit _{\alpha } t=\int_{a}^{b}f\left(
t\right) \diamondsuit _{\alpha } t+\int_{a}^{b}g\left( t\right) \diamondsuit _{\alpha } t\text{;}
\end{equation*}

\item $\lambda f$ is diamond-$\alpha$ integrable on $\left[ a,b\right]_{\mathbb{T}}$ and
\begin{equation*}
\int_{a}^{b}\lambda f\left( t\right) \diamondsuit _{\alpha } t
=\lambda \int_{a}^{b}f\left(t\right) \diamondsuit _{\alpha } t\text{;}
\end{equation*}

\item $fg$ is diamond-$\alpha$ integrable on $\left[ a,b\right]_{\mathbb{T}} $;

\item For $p>0$, $|f|^{p}$ is diamond-$\alpha$ integrable on $\left[ a,b\right]_{\mathbb{T}} $;

\item If $f\left( t\right) \geqslant 0$ on $\left[ a,b\right]_{\mathbb{T}}$, then
\begin{equation*}
\int_{a}^{b}f\left( t\right) \diamondsuit _{\alpha } t\geqslant 0;
\end{equation*}

\item If $\left \vert f\left( t\right) \right \vert \leqslant g\left(
t\right) $ on $\left[ a,b\right]_{\mathbb{T}} $, then
\begin{equation*}
\left \vert \int_{a}^{b}f\left( t\right) \diamondsuit _{\alpha } t\right \vert
\leqslant \int_{a}^{b}g\left( t\right) \diamondsuit_{\alpha } t.
\end{equation*}
\end{enumerate}
\end{theorem}

\begin{remark}
In \cite{Ammi}, Ammi et al. show that, in general, we do not have
$$
\left(s\rightarrow \int_{a}^{s}f\left( \tau\right)
\diamondsuit_{\alpha } \tau \right)^{\diamondsuit _{\alpha }}\left(t\right)
= f\left( t\right), \ \ \ t\in\mathbb{T}.
$$
Hence, we do not have an integral by parts formula for the diamond-$\alpha$ integral
and this is a great limitation for the development of the Calculus of Variations
for problems involving diamond-$\alpha$ integrals.
\end{remark}

\begin{remark}
\begin{enumerate}
\item If $\mathbb{T}=\mathbb{R}$, then a bounded function $f$ on $[a,b]$
is diamond-$\alpha$ integrable on $\left[a,b\right]$
if, and only if, is Riemann integrable on $[a,b]$, and in this case
\begin{equation*}
\int_{a}^{b}f\left( t\right) \diamondsuit _{\alpha } t=\int_{a}^{b}f\left( t\right) dt;
\end{equation*}

\item If $\mathbb{T}=\mathbb{Z}$ and $a,b \in\mathbb{T}$, $a<b$, then $f:\mathbb{T}\rightarrow\mathbb{R}$
is diamond-$\alpha$ integrable on $[a,b]_\mathbb{T}$ and
\begin{equation*}
\int_{a}^{b}f\left( t\right) \diamondsuit_{\alpha } t
=\alpha f\left(a\right) + \left(1-\alpha\right)f\left(b\right)
+\sum_{t=a+1}^{b-1}f\left( t\right).
\end{equation*}
\end{enumerate}
\end{remark}


\clearpage{\thispagestyle{empty}\cleardoublepage}


\chapter{Quantum Calculus}
\label{Quantum Calculus}

Quantum calculus is usually known as ``calculus without limits''.
There are several types of quantum calculus. In this thesis
we are concerned with $h$-calculus,
$q$--calculus and Hahn's calculus. In each type of quantum calculus,
we can make a study towards the future, known as the forward quantum
calculus, or towards the past, the backward quantum calculus.
Moreover, there are different approaches for each type of quantum calculus.
Some authors choose the set of study to be a subset of the real numbers containing isolated points,
 others use subintervals of $\mathbb{R}$.

In this chapter we review some definitions
and basic results about the quantum calculus.


\section{The $h$-calculus}

The $h$-calculus is also known as the calculus of finite
differences and Boole at \cite{Boole:old} described it as:

\begin{quote}
``The calculus of finite differences may be strictly defined as the science which is occupied about the
ratios of the simultaneous increments of quantities mutually dependent. The differential calculus is
occupied about the limits of which such ratios approach as the increments are indefinitely diminished.''
\end{quote}

Many authors contributed to the calculus of finite differences like Boole \cite{Boole:old,Boole},
Milne-Thomson \cite{Milne}, N\"{o}rlund \cite{Norlund}, just to name a few.
In this section we present the $h$-calculus as Kac and Cheung do in their book \cite{Kac}.

We consider that the set of study is
\begin{equation*}
h\mathbb{Z}:=\left\{ hz:z\in
\mathbb{Z}
\right\}
\end{equation*}
for some $h>0$.

\begin{definition}[cf. \cite{Kac}]
Let $f:h
\mathbb{Z}
\rightarrow
\mathbb{R}
$ be a function and let $t\in h
\mathbb{Z}
$. The \emph{$h$-derivative}\index{$h$-derivative} of $f$ or the
\emph{forward difference operator}\index{Forward difference operator}
of $f$ at $t$ is given by
\begin{equation*}
\Delta _{h}\left[ f\right] \left( t\right)
:=\frac{f\left( t+h\right) -f\left(t\right)}{h}.
\end{equation*}
\end{definition}

Note that, if a function $f$ is differentiable (in the classical sense) at $t$, then
\begin{equation*}
\lim_{h\rightarrow 0}\Delta _{h}\left[ f\right] \left( t\right)
=f^{\prime }\left( t\right) ,
\end{equation*}
where $f^{\prime }$ is the usual derivative.

\begin{example}
Let $f\left( t\right) =t^{n}$. Then
\begin{equation*}
\Delta _{h}\left[ f\right] \left( t\right) =nt^{n-1}+\frac{n\left( n-1\right) }{2}
t^{n-2}h+\ldots +h^{n-1}.
\end{equation*}
\end{example}

The $h$-derivative has the following properties.

\begin{theorem}[cf. \cite{Kac}]
Let $f,g:h\mathbb{Z}\rightarrow\mathbb{R}$ be functions, $t\in h\mathbb{Z}$ and $\alpha \in
\mathbb{R}.$ Then

\begin{enumerate}
\item $\displaystyle \Delta _{h}\left[ f+g\right] \left( t\right) =\Delta _{h}\left[ f
\right] \left( t\right) +\Delta _{h}\left[ g\right] \left( t\right) $;

\item $\displaystyle \Delta _{h}\left[ \alpha f\right] \left( t\right) =\alpha
\Delta _{h}\left[ f\right] \left( t\right) $;

\item $\displaystyle \Delta _{h}\left[ fg\right] \left( t\right) =f\left( t\right)
\Delta _{h}\left[ g\right] \left( t\right) +\Delta _{h}\left[ f\right] \left( t\right)
g\left( t+h\right) $;

\item $\displaystyle \Delta _{h}\left[ \frac{f}{g}\right] \left( t\right)
=\frac{\Delta _{h}\left[ f\right] \left( t\right) g\left( t\right) -f\left( t\right)
\Delta_{h}\left[ g\right] \left( t\right)}{g\left( t\right) g\left( t+h\right) }$.
\end{enumerate}
\end{theorem}

\begin{definition}[cf. \cite{Kac}]
A function $F:h
\mathbb{Z}
\rightarrow
\mathbb{R}
$ is said to be an \emph{$h$-antiderivative}\index{$h$-antiderivative} of $f:h
\mathbb{Z}
\rightarrow
\mathbb{R}
$ provided
\begin{equation*}
\Delta _{h}\left[ F\right] \left( t\right) =f\left( t\right) ,
\end{equation*}
for all $t\in h \mathbb{Z}$.
\end{definition}

\begin{definition}[cf. \cite{Kac}]
\label{h-integral}
Let $f:h
\mathbb{Z}
\rightarrow
\mathbb{R}
$ be a function. If $a,b\in h
\mathbb{Z}
$, we define the \emph{$h$-integral}\index{$h$-integral} of $f$ from $a$ to $b$ by
\begin{equation*}
\int_{a}^{b}f\left( t\right) \Delta _{h}t=\left\{
\begin{array}{lcc}
h\bigg{[}f\left( a\right) +f\left( a+h\right)
+\ldots +f\left( b-h\right)\bigg{]} & \text{ \ if \ } & a<b \\
&  &  \\
0 & \text{ \ if \ } & a=b \\
&  &  \\
-h\bigg{[}f\left( b\right) +f\left( b+h\right) +\ldots +f\left( a-h\right)
\bigg{]} & \text{ \ if \ } & a>b.
\end{array}
\right.
\end{equation*}
\end{definition}

\begin{theorem}[Fundamental theorem of the $h$-integral calculus (cf. \cite{Kac})]
Let $F:h
\mathbb{Z}
\rightarrow
\mathbb{R}
$ be an $h$-antiderivative of $f:h
\mathbb{Z}
\rightarrow
\mathbb{R}
$. If $a,b\in h
\mathbb{Z}
$, then
\begin{equation*}
\int_{a}^{b}f\left( t\right) \Delta _{h}t=F\left( b\right) -F\left( a\right) .
\end{equation*}
\end{theorem}

The $h$-integral has the following properties.

\begin{theorem}[cf. \cite{Kac}]
Let $f,g:h
\mathbb{Z}
\mathbb{\rightarrow \mathbb{R}}$ be\ functions.
Let $a,b,c\in h \mathbb{Z}$ and $\alpha \in \mathbb{R}$. Then
\begin{enumerate}
\item $\displaystyle\int_{a}^{b}f\left( t\right) \Delta _{h}t
=\int_{a}^{c}f\left(t\right) \Delta_{h}t+\int_{c}^{b}f\left( t\right) \Delta _{h}t$;

\item $\displaystyle\int_{a}^{b}f\left( t\right) \Delta _{h}t
=-\displaystyle \int_{b}^{a}f\left( t\right) \Delta _{h}t$;

\item $\displaystyle\int_{a}^{b}\left( f+g\right) \left( t\right)
\Delta _{h}t=\int_{a}^{b}f\left( t\right) \Delta _{h}t+\int_{a}^{b}g\left( t\right)
\Delta _{h}t$;

\item $\displaystyle\int_{a}^{b}\alpha f\left( t\right) \Delta _{h}t=\alpha
\int_{a}^{b}f\left( t\right) \Delta _{h}t$;

\item $
\displaystyle\int_{a}^{b}f\left( t\right) \Delta _{h}\left[ g\right] \left( t\right)
d_{h}t=f\left( t\right) g\left( t\right) \bigg{|}_{a}^{b}-\int_{a}^{b}D_h\left[
f\right] \left( t\right) g\left( t+h\right) \Delta _{h}t.
$
\end{enumerate}
\end{theorem}

We can also define a backward $h$-calculus where the
\emph{backward $h$-derivative}\index{Backward $h$-derivative},
or the \emph{backward difference operator}\index{Backward difference operator},
is defined by the following quotient
\begin{equation*}
\nabla_h\left[f\right](t):=\frac{f\left( t\right) -f\left( t-h\right) }{h}.
\end{equation*}
The results of the backward $h$-calculus are similar to the forward $h$-calculus.

It is worth noting that the $h$-calculus (as presented by Kac and Cheung \cite{Kac})
can be seen as a particular case of the time scale calculus.


\section{The $q$-calculus}

The $q$-derivative, like the $h$-derivatve,
is a discretization of the classical derivative and therefore,
has immediate applications in numerical analysis.
However, and according to Ernst \cite{Ernst}, is also a generalization
of many subjects, like hypergeometric series, complex analysis and particle physics.
The $q$-difference operator and its inverse operator, the Jackson $q$-integral, were first defined by
Jackson \cite{Jackson:old,Jackson} and due to its applications the $q$-calculus is a popular subject today.
In this section we present the $q$-calculus as Kac and Cheung do in their book \cite{Kac}.

Let $q \in \left]0,1\right[$ and let $I$ be a real interval containing $0$.

\begin{definition}[cf. \cite{Kac}]
Let $f:I\rightarrow \mathbb{R}$ be a function and let $t\in I$.
The \emph{$q$-derivative}\index{$q$-derivative},
or \emph{Jackson's difference operator}\index{Jackson's difference operator},
of $f$ at $t$ is given by
\begin{equation*}
D_{q}\left[ f\right] \left( t\right) :=\frac{f\left( qt\right) -f\left(
t\right) }{\left( q-1\right) t},\text{ if }t\neq 0\,,
\end{equation*}
and $D_{q}\left[ f\right] \left( 0\right) :=f^{\prime }\left( 0\right)$,
provided $f$ is differentiable at $0$.
\end{definition}

Note that, if a function is $f$ differentiable (in the classical sense) at $t$, then
\begin{equation*}
\lim_{q\rightarrow 1}D_{q}\left[ f\right] \left( t\right)
=\lim_{q\rightarrow 1}\frac{f\left( qt\right) -f\left( t\right) }{\left(
q-1\right) t}=f^{\prime }\left( t\right) ,
\end{equation*}
where $f^{\prime }$ is the usual derivative.

\begin{example}
Let $f\left( t\right) =t^{n}$. Then
\begin{equation*}
D_{q}\left[ f\right] \left( t\right) =\frac{q^{n}-1}{q-1}t^{n-1}=\left(
q^{n-1}+\ldots +1\right) t^{n-1}.
\end{equation*}
\end{example}

The $q$-derivative has the following properties.

\begin{theorem}[cf. \cite{Kac}]
Let $f,g:I\rightarrow
\mathbb{R}
$ be functions and $\alpha \in
\mathbb{R}
.$ Then

\begin{enumerate}
\item $\displaystyle D_{q}\left[ f+g\right] \left( t\right)
=D_{q}\left[ f\right]\left( t\right) + D_{q}\left[ g\right] \left( t\right) $;

\item $\displaystyle D_{q}\left[ \alpha f\right] \left( t\right)
=\alpha D_{q}\left[ f\right] \left( t\right) $;

\item $\displaystyle D_{q}\left[ fg\right] \left( t\right) =f\left( t\right)
D_{q}\left[ g\right] \left( t\right) +D_{q}\left[ f\right] \left( t\right)
g\left( qt\right) $;

\item $\displaystyle D_{q}\left[ \frac{f}{g}\right] \left( t\right)
=\frac{D_{q}\left[ f\right] \left( t\right) g\left( t\right)
-f\left( t\right) D_{q}\left[ g\right] \left( t\right)}{g\left( t\right) g\left( qt\right) }$.
\end{enumerate}
\end{theorem}

\begin{definition}[cf. \cite{Kac}]
A function $F:I\rightarrow \mathbb{R}$ is said to be a
\emph{$q$-antiderivative}\index{$q$-antiderivative} of $f:I\rightarrow \mathbb{R}$ provided
\begin{equation*}
D_{q}\left[ F\right] \left( x\right) =f\left( x\right) ,
\end{equation*}
for all $t\in I$.
\end{definition}

\begin{definition}[cf. \cite{Kac}]
Let $a,b\in I$ and $a<b$. For $f:I\rightarrow \mathbb{R}$ the \emph{$q$-integral}\index{$q$-integral},
 or \emph{Jackson integral}\index{Jackson integral}, of $f$ from $a$ to $b$ is given by
\begin{equation*}
\int_{a}^{b}f\left( t\right) d_{q}t:=\int_{0}^{b}f\left( t\right)
d_{q}t-\int_{0}^{a}f\left( t\right) d_{q}\text{,}
\end{equation*}
where
\begin{equation*}
\int_{0}^{x}f\left( t\right) d_{q }t:=\left( 1-q\right)
x\sum_{k=0}^{+\infty }q^{k}f\left( xq^{k}\right) \text{, }x\in I\,,
\end{equation*}
provided that the series converges at $x=a$ and $x=b$. In that case, $f$ is
called $q$-\emph{integrable on} $\left[ a,b\right] $. We say that $f$ is
$q$-\emph{integrable over} $I$ if it is $q $-\emph{integrable} over
$[a,b] $ for all $a,b\in I$.
\end{definition}

\begin{theorem}[Fundamental theorem of the $q$-integral calculus (cf. \cite{Kac})]
Let $F:I\rightarrow
\mathbb{R}
$ be an anti-derivative of $f:I\rightarrow
\mathbb{R}
$ and let $F$ be continuous at $0$. For $a,b\in I$ we have
\begin{equation*}
\int_{a}^{b}f\left( t\right) d_{q}t=F\left( b\right) -F\left( a\right) .
\end{equation*}
\end{theorem}

The $q$-integral has the following properties.

\begin{theorem}[cf. \cite{Kac}]
Let $f,g:I\mathbb{\rightarrow \mathbb{R}}$ be\ functions. Let $a,b,c\in I$
and $\alpha \in
\mathbb{R}
$. Then

\begin{enumerate}
\item $\displaystyle\int_{a}^{b}f\left( t\right) d_{q}t=\int_{a}^{c}f\left(
t\right) d_{q}t+\int_{c}^{b}f\left( t\right) d_{q}t$;

\item $\displaystyle\int_{a}^{b}f\left( t\right) d_{q}t
=-\displaystyle \int_{b}^{a}f\left( t\right) d_{q}t$;

\item $\displaystyle\int_{a}^{b}\left( f+g\right) \left( t\right)
d_{q}t=\int_{a}^{b}f\left( t\right) d_{q}t+\int_{a}^{b}g\left( t\right)
d_{q}t$;

\item $\displaystyle\int_{a}^{b}\alpha f\left( t\right) d_{q}t=\alpha
\int_{a}^{b}f\left( t\right) d_{q}t$;

\item $
\displaystyle\int_{a}^{b}f\left( t\right) D_q\left[ g\right] \left( t\right)
d_{h}t=f\left( t\right) g\left( t\right) \bigg{|}_{a}^{b}-\int_{a}^{b}D_q\left[
f\right] \left( t\right) g\left( qt\right) d_{h}t.
$
\end{enumerate}
\end{theorem}

The \emph{backward $q$-derivative}\index{Backward $q$-derivative}
of $f:I\rightarrow\mathbb{R}$ at $t\neq 0$ is defined by the quotient
\begin{equation*}
\frac{f\left( t\right) -f\left(q^{-1}
t\right) }{\left( 1-q^{-1}\right) t}
\end{equation*}
and from it one can obtain the backward $q$-quantum calculus.
Note that the $q$-calculus as defined in this section is not a particular case
of the time scale calculus. Namely, the $q$-integral does not coincide with
the delta integral for $\mathbb{T}=\overline{q^{\mathbb{Z}}}$.


\section{The Hahn calculus}
\label{sec:hoprelim}

Hahn, in \cite{Hahn}, introduced his difference operator
as a tool for constructing families of orthogonal polynomials.
Hahn's quantum difference operator unifies (in the limit) Jackson's $q$-difference operator
and the forward difference operator. Recently, Aldwoah \cite{Aldwoah} defined a proper inverse operator
of Hahn's difference operator, and the associated integral calculus was developed
in \cite{Aldwoah,Aldwoah:2,Annaby}. In this section we present the Hahn calculus
as Aldwoah did in his Ph.D. thesis \cite{Aldwoah}.

Let $q\in\left]0,1\right[$ and $\omega\geqslant0$. We introduce the real number
\index{$\omega_{0}$}
\begin{equation*}
\omega_{0} := \frac{\omega}{1-q}.
\end{equation*}
Let $I$ be a real interval containing $\omega_{0}$. For a function $f$
defined on $I$, the \emph{Hahn difference operator}\index{Hahn's difference operator}
of $f$ is given by
\begin{equation*}
D_{q,\omega}\left[f\right]\left(t\right)
:= \frac{f\left( qt+\omega\right) - f\left(t\right)}{\left(q-1\right)t + \omega}
\, , \text{ if } t\neq\omega_{0}\, ,
\end{equation*}
and $D_{q,\omega}\left[f\right]\left(\omega_{0}\right)
:= f^{\prime}\left(\omega_{0}\right)$, provided $f$ is differentiable at
$\omega_{0}$. We usually call $D_{q,\omega}\left[f\right]$ the
$q,\omega$-\emph{derivative of}\index{$q,\omega$-derivative} $f$,
and $f$ is said to be $q,\omega$-\emph{differentiable on} $I$ if
$D_{q,\omega}\left[f\right]\left(\omega_{0}\right)$ exists.

\begin{remark}
The $D_{q,\omega }$ operator generalizes (in the limit) the well known forward
difference and the Jackson difference operators \cite{Ernst,Kac}.
Indeed, when $q\rightarrow 1$ we obtain the
\emph{forward difference operator}\index{Forward difference operator}
\begin{equation*}
\Delta _{\omega}\left[ f\right] \left( t\right)
:= \frac{f\left( t+\omega\right) -f\left( t\right) }{\omega};
\end{equation*}
when $\omega =0$ we obtain the
\emph{Jackson difference operator}\index{Jackson $q$-difference operator}
\begin{equation*}
D_{q}\left[ f\right] \left( t\right)
:= \frac{f\left( qt\right) -f\left( t\right) }{\left( q-1\right) t}
\,,\text{ if }t\neq 0\,,
\end{equation*}
and $D_{q}[f]\left( 0\right) =f^{\prime }\left( 0\right)$ provided
$f^{\prime }\left( 0\right)$ exists. Notice also that,
under appropriate conditions,
\begin{equation*}
\lim_{q\rightarrow 1}D_{q,0}\left[ f\right]
\left( t\right) =f^{\prime }\left( t\right) .
\end{equation*}
\end{remark}

Hahn's difference operator has the following properties.

\begin{theorem}[\protect\cite{Aldwoah,Aldwoah:2,Annaby}]
Let $f$ and $g$ be $q,\omega$-differentiable on $I$ and $t\in I$. One has,
\begin{enumerate}
\item $D_{q,\omega}[f](t) \equiv 0$ on $I$ if, and only if, $f$ is constant;

\item $D_{q,\omega}\left[ f+g\right] \left( t\right) =D_{q,\omega}\left[
f\right] \left( t\right) +D_{q,\omega}\left[ g\right] \left( t\right)$;

\item $D_{q,\omega}\left[ fg\right] \left( t\right) =D_{q,\omega}\left[
f \right] \left( t\right) g\left( t\right) +f\left( qt+\omega\right)
D_{q,\omega}\left[ g\right] \left( t\right)$;

\item $\displaystyle D_{q,\omega}\left[ \frac{f}{g}\right] \left( t\right)
= \frac{D_{q,\omega}\left[ f\right] \left( t\right) g\left( t\right) -f\left(
t\right) D_{q,\omega}\left[ g\right] \left( t\right) }{g\left( t\right)
g\left( qt+\omega\right) }$ if $g\left( t\right) g\left( qt+\omega\right)
\neq0$;

\item $f\left( qt+\omega \right) =f\left( t\right) +\left( t\left(
q-1\right) +\omega \right) D_{q,\omega }\left[ f\right] \left( t\right) $.
\end{enumerate}
\end{theorem}

For $k\in\mathbb{N}_{0}:=\mathbb{N}\cup\left\{ 0\right\}$ define
$\displaystyle\left[ k\right] _{q}:=\frac{1-q^{k}}{1-q}$ and let
$\sigma\left(t\right) :=qt+\omega$, $t\in I$. Note that $\sigma$ is a
contraction, $\sigma(I)\subseteq I$, $\sigma\left(t\right)<t$ for
$t>\omega_{0}$, $\sigma\left( t\right)>t$ for $t<\omega_{0}$, and
$\sigma\left( \omega_{0}\right) =\omega_{0}$.
The following technical result is used several times in this thesis.

\begin{lemma}[\protect\cite{Aldwoah,Annaby}]
Let $k\in\mathbb{N}$ and $t\in I$. Then,

\begin{enumerate}
\item $\sigma^{k}\left( t\right) =\underset{k\text{-times}}{\underbrace{\sigma
\circ\sigma\circ \cdots \circ\sigma}}\left( t\right) =q^{k}t
+\omega\left[ k\right]_{q}$;

\item $\displaystyle \left( \sigma^{k}\left( t\right) \right)^{-1}
=\sigma^{-k}\left(t\right)=\frac{t-\omega\left[ k\right] _{q}}{q^{k}}$.
\end{enumerate}
\end{lemma}

Following \cite{Aldwoah,Aldwoah:2,Annaby} we define the notion of
$q,\omega$-\emph{integral}\index{$q,\omega$-integral}
(also known as the \emph{Jackson--N\"{o}rlund integral}
\index{Jackson--N\"{o}rlund integral}) as follows.

\begin{definition}
Let $a,b\in I$ and $a<b$. For $f:I\rightarrow\mathbb{R}$
the $q,\omega$-\emph{integral of} $f$ from $a$ to $b$ is given by
\begin{equation*}
\int_{a}^{b}f\left( t\right) d_{q,\omega}t:=\int_{\omega_{0}}^{b}f\left(
t\right) d_{q,\omega}t-\int_{\omega_{0}}^{a}f\left( t\right) d_{q,\omega }t ,
\end{equation*}
where
\begin{equation*}
\int_{\omega_{0}}^{x}f\left( t\right) d_{q,\omega}t:=\left( x\left(
1-q\right) -\omega\right) \sum_{k=0}^{+\infty}q^{k}f\left(q^{k}x +\omega
\left[ k\right] _{q}\right) \text{, }x\in I\, ,
\end{equation*}
provided that the series converges at $x=a$ and $x=b$. In that case, $f$ is
called $q,\omega$-\emph{integrable on} $\left[a,b\right]$. We say that $f$
is $q,\omega$-\emph{integrable over} $I$ if it is $q,\omega$-\emph{integrable}
over $[a,b]$ for all $a,b\in I$.
\end{definition}

\begin{remark}
The $q,\omega$-\emph{integral} generalizes (in the limit)
the Jackson $q$-integral and the N\"{o}rlund sum (cf. \cite{Kac}).
When $\omega=0$, we obtain the \emph{Jackson $q$-integral}\index{Jackson $q$-integral}
\begin{equation*}
\int_{a}^{b}f\left( t\right) d_{q}t:=\int_{0}^{b}f\left( t\right)
d_{q}t-\int_{0}^{a}f\left( t\right) d_{q}t ,
\end{equation*}
where
\begin{equation*}
\int_{0}^{x}f\left( t\right) d_{q}t:=x\left( 1-q\right) \sum_{k=0}
^{+\infty}q^{k}f\left( xq^{k}\right)
\end{equation*}
provided that the series converge at $x=a$ and $x=b$.
When $q\rightarrow1$, we obtain the \emph{N\"{o}rlund sum}\index{N\"{o}rlund sum}
\begin{equation*}
\int_{a}^{b}f\left( t\right) \Delta_{\omega}t :=\int_{+\infty}^{b}f\left(
t\right) \Delta_{\omega}t-\int_{+\infty}^{a}f\left( t\right) \Delta
_{\omega}t,
\end{equation*}
where
\begin{equation*}
\int_{+\infty}^{x}f\left( t\right) \Delta_{\omega}t :=-\omega\sum
_{k=0}^{+\infty}f\left( x+k\omega\right)
\end{equation*}
provided that the series converge at $x=a$ and $x=b$. Note that the N\"{o}rlund sum is,
in some sense, a generalization of the $h$-integral for the case where $a$ and $b$
are any real numbers (not necessarily $a,b\in h\mathbb{Z}$ like in Definition~\ref{h-integral}).
\end{remark}

\begin{proposition}[cf. \cite{Aldwoah,Aldwoah:2,Annaby}]
If $f:I\rightarrow\mathbb{R}$ is continuous at $\omega_{0}$,
then $f$ is $q,\omega$-\emph{integrable over} $I$.
\end{proposition}

\begin{theorem}[Fundamental theorem of Hahn's calculus \protect\cite{Aldwoah,Annaby}]
\index{Fundamental theorem of Hahn's calculus}
Assume that $f:I\rightarrow \mathbb{R}$ is continuous at
$\omega_{0}$ and, for each $x\in I$, define
\begin{equation*}
F\left( x\right) :=\int_{\omega_{0}}^{x}f\left( t\right) d_{q,\omega }t .
\end{equation*}
Then $F$ is continuous at $\omega_{0}$. Furthermore, $D_{q,\omega}\left[ F
\right] \left( x\right)$ exists for every $x\in I$ with
\begin{equation*}
D_{q,\omega}\left[ F\right] \left( x\right) =f\left( x\right).
\end{equation*}
Conversely,
\begin{equation*}
\int_{a}^{b}D_{q,\omega}\left[ f\right] \left( t\right) d_{q,\omega}t
=f\left( b\right) -f\left( a\right)
\end{equation*}
for all $a,b\in I$.
\end{theorem}

The $q,\omega$-\emph{integral} has the following properties.

\begin{theorem}[\protect\cite{Aldwoah,Aldwoah:2,Annaby}]
\label{hoPropriedades do integral} Let $f,g:I\rightarrow\mathbb{R}$
be $q,\omega$-\emph{integrable on} $I$, $a,b,c\in I$ and $k\in\mathbb{R}$. Then,

\begin{enumerate}
\item $\displaystyle\int_{a}^{a}f\left( t\right) d_{q,\omega}t=0$;

\item $\displaystyle\int_{a}^{b}kf\left( t\right) d_{q,\omega}t=k\int_{a}^{b}f\left(
t\right) d_{q,\omega}t$;

\item \label{eq:item3} $\displaystyle\int_{a}^{b}f\left( t\right)
d_{q,\omega}t=-\int_{b}^{a}f\left( t\right) d_{q,\omega}t$;

\item $\displaystyle\int_{a}^{b}f\left( t\right) d_{q,\omega}t=\int_{a}^{c}f\left(
t\right) d_{q,\omega}t+\int_{c}^{b}f\left( t\right) d_{q,\omega}t$;

\item $\displaystyle\int_{a}^{b}\left( f\left( t\right) +g\left( t\right) \right)
d_{q,\omega}t=\int_{a}^{b}f\left( t\right) d_{q,\omega}t
+\int_{a}^{b}g\left( t\right) d_{q,\omega}t$;

\item every Riemann integrable function $f$ on $I$ is
$q,\omega$-\emph{integrable on} $I$;

\item \label{itm:ip} if $f,g:I\rightarrow \mathbb{R}$ are
$q,\omega$-differentiable and $a,b\in I$, then
\begin{equation*}
\int_{a}^{b}f\left( t\right) D_{q,\omega}\left[ g\right] \left( t\right)
d_{q,\omega}t= f\left( t\right) g\left( t\right) \bigg|_{a}^{b}
-\int_{a}^{b}D_{q,\omega}\left[ f\right] \left( t\right) g\left(
qt+\omega\right) d_{q,\omega}t.
\end{equation*}
\end{enumerate}
\end{theorem}

Property~\ref{itm:ip} of Theorem~\ref{hoPropriedades do integral} is known
as $q,\omega$-\emph{integration by parts}\index{$q,\omega$-integration by parts}.
Note that
\begin{equation*}
\int_{\sigma\left( t\right) }^{t}f\left( \tau\right) d_{q,\omega}
\tau=\left( t\left( 1-q\right) -\omega\right) f\left( t\right)
\text{.}
\end{equation*}

\begin{lemma}[\textrm{cf.} \protect\cite{Brito:da:Cruz}]
\label{hopositividade} Let $b \in I$ and $f$ be $q,\omega$-\emph{integrable}
over $I$. Suppose that
\begin{equation*}
f(t)\geqslant 0, \quad \forall t\in\left\{ q^{n}b+\omega\left[ n\right] _{q}:n\in
\mathbb{N}_{0}\right\}.
\end{equation*}

\begin{enumerate}
\item If $\omega_0 \leqslant b$, then
\begin{equation*}
\int_{\omega_0}^b f(t)d_{q,\omega}t\geqslant 0.
\end{equation*}

\item If $\omega_0 > b$, then
\begin{equation*}
\int_b^{\omega_0} f(t)d_{q,\omega}t\geqslant 0.
\end{equation*}
\end{enumerate}
\end{lemma}

\begin{remark}
\label{rem:diff:int}
In general it is not true that
\begin{equation*}
\left\vert \int_{a}^{b}f\left( t\right) d_{q,\omega}t\right\vert \leqslant
\int_{a}^{b}| f\left( t\right)| d_{q,\omega}t , \ \ \ a,b \in I.
\end{equation*}
For a counterexample, see \cite{Aldwoah,Annaby}. This illustrates well the
difference with other non-quantum integrals, \textrm{e.g.}, the time scale
integrals \cite{Malinowska:2,Mozyrska}.
\end{remark}

Similarly to the previous sections, one can define the
\emph{backward Hahn's derivative}\index{Backward Hahn's derivative} for $ t\neq\omega_{0}$ by
\begin{equation*}
\frac{f\left( t\right)
- f\left(q^{-1}\left(t-\omega\right)\right)}{\left(1-q^{-1}\right)t + q^{-1}\omega}.
\end{equation*}

To conclude this section, we note that the Hahn Calculus is not a particular case
of the time scale calculus. We stress that the basic difference between the Hahn calculus
and the time scale calculus is that in the Hahn calculus we deal with functions defined
in real intervals and the derivative is a ``discrete derivative'',
while in the time scale calculus, this kind of derivative
is only defined for discrete time scales.


\clearpage{\thispagestyle{empty}\cleardoublepage}


\chapter{Calculus of Variations}
\label{Calculus of Variations}

\setcounter{theorem}{0}

In 1696, Johann Bernoulli challenge the mathematical community (and in
particular his brother James) to solve the following problem: in a vertical
plane, consider two fixed points $A$ and $B$, with $A$ higher than $B$.
Determine the curve from $A$ to $B$ along which a particle (initially at
rest) slides in minimum time, with the only external force acting on the
particle being gravity. This curve is called the brachistochrone (from
Greek, brachistos - the shortest, chronos - time). This problem was solved
by both Bernoulli brothers, but also by Newton, L'H\^{o}pital, Leibniz and
later by Euler and Lagrange. This problem is considered the first problem in
calculus of variations and the solution provided by Euler was so
resourceful that gave us the principles to solve other variational problems,
such as:

-Find among all the closed curves with the same length, the curve enclosing
the greatest area (this problem is called the isoperimetric problem);

-Find the shortest path (i.e., geodesic) between two given points on a
surface;

-Find the curve between two given points in the plane that yields a surface
of revolution of minimum area when revolved around a given axis. This curve
is called the catenary and the surface is called a catenoid.

In the eighteenth century mathematicians like the Bernoulli brothers,
Leibniz, Euler, Lagrange and Legendre all worked in theory of calculus of
variations. In the nineteenth century names like Jacobi and Weierstrass
contributed to the subject and in the twentieth century Hilbert, Noether,
Tonelli, Lebesgue, just to name a few, made advances to the theory. For a
deeper understanding of the history of the calculus of variations we refer
the reader to Goldstine \cite{Goldstine}.

Historically, there is an intricate relation between calculus of variations
and mechanics. In fact, many physical principles may be formulated in terms
of variational problems, via Hamilton's principle of least action.
Applications can be found in general relativity theory, quantum field
theory and particle physics. Even recently, calculus of variations provided
tools to discover the solution to the three-body problem \cite{Montgomery}.
However, applications of calculus of variations can be found in other areas
such as economics \cite{Attici} and optimal control theory \cite{Pontryagin}.

The calculus of variations is concerned with the problem of extremizing
functionals, and so, we can consider it as a branch of optimization.

Next, we review the basics of the theory of the calculus of variations.

Let $L\left( x,y,z\right) $ be a function with continuous first and second
partial derivatives with respect to all its arguments.
The simplest variational problem is defined by\index{Variational problem}
\begin{equation}
\left \{
\begin{array}{lll}
\mathcal{L}\left[ y\right] =\displaystyle\int_{a}^{b}L\left( t,y\left( t\right),
y^{\prime }\left( t\right) \right) dt & \rightarrow & \text{extremize} \\
\\
y\in C^{1}\left( \left[ a,b\right] ,
\mathbb{R}
\right) &  &  \\
\\
y\left( a\right) =\alpha &  &  \\
\\
y\left( b\right) =\beta &  &
\end{array}
\right.  \label{P}
\end{equation}
where $a,b,\alpha ,\beta \in \mathbb{R}$ and $a<b$.
By extremize we mean maximize or minimize.

We say that $y$ is an admissible function for problem \eqref{P} if $y\in
C^{1}\left( \left[ a,b\right] ,
\mathbb{R}
\right) $ and $y$ satisfies the boundary conditions $y\left( a\right)
=\alpha $ and $y\left( b\right) =\beta $. We say that $y_{\ast }$ is a local
minimizer (resp. local maximizer) for problem \eqref{P} if $y_{\ast }$ is an
admissible function an there exists $\delta >0$ such that
\[
\mathcal{L}\left[ y_{\ast }\right] \leqslant \mathcal{L}\left[ y\right]
\text{ \ \ (resp. }\mathcal{L}\left[ y_{\ast }\right]
\geqslant \mathcal{L}\left[ y\right] )
\]
for all admissible $y$ with $\left \Vert y_{\ast }-y\right \Vert <\delta $ and where
\[
\left \Vert y\right \Vert =\max_{t\in \left[ a,b\right] }\left \{ \left \vert
y\left( t\right) \right \vert +\left \vert y^{\prime }\left( t\right)
\right \vert \right \}.
\]

A necessary condition for a function to be an extremizer for problem
\eqref{P} is given by the Euler--Lagrange equation.

In what follows, $\frac{\partial L}{\partial y}$ denotes the
partial derivative of $L$ with respect to its $2nd$ argument and $\frac{\partial L }{\partial y^{\prime }}$ denotes the
partial derivative of $L$ with respect to its $3rd$ argument.

\begin{theorem}[Euler--Lagrange equation (cf. \cite{Brunt})]\index{Euler--Lagrange equation}
Under the hypotheses of the problem\\ \eqref{P}, if an admissible function
$y_{\ast }$ is a local extremizer for problem \eqref{P}, then $y_{\ast }$
satisfies the Euler--Lagrange equation
\begin{equation}
\label{EL:C}
\frac{\partial L}{\partial y}\left( t,y\left( t\right), y^{\prime }\left( t\right) \right)
=\frac{d}{dt}\left[ \tau\rightarrow\frac{\partial L }{\partial y^{\prime }}\left(\tau,
y\left( \tau\right) ,y^{\prime}\left( \tau\right) \right)\right]\left(t\right)
\end{equation}
for all $t\in \left[ a,b\right]$.
\end{theorem}

\begin{ex}
\label{ex1}
The problem of discovering the curve $y$ which has a minimum
arc length between the points $(0,0)$ and $(1,1)$ can be enunciate by
\[
\left \{
\begin{array}{lll}
\mathcal{L}\left[ y\right] =\displaystyle\int_{0}^{1}\sqrt{1+\left[ y^{\prime }\left(
t\right) \right] ^{2}}dt & \rightarrow  & \text{minimize} \\
&  &  \\
y\in C^{1}\left( \left[ 0,1\right] ,\mathbb{R}\right)  &  &  \\
&  &  \\
y\left( 0\right) =0 &  &  \\
&  &  \\
y\left( 1\right) =1. &  &
\end{array}
\right.
\]
Clearly, our knowledge of ordinary geometry suggest that the curve which minimizes
the arc length is the straight line connecting $(0,0)$ to $(1,1)$. Using the tools
of the calculus of variations we prove, in a rigorous way,
that our geometric intuition is indeed correct.

The Euler–-Lagrange equation associated to $\mathcal{L}$ is
\begin{align*}
&\frac{d}{dt}\left(\frac{y'\left(t\right)}{\sqrt{1+\left[ y^{\prime }\left(
t\right) \right] ^{2}}}\right)=0
\Leftrightarrow \frac{y''\left(t\right)}{\left(1
+\left[y'\left(t\right)\right]^2\right)^\frac{3}{2}}=0
\end{align*}
for all $t\in\left[0,1\right]$. Obviously, the solution of this second-order
equation is $y(t)=At+B$ for some $A,B\in\mathbb{R}$. Using the boundary conditions
$y(0)=0$ and $y(1)=1$ we conclude that $y\left(t\right)=t$, $t\in[0,1]$,
is the only candidate to be a minimizer. To prove that $y\left(t\right)=t$
is indeed the solution, we use a sufficient condition (see Example~\ref{ex2}).
\end{ex}

Historically, variational problems were often formulated with some kind of
constraint. Dido, the queen founder of Carthage, when arrived on the coast
of Tunisia, asked for a piece of land. Her request was satisfied provided
that the land could be encompassed by an oxhide. She cut the oxhide into
very narrow strips and joined them together into a long thin strip and used
it to encircle the land. This land became Carthage. Dido's problem, the
\emph{isoperimetric }(same perimeter) \emph{problem}\index{Isoperimetric problem}
is to find a curve $y$ that satisfies the variational problem
\begin{equation}
\label{I}
\left\{
\begin{array}{lll}
\mathcal{L}\left[ y\right] =\displaystyle\int_{a}^{b}L\left( t,y\left( t\right)
,y^{\prime }\left( t\right) \right) dt & \rightarrow  & \text{extremize} \\
\\
y\in C^{1}\left( \left[ a,b\right] ,
\mathbb{R}
\right)  &  &  \\
\\
y\left( a\right) =\alpha  &  &  \\
\\
y\left( b\right) =\beta  &  &
\end{array}
\right.
\end{equation}
but subject to the constraint
\begin{equation}
\label{IC}
\gamma=\int_{a}^{b}G\left( t,y\left( t\right) ,y^{\prime }\left( t\right) \right) dt,
\end{equation}
where $G$ is a given function satisfying the same condition
that the Lagrangian $L$ and $\gamma$ is a real fixed value.

A necessary condition for a function to be an extremizer
for the isoperimetric problem \eqref{I}--\eqref{IC}
is given by the following result.

\begin{theorem}[cf. \cite{Brunt}]
Suppose that $\mathcal{L}$ has an extremum at
$y\in C^{1}\left( \left[ a,b\right],\mathbb{R}\right)$
subject to the boundary conditions $y\left( a\right) =\alpha$
and $y\left( b\right) =\beta$ and the isoperimetric constraint \eqref{IC}.
Then there exist two real numbers $\lambda_0$ and $\lambda_1$, not both zero, such that
$$
\frac{d}{dt}\frac{\partial K}{\partial y'} = \frac{\partial K}{\partial y},
$$
where $$K=\lambda_0 L - \lambda_1 G.$$
\end{theorem}

Another necessary condition for problem \eqref{P}
is given by the Legendre condition.\index{Legendre's condition}

\begin{theorem}[Legendre's condition (cf. \cite{Brunt})]
Let $\mathcal{L}$ be the functional of problem \eqref{P}
and suppose that $\mathcal{L}$ has a local minimum for the curve $y$.
Then,
\[
\frac{\partial }{\partial y^{^{\prime }}}\left[\tau\rightarrow \frac{\partial L }{\partial
y^{^{\prime }}}\left(
\tau,y\left( \tau\right) ,y^{\prime }\left( \tau\right) \right)\right] \left(t\right)\geqslant 0
\]
for all $t\in \left[ a,b\right]$.
\end{theorem}

The first-order variational problem \eqref{P} can be extended to functionals
involving higher-order derivatives. We formulate the $r$-order variational
problem as follows
\begin{equation}
\left\{
\begin{array}{lll}\index{Higher-order variational problem}
\mathcal{L}\left[ y\right] =\displaystyle\int_{a}^{b}L\left( t,y\left( t\right)
,y^{\prime }\left( t\right) ,\ldots ,y^{\left( r\right) }\left( t\right)
\right) dt & \rightarrow  & \text{extremize} \\
\\
y\in C^{r}\left( \left[ a,b\right] ,
\mathbb{R}
\right)  &  &  \\
\\
y\left( a\right) =\alpha _{0},\text{ \ \ }y\left( b\right) =\beta _{0} &  &
\\
\vdots  &  &  \\
y^{\left( r-1\right) }\left( a\right) =\alpha _{r-1},\text{ \ \ }y^{\left(
r-1\right) }\left( b\right) =\beta _{r-1} &  &
\end{array}
\right.   \label{P_r}
\end{equation}
where $r\in
\mathbb{N}
$, $a,b,\in
\mathbb{R}
$ with \thinspace $a<b$ and $\alpha _{i},\beta _{i}\in
\mathbb{R}
$, $i=0,\ldots ,r-1$, are given. Also, we must assume that the Lagrangian
function $L$ has continuous partial derivatives up to the order $r+1$ with
respect to all its\ arguments. Clearly, in the problem of order $r$ we need
to restrict the space of functions and we say that $y$ is an admissible
function if $y\in C^{r}\left( \left[ a,b\right] ,
\mathbb{R}
\right) $ and satisfies all the boundary conditions of problem \eqref{P_r}.
Like the first-order case, we can obtain a necessary condition for a
function to be an extremizer.

\begin{theorem}[Higher-order Euler--Lagrange equation (cf. \cite{Brunt})]\index{Higher-order Euler--Lagrange equation}
Under the hypotheses of problem \eqref{P_r}, if an admissible function
$y_{\ast }$ is a local extremizer for problem \eqref{P}, then $y_{\ast }$
satisfies the Euler--Lagrange equation
$$
\sum_{i=0}^{r}\left(-1\right)^{i}
\frac{d^i}{d t^i}\left[\tau\rightarrow\partial_{i+2}L\left(\tau,
y\left(\tau\right),y'\left(\tau\right), \ldots, y^{(r)}(\tau)\right)\right]\left(  t\right)=0
$$
for all $t\in \left[ a,b\right] $ and where $\partial _{i}L$ denotes the
partial derivative of $L$ with respect to its $ith$ argument.
\end{theorem}

Let us recall the definition of a convex set and a convex function.

\begin{definition}
Let $\Omega\subseteq\mathbb{R}^2$.
The set $\Omega$ is said to be convex if
$$
\left(1-t\right)z_1+tz_2\in\Omega
$$
for all $t\in\left[0,1\right]$ and $z_1,z_2\in\Omega$.
\end{definition}

\begin{definition}
Let $\Omega\subseteq\mathbb{R}^2$ be a convex set.
A function  $f:\Omega\rightarrow\mathbb{R}$ is said to be convex if
$$
f\left[\left(1-t\right)z_1+tz_2\right]\leqslant\left(1-t\right)f(z_1)+tf(z_2)
$$
for all $z_1,z_2\in\Omega$ and $t\in\left[0,1\right]$.
\end{definition}

A sufficient condition for a minimum for problem \eqref{I} is now presented.

\begin{theorem}[cf. \cite{Brunt}]
\label{SuffCondt}
Let $L:D\subseteq\mathbb{R}^3\rightarrow\mathbb{R}$ be a function.
Suppose that, for each $t\in\left[a,b\right]$, the set
$$
\Omega_t:=\{ \left( y,y' \right) \in \mathbb{R}^2: \left( t,y,y' \right) \in D \}
$$
is convex and that $L$ is a convex function of the variables
$\left(y,y' \right)\in\Omega_t$. If $y_*$ satisfies
the Euler--Lagrange equation \eqref{EL:C},
then $y_*$ is a minimizer to problem \eqref{I}.
\end{theorem}

Note that when a function is smooth or has continuous first-
and second-order partial derivatives we can apply
the following result which is more practical
to prove that a function is convex.

\begin{theorem}[cf. \cite{Brunt}]
Let $\Omega\subseteq\mathbb{R}^2$ be a convex set and let $f:\Omega\rightarrow\mathbb{R}$
be a function with continuous first- and second-order partial derivatives.
The function $f$ is convex if, and only if, for each $(y,w)\in\Omega$:
\[
\begin{array}{rcc}

\displaystyle\frac{\partial ^{2}f}{\partial y^{2}}\left( y,w\right)  & \geqslant  & 0 \\
&  &  \\
\displaystyle\frac{\partial ^{2}f}{\partial w^{2}}\left( y,w\right)  & \geqslant  & 0 \\
&  &  \\
\displaystyle\frac{\partial ^{2}f}{\partial y^{2}}\left( y,w\right)
\frac{\partial^{2}f}{\partial w^{2}}\left( y,w\right)
- \frac{\partial ^{2}f}{\partial w\partial y}
\left( y,w\right)  & \geqslant  & 0.
\end{array}
\]
\end{theorem}

\begin{ex}\label{ex2}
In Example~\ref{ex1}, we have proven that the function $y(t)=t$ is a candidate
to the problem of finding the curve with the shortest arc length connecting
the points $(0,0)$ and $(1,1)$ of the plane. In this example
$\Omega_t=\mathbb{R}^2$ is a convex set and since for all $(y,y')\in\Omega_t$ we have
$$
\frac{\partial ^{2}L}{\partial y^{2}}=\frac{\partial ^{2}L}{\partial y'\partial y}=0
$$
and
$$
\frac{\partial ^{2}L}{\partial y'^{2}}= \frac{1}{\left[\sqrt{1+\left[ y^{\prime }\left(
t\right) \right] ^{2}}\right]^3}>0,
$$
hence $L$ is convex. By Theorem~\ref{SuffCondt}, we can conclude
that the line segment $y(t)=t$ is the solution to the problem.
\end{ex}

For further information about the classical calculus of variations we suggest the reader
the following books \cite{Gelfand,Brunt}.

The study of the calculus of variations in the context of the time scale theory
is very recent (2004) and was initiated with the well known paper of Martin Bohner \cite{Bohner:3}.
One of the main results of \cite{Bohner:3} is the Euler--Lagrange equation
for the first-order variational problem involving the delta
derivative:\index{Variational problem on time scales}
\begin{equation}
\label{Variational problem TS}
\left \{
\begin{array}{lll}
\mathcal{L}\left[ y\right] =\displaystyle\int_{a}^{b}L\left( t,y^\sigma\left( t\right)
,y^\Delta\left( t\right) \right) \Delta t & \rightarrow & \text{extremize} \\
\\
y\left( a\right) =\alpha &  &  \\
\\
y\left( b\right) =\beta &  &
\end{array}
\right.
\end{equation}
where $a,b\in\mathbb{T}$, $a<b$; $\alpha ,\beta \in \mathbb{R}^n$
with $n\in\mathbb{N}$, and $L:\mathbb{T}\times\mathbb{R}^{2n}\rightarrow\mathbb{R}$ is a $C^2$ function.

\begin{theorem}[Euler--Lagrange equation on time scales \cite{Bohner:3}]\index{Euler--Lagrange equation on time scales}
If $y_{\ast }\in C_{rd}^1$ is a local extremum of \eqref{Variational problem TS},
then the Euler--Lagrange equation
\[
\left( \partial _{3}L\right) ^{\Delta }\left( t,y_{\ast }^{\sigma }\left(
t\right) ,y_{\ast }^{\Delta }\left( t\right) \right) =\partial _{2}L\left(
t,y_{\ast }^{\sigma }\left( t\right) ,y_{\ast }^{\Delta }\left( t\right)
\right)
\]
holds for all $t\in [a,b]^k$.
\end{theorem}

Since the pioneer work of Martin Bohner \cite{Bohner:3}, many researchers developed
the calculus of variations on time scales in several different directions, namely:
with non-fixed boundary conditions \cite{Hilscher,Hilscher:2},
with two independent variables \cite{Bohner:4},
with higher-order delta derivatives \cite{Ferreira,Ferreira:3},
with higher-order nabla derivatives \cite{Martins},
with isoperimetric constraints \cite{Almeida,Ferreira:2,Malinowska:34},
infinite horizon variational problems \cite{Malinowska35,Martins:4},
with a Lagrangian depending on a delta indefinite integral
that depends on the unknown function \cite{Martins:25},
and with an invariant group of parameter-transformations
\cite{Bartosiewicz,Bartosiewicz05,Martins15}.

To end this chapter we refer the reader to the survey paper \cite{Torres},
that gives an excellent overview on a more general approach of the calculus of variations
on time scales that allows to obtain as particular cases the delta and nabla calculus of variations.


\clearpage{\thispagestyle{empty}\cleardoublepage}


\part{Original Work}
\label{Original Work}

\clearpage{\thispagestyle{empty}\cleardoublepage}


\chapter{Higher-order Hahn's Quantum Variational Calculus}
\label{Higher-order Hahn's Quantum Variational Calculus}

In this chapter we present our achievements made in the calculus of variations
within the Hahn's quantum calculus. In Section~\ref{sub:sec:hoCV} we prove a
higher-order fundamental lemma of the calculus of variations
with Hahn's operator (Lemma~\ref{hoordem n}).
In Section~\ref{sub:sec:HOELHCV} we deduce a higher-order
Euler--Lagrange equation for Hahn's variational calculus
(Theorem~\ref{hoHigher order E-L}).
Finally, we provide in Section~\ref{subsec:Ex}
a simple example of a quantum optimization problem
where our Theorem~\ref{hoHigher order E-L} leads
to the global minimizer, which is not a continuous function.


\section{Introduction}

Many physical phenomena are described
by equations involving nondifferentiable functions,
\textrm{e.g.}, generic trajectories of quantum mechanics \cite{Feynman}.
Several different approaches to deal with nondifferentiable functions
are followed in the literature of variational calculus, including
the time scale approach, which typically deal with delta or nabla differentiable
functions \cite{Ferreira,Malinowska,Martins}, the fractional approach, allowing to consider
functions that have no first-order derivative but have fractional derivatives
of all orders less than one \cite{Almeida:3,El-Nabulsi,Frederico},
and the quantum approach, which is particularly useful to model physical
and economical systems \cite{Bangerezako,Cresson,Malinowska:5,Malinowska:3,Martins:2}.

Roughly speaking, a quantum calculus substitutes the classical derivative
by a difference operator, which allows one to deal with sets of nondifferentiable functions.
Several dialects of quantum calculus are available \cite{Ernst:2,Kac}.
For motivation to study a nondifferentiable quantum variational calculus
we refer the reader to \cite{Almeida:2,Bangerezako,Cresson}.

In 1949 Hahn introduced the difference operator $D_{q,\omega}$ defined by
\[
D_{q,\omega}\left[  f\right]  \left(  t\right)  :=\frac{f\left(  qt+\omega
\right)  -f\left(  t\right)  }{\left(  q-1\right)  t+\omega},
\]
where $f$ is a real function, $q\in ]0,1[$
and $\omega\geqslant0$ are real fixed numbers \cite{Hahn}.
Hahn's difference operator has been applied successfully in the construction
of families of orthogonal polynomials as well as in approximation problems
\cite{alvares,Dobrogowska,Petronilho}.
However, during 60 years, the construction of the proper inverse
of Hahn's difference operator remained an open question.
Eventually, the problem was solved in 2009 by Aldwoah \cite{Aldwoah}
(see also \cite{Aldwoah:2,Annaby}). Here we introduce
the higher-order Hahn's quantum variational calculus,
proving Hahn's quantum analog of the higher-order Euler--Lagrange equation.
As particular cases we obtain the $q$-calculus Euler--Lagrange equation
\cite{Bangerezako} and the $h$-calculus Euler--Lagrange equation \cite{Bastos,Kelley}.

Variational functionals that depend on higher-order
derivatives arise in a natural way
in applications of engineering, physics, and economics.
Let us consider, for example, the equilibrium of an elastic
bending beam. Let us denote by $y(x)$ the deflection of the point $x$ of the beam,
$E(x)$ the elastic stiffness of the material, that can vary with $x$,
and $\xi(x)$ the load that bends the beam.
One may assume that, due to some constraints of physical nature, the dynamics
does not depend on the usual derivative $y'(x)$ but on some quantum derivative
$D_{q,\omega}\left[y\right]\left(x\right)$. In this condition,
the equilibrium of the beam correspond to the solution of the following
higher-order Hahn's quantum variational problem:
\begin{equation}
\label{ex:intro:ar}
\int_0^L \left[\frac{1}{2} \left(E(x) D_{q,\omega}^{2}\left[y\right](x)\right)^2
- \xi(x) y\left(q^2 x + q \omega + \omega\right)\right] dx \longrightarrow \min.
\end{equation}
Note that we recover the classical problem of the equilibrium
of the elastic bending beam when $(\omega,q)\rightarrow ]0,1[$.
Problem \eqref{ex:intro:ar} is a particular case of
problem \eqref{hoP} investigated in Section~\ref{sec:homr}.
Our higher-order Hahn's quantum Euler--Lagrange equation
(Theorem~\ref{hoHigher order E-L}) gives the main tool to solve such problems.


\section{Main results}
\label{sec:homr}

Let $I$ be a real interval containing $\displaystyle\omega_{0} := \frac{\omega}{1-q}.$
We define the $q,\omega$-derivatives of higher-order in the usual way:
the $r$th $q,\omega$-derivative ($r \in \mathbb{N}$)
of $f:I\rightarrow \mathbb{R}$ is the function
$D_{q,\omega}^{r}[f]: I\rightarrow \mathbb{R}$ given by
$D_{q,\omega}^{r}[f]:=D_{q,\omega}[D_{q,\omega}^{r-1}[f]]$,
provided $D_{q,\omega}^{r-1}[f]$ is $q,\omega$-differentiable on $I$
and where $D_{q,\omega}^{0}[f]:=f$.

For $s\in I$ we define
\begin{equation}
\label{eq:def:tar}
\left[  s\right]_{q,\omega}:=\left\{  q^{n}s+\omega\left[  n\right]_q  :n\in
\mathbb{N}_{0}\right\}  \cup\left\{  \omega_{0}\right\}  \text{.}
\end{equation}

Let $a,b \in I$ and $a<b$. We define the $q,\omega$-interval from $a$ to $b$ by
\[
\left[  a,b\right]_{q,\omega}:=\left\{  q^{n}a+\omega\left[  n\right]
_{q}:n\in
\mathbb{N}_{0}\right\}  \cup
\left\{  q^{n}b+\omega\left[  n\right]_{q}:n\in\mathbb{N}_{0}\right\}
\cup\left\{  \omega_{0}\right\},
\]
\textrm{i.e.}, $\left[a,b\right]_{q,\omega}=[a]_{q,\omega}\cup [b]_{q,\omega}$,
where $[a]_{q,\omega}$ and $[b]_{q,\omega}$ are given by \eqref{eq:def:tar}.

We introduce the linear space
$\mathcal{Y}^{r} = \mathcal{Y}^{r}\left(\left[a,b\right],\mathbb{R}\right)$ by
$$
\mathcal{Y}^{r} :=
\left\{  y: I  \rightarrow \mathbb{R}\, |\,
D_{q,\omega}^{i}[y], i = 0,\ldots, r,
\text{ are bounded on $[a,b]_{q,\omega}$ and continuous at } \omega_{0}\right\}
$$
endowed with the norm
$\left\Vert y\right\Vert_{r,\infty}:=\sum_{i=0}^{r}\left\Vert D_{q,\omega}
^{i}\left[  y\right]  \right\Vert_{\infty}$,
where
$\left\Vert y\right\Vert_{\infty}:=\sup_{t\in\left[  a,b\right]_{q,\omega}  }\left\vert
y\left(  t\right)  \right\vert$. The following notations are in order:
$\sigma(t):=qt+\omega$, $y^{\sigma}(t) = y^{\sigma^1}(t)
= (y \circ \sigma)(t) = y\left(  qt+\omega\right)$,
and $y^{\sigma^{k}} = y^{\sigma^{k-1}} \circ \sigma$, $k = 2, 3, \ldots$
Our main goal is to establish necessary optimality conditions
for the higher-order $q,\omega$-variational
problem \index{Higher-order $q,\omega$-variational problem}
\begin{equation}
\label{hoP}
\left\{
\begin{array}[c]{l}
\mathcal{L}\left[  y\right] = \displaystyle \int_{a}^{b}L\left(
t,y^{\sigma^{r}}\left(  t\right)  ,D_{q,\omega}\left[  y^{\sigma^{r-1}}\right]
\left(  t\right)  ,\ldots,D_{q,\omega}^{r}\left[  y\right]  \left(  t\right)
\right)  d_{q,\omega}t \longrightarrow \textrm{extremize}\\
\\
y \in \mathcal{Y}^{r}\left(  \left[  a,b\right]  ,\mathbb{R}\right)\\
\\
y\left(  a\right)  =\alpha_{0} \, , \quad  y\left(  b\right)  =\beta_{0} \, ,\\
\vdots\\
D_{q,\omega}^{r-1}\left[  y\right]  \left(  a\right)  =\alpha_{r-1} \, , \quad
D_{q,\omega}^{r-1}\left[  y\right]  \left(  b\right)  =\beta_{r-1}\, ,
\end{array}
\right.
\end{equation}
where $r\in \mathbb{N}$ and $\alpha_{i}, \beta_{i}\in
\mathbb{R}$, $i=0,\ldots,r-1$, are given. By extremize we mean maximize or minimize.
\begin{definition}
We say that $y$ is an admissible function for \eqref{hoP}
if $y \in\mathcal{Y}^{r}\left(  \left[  a,b\right], \mathbb{R}\right)$
and $y$ satisfies the boundary conditions
$D_{q,\omega}^{i}\left[  y\right]  \left(  a\right)  =\alpha_{i}$
and $D_{q,\omega}^{i}\left[  y\right]  \left(  b\right)  =\beta_{i}$
of problem \eqref{hoP}, $i = 0,\ldots,r-1$.
\end{definition}

The Lagrangian $L$ is assumed to satisfy the following hypotheses:
\begin{enumerate}
\item[(H1)] $(u_0, \ldots, u_r)\rightarrow L(t,u_0, \ldots, u_r)$
is a $C^1(\mathbb{R}^{r+1}, \mathbb{R})$ function for any $t \in [a,b]$;

\item[(H2)] $t \rightarrow L(t, y(t), D_{q,\omega}\left[  y\right](t),
\ldots, D_{q,\omega}^{r}\left[  y\right](t))$ is continuous
at $\omega_0$ for any admissible $y$;

\item[(H3)]  functions $t \rightarrow \partial_{i+2}L(t, y(t),
D_{q,\omega}\left[  y\right](t), \cdots, D^{r}_{q,\omega}\left[  y\right](t))$,
$i=0, 1, \cdots, r$, belong to\\
$\mathcal{Y}^{1}\left(  \left[  a,b\right]  ,\mathbb{R}\right)$
for all admissible $y$,
where $\partial_{i}L$ denotes the partial derivative
of $L$ with respect to its $i$th argument.
\end{enumerate}

\begin{definition}
We say that $y_{\ast}$ is a local minimizer (resp. local maximizer) for problem
\eqref{hoP} if $y_{\ast}$ is an admissible function and there exists $\delta>0$ such that
\[
\mathcal{L}\left[  y_{\ast}\right]  \leqslant\mathcal{L}\left[  y\right]
\text{ \ \ (resp. }\mathcal{L}\left[  y_{\ast}\right]  \geqslant
\mathcal{L}\left[  y\right]  \text{) }
\]
for all admissible $y$ with
$\left\Vert y_{\ast}-y\right\Vert_{r,\infty}<\delta$.
\end{definition}

\begin{definition}
We say that $\eta\in \mathcal{Y}^{r}\left(  \left[  a,b\right],
\mathbb{R}\right) $ is a \emph{variation}
if $\eta\left(  a\right)  =\eta\left(  b\right)  =0$, \ldots,
$D_{q,\omega}^{r-1}\left[  \eta\right]  \left(  a\right)
=D_{q,\omega}^{r-1}\left[  \eta\right]  \left(  b\right)  =0$.
\end{definition}


\subsection{Higher-order fundamental lemma of Hahn's variational calculus}
\label{sub:sec:hoCV}

The chain rule, as known from classical calculus,
does not hold in Hahn's quantum context
(see a counterexample in \cite{Aldwoah,Annaby}).
However, we can prove the following.

\begin{lemma}
\label{hoq}
If $f$ is $q,\omega$-differentiable on $I$, then the following equality holds:
\[
D_{q,\omega}\left[  f^{\sigma}\right]  \left(  t\right)  =q\left(
D_{q,\omega}\left[  f\right]  \right)  ^{\sigma}\left(  t\right)  \text{,
\ }t\in I\text{.}
\]
\end{lemma}

\begin{proof}
For $t\neq\omega_{0}$ we have
\begin{equation*}
\left(  D_{q,\omega}\left[  f\right]  \right)  ^{\sigma}\left(  t\right)
=\frac{f\left(  q\left(  qt+\omega\right)  +\omega\right)  -f\left(
qt+\omega\right)  }{\left(  q-1\right)  \left(  qt+\omega\right)  +\omega}
  =\frac{f\left(  q\left(  qt+\omega\right)  +\omega\right)  -f\left(
qt+\omega\right)  }{q\left(  \left(  q-1\right)  t+\omega\right)  }
\end{equation*}
and
\begin{equation*}
D_{q,\omega}\left[  f^{\sigma}\right]  \left(  t\right)
=\frac{f^{\sigma}\left(  qt+\omega\right)
-f^{\sigma}\left(  t\right)  }{\left(  q-1\right) t+\omega}
=\frac{f\left(  q\left(  qt+\omega\right)  +\omega\right)
-f\left(qt+\omega\right)}{\left(  q-1\right)  t+\omega}.
\end{equation*}
Therefore,
$D_{q,\omega}\left[  f^{\sigma}\right]  \left(  t\right)  =q\left(
D_{q,\omega}\left[  f\right]  \right)  ^{\sigma}\left(  t\right)$.
If $t=\omega_{0}$, then $\sigma\left(  \omega_{0}\right)  =\omega_{0}$. Thus,
\begin{equation*}
\left(  D_{q,\omega}\left[  f\right]  \right)  ^{\sigma}\left(  \omega
_{0}\right) =\left(  D_{q,\omega}\left[  f\right]  \right)  \left(
\sigma\left(  \omega_{0}\right)  \right)
=\left(  D_{q,\omega}\left[  f\right]  \right)  \left(  \omega_{0}\right)
=f^{\prime}\left(  \omega_{0}\right)
\end{equation*}
and
$D_{q,\omega}\left[  f^{\sigma}\right]  \left(  \omega_{0}\right) =\left(
f^{\sigma}\right)  ^{\prime}\left(  \omega_{0}\right)
=f^{\prime}\left(  \sigma\left(
\omega_{0}\right)  \right)  \sigma^{\prime}\left(  \omega_{0}\right)
=qf^{^{\prime}}\left(  \omega_{0}\right)$.
\end{proof}

\begin{lemma}
\label{hoanula}
If $\eta\in \mathcal{Y}^{r}\left(  \left[  a,b\right]  ,\mathbb{R}\right)$ is such that
$D_{q,\omega}^{i}\left[  \eta\right]  \left(
a\right)  =0$ (resp. $D_{q,\omega}^{i}\left[  \eta\right]  \left(  b\right)
=0$) for all $i\in\left\{  0,1,\ldots,r\right\}$, then $D_{q,\omega}
^{i-1}\left[  \eta^{\sigma}\right]  \left(  a\right)  =0$ (resp. $D_{q,\omega
}^{i-1}\left[  \eta^{\sigma}\right]  \left(  b\right)  =0$) for all $i\in\left\{
1,\ldots,r\right\}  $.
\end{lemma}

\begin{proof}
If $a=\omega_{0}$ the result is trivial (because $\sigma\left(  \omega
_{0}\right)  =\omega_{0}$). Suppose now that $a\neq\omega_{0}$ and fix
$i\in\left\{  1,\ldots,r\right\}  $. Note that
\[
D_{q,\omega}^{i}\left[  \eta\right]  \left(  a\right)  =\frac{\left(
D_{q,\omega}^{i-1}\left[  \eta\right]  \right)  ^{\sigma}\left(  a\right)
-D_{q,\omega}^{i-1}\left[  \eta\right]  \left(  a\right)  }{\left(
q-1\right)  a+\omega}\text{.}
\]
Since, by hypothesis, $D_{q,\omega}^{i}\left[  \eta\right]  \left(  a\right)
=0$ and $D_{q,\omega}^{i-1}\left[  \eta\right]  \left(  a\right)  =0$, then
$\left(  D_{q,\omega}^{i-1}\left[  \eta\right]  \right)  ^{\sigma}\left(
a\right)  =0$. \\Lemma~\ref{hoq} shows that
\[
\left(  D_{q,\omega}^{i-1}\left[  \eta\right]  \right)  ^{\sigma}\left(
a\right)  =\left(  \frac{1}{q}\right)  ^{i-1}D_{q,\omega}^{i-1}\left[
\eta^{\sigma}\right]  \left(  a\right).
\]
We conclude that
$D_{q,\omega}^{i-1}\left[  \eta^{\sigma}\right]  \left(  a\right)  =0$.
The case $t=b$ is proved in the same way.
\end{proof}

\begin{lemma}
\label{hoLema F CV}
Suppose that $f \in \mathcal{Y}^{1}\left(\left[a,b\right],\mathbb{R}\right)$.
One has
\[
\int_{a}^{b}f\left(  t\right)  D_{q,\omega}\left[  \eta\right]  \left(
t\right)  d_{q,\omega}t=0
\]
for all functions $\eta\in \mathcal{Y}^{1}\left(
\left[  a,b\right]  ,\mathbb{R}\right)$ such that
$\eta\left(  a\right)  =\eta\left(  b\right)  =0$ if, and only if,
$f\left(  t\right)  =c$, $c\in \mathbb{R}$, for all $t\in\left[  a,b\right]_{q,\omega}$.
\end{lemma}

\begin{proof}
The implication ``$\Leftarrow$'' is obvious.
We prove ``$\Rightarrow$''. We begin by noting that
\[
\underset{=0}{\underbrace{\int_{a}^{b}f\left(  t\right)  D_{q,\omega}\left[
\eta\right]  \left(  t\right)  d_{q,\omega}t}}=\underset{=0}{\underbrace
{f\left(  t\right)  \eta\left(  t\right)  \bigg|_{a}^{b}}}-\int_{a}
^{b}D_{q,\omega}\left[  f\right]  \left(  t\right)  \eta^{\sigma}\left( t
\right)  d_{q,\omega}t.
\]
Hence,
\[
\int_{a}^{b}D_{q,\omega}\left[  f\right]  \left(  t\right)  \eta\left(
qt+\omega\right)  d_{q,\omega}t=0
\]
for any $\eta\in \mathcal{Y}^{1}\left(  \left[  a,b\right]  ,
\mathbb{R}
\right)  $  such that $\eta\left(  a\right)  =\eta\left(  b\right)  =0$.
We need to prove that, for some $c\in\mathbb{R}$,
$f\left(  t\right)  =c$ for all $t\in\left[  a,b\right]  _{q,\omega}$,
that is,
$D_{q,\omega}\left[  f\right]  \left(  t\right)  =0$ for all $t\in\left[
a,b\right]  _{q,\omega}$.
Suppose, by contradiction, that there exists $p\in\left[  a,b\right]_{q,\omega}$
such that $D_{q,\omega}\left[  f\right]  \left(  p\right)  \neq 0$.

\noindent (1) If $p\neq\omega_{0}$, then
$p=q^{k}a+\omega\left[  k\right]_{q}$ or $p=q^{k}b+\omega\left[k\right]_{q}$
for some $k\in \mathbb{N}_{0}$. Observe that
$a\left(  1-q\right) -\omega$ and $b\left(  1-q\right)-\omega$
cannot vanish simultaneously.

\noindent (a) Suppose that $a\left(  1-q\right)  -\omega\neq0$ and $b\left(
1-q\right)  -\omega\neq0$. In this case we can assume, without loss of
generality, that
$p=q^{k}a+\omega\left[  k\right]_{q}$
and we can define
\begin{equation*}
\eta\left(  t\right) =
\begin{cases}
D_{q,\omega}\left[  f\right]\left( q^{k}a+\omega\left[  k\right]_{q}\right)
&  \text{ if } t=q^{k+1}a+\omega\left[  k+1\right]  _{q}\\
\\
0 &  \text{ otherwise.}
\end{cases}
\end{equation*}
Then,
\begin{multline*}
\int_{a}^{b}D_{q,\omega}\left[  f\right]  \left(  t\right)  \cdot\eta\left(
qt+\omega\right)  d_{q,\omega}t\\
=-\left(  a\left(  1-q\right)
-\omega\right)  q^{k}D_{q,\omega}\left[  f\right]  \left(  q^{k}
a+\omega\left[  k\right]  _{q}\right)\cdot
D_{q,\omega}\left[  f\right]  \left(  q^{k}a+\omega\left[  k\right]_{q}\right) \neq 0,
\end{multline*}
which is a contradiction.

\noindent (b) If $a\left(  1-q\right)  -\omega\neq0$ and $b\left(  1-q\right)
-\omega=0$, then $b=\omega_{0}$. Since
$q^{k}\omega_{0}+\omega\left[  k\right]_{q}=\omega_{0}$ for all
$k\in \mathbb{N}_{0}$, then
$p\neq q^{k}b+\omega\left[  k\right]  _{q}$ $\forall k\in
\mathbb{N}_{0}$ and, therefore,
\[
p=q^{k}a+\omega\left[  k\right]_{q,\omega}\text{ for some }
k\in \mathbb{N}_{0}\text{.}
\]
Repeating the proof of $\left(  a\right)$ we obtain again a contradiction.

\noindent (c) If $a\left(  1-q\right)  -\omega=0$ and $b\left(  1-q\right)
-\omega\neq0$ then the proof is similar to $\left(  b\right)$.\\

\noindent (2) If $p=\omega_{0}$ then, without loss of generality, we can assume
$D_{q,\omega}\left[  f\right]  \left(  \omega_{0}\right)  >0$. Since
\[
\lim_{n\rightarrow+\infty}\left(  q^{n}a+\omega\left[  k\right]  _{q}\right)
=\lim_{n\rightarrow+\infty}\left(  q^{n}b+\omega\left[  k\right]  _{q}\right)
=\omega_{0}
\]
(see \cite{Aldwoah}) and $D_{q,\omega}\left[  f\right]  $ is continuous at
$\omega_{0}$, then
\begin{equation*}
\lim_{n\rightarrow+\infty}D_{q,\omega}\left[  f\right]  \left(  q^{n}
a+\omega\left[  k\right]  _{q}\right)   =\lim_{n\rightarrow
+\infty}D_{q,\omega}\left[  f\right]
\left(  q^{n}b+\omega\left[  k\right]_{q}\right)
=D_{q,\omega}\left[  f\right]  \left(  \omega_{0}\right)>0.
\end{equation*}
Thus, there exists $N\in \mathbb{N}$ such that for all $n\geqslant N$
one has
$D_{q,\omega}\left[  f\right]  \left(  q^{n}a+\omega\left[  k\right]_{q}\right)  >0$
and $D_{q,\omega}\left[  f\right]  \left(  q^{n}b+\omega\left[  k\right]_{q}\right)  >0$.

\noindent (a) If $\omega_{0}\neq a$ and $\omega_{0}\neq b$, then we can define
\[
\eta\left(  t\right)  =\left\{
\begin{tabular}[c]{lll}
$D_{q,\omega}\left[  f\right]  \left(  q^{N}b+\omega\left[  N\right]
_{q}\right)  $ & if & $t=q^{N+1}a+\omega\left[  N+1\right]  _{q}$ \\
&  & \\
$D_{q,\omega}\left[  f\right]  \left(  q^{N}a+\omega\left[  N\right]
_{q}\right)  $ & if & $t=q^{N+1}b+\omega\left[  N+1\right]  _{q}$ \\
&  & \\
0 &  & \text{otherwise}.
\end{tabular}
\right.
\]
Hence,
\begin{multline*}
\int_{a}^{b}D_{q,\omega}\left[  f\right]  \left(  t\right)  \eta\left(
qt+\omega\right)  d_{q,\omega}t\\
=\left(  b-a\right)  \left(  1-q\right)q^{N}D_{q,\omega}\left[  f\right]
\left(  q^{N} b+\omega\left[ N\right]  _{q}\right)\cdot  D_{q,\omega}\left[  f\right]
\left(  q^{N}a+\omega\left[  N\right]  _{q}\right) \neq 0,
\end{multline*}
which is a contradiction.

\noindent (b) If $\omega_{0}=b$, then we define
\[
\eta\left(  t\right)  =\left\{
\begin{tabular}
[c]{lll}
$D_{q,\omega}\left[  f\right]  \left(  \omega_{0}\right)  $ & if &
$t=q^{N+1}a+\omega\left[  N+1\right]  _{q}$\\
&  & \\
$0$ &  & otherwise.
\end{tabular}
\right.
\]
Therefore,
\begin{equation*}
\begin{split}
\int_{a}^{b} & D_{q,\omega}\left[  f\right]  \left(  t\right)  \eta\left(
qt+\omega\right)  d_{q,\omega}t\\
 &=  -\int_{\omega_{0}}^{a}D_{q,\omega}\left[
f\right]  \left(  t\right)  \eta\left(  qt+\omega\right)  d_{q,\omega}t\\
 &= -\left(  a\left(  1-q\right)  -\omega\right)   q^{N}D_{q,\omega
}\left[  f\right]  \left(  q^{N}a+\omega\left[  k\right]  _{q}\right)
\cdot D_{q,\omega}\left[  f\right]  \left(  \omega_{0}\right) \neq  0,
\end{split}
\end{equation*}
which is a contradiction.

\noindent (c) When $\omega_{0}=a$, the proof is similar to $\left(  b\right)$.
\end{proof}

\begin{lemma}[Fundamental lemma of Hahn's variational calculus]
\label{hoOrdem 1}
Let\index{Fundamental lemma of Hahn's variational calculus}
$f,g \in \mathcal{Y}^{1}\left(  \left[  a,b\right], \mathbb{R}\right)$. If
\[
\int_{a}^{b}\left(  f\left(  t\right)  \eta^{\sigma}\left(  t\right)
+g\left(  t\right)  D_{q,\omega}\left[  \eta\right]  \left(  t\right)
\right)  d_{q,\omega}t=0
\]
for all $\eta\in \mathcal{Y}^{1}\left(  \left[  a,b\right], \mathbb{R}\right)$
such that $\eta\left(  a\right)=\eta\left(  b\right)  =0$, then
\[
D_{q,\omega}\left[  g\right]  \left(  t\right)  =f\left(  t\right)  \text{
\ }\forall t\in\left[a,b\right]_{q,\omega}.
\]
\end{lemma}

\begin{proof}
Define the function $A$ by $A\left(  t\right)
:=\int_{\omega_{0}}^{t}f\left(  \tau\right)  d_{q,\omega}\tau$. Then,
$D_{q,\omega}\left[  A\right]  \left(  t\right)  =f\left(  t\right)$ for all
$t\in\left[  a,b\right]$ and
\begin{align*}
\int_{a}^{b}A\left(  t\right)  D_{q,\omega}\left[  \eta\right]  \left(
t\right)  d_{q,\omega}t  & =  A\left(  t\right)  \eta\left(  t\right)
\bigg|_{a}^{b}-\int_{a}^{b}D_{q,\omega}\left[  A\right]  \left(  t\right)
\eta^{\sigma}\left(  t\right)  d_{q,\omega}t\\
& =-\int_{a}^{b}D_{q,\omega}\left[  A\right]  \left(  t\right)  \eta^{\sigma
}\left(  t\right)  d_{q,\omega}t\\
& =-\int_{a}^{b}f\left(  t\right)  \eta^{\sigma}\left(  t\right)  d_{q,\omega}t.
\end{align*}
Hence,
\begin{equation*}
\begin{split}
&\int_{a}^{b}\left(  f\left(  t\right)  \eta^{\sigma}\left(  t\right)
+g\left(  t\right)  D_{q,\omega}\left[  \eta\right]  \left(  t\right)
\right)  d_{q,\omega}t=0\\
\Leftrightarrow & \int_{a}^{b}\left(  -A\left(  t\right)  +g\left(  t\right)
\right)  D_{q,\omega}\left[  \eta\right]  \left(  t\right)  d_{q,\omega}t=0.
\end{split}
\end{equation*}
By Lemma~\ref{hoLema F CV} there is a
$c\in \mathbb{R}$ such that $-A\left(  t\right)  +g\left(  t\right)  =c$ for all
$t\in\left[a,b\right]_{q,\omega}$. Hence $D_{q,\omega}\left[  A\right]  \left(
t\right)  =D_{q,\omega}\left[  g\right]  \left(  t\right)  $ for all $t\in\left[
a,b\right]_{q,\omega}$, which provides the desired result:
$$
D_{q,\omega}\left[  g\right]  \left(  t\right)  =f\left(  t\right) \text{
,\ }\forall t\in\left[  a,b\right]_{q,\omega}.
$$
\end{proof}

We are now in conditions to deduce a higher-order fundamental lemma of the
Hahn quantum variational calculus.\index{Higher-order fundamental lemma of the Hahn variational calculus}

\begin{lemma}[Higher-order fundamental lemma of the Hahn variational calculus]
\label{hoordem n}
Let \\$f_{0},f_{1},\ldots,f_{r} \in \mathcal{Y}^{1}\left(  \left[  a,b\right], \mathbb{R}\right)$. If
\[
\int_{a}^{b}\left(  \sum_{i=0}^{r}f_{i}\left(  t\right)  D_{q,\omega}^{i}\left[
\eta^{\sigma^{r-i}}\right]  \left(  t\right)  \right) d_{q,\omega}t=0
\]
for any variation $\eta$, then
\[
\sum_{i=0}^{r}\left(  -1\right)  ^{i}\left(  \frac{1}{q}\right)
^{\frac{\left(  i-1\right)  i}{2}}D_{q,\omega}^{i}\left[  f_{i}\right]
\left(  t\right) = 0
\]
for all $t\in\left[a,b\right]_{q,\omega}$.
\end{lemma}

\begin{proof}
We proceed by mathematical induction. If $r=1$ the result is true
by Lemma~\ref{hoOrdem 1}. Assume that
\[
\int_{a}^{b}\left(  \sum_{i=0}^{r+1}f_{i}\left(  t\right)  D_{q,\omega}^{i}
\left[  \eta^{\sigma^{r+1-i}}\right]  \left(  t\right)  \right)  d_{q,\omega
}t=0
\]
for all functions $\eta$ such that $\eta\left(  a\right)
= \eta\left(b\right)  =0$, \ldots,
$D_{q,\omega}^{r}\left[  \eta\right]  \left(  a\right)
= D_{q,\omega}^{r}\left[  \eta\right]  \left(  b\right)  =0$. Note that
\begin{align*}
\int_{a}^{b} f_{r+1}&\left(  t\right)  D_{q,\omega}^{r+1}\left[  \eta\right]
\left(  t\right)  d_{q,\omega}t\\
 &=f_{r+1}\left(  t\right)
D_{q,\omega}^{r}\left[  \eta\right]  \left(  t\right)  \bigg|_{a}^{b}
 -\int_{a}^{b}D_{q,\omega}\left[  f_{r+1}\right]  \left(  t\right)  \left(
D_{q,\omega}^{r}\left[  \eta\right]  \right)  ^{\sigma}\left(  t\right)
d_{q,\omega}t\\
 &= -\int_{a}^{b}D_{q,\omega}\left[  f_{r+1}\right]  \left(  t\right)  \left(
D_{q,\omega}^{r}\left[  \eta\right]  \right)  ^{\sigma}
\left(  t\right)  d_{q,\omega}t
\end{align*}
and, by Lemma~\ref{hoq},
\[
\int_{a}^{b}f_{r+1}\left(  t\right)  D_{q,\omega}^{r+1}\left[  \eta\right]
\left(  t\right)  d_{q,\omega}t=-\int_{a}^{b}D_{q,\omega}\left[
f_{r+1}\right] \left(  t\right)  \left(  \frac{1}{q}\right)^{r}D_{q,\omega
}^{r}\left[  \eta^{\sigma}\right]  \left(  t\right)  d_{q,\omega}t.
\]
Therefore,
\begin{align*}
\int_{a}^{b}&\left(  \sum_{i=0}^{r+1}f_{i}\left(  t\right)  D_{q,\omega}^{i}
\left[  \eta^{\sigma^{r+1-i}}\right]  \left(  t\right)  \right)  d_{q,\omega}t\\
&=\int_{a}^{b}\left(  \sum_{i=0}^{r}f_{i}\left(  t\right)  D_{q,\omega}^{i}\left[
\eta^{\sigma^{r+1-i}}\right]  \left(  t\right)\right)  d_{q,\omega}t\\
&\qquad -\int_{a}^{b}D_{q,\omega}\left[  f_{r+1}\right]  \left(  t\right)
\left(\frac{1}{q}\right)  ^{r}D_{q,\omega}^{r}\left[  \eta^{\sigma}\right]
\left(t\right)  d_{q,\omega}t\\
&=\int_{a}^{b}
\biggl[
\sum_{i=0}^{r-1}f_{i}\left(  t\right)  D_{q,\omega}^{i}\left[  \left(
\eta^{\sigma}\right)  ^{\sigma^{r-i}}\right]  \left(  t\right)  d_{q,\omega
}t\\
& \qquad +\left(  f_{r}-\left(  \frac{1}{q}\right)  ^{r}D_{q,\omega}\left[
f_{r+1}\right]  \right)  \left(  t\right)  D_{q,\omega}^{r}\left[
\eta^{\sigma}\right]  \left(  t\right)
\biggr]
d_{q,\omega}t.
\end{align*}
By Lemma~\ref{hoanula}, $\eta^{\sigma}$ is a variation.
Hence, using the induction hypothesis,
\begin{align*}
\sum_{i=0}^{r-1}&\left(-1\right)^{i}\left(  \frac{1}{q}\right)^{\frac{\left(
i-1\right)  i}{2}}D_{q,\omega}^{i}\left[  f_{i}\right]
\left(  t\right)\\
& \qquad +\left(  -1\right)  ^{r}\left(  \frac{1}{q}\right)^{\frac{\left(
r-1\right)  r}{2}}D_{q,\omega}^{r}\left[  \left(  f_{r}
-\left(  \frac{1}{q}\right)  ^{r}D_{q,\omega}\left[
f_{r+1}\right]  \right)  \right]  \left(t\right)\\
&=\sum_{i=0}^{r-1}\left(  -1\right)  ^{i}\left(  \frac{1}{q}\right)^{\frac{\left(
i-1\right)  i}{2}}D_{q,\omega}^{i}\left[  f_{i}\right]
\left(  t\right)  +\left(  -1\right)  ^{r}\left(  \frac{1}{q}\right)^{\frac{\left(
r-1\right)  r}{2}}D_{q,\omega}^{r}\left[  f_{r}\right]
\left(  t\right) \\
& \qquad +\left(  -1\right)  ^{r+1}\left(  \frac{1}{q}\right)  ^{\frac{\left(
r-1\right)  r}{2}}\left(  \frac{1}{q}\right)  ^{r}D_{q,\omega}^{r}\left[  D_{q,\omega}\left[
f_{r+1}\right]  \right]  \left(  t\right) \\
&=0
\end{align*}
for all $t\in\left[  a,b\right]_{q,\omega}$, which leads to
\[
\sum_{i=0}^{r+1}\left(  -1\right)  ^{i}\left(  \frac{1}{q}\right)
^{\frac{\left(  i-1\right)  i}{2}}D_{q,\omega}^{i}\left[  f_{i}\right]
\left(  t\right)  =0\text{, \ }t\in\left[  a,b\right]  _{q,\omega}.
\]
\end{proof}


\subsection{Higher-order Hahn's quantum Euler--Lagrange equation}
\label{sub:sec:HOELHCV}

For a variation $\eta$ and an admissible function $y$, we define the function
$\phi:\left]-\bar{\epsilon},\bar{\epsilon}\right[  \rightarrow \mathbb{R}$ by
\[
\phi\left(  \epsilon\right)
:=\mathcal{L}\left[  y+\epsilon\eta\right]  .
\]
The first variation of the variational problem \eqref{hoP} is defined by
\[
\delta\mathcal{L}\left[  y,\eta\right]  :=
\phi^{\prime}\left(  0\right)  .
\]
Observe that
\begin{align*}
\mathcal{L}\left[  y+\epsilon\eta\right]
=&\int_{a}^{b}L\Bigg(t,y^{\sigma^{r}}\left(  t\right)
+\epsilon\eta^{\sigma^{r}}\left(  t\right),
D_{q,\omega}\left[  y^{\sigma^{r-1}}\right]\left(t\right)
+\epsilon D_{q,\omega}\left[  \eta^{\sigma^{r-1}}\right]  \left(  t\right),\\
& \qquad\qquad\qquad \ldots,
D_{q,\omega}^{r}\left[  y\right]  \left(  t\right)
+\epsilon D_{q,\omega}^{r}\left[
\eta\right]  \left(  t\right)  \Bigg)d_{q,\omega}t\\
&=\mathcal{L}_{b}\left[y+\epsilon\eta\right]
-\mathcal{L}_{a}\left[  y+\epsilon\eta\right]
\end{align*}
with
\begin{multline*}
\mathcal{L}_{\xi}\left[  y+\epsilon\eta\right]
=\int_{\omega_{0}}^{\xi}L\Bigg(t,y^{\sigma^{r}}\left(  t\right)
+\epsilon\eta^{\sigma^{r}}\left(t\right), D_{q,\omega}\left[
y^{\sigma^{r-1}}\right]  \left(  t\right)
+\epsilon D_{q,\omega}\left[  \eta^{\sigma^{r-1}}\right]  \left(  t\right),\\
\ldots, D_{q,\omega}^{r}\left[  y\right]  \left(  t\right)
+\epsilon D_{q,\omega}^{r}\left[  \eta\right]  \left(  t\right)\Bigg)d_{q,\omega}t,
\end{multline*}
$\xi \in \{a, b\}$. Therefore,
\begin{equation}
\label{hoeq:le}
\delta\mathcal{L}\left[  y,\eta\right]  =\delta\mathcal{L}_{b}\left[
y,\eta\right]  -\delta\mathcal{L}_{a}\left[  y,\eta\right] .
\end{equation}

The following definition and lemma are important for our purposes.

\begin{definition}
Let $s \in I$ and $g:I\times]-\bar{\theta},\bar{\theta}[ \rightarrow \mathbb{R}$.
We say that $g\left(t,\cdot\right)$ is differentiable at $\theta_{0}$ uniformly
in $\left[s\right]_{q,\omega}$ if, for every $\varepsilon>0$, there exists
$\delta>0$ such that
\[
0<\left\vert \theta-\theta_{0}\right\vert <\delta
\Rightarrow
\left\vert \frac{g\left(  t,\theta\right)  -g\left(  t,\theta_{0}\right)
}{\theta-\theta_{0}}-\partial_{2} g\left(  t,\theta_{0}\right)  \right\vert
<\varepsilon
\]
for all $t\in\left[  s\right]_{q,\omega}$,
where $\displaystyle\partial_{2}g=\frac{\partial g}{\partial\theta}$.
\end{definition}

\begin{lemma}[\textrm{cf.} \cite{Malinowska:3}]
\label{hoderivada do integral}
Let $s \in I$. Assume that
$g:I\times]-\bar{\theta},\bar{\theta}[ \rightarrow \mathbb{R}$
is differentiable at $\theta_{0}$ uniformly in $\left[s\right]_{q,\omega}$,
and $\displaystyle\int_{\omega_{0}}^{s}\partial_{2}g\left(t,\theta_{0}\right)
d_{q,\omega}t$ exist. Then,
$$
G\left(  \theta\right):=\int_{\omega_{0}}^{s}g\left(t,\theta\right) d_{q,\omega}t,
$$
for $\theta$ near $\theta_{0}$, is differentiable at
$\theta_{0}$ with
$$
G^{\prime}\left(  \theta_{0}\right)
=\displaystyle\int_{\omega_{0}}^{s}\partial_{2}g\left(t,\theta_{0}\right) d_{q,\omega}t.
$$
\end{lemma}

Considering \eqref{hoeq:le}, the following lemma is a direct
consequence of Lemma~\ref{hoderivada do integral}.

\begin{lemma}
\label{houniformly}
For a variation $\eta$ and an admissible function $y$, let
\begin{multline*}
g\left(  t,\epsilon\right) :=L\bigg(t,y^{\sigma^{r}}\left(  t\right)
+\epsilon\eta^{\sigma^{r}}\left(  t\right)  ,D_{q,\omega}\left[
y^{\sigma^{r-1}}\right]  \left(  t\right)  +\epsilon D_{q,\omega}\left[
\eta^{\sigma^{r-1}}\right]  \left(  t\right),\\
\ldots, D_{q,\omega}^{r}\left[  y\right]  \left(  t\right)
+\epsilon D_{q,\omega}^{r}\left[  \eta\right]  \left(  t\right)  \bigg),
\end{multline*}
$\epsilon\in\left]-\bar{\epsilon},\bar{\epsilon}\right[$. Assume that

\noindent (1) $g\left(  t,\cdot\right)$ is differentiable at $\omega_0$ uniformly in
$\left[a,b\right]_{q,\omega}$;

\noindent (2) $\mathcal{L}_{a}\left[  y+\epsilon\eta\right]
=\displaystyle\int_{\omega_{0}}^{a}g\left(
t,\epsilon\right)  d_{q,\omega}t $ and $\mathcal{L}_{b}\left[
y+\epsilon\eta\right]=\displaystyle\int_{\omega_{0}}^{b}g\left(
t,\epsilon\right)  d_{q,\omega}t$ exist for $\epsilon \approx 0$;

\noindent (3) $\displaystyle\int_{\omega_{0}}^{a}\partial_{2}g\left(  t,0\right)
d_{q,\omega}t$ and $\displaystyle\int_{\omega_{0}}^{b}\partial_{2}g\left(
t,0\right)  d_{q,\omega}t$ exist.

\noindent Then,
\begin{multline*}
\phi^{\prime}\left(  0\right) = \delta\mathcal{L}\left[  y,\eta\right]
= \int_{a}^{b}\bigg[  \sum_{i=0}^{r}\partial_{i+2}L\left(t,
y^{\sigma^{r}}\left(  t\right), D_{q,\omega}\left[
y^{\sigma^{r-1}}\right]  \left(  t\right),
\ldots, D_{q,\omega}^{r}\left[  y\right]  \left(  t\right)\right)\\
\cdot D_{q,\omega}^{i}\left[
\eta^{\sigma^{r-i}}\right]  \left(  t\right)\bigg] d_{q,\omega}t.
\end{multline*}
\end{lemma}

The following result gives a necessary condition of Euler--Lagrange type
for an admissible function to be a local extremizer
for \eqref{hoP}.\index{Higher-order Hahn's Quantum \\ Euler--Lagrange Equation}

\begin{theorem}[Higher-order Hahn's quantum Euler--Lagrange equation]
\label{hoHigher order E-L}
Under hypotheses (H1)--(H3) and conditions (1)--(3)
of Lemma~\ref{houniformly} on the Lagrangian $L$, if
$y_{\ast}\in \mathcal{Y}^{r}$ is a local extremizer for problem \eqref{hoP},
then $y_{\ast}$ satisfies the Hahn quantum Euler--Lagrange equation

\begin{eqnarray}
&&\sum_{i=0}^{r}\left(  -1\right)  ^{i}\left(  \frac{1}{q}\right)^{\frac{\left(
i-1\right)  i}{2}}D_{q,\omega}^{i}
\left[\tau\rightarrow
\partial_{i+2}L\left(\tau,y^{\sigma^{r}}\left(  \tau\right), D_{q,\omega}\left[
y^{\sigma^{r-1}}\right]  \left(  \tau\right),
\ldots, D_{q,\omega}^{r}\left[  y\right]\left(\tau\right)\right)\right]\left(  t\right) \notag \\
&&=0\label{hoeq:E-L}
\end{eqnarray}
for all $t\in\left[  a,b\right]_{q,\omega}$.
\end{theorem}

\begin{proof}
Let $y_{\ast}$ be a local extremizer for problem
\eqref{hoP} and $\eta$ a variation. Define
$\phi:]-\bar{\epsilon},\bar{\epsilon}[ \rightarrow \mathbb{R}$
by $\phi\left(  \epsilon\right)  :=\mathcal{L}\left[  y_{\ast}+\epsilon\eta\right]$.
A necessary condition for $y_{\ast}$ to be an extremizer is given by
$\phi^{\prime}\left(  0\right)  =0$. By Lemma~\ref{houniformly} we conclude that
\begin{multline*}
\int_{a}^{b}\Biggl[  \sum_{i=0}^{r}\partial_{i+2}
L\left(  t,y^{\sigma^{r}}\left(  t\right)  ,D_{q,\omega}\left[
y^{\sigma^{r-1}}\right]  \left(  t\right),
\ldots, D_{q,\omega}^{r}\left[y\right] \left(  t\right)  \right)\\
\cdot D_{q,\omega}^{i}\left[\eta^{\sigma^{r-i}}\right]\left(
t\right)  \Biggr]  d_{q,\omega}t=0
\end{multline*}
and equation \eqref{hoeq:E-L} follows from Lemma~\ref{hoordem n}.
\end{proof}

\begin{remark}
In practical terms the hypotheses of Theorem~\ref{hoHigher order E-L}
are not so easy to verify \emph{a priori}. One can, however,
assume that all hypotheses are satisfied and apply the
Hahn quantum Euler--Lagrange equation \eqref{hoeq:E-L} heuristically
to obtain a \emph{candidate}. If such a candidate is, or not,
a solution to problem \eqref{hoP} is a different question that always requires
further analysis (see an example in Section~\ref{subsec:Ex}).
\end{remark}

When $\omega= 0$ one obtains from \eqref{hoeq:E-L}
the \emph{higher-order $q$-Euler--Lagrange equation}\index{Higher-order $q$-Euler--Lagrange equation}:
\[
\sum_{i=0}^{r}\left(  -1\right)  ^{i}\left(  \frac{1}{q}\right)
^{\frac{\left(  i-1\right)  i}{2}}D_{q}^{i}\left[ \tau\rightarrow \partial_{i+2}
L\left(  \tau,y^{\sigma^{r}}\left(  \tau\right)  ,D_{q}\left[
y^{\sigma^{r-1}}\right]  \left(  \tau\right),
\ldots, D_{q}^{r}\left[  y\right]
\left(  \tau\right)  \right)  \right]\left(  t\right)  =0
\]
for all $t\in\left\{  aq^{n}:n\in
\mathbb{N}_{0}\right\}  \cup\left\{  bq^{n}:n\in
\mathbb{N}_{0}\right\}  \cup\left\{  0\right\}$.
The \emph{higher-order $h$-Euler--Lagrange equation}\index{Higher-order $h$-Euler--Lagrange equation}
is obtained from \eqref{hoeq:E-L} taking the limit $q\rightarrow 1$:
\[
\sum_{i=0}^{r}\left(  -1\right)  ^{i}
\Delta_{h}^{i}\left[\tau\rightarrow  \partial_{i+2} L\left(\tau,
y^{\sigma^{r}}\left(  \tau\right),\Delta_{h}\left[
y^{\sigma^{r-1}}\right]  \left(  \tau\right),
\ldots, \Delta_{h}^{r}\left[  y\right]  \left(  \tau\right)\right)\right]\left(  t\right)=0
\]
for all $t\in\left\{  a+nh:n\in
\mathbb{N}_{0}\right\}
\cup\left\{  b+nh:n\in \mathbb{N}_{0}\right\}$.
The \emph{classical higher-order Euler--Lagrange equation}
\cite{Brunt}\index{Higher-order Euler--Lagrange equation}
is recovered when $(\omega, q)\rightarrow (0, 1)$:
$$
\sum_{i=0}^{r}\left(-1\right)^{i}
\frac{d^i}{d t^i}\left[\tau\rightarrow\partial_{i+2}L\left(\tau,
y\left(\tau\right),y'\left(\tau\right),
\ldots, y^{(r)}(\tau)\right)\right]\left(  t\right)=0
$$
for all $t \in [a,b]$.

We now illustrate the usefulness
of our Theorem~\ref{hoHigher order E-L}
by means of an example that is not covered
by previous available results in the literature.


\subsection{An example}
\label{subsec:Ex}

Let $q=\frac{1}{2}$ and $\omega=\frac{1}{2}$. Consider the following problem:
\begin{equation}
\label{example}
\mathcal{L}\left[  y\right]
=  \int_{-1}^{1}\left(y^\sigma(t) + \frac{1}{2}\right)^2  \left(
\left( D_{q,\omega}\left[  y\right] (t)\right)^2 -1 \right)^2 d_{q,\omega}t
\longrightarrow \min
\end{equation}
over all $y \in \mathcal{Y}^1$ satisfying the boundary conditions
\begin{equation}
\label{eq:bc:ex}
y(-1)=0 \ \ \ \text{and} \ \ \ y(1)=-1.
\end{equation}
This is an example of problem \eqref{hoP} with $r=1$.
Our Hahn's quantum Euler--Lagrange equation \eqref{hoeq:E-L}
takes the form
\[
D_{q,\omega}\left[\tau\rightarrow\partial_{3}L\left(
\tau,y^{\sigma}\left(  \tau\right),
D_{q,\omega}\left[  y\right]  \left(\tau\right)  \right)\right]\left(  t\right)
=\partial_{2}L\left(  t,y^{\sigma} \left(  t\right)  ,D_{q,\omega}\left[
y\right]  \left(  t\right)  \right).
\]
Therefore, we look for an admissible function $y_{\ast}$
of \eqref{example}--\eqref{eq:bc:ex} satisfying
\begin{multline}
\label{hoeq:EL:ex}
D_{q,\omega}\left[4 \left(y^\sigma
+ \frac{1}{2}\right)^2\left( \left( D_{q,\omega}\left[
y\right]\right)^2 -1 \right)D_{q,\omega}\left[  y\right] \right](t)\\
= 2\left(y^\sigma(t) + \frac{1}{2}\right) \left(
\left( D_{q,\omega}\left[  y\right](t)\right)^2 -1 \right)^2
\end{multline}
for all $ t\in\left[  -1,1\right]_{q,\omega}$. It is easy to see that
\begin{equation*}
y_{\ast}(t) =
\begin{cases}
-t & \text{ if } t \in ]-1,0[\cup ]0,1]\\
0 & \text{ if } t=-1\\
1 & \text{ if } t=0
\end{cases}
\end{equation*}
is an admissible function for \eqref{example}--\eqref{eq:bc:ex} with
\begin{equation*}
D_{q,\omega}\left[  y_{\ast}\right] (t) =
\begin{cases}
-1 & \text{ if } t \in ]-1,0[\cup ]0,1]\\
1 & \text{ if } t=-1\\
-3 & \text{ if }  t=0,
\end{cases}
\end{equation*}
satisfying the Hahn quantum Euler--Lagrange equation \eqref{hoeq:EL:ex}.
We now prove that the \emph{candidate} $y_{\ast}$ is indeed a minimizer
for \eqref{example}--\eqref{eq:bc:ex}.
Note that here $\omega_0=1$ and, by Lemma~\ref{hopositividade},
\begin{equation*}
\mathcal{L}\left[  y\right] = \int_{-1}^{1}\left(y^\sigma(t) + \frac{1}{2}\right)^2\left(
\left( D_{q,\omega}\left[  y\right] (t)\right)^2 -1 \right)^2 d_{q,\omega}t   \geqslant 0
\end{equation*}
for all admissible functions
$y \in \mathcal{Y}^1\left(\left[-1,1\right],\mathbb{R}\right)$.
Since $\mathcal{L}\left[  y_{\ast}\right]=0$,
we conclude that $y_{\ast}$ is a minimizer
for problem \eqref{example}--\eqref{eq:bc:ex}.

It is worth to mention that the minimizer $y_{\ast}$ of \eqref{example}--\eqref{eq:bc:ex}
is not continuous while the classical calculus of variations \cite{Brunt},
the calculus of variations on time scales \cite{Ferreira,Malinowska,Martins},
or the nondifferentiable scale variational calculus \cite{Almeida:2,Almeida:4,Cresson},
deal with functions which are necessarily continuous.
As an open question, we pose the problem of determining
conditions on the data of problem \eqref{hoP} ensuring, \emph{a priori},
the minimizer to be regular.


\section{State of the Art}

The results of this chapter are already published in \cite{Brito:da:Cruz} and were
presented by the author at EUROPT Workshop ``Advances in Continuous Optimization'',
July 9-10, 2010, University of Aveiro, Portugal.
It is worth mentioning that nowadays other
researchers are dedicating their time to the development of the theory
of Hahn's quantum calculus (see \cite{Aldwoah,Aldwoah:2,Koekoek,Malinowska:5}
and references within). The theory of calculus of variations within the Hahn
quantum calculus is quite new and was first presented
by Malinowska and Torres in 2010 \cite{Malinowska:3}.


\clearpage{\thispagestyle{empty}\cleardoublepage}


\chapter{A Symmetric Quantum Calculus}
\label{A Symmetric Quantum Calculus}

In this chapter we present a symmetric quantum calculus.
We define and prove the properties of
the $\alpha,\beta$-symmetric derivative (Section~\ref{hs:sec:3.1});
Section~\ref{hs:sec:3.2} is devoted to the development of the
$\alpha,\beta$-symmetric N\"{o}rlund sum as well to some
of its properties.
Section~\ref{hs:sec:3.3} is dedicated to mean value theorems for the
$\alpha,\beta$-symmetric calculus: we prove $\alpha,\beta$-symmetric
versions of Fermat's theorem for stationary points,
Rolle's, Lagrange's, and Cauchy's mean value theorems.
In Section~\ref{hi:sec:ineq} we present and
prove $\alpha,\beta$-symmetric versions of H\"{o}lder's,
Cauchy--Schwarz's and Minkowski's inequalities.


\section{Introduction}

Quantum derivatives and integrals play a leading role
in the understanding of complex physical systems.
The subject has been under strong development
since the beginning of the 20th century
\cite{Bangerezako,Cresson,Dobrogowska,Ernst:1meio,Ernst:2,Hahn,Jackson,Kac,Lavagno2,
Malinowska:5,Malinowska:3,Martins:2,Martins:3,Milne}.
Roughly speaking, two approaches to quantum calculus are available.
The first considers the set of points of study to be the lattice
$\overline{q^{\mathbb{Z}}}$ or $h\mathbb{Z}$ and is nowadays
part of the more general time scale calculus
\cite{Agarwal:2,Bohner,Bohner:2,Hilger,Malinowska:2};
the second uses the same formulas for the quantum
derivatives but the set of study is the set $\mathbb{R}$
\cite{Aldwoah,Aldwoah:2,Aldwoah:3,Almeida:5,Brito:da:Cruz,Malinowska:3}.
Here we take the second perspective.

Given a function $f$ and a positive real number $h$,
the $h$-derivative of $f$ at $x$ is defined by the ratio
$$\frac{f\left(x+h\right)-f\left(x\right)}{h}.$$
When $h\rightarrow0$, one obtains the usual
derivative of $f$ at $x$. The symmetric $h$-derivative of $f$ at $x$ ($h>0$) is defined by
$$
\frac{f\left(  x+h\right)  -f\left(  x-h\right)}{2h},
$$
which coincides with the standard symmetric derivative \cite{Thomson}
when we let $h\rightarrow 0$. The notion of symmetrically differentiable is interesting
because if a function is differentiable at a point then
it is also symmetrically differentiable, but the converse
is not true. The best known example of this fact
is the absolute value function: $f(x) = |x|$ is not
differentiable at $x = 0$ but is symmetrically differentiable
at $x = 0$ with symmetric derivative zero \cite{Thomson}.

The aim of this chapter is to introduce the $\alpha,\beta$-symmetric difference derivative
and N\"{o}rlund sum, and then develop the associated calculus.
Such $\alpha,\beta$-symmetric calculus gives
a generalization to (both forward and backward) quantum $h$-calculus.


\section{Forward and backward N\"{o}rlund sums}
\label{hi:sec:2}

In what follows we denote by $\left\vert I\right\vert $
the measure of the interval $I$.

\begin{definition}
Let $\alpha$ and $\beta$ be two positive real numbers,
$I\subseteq\mathbb{R}$ be an interval
with $\left\vert I\right\vert >\alpha$,
and $f:I\rightarrow\mathbb{R}$.
The \emph{$\alpha$-forward difference
operator} \index{$\alpha$-forward difference
operator} $\Delta_{\alpha}$ is defined by
\[
\Delta_{\alpha}\left[  f\right]  \left(  t\right)  :=\frac{f\left(
t+\alpha\right)  -f\left(  t\right)  }{\alpha}
\]
for all $t\in I\backslash\left[\sup I-\alpha,\sup I\right]$,
in case $\sup I$ is finite, or, otherwise, for all $t \in I$.
Similarly, for $\left\vert I\right\vert >\beta$
the \emph{$\beta$-backward difference operator}\index{$\beta$-backward difference operator}
$\nabla_{\beta}$ is defined by
\[
\nabla_{\beta}\left[  f\right]  \left(  t\right)  :=\frac{f\left(  t\right)
-f\left(t-\beta\right)}{\beta}
\]
for all $t\in I\backslash\left[  \inf I,\inf I+\beta\right]$,
in case $\inf I$ is finite, or, otherwise, for all $t\in I$.
We call to $\Delta_{\alpha}\left[  f\right]$ the \emph{$\alpha$-forward
difference derivative} \index{$\alpha$-forward
difference derivative} of $f$ and to $\nabla_{\beta}\left[f\right]$
the \emph{$\beta$-backward difference derivative}\index{$\beta$-backward difference derivative} of $f$.
\end{definition}

This section is dedicated to the inverse operators of
the $\alpha$-forward and $\beta$-backward difference operators.

\begin{definition}
\label{hi:def:o1}
Let $I \subseteq \mathbb{R}$ be such that
$a,b\in I$ with $a<b$ and $\sup I=+\infty$.
For $f:I\rightarrow\mathbb{R}$ and $\alpha >0$ we define the \emph{N\"{o}rlund sum} \index{N\"{o}rlund sum}
(the \emph{$\alpha$-forward integral} \index{$\alpha$-forward integral}) of $f$ from $a$ to $b$ by
\[
\int_{a}^{b}f\left(  t\right)  \Delta_{\alpha}t=\int_{a}^{+\infty}f\left(
t\right)  \Delta_{\alpha}t-\int_{b}^{+\infty}f\left(  t\right)
\Delta_{\alpha}t,
\]
where
$$
\int_{x}^{+\infty}f\left(  t\right)\Delta_{\alpha}t
=\alpha\sum_{k=0}^{+\infty}f\left(  x+k\alpha\right),
$$
provided the series converges at $x=a$ and $x=b$. In that case, $f$ is
said to be $\alpha$-forward integrable on $\left[  a,b\right]$. We say that $f$
is $\alpha$-forward integrable over $I$ if it is $\alpha$-forward integrable
for all $a,b\in I$.
\end{definition}

\begin{remark}
If $f:I\rightarrow\mathbb{R}$
is a function such that $\sup I<+\infty$,
then we can easily extend $f$ to $\tilde{f}:\tilde{I}\rightarrow\mathbb{R}$
with $\sup\tilde{I}=+\infty$ by letting
$\tilde{f}|_{I}=f$ and $\tilde{f}|_{\tilde{I}\backslash I}=0$.
\end{remark}

\begin{remark}
Definition~\ref{hi:def:o1} is valid for any two real points $a,b$
and not only for points belonging to $\alpha\mathbb{Z}$.
This is in contrast with the theory of time scales \cite{Agarwal:2,Bohner,Bohner:2}.
\end{remark}

Using the same techniques that Aldwoah used in his Ph.D. thesis \cite{Aldwoah},
it can be proved that the $\alpha$-forward integral has the following properties.

\begin{theorem}
If $f,g: I \rightarrow\mathbb{R}$ are
$\alpha$-forward integrable on $[a,b]$,
$c\in\left[a,b\right]$, $k\in\mathbb{R}$, then
\begin{enumerate}
\item $\displaystyle\int_{a}^{a}f\left(  t\right)\Delta_{\alpha}t=0$;

\item $\displaystyle\int_{a}^{b}f\left(  t\right)\Delta_{\alpha}t
=\int_{a}^{c}f\left(  t\right)  \Delta_{\alpha}t+\int_{c}^{b}f\left(  t\right)
\Delta_{\alpha}t$, when the integrals exist;

\item $\displaystyle\int_{a}^{b}f\left(  t\right)\Delta_{\alpha}t
=-\int_{b}^{a}f\left(  t\right)  \Delta_{\alpha}t$;

\item $kf$ is $\alpha$-forward integrable on $\left[a,b\right]$ and
$\displaystyle \int_{a}^{b}kf\left(t\right)\Delta_{\alpha}t
=k\int_{a}^{b}f\left(t\right)  \Delta_{\alpha}t$;

\item $f+g$ is $\alpha$-forward integrable on $\left[a,b\right]$ and
\[
\int_{a}^{b}\left(  f+g\right)  \left(  t\right)  \Delta_{\alpha}t=\int
_{a}^{b}f\left(  t\right)  \Delta_{\alpha}t+\int_{a}^{b}g\left(  t\right)
\Delta_{\alpha}t\text{;}
\]

\item if $f\equiv0$, then $\displaystyle\int_{a}^{b}f\left(  t\right)
\Delta_{\alpha}t=0$.
\end{enumerate}
\end{theorem}

\begin{remark}
Since for $a>b$ we have
\begin{align*}
\int_{a}^{b}f\left(  t\right)  \Delta_{\alpha}t
&=-\left(\int_{b}^{a}f\left(  t\right)  \Delta_{\alpha}t\right)\\
&=-\left(\int_{b}^{+\infty}f\left(t\right)  \Delta_{\alpha}t
-\int_{a}^{+\infty}f\left(  t\right)\Delta_{\alpha}t\right)\\
&=\int_{a}^{+\infty}f\left(
t\right)  \Delta_{\alpha}t-\int_{b}^{+\infty}f\left(  t\right)\Delta_{\alpha}t,
\end{align*}
then we could defined the N\"{o}rlund sum (Definition~\ref{hi:def:o1})
for any $a,b\in I$ instead for $a<b$.
\end{remark}

\begin{theorem}
Let $f: I \rightarrow\mathbb{R}$ be $\alpha$-forward integrable
on $\left[a,b\right]$.
If $g:I\rightarrow\mathbb{R}$ is a nonnegative
$\alpha$-forward integrable function on $\left[a,b\right]$,
then $fg$ is $\alpha$-forward integrable on $\left[a,b\right]$.
\end{theorem}

\begin{proof}
Since $g$ is $\alpha$-forward integrable, then both series
$$\alpha\sum_{k=0}^{+\infty}g\left(  a+k\alpha\right)
\text{ \ \ and \ \ } \alpha\sum_{k=0}^{+\infty}g\left(  b+k\alpha\right)$$
converge. We want to study the nature of series
$$
\alpha\sum_{k=0}^{+\infty}fg\left(  a+k\alpha\right) \text{ \ \ and \ \ }
\alpha\sum_{k=0}^{+\infty}fg\left(  b+k\alpha\right).
$$
Since there exists an order $N\in\mathbb{N}$ such that
$$\left\vert fg\left(  b+k\alpha\right)  \right\vert \leqslant g\left(
b+k\alpha\right) \text{ \ \ and \ \ } \left\vert fg\left(  a+k\alpha\right)
\right\vert \leqslant g\left(  a+k\alpha\right)$$
for all $k>N$, then both
$$
\alpha\sum_{k=0}^{+\infty}fg\left(a+k\alpha\right) \text{ \ \ and \ \ }
\alpha\sum_{k=0}^{+\infty}fg\left(  b+k\alpha\right)
$$
converge absolutely. The intended conclusion follows.
\end{proof}

\begin{theorem}
\label{hi:p}
Let $f:I\rightarrow\mathbb{R}$ and $p>1$.
If $\left\vert f\right\vert $ is $\alpha$-forward integrable
on $\left[  a,b\right]$, then $\left\vert f\right\vert ^{p}$
is also $\alpha$-forward integrable on $\left[a,b\right]$.
\end{theorem}

\begin{proof}
There exists $N\in\mathbb{N}$ such that
$$\left\vert f\left(b+k\alpha\right)\right\vert^{p}
\leqslant
\left\vert f\left(b+k\alpha\right) \right\vert$$
and
$$\left\vert f\left(a+k\alpha\right)\right\vert^{p}
\leqslant
\left\vert f\left(a+k\alpha\right) \right\vert$$
for all $k>N$. Therefore, $\left\vert f\right\vert^{p}$
is $\alpha$-forward integrable on $\left[a,b\right]$.
\end{proof}

\begin{theorem}
\label{hi:desigualdade}
Let $f,g:I\rightarrow\mathbb{R}$ be $\alpha$-forward integrable on $\left[  a,b\right]$.
If $\left\vert f\left(  t\right)  \right\vert \leqslant g\left(  t\right)$
for all $t\in\left\{  a+k\alpha:k\in\mathbb{N}_{0}\right\}$,
then for $b\in\left\{  a+k\alpha:k\in\mathbb{N}_{0}\right\}$ one has
\[
\left\vert \int_{a}^{b}f\left(  t\right)  \Delta_{\alpha}t\right\vert
\leqslant\int_{a}^{b}g\left(  t\right)  \Delta_{\alpha}t.
\]
\end{theorem}

\begin{proof}
Since $b\in\left\{  a+k\alpha:k\in\mathbb{N}_{0}\right\}$,
there exists $k_{1}$ such that $b=a+k_{1}\alpha$. Thus,
\begin{align*}
\left\vert \int_{a}^{b}f\left(  t\right)  \Delta_{\alpha}t\right\vert  &
=\left\vert \alpha\sum_{k=0}^{+\infty}f\left(  a+k\alpha\right)  -\alpha
\sum_{k=0}^{+\infty}f\left(  a+\left(  k_{1}+k\right)  \alpha\right)
\right\vert \\
&=\left\vert \alpha\sum_{k=0}^{+\infty}f\left(  a+k\alpha\right)  -\alpha
\sum_{k=k_{1}}^{+\infty}f\left(  a+k\alpha\right)  \right\vert
=\left\vert \alpha\sum_{k=0}^{k_{1}-1}f\left(  a+k\alpha\right)  \right\vert\\
&\leqslant \alpha\sum_{k=0}^{k_{1}-1}\left\vert f\left(  a+k\alpha\right)
\right\vert
\leqslant\alpha\sum_{k=0}^{k_{1}-1}g\left(  a+k\alpha\right) \\
& =\alpha\sum_{k=0}^{+\infty}g\left(  a+k\alpha\right)  -\alpha\sum_{k=k_{1}
}^{+\infty}g\left(  a+k\alpha\right)
=\int_{a}^{b}g\left(  t\right)  \Delta_{\alpha}t.
\end{align*}
\end{proof}

\begin{corollary}
\label{hi:desigualdade2}
Let $f,g:I\rightarrow\mathbb{R}$ be
$\alpha$-forward integrable on $\left[a,b\right]$
with $b = a+k\alpha$ for some $k\in\mathbb{N}_{0}$.
\begin{enumerate}
\item If $f\left(  t\right)  \geqslant 0$
for all $t\in\left\{  a+k\alpha:k\in\mathbb{N}_{0}\right\}$,
then $$\int_{a}^{b}f\left(  t\right)  \Delta_{\alpha}t \geqslant 0;$$

\item If $g\left(  t\right)  \geqslant f$ $\left(  t\right)$ for all
$t\in\left\{  a+k\alpha:k\in\mathbb{N}_{0}\right\}$, then
$$\int_{a}^{b}g\left(  t\right)  \Delta_{\alpha}t\geqslant\int_{a}^{b}f\left(
t\right)  \Delta_{\alpha}t.$$
\end{enumerate}
\end{corollary}

We can now prove the following fundamental
theorem of the $\alpha$-forward integral calculus.

\begin{theorem}[Fundamental theorem of N\"{o}rlund calculus]\index{Fundamental theorem of N\"{o}rlund's calculus}
Let $f:I\rightarrow\mathbb{R}$ be $\alpha$-forward integrable
over $I$. Let $a,b,x\in I$ and define
$$
F\left(  x\right)  :=\int_{a}^{x}f\left(  t\right)  \Delta_{\alpha}t.
$$
Then,
$$
\Delta_{\alpha}\left[  F\right]  \left(  x\right)  =f\left(  x\right).
$$
Conversely,
$$
\int_{a}^{b}\Delta_{\alpha}\left[  f\right]  \left(  t\right)  \Delta_{\alpha}t
=f\left(  b\right)  -f\left(  a\right).
$$
\end{theorem}

\begin{proof}
If
$$
G\left(  x\right)
=-\int_{x}^{+\infty}f\left(t\right)\Delta_{\alpha}t,
$$
then
\begin{align*}
\Delta_{\alpha}\left[  G\right]  \left(  x\right)
&=\frac{G\left(x+\alpha\right)  -G\left(  x\right)  }{\alpha}\\
&=\frac{-\alpha\sum_{k=0}^{+\infty}f\left(  x+\alpha+k\alpha\right)
+\alpha\sum_{k=0}^{+\infty}f\left(  x+k\alpha\right)  }{\alpha}\\
& =\sum_{k=0}^{+\infty}f\left(  x+k\alpha\right)
-\sum_{k=0}^{+\infty}f\left(  x+\left(  k+1\right)  \alpha\right)
=f\left(  x\right).
\end{align*}
Therefore,
$$\Delta_{\alpha}\left[  F\right]  \left(  x\right)
=\Delta_{\alpha}\left(\int_{a}^{+\infty}f\left(  t\right)
\Delta_{\alpha}t-\int_{x}^{+\infty}
f\left(  t\right)  \Delta_{\alpha}t\right)
=f\left(  x\right).$$
Using the definition of $\alpha$-forward difference operator, the second part
of the theorem is also a consequence of the properties of Mengoli's series.
Since
\begin{align*}
\int_{a}^{+\infty}\Delta_{\alpha}\left[  f\right]  \left(  t\right)
\Delta_{\alpha}t  & =\alpha\sum_{k=0}^{+\infty}\Delta_{\alpha}\left[
f\right]  \left(  a+k\alpha\right)\\
&=\alpha\sum_{k=0}^{+\infty}\frac{f\left(  a+k\alpha+\alpha\right)  -f\left(
a+k\alpha\right)  }{\alpha}\\
& =\sum_{k=0}^{+\infty}\bigg(f\left(  a+\left(  k+1\right)  \alpha\right)
-f\left(  a+k\alpha\right)  \bigg)\\
& =-f\left(  a\right)
\end{align*}
and
$$
\int_{b}^{+\infty}\Delta_{\alpha}\left[  f\right]  \left(  t\right)
\Delta_{\alpha}t=-f\left(  b\right),
$$
it follows that
\begin{equation*}
\int_{a}^{b}\Delta_{\alpha}\left[  f\right]  \left(  t\right)  \Delta_{\alpha}t
=\int_{a}^{+\infty}f\left(  t\right)  \Delta_{\alpha}t-\int_{b}^{+\infty}
f\left(  t\right)  \Delta_{\alpha}t
=f\left(  b\right)  -f\left(  a\right).
\end{equation*}
\end{proof}

\begin{corollary}[$\alpha$-forward integration by parts]
\label{hi:partes}
Let $f,g:I\rightarrow\mathbb{R}$.\index{$\alpha$-forward integration by parts}
If $f g$ and $f\Delta_{\alpha}\left[  g\right]  $ are $\alpha
$-forward integrable on $\left[  a,b\right]  $, then
\[
\int_{a}^{b}f\left(  t\right)  \Delta_{\alpha}\left[  g\right]  \left(
t\right)  \Delta_{\alpha}t=f\left(  t\right)  g\left(  t\right)
\bigg|_{a}^{b}-\int_{a}^{b}\Delta_{\alpha}\left[  f\right]  \left(  t\right)
g\left(  t+\alpha\right)  \Delta_{\alpha}t
\]
\end{corollary}

\begin{proof}
Since
$$
\Delta_{\alpha}\left[  fg\right]  \left(  t\right)
=\Delta_{\alpha}\left[
f\right]  \left(  t\right)  g\left(  t+\alpha\right)  +f\left(  t\right)
\Delta_{\alpha}\left[  g\right]  \left(  t\right),
$$ then
\begin{align*}
\int_{a}^{b}f\left(  t\right)  \Delta_{\alpha}\left[  g\right]  \left(
t\right)  \Delta_{\alpha}t  & =\int_{a}^{b}\bigg(\Delta_{\alpha}\left[
fg\right]  \left(  t\right)  -\Delta_{\alpha}\left[  f\right]  \left(
t\right)  g\left(  t+\alpha\right)  \bigg)\Delta_{\alpha}t\\
& =\int_{a}^{b}\Delta_{\alpha}\left[  fg\right]  \left(  t\right)
\Delta_{\alpha}t-\int_{a}^{b}\Delta_{\alpha}\left[  f\right]  \left(
t\right)  g\left(  t+\alpha\right)  \Delta_{\alpha}t\\
& =f\left(  t\right)  g\left(  t\right)  \bigg|_{a}^{b}-\int_{a}^{b}
\Delta_{\alpha}\left[  f\right]  \left(  t\right)  g\left(  t+\alpha\right)
\Delta_{\alpha}t\text{.}
\end{align*}
\end{proof}

\begin{remark}
Our study of the N\"{o}rlund sum is in agreement with
the Hahn quantum calculus \cite{Aldwoah,Brito:da:Cruz,Malinowska:3}.
In \cite{Kac} the N\"{o}rlund sum is defined by
$$\int_{a}^{b}f\left(  t\right)  \Delta_{\alpha}t=\alpha\left[  f\left(
a\right)  +f\left(  a+\alpha\right)
+ \cdots +f\left(  b-\alpha\right)  \right]$$
for $a<b$ such that $b-a\in\alpha\mathbb{Z}$,
$\alpha\in\mathbb{R}^{+}$. In contrast with \cite{Kac}, our definition is valid
for any two real points $a,b$ and not only for those points belonging to the time
scale $\alpha\mathbb{Z}$. The definitions (only) coincide
if the function $f$ is $\alpha$-forward
integrable on $\left[a,b\right]$.
\end{remark}

Similarly, one can introduce the $\beta$-backward integral.

\begin{definition}
\label{hi:beta}
Let $I$ be an interval of $\mathbb{R}$ such that $a,b\in I$
with $a<b$ and $\inf I=-\infty$. For $f:I\rightarrow\mathbb{R}$
and $\beta >0$ we define the
\emph{$\beta$-backward integral}\index{$\beta$-backward integral}
of $f$ from $a$ to $b$ by
\[
\int_{a}^{b}f\left(  t\right)  \nabla_{\beta}t=\int_{-\infty}^{b}f\left(
t\right)  \nabla_{\beta}t-\int_{-\infty}^{a}f\left(  t\right)  \nabla_{\beta}t,
\]
where
$$
\int_{-\infty}^{x}f\left(  t\right)\nabla_{\beta}t
=\beta\sum_{k=0}^{+\infty}f\left(  x-k\beta\right),
$$
provided the series converges at $x=a$ and $x=b$. In that case, $f$ is
called $\beta$-backward integrable on $\left[a,b\right]$. We say that $f$
is $\beta$-backward integrable over $I$ if it is $\beta$-backward integrable
for all $a,b\in I$.
\end{definition}

The $\beta$-backward N\"{o}rlund sum has similar results and properties as the
$\alpha$-forward N\"{o}rlund sum. In particular, the $\beta$-backward integral
is the inverse operator of $\nabla_\beta$.


\section{The $\alpha,\beta$-symmetric quantum calculus}
\label{hs:sec:3}

We begin by introducing in Section~\ref{hs:sec:3.1}
the $\alpha,\beta$-symmetric derivative;
in Section~\ref{hs:sec:3.2} we define
the $\alpha,\beta$-symmetric N\"{o}rlund sum;
Section~\ref{hs:sec:3.3} is dedicated to mean value theorems
for the new $\alpha,\beta$-symmetric calculus and in the last
Section~\ref{hi:sec:ineq} we prove some
$\alpha,\beta$-Symmetric integral inequalities.


\subsection{The $\alpha,\beta$-symmetric derivative}
\label{hs:sec:3.1}

In what follows, $\alpha,\beta\in\mathbb{R}_{0}^{+}$
with at least one of them positive and $I$ is an interval such that
$\left\vert I\right\vert >\max\left\{  \alpha,\beta\right\}$.
We denote by $I_{\beta}^{\alpha}$ the set
\[
I_{\beta}^{\alpha}=\left\{
\begin{array}
[c]{ccc}
I\backslash\left(  \left[  \inf I,\inf I+\beta\right]  \cup\left[  \sup
I-\alpha,\sup I\right]  \right) & \text{if} & \inf I\neq-\infty \wedge \sup I\neq+\infty\\
\\
I\backslash\left(  \left[  \inf I,\inf I+\beta\right]  \right)  & \text{if}
& \inf I\neq-\infty\wedge\sup I=+\infty\\
\\
I\backslash\left(  \left[  \sup I-\alpha,\sup I\right]  \right)  & \text{if}
& \inf I=-\infty\wedge\sup I\neq+\infty\\
\\
I & \text{if} & \inf I=-\infty\wedge\sup I=+\infty.
\end{array}
\right.
\]

\begin{definition}
\label{hs:def:s:ab:dd}
The \emph{$\alpha,\beta$-symmetric difference
derivative}\index{$\alpha,\beta$-symmetric difference derivative}
of $f:I\rightarrow\mathbb{R}$ is given by
$$
D_{\alpha,\beta}\left[  f\right]  \left(  t\right)  =\frac{f\left(
t+\alpha\right)  -f\left(  t-\beta\right)}{\alpha+\beta}
$$
for all $t\in I_{\beta}^{\alpha}$.
\end{definition}

\begin{remark}
The $\alpha,\beta$-symmetric difference operator
is a generalization of both the $\alpha$-forward
and the $\beta$-backward difference operators.
Indeed, the $\alpha$-forward difference operator is obtained
for $\alpha>0$ and $\beta=0$; while for $\alpha=0$ and $\beta>0$
we obtain the $\beta$-backward difference operator.
\end{remark}

\begin{remark}
The classical symmetric derivative \cite{Thomson}
is obtained by choosing $\beta = \alpha$ and taking the limit
$\alpha \rightarrow 0$. When $\alpha=\beta=h > 0$, we call
$h$-symmetric derivative to the $\alpha,\beta$-symmetric difference operator.
\end{remark}

\begin{remark}
\label{hs:rem:lc:der}
If $\alpha,\beta\geqslant 0$ with $\alpha + \beta >0$, then
$$
D_{\alpha,\beta}\left[  f\right]\left(t\right)
=\frac{\alpha}{\alpha+\beta}\Delta_{\alpha}\left[  f\right]\left(t\right)
+\frac{\beta}{\alpha+\beta}\nabla_{\beta}\left[  f\right]\left(t\right),
$$
where $\Delta_{\alpha}$ and $\nabla_{\beta}$ are, respectively,
the $\alpha$-forward and the $\beta$-backward difference operators.
\end{remark}

The symmetric difference operator has the following properties.

\begin{theorem}
Let $f,g:I\rightarrow\mathbb{R}$ and $c,\lambda\in\mathbb{R}$.
For all $t\in I_{\beta}^{\alpha}$ one has:

\begin{enumerate}
\item $D_{\alpha,\beta}\left[  c\right]  \left(  t\right)  =0$;

\item $D_{\alpha,\beta}\left[  f+g\right]  \left(  t\right)=D_{\alpha,\beta
}\left[  f\right]  \left(  t\right)  +D_{\alpha,\beta}\left[  g\right]
\left(  t\right)$;

\item $D_{\alpha,\beta}\left[  \lambda f\right]  \left(  t\right)
=\lambda D_{\alpha,\beta}\left[  f\right]  \left(  t\right)$;

\item $D_{\alpha,\beta}\left[  fg\right]  \left(  t\right)  =D_{\alpha,\beta
}\left[  f\right]  \left(  t\right)  g\left(  t+\alpha\right)  +f\left(
t-\beta\right)  D_{\alpha,\beta}\left[  g\right]  \left(  t\right)$;

\item $D_{\alpha,\beta}\left[  fg\right]  \left(  t\right)  =D_{\alpha,\beta
}\left[  f\right]  \left(  t\right)  g\left(  t-\beta\right)  +f\left(
t+\alpha\right)  D_{\alpha,\beta}\left[  g\right]  \left(  t\right)$;

\item $\displaystyle D_{\alpha,\beta}\left[  \frac{f}{g}\right]  \left(
t\right)  =\frac{D_{\alpha,\beta}\left[  f\right]  \left(  t\right)  g\left(
t-\beta\right)  -f\left(  t-\beta\right)  D_{\alpha,\beta}\left[  g\right]
\left(  t\right)  }{g\left(  t+\alpha\right)  g\left(  t-\beta\right)}$
\newline \newline provided $g\left(  t+\alpha\right)  g\left(  t-\beta\right)
\neq 0$;

\item $\displaystyle D_{\alpha,\beta}\left[  \frac{f}{g}\right]  \left(
t\right)  =\frac{D_{\alpha,\beta}\left[  f\right]  \left(  t\right)  g\left(
t+\alpha\right)  -f\left(  t+\alpha\right)  D_{\alpha,\beta}\left[  g\right]
\left(  t\right)  }{g\left(  t+\alpha\right)  g\left(  t-\beta\right)  }$
\newline \newline provided $g\left(  t+\alpha\right)  g\left(  t-\beta\right)
\neq0$.
\end{enumerate}
\end{theorem}

\begin{proof}
Property 1 is a trivial consequence of Definition~\ref{hs:def:s:ab:dd}.
Properties 2, 3 and 4 follow by direct computations:
\begin{align*}
D_{\alpha,\beta}\left[  f+g\right]  \left(  t\right)   & =\frac{\left(
f+g\right)  \left(  t+\alpha\right)  -\left(  f+g\right)  \left(
t-\beta\right)  }{\alpha+\beta}\\
& =\frac{f\left(  t+\alpha\right)  -f\left(  t-\beta\right)  }{\alpha+\beta
}+\frac{g\left(  t+\alpha\right)  -g\left(  t-\beta\right)  }{\alpha+\beta}\\
& =D_{\alpha,\beta}\left[  f\right]  \left(  t\right)  +D_{\alpha,\beta
}\left[  g\right]  \left(  t\right);
\end{align*}
\begin{align*}
D_{\alpha,\beta}\left[  \lambda f\right]  \left(  t\right)   & =\frac{\left(
\lambda f\right)  \left(  t+\alpha\right)  -\left(  \lambda f\right)  \left(
t-\beta\right)  }{\alpha+\beta}\\
&=\lambda\frac{f\left(  t+\alpha\right)  -f\left(  t-\beta\right)}{\alpha+\beta}\\
& =\lambda D_{\alpha,\beta}\left[  f\right]  \left(  t\right);
\end{align*}
\begin{align*}
D_{\alpha,\beta}\left[  fg\right]  \left(  t\right)   & =\frac{\left(
fg\right)  \left(  t+\alpha\right)  -\left(  fg\right)  \left(  t-\beta
\right)  }{\alpha+\beta}\\
& =\frac{f\left(  t+\alpha\right)  g\left(  t+\alpha\right)  -f\left(
t-\beta\right)  g\left(  t-\beta\right)  }{\alpha+\beta}\\
& =\frac{f\left(  t+\alpha\right)  -f\left(  t-\beta\right)  }{\alpha+\beta
}g\left(  t+\alpha\right)
+\frac{g\left(  t+\alpha\right)  -g\left(  t-\beta\right)  }{\alpha+\beta
}f\left(  t-\beta\right) \\
& =D_{\alpha,\beta}\left[  f\right]  \left(  t\right)  g\left(  t+\alpha
\right)  +f\left(  t-\beta\right)  D_{\alpha,\beta}\left[  g\right]  \left(
t\right).
\end{align*}
Property 5 is obtained from property 4 interchanging the role of $f$ and $g$.
To prove property 6 we begin by noting that
\begin{align*}
D_{\alpha,\beta}\left[  \frac{1}{g}\right]  \left(  t\right)   & =\frac
{\frac{1}{g}\left(  t+\alpha\right)  -\frac{1}{g}\left(  t-\beta\right)
}{\alpha+\beta}
=\frac{\frac{1}{g\left(  t+\alpha\right)  }-\frac{1}{g\left(  t-\beta
\right)  }}{\alpha+\beta}\\
& =\frac{g\left(  t-\beta\right)  -g\left(  t+\alpha\right)  }{\left(
\alpha+\beta\right)  g\left(  t+\alpha\right)  g\left(  t-\beta\right)}
=-\frac{D_{\alpha,\beta}\left[  g\right]  \left(  t\right)  }{g\left(
t+\alpha\right)  g\left(  t-\beta\right)  }\text{.}
\end{align*}
Hence,
\begin{align*}
D_{\alpha,\beta}\left[  \frac{f}{g}\right]  \left(  t\right)
&=D_{\alpha,\beta}\left[  f\frac{1}{g}\right]  \left(  t\right)
=D_{\alpha,\beta}\left[  f\right]  \left(  t\right)  \frac{1}{g}\left(
t+\alpha\right)  +f\left(  t-\beta\right)  D_{\alpha,\beta}\left[  \frac{1}
{g}\right]  \left(  t\right) \\
& =\frac{D_{\alpha,\beta}\left[  f\right]  \left(  t\right)  }{g\left(
t+\alpha\right)  }-f\left(  t-\beta\right)  \frac{D_{\alpha,\beta}\left[
g\right]  \left(  t\right)  }{g\left(  t+\alpha\right)  g\left(
t-\beta\right)  }\\
& =\frac{D_{\alpha,\beta}\left[  f\right]  \left(  t\right)  g\left(
t-\beta\right)  -f\left(  t-\beta\right)  D_{\alpha,\beta}\left[  g\right]
\left(  t\right)  }{g\left(  t+\alpha\right)  g\left(  t-\beta\right)  }.
\end{align*}
Property 7 follows from simple calculations:
\begin{align*}
D_{\alpha,\beta}\left[  \frac{f}{g}\right]  \left(  t\right)
&=D_{\alpha,\beta}\left[  f\frac{1}{g}\right]  \left(  t\right)
=D_{\alpha,\beta}\left[  f\right]  \left(  t\right)  \frac{1}{g}\left(
t-\beta\right)  +f\left(  t+\alpha\right)  D_{\alpha,\beta}\left[  \frac{1}
{g}\right]  \left(  t\right) \\
& =\frac{D_{\alpha,\beta}\left[  f\right]  \left(  t\right)  }{g\left(
t-\beta\right)  }-f\left(  t+\alpha\right)  \frac{D_{\alpha,\beta}\left[
g\right]  \left(  t\right)  }{g\left(  t+\alpha\right)  g\left(
t-\beta\right)  }\\
& =\frac{D_{\alpha,\beta}\left[  f\right]  \left(  t\right)  g\left(
t+\alpha\right)  -f\left(  t+\alpha\right)  D_{\alpha,\beta}\left[  g\right]
\left(  t\right)  }{g\left(  t+\alpha\right)  g\left(  t-\beta\right)  }.
\end{align*}
\end{proof}


\subsection{The $\alpha,\beta$-symmetric N\"{o}rlund sum}
\label{hs:sec:3.2}

Having in mind Remark~\ref{hs:rem:lc:der},
we define the $\alpha,\beta$-symmetric integral
as a linear combination of the $\alpha$-forward and the
$\beta$-backward integrals.

\begin{definition}
\label{hi:def:3}
Let $f:\mathbb{R}\rightarrow\mathbb{R}$ and $a,b\in\mathbb{R}$, $a<b$.
If $f$ is $\alpha$-forward and $\beta$-backward integrable
on $\left[  a,b\right]$, $\alpha, \beta \ge 0$ with $\alpha + \beta > 0$,
then we define the \emph{$\alpha,\beta$-symmetric
integral}\index{$\alpha,\beta$-symmetric integral} of $f$ from $a$ to $b$ by
\[
\int_{a}^{b}f\left(  t\right)  d_{\alpha,\beta}t=\frac{\alpha}{\alpha+\beta
}\int_{a}^{b}f\left(  t\right)  \Delta_{\alpha}t+\frac{\beta}{\alpha+\beta
}\int_{a}^{b}f\left(  t\right)  \nabla_{\beta}t\text{.}
\]
Function $f$ is $\alpha,\beta$-symmetric integrable if it is
$\alpha,\beta$-symmetric integrable for all $a,b\in\mathbb{R}$.
\end{definition}

\begin{remark}
Note that if $ \alpha\in\mathbb{R}^{+}$ and $\beta=0$,
then $$\displaystyle\int_{a}^{b}f\left(  t\right)
d_{\alpha,\beta}t=\int_{a}^{b}f\left(  t\right)  \Delta_{\alpha}t$$
and we do not need to assume in Definition~\ref{hi:def:3} that
$f$ is $\beta$-backward integrable;
if $\alpha=0$ and $\beta\in\mathbb{R}^{+}$, then
$$\displaystyle\int_{a}^{b}f\left(  t\right)  d_{\alpha,\beta}t
=\int_{a}^{b}f\left(  t\right)  \nabla_{\beta}t$$
and we do not need to assume that
$f$ is $\alpha$-forward integrable.
\end{remark}

\begin{example}
Let $f\left(  t\right)=\displaystyle\frac{1}{t^{2}}$.
\begin{align*}
\int_{1}^{3}\frac{1}{t^{2}}d_{2,2}t  & =\frac{1}{2}\int_{1}^{3}\frac{1}{t^{2}}
\Delta_{2}t+\frac{1}{2}\int_{1}^{3}\frac{1}{t^{2}}\nabla_{2}t\\
& =\frac{1}{2}\left(  2\sum_{k=0}^{+\infty}f\left(  1+2k\right)  -2\sum
_{k=0}^{+\infty}f\left(  3+2k\right)  \right) \\
& +\frac{1}{2}\left(  2\sum_{k=0}^{+\infty}f\left(  3-2k\right)  -2\sum
_{k=0}^{+\infty}f\left(  1-2k\right)  \right) \\
& =\left(  \sum_{k=0}^{+\infty}f\left(  1+2k\right)  -\sum_{k=0}^{+\infty
}f\left(  1+2\left(  k+1\right)  \right)  \right) \\
& +\left(  \sum_{k=0}^{+\infty}f\left(  3-2k\right)  -\sum_{k=0}^{+\infty
}f\left(  3-2\left(  k+1\right)  \right)  \right) \\
& =f\left(  1\right)  +f\left(  3\right) \\
& =1+\frac{1}{9}=\frac{10}{9}.
\end{align*}
\end{example}

The $\alpha,\beta$-symmetric integral has the following properties.

\begin{theorem}
\label{hi:propriedades}
Let $f,g:\mathbb{R}\rightarrow\mathbb{R}$
be $\alpha,\beta$-symmetric integrable on $\left[a,b\right]$.
Let $c\in\left[  a,b\right]$ and $k\in\mathbb{R}$. Then,
\begin{enumerate}
\item $\displaystyle\int_{a}^{a}f\left(  t\right)  d_{\alpha,\beta}t=0$;

\item $\displaystyle\int_{a}^{b}f\left(  t\right)  d_{\alpha,\beta}t=\int
_{a}^{c}f\left(  t\right)  d_{\alpha,\beta}t+\int_{c}^{b}f\left(  t\right)
d_{\alpha,\beta}t$, when the integrals exist;

\item $\displaystyle\int_{a}^{b}f\left(  t\right)  d_{\alpha,\beta}t
=-\int_{b}^{a}f\left(  t\right)  d_{\alpha,\beta}t$;

\item $kf$ is $\alpha,\beta$-symmetric integrable on $\left[a,b\right]$
and
$\displaystyle \int_{a}^{b}kf\left(t\right) d_{\alpha,\beta}t
=k\int_{a}^{b}f\left(t\right) d_{\alpha,\beta}t$;

\item $f+g$ is $\alpha,\beta$-symmetric integrable
on $\left[a,b\right]$ and
\[
\int_{a}^{b}\left(  f+g\right)  \left(  t\right)  d_{\alpha,\beta}t=\int
_{a}^{b}f\left(  t\right)  d_{\alpha,\beta}t+\int_{a}^{b}g\left(  t\right)
d_{\alpha,\beta}t\text{;}
\]

\item if $g$ is  a nonnegative function, then $fg$
is $\alpha,\beta$-symmetric integrable on $\left[  a,b\right]$.
\end{enumerate}
\end{theorem}

\begin{proof}
These results are easy consequences of the $\alpha$-forward
and $\beta$-backward integral properties.
\end{proof}

The properties of the $\alpha,\beta$-symmetric integral
follow from the corresponding $\alpha$-forward
and $\beta$-backward integral properties.
It should be noted, however, that the equality
$$D_{\alpha,\beta}\left[s \mapsto  \int_{a}^{s}f\left(\tau\right) d_{\alpha,\beta}
\tau\right](t) =f\left(  t\right)$$
is not always true in the $\alpha,\beta$-symmetric calculus,
despite both forward and backward integrals satisfy the
corresponding fundamental theorem of calculus. Indeed, let
$$f\left(  t\right)
=\displaystyle
\begin{cases}
\frac{1}{2^{t}} & \text{ if } t\in\mathbb{N},\\
\\
0 & \text{otherwise}.
\end{cases}$$
Then, for a fixed $t \in \mathbb{N}$,
\begin{align*}
\int_{0}^{t}\frac{1}{2^{\tau}}d_{1,1}\tau & =\frac
{1}{2}\int_{0}^{t}\frac{1}{2^{\tau}}\Delta_{1}
\tau+\frac{1}{2}\int_{0}^{t}\frac{1}{2^{\tau}}
\nabla_{1}\tau\\
& =\frac{1}{2}\left(  \sum_{k=0}^{+\infty}f\left(  0+k\right)
-\sum_{k=0}^{+\infty}f\left(  t+k\right)  \right)
+\frac{1}{2}\left(  \sum_{k=0}^{+\infty}f\left(  t-k\right)
-\sum_{k=0}^{+\infty}f\left(  0-k\right)  \right)  \\
&=\frac{1}{2}\left(  1+\frac{1}{2}+\cdots+\frac{1}{2^{t-1}}\right)  +\frac
{1}{2}\left(  \frac{1}{2^{t}}+\frac{1}{2^{t-1}}+\cdots+\frac{1}{2}\right)  \\
&=\frac{1}{2}\frac{1-\frac{1}{2^{t}}}{1-\frac{1}{2}}+\frac{1}{4}\frac
{1-\frac{1}{2^{t}}}{1-\frac{1}{2}}=\frac{3}{2}\left(  1-\frac{1}{2^{t}}\right)
\end{align*}
and $$\displaystyle
D_{1,1}\left[s \mapsto \int_{0}^{s}\frac{1}{2^{\tau}}
d_{1,1}\tau\right](t) =\frac{3}{2}D_{1,1}\left[ s \mapsto  1-\frac{1}{2^{s}}\right](t)
=-\frac{3}{2}\frac{\frac{1}{2^{t+1}}-\frac{1}{2^{t-1}}}{2}
=\frac{9}{2^{t+3}}.$$
Therefore,
$$
\displaystyle D_{1,1}\left[s \mapsto \int_{0}^{s}\frac{1}{2^{\tau}}
d_{1,1}\tau\right](t) \neq \frac{1}{2^{t}}.
$$

The next result follows immediately from Theorem~\ref{hi:p} and the
corresponding \\$\beta$-backward version.

\begin{theorem}
\label{hi:modulo}
Let $f:\mathbb{R}\rightarrow\mathbb{R}$ and $p>1$.
If $\left\vert f\right\vert $ is $\alpha,\beta$-symmetric integrable
on $\left[  a,b\right]$, then $\left\vert f\right\vert ^{p}$
is also $\alpha,\beta$-symmetric integrable on $\left[a,b\right]$.
\end{theorem}

\begin{theorem}
\label{hi:des}
Let $f,g:\mathbb{R}\rightarrow\mathbb{R}$
be $\alpha,\beta$-symmetric integrable functions on $\left[a,b\right]$,
$\mathcal{A} := \left\{  a+k\alpha:k\in\mathbb{N}_{0}\right\}$
and $\mathcal{B} := \left\{  b-k\beta:k\in\mathbb{N}_{0}\right\}$.
For $b\in\mathcal{A}$ and $a\in\mathcal{B}$ one has:
\begin{enumerate}
\item if $\left\vert f\left(  t\right)  \right\vert
\leqslant g\left(t\right)$ for all $t\in\mathcal{A}\cup\mathcal{B}$, then
$\displaystyle \left\vert \int_{a}^{b}f\left(  t\right)  d_{\alpha,\beta}t\right\vert
\leqslant\int_{a}^{b}g\left(  t\right)  d_{\alpha,\beta}t$;

\item if $f\left(  t\right) \geqslant0$ for all
$t\in\mathcal{A}\cup\mathcal{B}$, then
$\displaystyle \int_{a}^{b}f\left(  t\right)  d_{\alpha,\beta}t\geqslant0$;

\item if $g\left(  t\right)  \geqslant f\left(  t\right)$
for all $t\in\mathcal{A}\cup\mathcal{B}$, then
$\displaystyle \int_{a}^{b}g\left(  t\right)
d_{\alpha,\beta}t\geqslant\int_{a}^{b}f\left(t\right)d_{\alpha,\beta}t$.
\end{enumerate}
\end{theorem}

\begin{proof}
It follows from Theorem~\ref{hi:desigualdade}
and Corollary~\ref{hi:desigualdade2}
and the corresponding $\beta$-backward versions.
\end{proof}


\subsection{Mean value theorems}
\label{hs:sec:3.3}

We begin by remarking that if $f$ assumes its local maximum at
$t_{0}$, then there exist $\alpha,\beta\in\mathbb{R}_{0}^{+}$
with at least one of them positive, such that
$f\left(  t_{0}+\alpha\right)  \leqslant f\left(  t_{0}\right)$
and
$f\left(  t_{0}\right)  \geqslant f\left(  t_{0}-\beta\right)$.
If $\alpha,\beta\in\mathbb{R}^{+}$ this means that
$\Delta_{\alpha}\left[  f\right]  \left(  t\right)  \leqslant 0$
and $\nabla_{\beta}\left[  f\right]  \left(  t\right)  \geqslant 0$.
Also, we have the corresponding result for a local minimum: if $f$ assumes its
local minimum at $t_{0}$, then there exist $\alpha,\beta\in\mathbb{R}^{+}$ such that
$\Delta_{\alpha}\left[  f\right]  \left(  t\right)  \geqslant 0$ and
$\nabla_{\beta}\left[  f\right]  \left(  t\right)  \leqslant 0$.

\begin{theorem}[The $\alpha,\beta$-symmetric Fermat theorem for stationary points]
\label{hs:Symmetric Fermat's Theorem}
Let $f:\left[  a,b\right]  \rightarrow\mathbb{R}$
be a continuous function.\index{$\alpha,\beta$-symmetric Fermat theorem for stationary points}
If $f$ assumes a local extremum at $t_{0} \in\left]  a,b\right[$, then there exist
two positive real numbers $\alpha$ and $\beta$ such that
$$
D_{\alpha,\beta}\left[  f\right]  \left(  t_{0}\right)  =0.
$$
\end{theorem}

\begin{proof}
We prove the case where $f$ assumes a local maximum at $t_{0}$. Then
there exist $\alpha_{1},\beta_{1}\in\mathbb{R}^{+}$ such that
$\Delta_{\alpha_{1}}\left[  f\right]  \left(  t_{0}\right)  \leqslant 0$
and
$\nabla_{\beta_{1}}\left[  f\right]  \left(  t_{0}\right)  \geqslant 0$.
If $f\left(  t_{0}+\alpha_{1}\right)  =f\left(  t_{0}-\beta_{1}\right)$,
then $D_{\alpha_{1},\beta_{1}}\left[  f\right]  \left(  t_{0}\right)  =0$.
If $f\left(  t_{0}+\alpha_{1}\right)  \neq f\left(  t_{0}-\beta_{1}\right)$,
then let us choose $\gamma=\min\left\{  \alpha_{1},\beta_{1}\right\}$.
Suppose (without loss of generality) that $f\left(  t_{0}-\gamma\right)
>f\left(  t_{0}+\gamma\right)$. Then,
$f\left(  t_{0}\right) \geqslant f\left(  t_{0}-\gamma\right)
>f\left(  t_{0} +\gamma\right)$
and, since $f$ is continuous, by the intermediate value theorem there
exists $\rho$ such that $0<\rho<\gamma$ and
$f\left(  t_{0}+\rho\right)  =f\left(  t_{0}-\gamma\right)$.
Therefore,
$D_{\rho,\gamma}\left[  f\right]  \left(  t_{0}\right)  =0$.
\end{proof}

\begin{theorem}[The $\alpha,\beta$-symmetric Rolle mean value theorem]
\label{hs:Symmetric Rolle's Mean Value Theorem}
Let $f:\left[  a,b\right] \rightarrow\mathbb{R}$
be a continuous function with $f\left(  a\right)
=f\left(  b\right)$.\index{$\alpha,\beta$-symmetric Rolle mean value theorem}
Then there exist $\alpha$, $\beta\in\mathbb{R}^{+}$
and $c\in\left]  a,b\right[$ such that
$$
D_{\alpha,\beta}\left[  f\right]  \left(  c\right)  =0.
$$
\end{theorem}

\begin{proof}
If $f=const$, then the result is obvious.
If $f$ is not a constant function, then
there\ exists $t\in\left]a,b\right[$
such that $f\left(  t\right)  \neq f\left(  a\right)$.
Since $f$ is continuous on the compact set $\left[a,b\right]$,
$f$ has an extremum $M=f\left(  c\right)$
with $c\in\left]a,b\right[$. Since $c$ is also a local extremizer, then,
by Theorem~\ref{hs:Symmetric Fermat's Theorem},
there exist $\alpha$, $\beta\in\mathbb{R}^{+}$ such that
$D_{\alpha,\beta}\left[  f\right]  \left(  c\right)=0$.
\end{proof}

\begin{theorem}[The $\alpha,\beta$-symmetric Lagrange mean value theorem]
Let $f:\left[a,b\right]\rightarrow\mathbb{R}$ be a continuous
function.\index{$\alpha,\beta$-symmetric Lagrange mean value \\theorem}
Then there exist $c\in\left]  a,b\right[$
and $\alpha$, $\beta\in\mathbb{R}^{+}$ such that
$$D_{\alpha,\beta}\left[  f\right]  \left(  c\right)
=\frac{f\left(  b\right)-f\left(  a\right)  }{b-a}.$$
\end{theorem}

\begin{proof}
Let function $g$ be defined on $\left[  a,b\right]$ by
$g\left(  t\right)  =f\left(  a\right)  -f\left(  t\right)  +\left(
t-a\right)  \frac{f\left(  b\right)  -f\left(  a\right)  }{b-a}$.
Clearly, $g$ is continuous on $\left[  a,b\right]$ and
$g\left(  a\right)  =g\left(  b\right)  =0$.
Hence, by Theorem~\ref{hs:Symmetric Rolle's Mean Value Theorem}, there exist
$\alpha$, $\beta\in\mathbb{R}^{+}$ and $c\in\left]  a,b\right[$ such that
$D_{\alpha,\beta}\left[  g\right]  \left(  c\right)  =0$.
Since
\begin{align*}
D_{\alpha,\beta}\left[  g\right]  \left(  t\right)   & =\frac{g\left(
t+\alpha\right)  -g\left(  t-\beta\right)  }{\alpha+\beta}\\
& =\frac{1}{\alpha+\beta}\left(  f\left(  a\right)  -f\left(  t+\alpha\right)
+\left(  t+\alpha-a\right)  \frac{f\left(  b\right)  -f\left(  a\right)
}{b-a}\right) \\
& -\frac{1}{\alpha+\beta}\left(  f\left(  a\right)  -f\left(  t-\beta\right)
+\left(  t-\beta-a\right)  \frac{f\left(  b\right)  -f\left(  a\right)  }
{b-a}\right) \\
& =\frac{1}{\alpha+\beta}\left(  f\left(  t-\beta\right)  -f\left(
t+\alpha\right)  +\left(  \alpha+\beta\right)  \frac{f\left(  b\right)
-f\left(  a\right)  }{b-a}\right) \\
& =\frac{f\left(  b\right)  -f\left(  a\right)  }{b-a}-D_{\alpha,\beta}\left[
f\right]  \left(  t\right),
\end{align*}
we conclude that
$$
D_{\alpha,\beta}\left[  f\right]  \left(  c\right)  =\frac{f\left(  b\right)
-f\left(  a\right)  }{b-a}.
$$
\end{proof}

\begin{theorem}[The $\alpha,\beta$-symmetric Cauchy mean value theorem]
Let $f,g:\left[a,b\right] \rightarrow\mathbb{R}$ be continuous
functions.\index{$\alpha,\beta$-symmetric Cauchy mean value theorem}
Suppose that $D_{\alpha,\beta}\left[  g\right]\left(  t\right)  \neq 0$
for all $t\in\left]  a,b\right[$ and all $\alpha$, $\beta\in\mathbb{R}^{+}$.
Then there exist $\bar{\alpha},\bar{\beta}\in\mathbb{R}^{+}$
and $c\in\left]  a,b\right[$ such that
$$\frac{f\left(  b\right)  -f\left(  a\right)  }{g\left(  b\right)  -g\left(
a\right)  }=\frac{D_{\bar{\alpha},\bar{\beta}}\left[  f\right]  \left(
c\right)  }{D_{\bar{\alpha},\bar{\beta}}\left[  g\right]  \left(  c\right)}.$$
\end{theorem}

\begin{proof}
From condition $D_{\alpha,\beta}\left[  g\right]  \left(  t\right)  \neq 0$
for all $t\in\left]a,b\right[$ and all $\alpha,\beta\in\mathbb{R}^{+}$
and the $\alpha,\beta$-symmetric Rolle mean value theorem
(Theorem~\ref{hs:Symmetric Rolle's Mean Value Theorem}), it follows that
$g\left(  b\right)  \neq g\left(  a\right)$.
Let us consider function $F$ defined on $\left[a,b\right]$ by
$$F\left(  t\right)  =f\left(  t\right)  -f\left(  a\right)  -\frac{f\left(
b\right)  -f\left(  a\right)  }{g\left(  b\right)  -g\left(  a\right)
}\left[  g\left(  t\right)  -g\left(  a\right)  \right].$$
Clearly, $F$ is continuous on $\left[  a,b\right]$
and $F\left(  a\right)  =F\left(b\right)$. Applying the
$\alpha,\beta$-symmetric Rolle mean value theorem to
function $F$, we conclude that there exist
$\bar{\alpha},\bar{\beta}\in\mathbb{R}^{+}$ and $c\in\left]a,b\right[$
such that
\begin{equation*}
0 = D_{\bar{\alpha},\bar{\beta}}\left[  F\right]  \left(  c\right)\\
=D_{\bar{\alpha},\bar{\beta}}\left[  f\right]  \left(  c\right)
-\frac{f\left(  b\right)-f\left(  a\right)}{g\left(  b\right)
-g\left(a\right)}D_{\bar{\alpha},\bar{\beta}}\left[g\right]\left(c\right),
\end{equation*}
proving the intended result.
\end{proof}

\begin{theorem}[Mean value theorem for the $\alpha,\beta$- integral]
\label{hi:thm:mvt}
Let $a,b\in\mathbb{R}$ such that\index{Mean value theorem for the $\alpha,\beta$-integral}
$a<b$ and $b\in\mathcal{A} := \left\{  a+k\alpha:k\in\mathbb{N}_{0}\right\}$ and
$a\in\mathcal{B} := \left\{  b-k\beta:k\in\mathbb{N}_{0}\right\}$, where
$\alpha,\beta\in\mathbb{R}_{0}^{+}$, $\alpha + \beta \ne 0$.
Let $f,g:\mathbb{R}\rightarrow\mathbb{R}$
be bounded and $\alpha,\beta$-symmetric integrable on $[a,b]$
with $g$ nonnegative. Let $m$ and $M$ be
the infimum and the supremum, respectively, of function $f$.
Then, there exists a real number $K$ satisfying the inequalities
$$
m\leqslant K\leqslant M
$$
such that
$$
\int_{a}^{b}f\left(  t\right)  g\left(  t\right) d_{\alpha,\beta}t
=K\int_{a}^{b}g\left(  t\right)  d_{\alpha,\beta}t.
$$
\end{theorem}

\begin{proof}
Since
$$
m\leqslant f\left(  t\right)  \leqslant M
$$
for all $t\in\mathbb{R}$ and
$$
g\left(  t\right)  \geqslant 0,
$$
then
$$
mg\left(  t\right)  \leqslant f\left(  t\right)  g\left(  t\right)
\leqslant Mg\left(  t\right)
$$
for all $t\in\mathcal{A} \cup \mathcal{B}$.
All functions $mg$, $fg$ and $Mg$ are
$\alpha,\beta$-symmetric integrable on $[a,b]$.
By Theorems~\ref{hi:propriedades} and \ref{hi:des},
$$
m\int_{a}^{b}g\left(  t\right)  d_{\alpha,\beta}t
\leqslant\int_{a}^{b}f\left(t\right) g\left(t\right) d_{\alpha,\beta}t
\leqslant M\int_{a}^{b}g\left(
t\right)  d_{\alpha,\beta}t.
$$
If
$$
\int_{a}^{b}g\left(  t\right)  d_{\alpha,\beta}t=0,
$$
then
$$
\int_{a}^{b}f\left(  t\right)  g\left(  t\right)
d_{\alpha,\beta}t=0;
$$
if
$$
\int_{a}^{b}g\left(  t\right) d_{\alpha,\beta}t>0,
$$
then
$$
m\leqslant\frac{\int_{a}^{b}f\left(  t\right)  g\left(  t\right)
d_{\alpha,\beta}t}{\int_{a}^{b}g\left(  t\right)
d_{\alpha,\beta}t} \leqslant M.
$$
Therefore, the middle term of these inequalities
is equal to a number $K$, which yields the intended result.
\end{proof}


\subsection{$\alpha,\beta$-Symmetric integral inequalities}
\label{hi:sec:ineq}

Inspired in the work by Agarwal et al. \cite{Agarwal},
we now present $\alpha,\beta$-symmetric versions of H\"{o}lder's,
Cauchy--Schwarz's and Minkowski's inequalities.

\begin{theorem}[$\alpha,\beta$-symmetric H\"{o}lder's inequality]
\label{hi:Holders Inequality}
Let\index{$\alpha,\beta$-symmetric H\"{o}lder's inequality}
$f,g:\mathbb{R}\rightarrow\mathbb{R}$
and $a,b\in\mathbb{R}$ with $a<b$ such that
$b\in\mathcal{A} := \left\{  a+k\alpha:k\in\mathbb{N}_{0}\right\}$
and $a\in\mathcal{B} := \left\{  b-k\beta:k\in\mathbb{N}_{0}\right\}$,
where $\alpha,\beta\in\mathbb{R}_{0}^{+}$, $\alpha + \beta \ne 0$.
If $\left\vert f\right\vert $ and $\left\vert g\right\vert$ are
$\alpha,\beta$-symmetric integrable on $\left[a,b\right]$, then
\begin{equation}
\label{hi:eq:hi}
\int_{a}^{b}\left\vert f\left(  t\right)  g\left(  t\right)  \right\vert
d_{\alpha,\beta}t\leqslant\left(  \int_{a}^{b}\left\vert f\left(  t\right)
\right\vert ^{p}d_{\alpha,\beta}t\right)  ^{\frac{1}{p}}\left(
\int_{a}^{b}\left\vert g\left(  t\right)\right\vert^{q}
d_{\alpha,\beta}t\right)^{\frac{1}{q}},
\end{equation}
where $p>1$ and $q=p/(p-1)$.
\end{theorem}

\begin{proof}
For $\alpha,\beta\in\mathbb{R}_{0}^{+}$, $\alpha + \beta \ne 0$,
the following inequality holds (Young's inequality):
$$
\alpha^{\frac{1}{p}}\beta^{\frac{1}{q}}
\leqslant\frac{\alpha}{p}+\frac{\beta}{q}.
$$
Without loss of generality, suppose that
$$\displaystyle \left(  \int_{a}^{b}\left\vert f\left(  t\right)  \right\vert ^{p}
d_{\alpha,\beta}t\right)  \left(  \int_{a}^{b}\left\vert g\left(  t\right)
\right\vert ^{q}d_{\alpha,\beta}t\right)  \neq 0$$
(note that both integrals exist by Theorem~\ref{hi:modulo}). Set
$$\xi\left( t\right)
:=\frac{\left\vert f\left(  t\right)  \right\vert ^{p}}{\int_{a}^{b}\left\vert f\left(  \tau\right)  \right\vert ^{p}d_{\alpha,\beta}\tau}
$$
and $$\gamma\left(  t\right)
:=\frac{\left\vert g\left(t\right)\right\vert ^{q}}{\int_{a}^{b}\left\vert g\left(\tau\right)
\right\vert ^{q}d_{\alpha,\beta}\tau}.$$
Since both functions $\xi$ and $\gamma$ are
$\alpha,\beta$-symmetric integrable on $\left[a,b\right]$, then
\begin{align*}
\int_{a}^{b}&\frac{\left\vert f\left(  t\right)  \right\vert }{\left(
\int_{a}^{b}\left\vert f\left(  \tau\right)  \right\vert ^{p}d_{\alpha,\beta
}\tau\right)  ^{\frac{1}{p}}}\frac{\left\vert g\left(  t\right)  \right\vert
}{\left(  \int_{a}^{b}\left\vert g\left(  \tau\right)  \right\vert
^{q}d_{\alpha,\beta}\tau\right)  ^{\frac{1}{q}}}d_{\alpha,\beta}t
=\int_{a}^{b}\xi\left(  t\right)  ^{\frac{1}{p}}\gamma\left(  t\right)
^{\frac{1}{q}}d_{\alpha,\beta}t\\
& \leqslant\int_{a}^{b}\left(  \frac{\xi\left(  t\right)  }{p}+\frac
{\gamma\left(  t\right)  }{q}\right)  d_{\alpha,\beta}t\\
& =\frac{1}{p}\int_{a}^{b}\left(  \frac{\left\vert f\left(  t\right)
\right\vert ^{p}}{\int_{a}^{b}\left\vert f\left(  \tau\right)  \right\vert
^{p}d_{\alpha,\beta}\tau}\right)  d_{\alpha,\beta}t
+\frac{1}{q}\int_{a}^{b}\left(  \frac{\left\vert g\left(  t\right)
\right\vert ^{q}}{\int_{a}^{b}\left\vert g\left(  \tau\right)
\right\vert ^{q}d_{\alpha,\beta}\tau}\right)  d_{\alpha,\beta}t = 1
\end{align*}
proving the intended result.
\end{proof}

The particular case $p=q=2$ of \eqref{hi:eq:hi}
gives the $\alpha,\beta$-symmetric Cauchy--Schwarz's inequality.

\begin{corollary}[$\alpha,\beta$-symmetric Cauchy--Schwarz's inequality]
Let\index{$\alpha,\beta$-symmetric Cauchy--Schwarz's inequality}
$f,g:\mathbb{R}\rightarrow\mathbb{R}$
and $a,b\in\mathbb{R}$ with $a<b$ such that
$b\in\mathcal{A} := \left\{  a+k\alpha:k\in\mathbb{N}_{0}\right\}$ and
$a\in\mathcal{B} := \left\{  b-k\beta:k\in\mathbb{N}_{0}\right\}$, where
$\alpha,\beta\in\mathbb{R}_{0}^{+}$, $\alpha + \beta \ne 0$. If $f$ and $g$
are $\alpha,\beta$-symmetric integrable on $\left[  a,b\right]$, then
\[
\int_{a}^{b}\left\vert f\left(  t\right)  g\left(  t\right)  \right\vert
d_{\alpha,\beta}t\leqslant\sqrt{\left(  \int_{a}^{b}\left\vert f\left(
t\right)  \right\vert ^{2}d_{\alpha,\beta}t\right)  \left(
\int_{a}^{b}\left\vert g\left(  t\right)
\right\vert ^{2}d_{\alpha,\beta}t\right)}\text{.}
\]
\end{corollary}

We prove the $\alpha,\beta$-symmetric Minkowski inequality using
$\alpha,\beta$-symmetric H\"{o}lder's inequality.

\begin{theorem}[$\alpha,\beta$-symmetric Minkowski's inequality]
Let $f,g:\mathbb{R}
\rightarrow\mathbb{R}$\index{$\alpha,\beta$-symmetric Minkowski's inequality}
and $a,b\in\mathbb{R}$ with $a<b$ such that $b\in\mathcal{A}
:= \left\{  a+k\alpha:k\in\mathbb{N}_{0}\right\}$ and
$a\in\mathcal{B} := \left\{  b-k\beta:k\in\mathbb{N}_{0}\right\}$, where
$\alpha,\beta\in\mathbb{R}_{0}^{+}$, $\alpha + \beta \ne 0$.
If $f$ and $g$ are $\alpha,\beta$-symmetric integrable
on $\left[a,b\right]$, then, for any $p>1$,
\begin{equation}
\label{A}
\left(  \int_{a}^{b}\left\vert f\left(  t\right)  +g\left(  t\right)
\right\vert ^{p}d_{\alpha,\beta}t\right)  ^{\frac{1}{p}}\leqslant\left(
\int_{a}^{b}\left\vert f\left(  t\right)  \right\vert ^{p}d_{\alpha,\beta
}t\right)  ^{\frac{1}{p}}+\left(  \int_{a}^{b}\left\vert g\left(  t\right)
\right\vert ^{p}d_{\alpha,\beta}t\right)  ^{\frac{1}{p}}.
\end{equation}
\end{theorem}

\begin{proof}
If $\displaystyle\int_{a}^{b}\left\vert f\left(  t\right)
+g\left(  t\right)  \right\vert^{p}d_{\alpha,\beta}t=0$, the result is trivial.
Suppose that $$\displaystyle\int_{a}^{b}\left\vert f\left(  t\right)
+g\left(  t\right)  \right\vert^{p}d_{\alpha,\beta}t\neq0.$$
Since
\begin{multline*}
\int_{a}^{b}\left\vert f\left(  t\right)
+g\left(  t\right)  \right\vert^{p}d_{\alpha,\beta}t
=\int_{a}^{b}\left\vert f\left(  t\right)  +g\left(
t\right)  \right\vert ^{p-1}\left\vert f\left(  t\right)  +g\left(  t\right)
\right\vert d_{\alpha,\beta}t\\
\leqslant\int_{a}^{b}\left\vert f\left(  t\right)  \right\vert \left\vert
f\left(  t\right)  +g\left(  t\right)  \right\vert ^{p-1}d_{\alpha,\beta
}t+\int_{a}^{b}\left\vert g\left(  t\right)  \right\vert \left\vert f\left(
t\right)  +g\left(  t\right)  \right\vert ^{p-1}d_{\alpha,\beta}t,
\end{multline*}

\noindent then applying $\alpha,\beta$-symmetric H\"{o}lder's inequality
(Theorem~\ref{hi:Holders Inequality}) with $q=p/(p-1)$, we obtain
\begin{align*}
& \int_{a}^{b}\left\vert f\left(  t\right)  +g\left(  t\right)  \right\vert
^{p}d_{\alpha,\beta}t
\leqslant \left(  \int_{a}^{b}\left\vert f\left(
t\right)  \right\vert ^{p}d_{\alpha,\beta}t\right)  ^{\frac{1}{p}}\left(
\int_{a}^{b}\left\vert f\left(  t\right)  +g\left(  t\right)  \right\vert
^{\left(  p-1\right)  q}d_{\alpha,\beta}t\right)  ^{\frac{1}{q}}\\
& +\left(  \int_{a}^{b}\left\vert g\left(  t\right)  \right\vert ^{p}
d_{\alpha,\beta}t\right)  ^{\frac{1}{p}}\left(  \int_{a}^{b}\left\vert
f\left(  t\right)  +g\left(  t\right)  \right\vert ^{\left(  p-1\right)
q}d_{\alpha,\beta}t\right)  ^{\frac{1}{q}}\\
=& \left[  \left(  \int_{a}^{b}\left\vert f\left(  t\right)  \right\vert
^{p}d_{\alpha,\beta}t\right)  ^{\frac{1}{p}}+\left(  \int_{a}^{b}\left\vert
g\left(  t\right)  \right\vert ^{p}d_{\alpha,\beta}t\right)  ^{\frac{1}{p}
}\right]
\left(  \int_{a}^{b}\left\vert f\left(  t\right)  +g\left(
t\right)  \right\vert ^{p}d_{\alpha,\beta}t\right)
^{\frac{1}{q}}.
\end{align*}
and therefore we get
\begin{equation*}
\frac{\int_{a}^{b}\left\vert f\left(  t\right)
+g\left(  t\right)  \right\vert^{p}d_{\alpha,\beta}t}{\left(
\int_{a}^{b}\left\vert f\left(
t\right)  +g\left(  t\right)  \right\vert ^{p}
d_{\alpha,\beta}t\right)^{\frac{1}{q}}}
\leqslant
\left(  \int_{a}^{b}\left\vert f\left(
t\right)  \right\vert ^{p}d_{\alpha,\beta}t\right)^{\frac{1}{p}}
+\left(\int_{a}^{b}\left\vert g\left(  t\right)
\right\vert^{p}d_{\alpha,\beta}t\right)^{\frac{1}{p}},
\end{equation*}
proving inequality \eqref{A}.
\end{proof}

To conclude this section we remark that our
$\alpha,\beta$-symmetric calculus is more general
than the standard $h$-calculus and $h$-symmetric calculus.
In particular, all our results give, as corollaries,
results in the classical quantum $h$-calculus by choosing
$\alpha=h>0$ and $\beta=0$, and results in the
$h$-symmetric calculus by choosing $\alpha=\beta=h>0$.


\section{State of the Art}

The results of this chapter were accepted to be published
on the Proceedings of the International Conference
on Differential \& Difference Equations and Applications
\cite{Brito:da:Cruz:3,Brito:da:Cruz:4}.
Quantum calculus is receiving an increase of interest
due to its applications in physics, economics
and calculus of variations. In this chapter
we have developed new results in this field:
we introduced a new symmetric quantum calculus.


\clearpage{\thispagestyle{empty}\cleardoublepage}


\chapter{The \lowercase{$q$}-Symmetric Variational Calculus}
\label{The $q$-Symmetric Variational Calculus}

In this chapter we bring a new approach to the study of quantum calculus
and introduce the $q$-symmetric variational calculus. We prove a necessary
optimality condition of Euler--Lagrange type and a sufficient optimality
condition for symmetric quantum variational problems.
The results are illustrated with an example.


\section{Introduction}

Quantum calculus, also known as calculus without limits,
is a very interesting field in mathematics. Moreover,
it plays an important role in several fields of physics such as cosmic strings
and black holes \cite{Strominger}, conformal quantum mechanics \cite{Youm},
nuclear and high energy physics \cite{Lavagno}, just to name a few.
For a deeper understanding of quantum calculus we refer the reader
to \cite{Boole,Ernst:2,Kac,Koekoek}.

Usually we are concern with two types of quantum calculus:
the $q$-calculus and the $h$-calculus. In this chapter we are concerned with the $q$-calculus.
Historically, the $q$-calculus was first introduced by Jackson \cite{Jackson}
and is a calculus based on the notion of the $q$-derivative
$$
\frac{f(qt)-f(t)}{(q-1)t},
$$
where $q$ is a fixed number different from 1, $t\neq 0$ and $f$ is a real function.
In contrast to the classical derivative, which measures the rate of change of the
function of an incremental translation of its argument, the $q$-derivative measures
the rate of change with respect to a dilatation of its argument by a factor $q$.
It is clear that if $f$ is differentiable at $t\neq 0$, then
$$
f^\prime(t)= \displaystyle\lim_{q\rightarrow 1}\frac{f(qt)-f(t)}{(q-1)t}.
$$

For a fixed $q\in\left]  0,1\right[  $ and $t\neq0$ the $q$-symmetric
derivative of a function $f$ at point $t$ is defined by
\[
\frac{f\left(  qt\right)  -f\left(  q^{-1}t\right)  }{\left(  q-q^{-1}\right) t}.
\]

The $q$-symmetric quantum calculus has proven to be useful in several fields, in particular
in quantum mechanics \cite{Lavagno2}. As noticed in \cite{Lavagno2}, the $q$-symmetric
derivative has important properties for the $q$-exponential function
which turns out to be not true with the usual derivative.

It is well known that the derivative of a differentiable
real function $f$ at a point $t\neq 0$ can be approximated by
the $q$-symmetric derivative. We believe that the $q$-symmetric
derivative  has, in general, better convergence properties than the $h$-derivative
$$
\frac{f(t+h)-f(t)}{h}
$$
and the $q$-derivative
$$
\frac{f(qt)-f(t)}{(q-1)t},
$$
but this requires additional investigation.

Our goal is to establish a necessary optimality condition and a sufficient optimality
condition for the $q$-symmetric variational problem\index{$q$-symmetric variational problem}
\begin{equation}
\left\{
\begin{array}[c]{l}
\mathcal{L}\left[  y\right]
=\displaystyle\int_{a}^{b}L\left(  t,y\left(  qt\right),
\tilde{D}_{q}\left[  y\right]\left(  t\right)\right)
\tilde{d}_{q}t\rightarrow\text{extremize}\\
\\
y\in\mathcal{Y}^{1}\left(  \left[  a,b\right]_q  ,\mathbb{R}\right)  \\
\\
y\left(  a\right)  =\alpha\\
\\
y\left(  b\right)  =\beta,
\end{array}
\label{qs:P}
\right.
\end{equation}
where $\alpha,\beta\ $are fixed real numbers and $\tilde{D}_{q}$ denotes the $q$-symmetric derivative operator (see \ref{q:symmetric:derivative}). By extremize
we mean maximize or minimize. Also we must assume that the
Lagrangian function $L$  satisfies the following hypotheses:

\begin{enumerate}
\item[(H$_{q}$1)] $\left(  u,v\right)  \rightarrow L\left(  t,u,v\right)  $ is a
$C^{1}\left(\mathbb{R}^{2},\mathbb{R}\right)  $ function for any $t\in\left[  a,b\right]$;

\item[(H$_{q}$2)] $t\rightarrow L\left(  t,y\left(  qt\right)  ,\tilde{D}_{q}\left[
y\right] \left( t \right) \right)  $ is continuous at $0$ for any admissible function $y$;

\item[(H$_{q}$3)] functions $t\rightarrow\partial_{i+2}L\left(  t,y\left(
qt\right)  ,\tilde{D}_{q}\left[  y\right]\left(t\right)  \right)$, $i=0,1$ belong to
$\mathcal{Y}^{1}\left(  \left[  a,b\right]  _{q},\mathbb{R}\right)$ for all admissible functions
$y$, where $\partial_i L$ denotes the partial derivative of $L$ with respect to its $i$th argument.
\end{enumerate}

This chapter is organized as follows. In Section~\ref{qs:sec2} we review the necessary
definitions and prove some new results of the $q$-symmetric calculus. Usually the set of
study in $q$-calculus is the lattice $\left\{  a,aq,aq^{2},aq^{3},
\ldots\right\}$. In this chapter we work in an arbitrary real interval
containing $0$. This new approach follows the ideas that Aldwoah used in his
Ph.D. thesis \cite{Aldwoah} (see also \cite{Aldwoah:3}).
In Section~\ref{qs:sec3} we formulate and prove our results
for the $q$-symmetric variational calculus.


\section{Preliminaries}
\label{qs:sec2}

Let $q\in\left]  0,1\right[$ and let $I$ be an interval (bounded or unbounded)
of $\mathbb{R}$ containing $0$. We denote by $I^{q}$ the set
$$
I^{q}:=qI:=\{qx:x\in I\} .
$$
Note that $I^{q}\subseteq I$.

\begin{definition}[cf. \cite{Kac}]
\label{q:symmetric:derivative}
Let $f$ be a real function defined on $I$.
The \emph{$q$-symmetric difference operator}\index{$q$-symmetric difference operator}
of  $f$ is defined by
\[
\tilde{D}_{q}\left[  f\right]  \left(  t\right)  =\frac{f\left(  qt\right)
-f\left(  q^{-1}t\right)  }{\left(  q-q^{-1}\right)  t},\text{ if }
t\in I^{q}\backslash\{0\},
\]
and $\tilde{D}_{q}\left[  f\right]  \left(  0\right)  :=f^{\prime}\left(
0\right)$, provided $f$ is differentiable at $0$. We usually call
$\tilde{D}_{q}\left[  f\right]$ the
\emph{$q$-symmetric derivative}\index{$q$-symmetric derivative} of $f$.
\end{definition}

\begin{remark}
Notice that if $f$ is differentiable (in the classical sense) at $t\in I^{q}$, then
\[
\lim_{q\rightarrow1}\tilde{D}_{q}\left[  f\right]  \left(  t\right)
=f^{\prime}\left(  t\right)  .
\]
\end{remark}

The $q$-symmetric difference operator has the following properties.

\begin{theorem}[cf. \cite{Kac}]
\label{qs:props derivada}
Let $f$ and $g$ be $q$-symmetric differentiable
on $I$, let $\alpha,\beta\in\mathbb{R}$ and $t\in I^{q}$. One has
\begin{enumerate}
\item $\tilde{D}_{q}\left[  f\right]  \equiv0$ if and only if $f$ is constant on $I$;

\item $\tilde{D}_{q}\left[  \alpha f+\beta g\right]  \left(  t\right)
=\alpha\tilde{D}_{q}\left[  f\right]  \left(  t\right)
+ \beta\tilde{D}_{q}\left[  g\right]  \left(  t\right)$;

\item $\tilde{D}_{q}\left[  fg\right]  \left(  t\right)  =\tilde{D}_{q}\left[
f\right]  \left(  t\right)  g\left(  qt\right)  +f\left(  q^{-1}t\right)
\tilde{D}_{q}\left[  g\right]  \left(  t\right)$;

\item $\displaystyle\tilde{D}_{q}\left[  \frac{f}{g}\right]  \left(  t\right)
=\frac{\tilde{D}_{q}\left[  f\right]  \left(  t\right)  g\left(
q^{-1}t\right)  -f\left(  q^{-1}t\right)  \tilde{D}_{q}\left[  g\right]
\left(  t\right)  }{g\left(  qt\right)  g\left(  q^{-1}t\right)  }$ if
$g\left(  qt\right)  g\left(  q^{-1}t\right)  \neq0$.
\end{enumerate}
\end{theorem}

\begin{proof}
If $f$ is constant, then it is clear that
$\tilde{D}_{q}\left[f\right]\equiv0$.
For each $t\in I$ if $\tilde{D}_{q}\left[
f\right]  \left(  qt\right)  =0$, then $f\left(  t\right)
=f\left(q^{2}t\right)$ and one has, for $n \in \mathbb{N}$,
\[
f\left(  t\right)  =f\left(  q^{2}t\right)
=\cdots=f\left(  q^{2n}t\right).
\]
Since
\[
\lim_{n\rightarrow+\infty}f\left(  t\right)
=\lim_{n\rightarrow+\infty}f\left(  q^{2n}t\right)
\]
and $f$ is continuous at $0$, then
\[
f\left(  t\right)  =f\left(  0\right)  \text{, }\forall t\in I.
\]
The other properties are trivial for $t=0$ and for $t\neq0$ see \cite{Kac}.
\end{proof}

\begin{definition}[cf. \cite{Kac}]
\label{$q$-symmetric integral}
Let $a,b\in I$ and $a<b$. For $f:I\rightarrow\mathbb{R}$ and for
$q\in\left]  0,1\right[$ the \emph{$q$-symmetric integral}\index{$q$-symmetric integral}
of $f$ from $a$ to $b$ is given by
\[
\int_{a}^{b}f\left(  t\right)  \tilde{d}_{q}t=\int_{0}^{b}f\left(  t\right)
\tilde{d}_{q}t-\int_{0}^{a}f\left(  t\right)  \tilde{d}_{q}t,
\]
where
\begin{align*}
\int_{0}^{x}f\left(  t\right)  \tilde{d}_{q}t  & =\left(  q^{-1}-q\right)
x\sum_{n=0}^{+\infty}q^{2n+1}f\left(  q^{2n+1}x\right)  \\
& =\left(  1-q^{2}\right)  x\sum_{n=0}^{+\infty}q^{2n}f\left(  q^{2n+1}
x\right)  \text{, \ \ \ \ }x\in I,
\end{align*}
provided that the series converges at $x=a$ and $x=b$. In that case, $f$ is
called $q$-symmetric integrable on $[a,b]$. We say that $f$ is $q$-integrable
on $I$ if it is $q$-integrable on $[a,b]$ for all $a,b\in I$.
\end{definition}

We now present two technical results that are useful to prove the
Fundamental Theorem of the $q$-Symmetric Integral Calculus (Theorem~\ref{qs:Fundamental}).

\begin{lemma}[\cite{Aldwoah}]
\label{qs:lema integral}
Let $a,b\in I $, $a<b$ and $f:I\rightarrow\mathbb{R}$ continuous at $0$.
Then for $s\in\left[  a,b\right]$ the sequence
$\left(  f\left(  q^{2n+1}s\right)\right)_{n\in\mathbb{N}}$
converges uniformly to $f\left(  0\right)$ on $I$.
\end{lemma}

\begin{corollary}[\cite{Aldwoah}]
\label{qs:corolario integral}
If $f:I\rightarrow\mathbb{R}$ is continuous at $0$,
then for $s\in\left[a,b\right]$ the series
$\sum_{n=0}^{+\infty}q^{2n}f\left(
q^{2n+1}s\right)$ is uniformly convergent on $I$,
and consequently, $f$ is $q$-symmetric integrable
on $\left[a,b\right]$.
\end{corollary}

\begin{theorem}[Fundamental theorem of the $q$-symmetric integral calculus]
\label{qs:Fundamental}
Assume\index{Fundamental theorem of the $q$-symmetric integral calculus}
that $f:I\rightarrow\mathbb{R}$ is continuous at $0$ and, for each $x\in I$, define
\[
F(x):=\int_{0}^{x}f\left(  t\right)  \tilde{d}_{q}t\text{.}
\]
Then $F$ is continuous at $0$. Furthermore, $\tilde{D}_{q}[F](x)$ exists for
every $x\in I^{q}$ with $$\tilde{D}_{q}[F](x)=f(x).$$ Conversely,
\[
\int_{a}^{b}\tilde{D}_{q}\left[  f\right]  \left(  t\right)  \tilde{d}_{q}t
=f\left(  b\right)  -f\left(  a\right)
\]
for all $a,b\in I$.
\end{theorem}

\begin{proof}
By Corollary~\ref{qs:corolario integral},
the function $F$ is continuous at $0$.
If $x \in I^{q}\setminus\{0\}$, then
\begin{align*}
& \tilde{D}_{q}\left(  \int_{0}^{x}f\left(  t\right)  \tilde{d}_{q}t\right)
=\frac{\int_{0}^{qx}f\left(  t\right)  \tilde{d}_{q}t-\int_{0}^{q^{-1} x}
f\left(  t\right)  \tilde{d}_{q}t}{\left(  q-q^{-1}\right)  x}\\
&  =\frac{q}{\left(  q^{2}-1\right)  x}\bigg[\left(  1-q^{2}\right)
qx\sum_{n=0}^{+\infty}q^{2n}f\left(  q^{2n+1}qx\right)  -\left(
1-q^{2}\right)  q^{-1}x\sum_{n=0}^{+\infty}q^{2n}f\left(  q^{2n+1}
q^{-1}x\right)  \bigg]\\
&  =\sum_{n=0}^{+\infty}q^{2n}f\left(  q^{2n}x\right)  -\sum_{n=0}^{+\infty
}q^{2n+2}f\left(  q^{2n+2}x\right)  \\
&  =\sum_{n=0}^{+\infty}\bigg[q^{2n}f\left(  q^{2n}x\right)  -q^{2\left(
n+1\right)  }f\left(  q^{2\left(  n+1\right)  }x\right)  \bigg]\\
&  =f\left(  x\right) .
\end{align*}
If $x=0$, then
\begin{align*}
\tilde{D}_{q}\left[F\right]\left(  0\right)   &  =\lim_{h\rightarrow0}\frac{F\left(
h\right)  -F\left(  0\right)  }{h}\\
&  =\lim_{h\rightarrow0}\frac{1}{h}\left(  1-q^{2}\right)  h\sum
_{n=0}^{+\infty}q^{2n}f\left(  q^{2n+1}h\right)  \\
&  =\lim_{h\rightarrow0}\left(  1-q^{2}\right)  \sum_{n=0}^{+\infty}
q^{2n}f\left(  q^{2n+1}h\right)  \\
&  =\left(  1-q^{2}\right)  \sum_{n=0}^{+\infty}q^{2n}f\left(  0\right)
\text{ (by the continuity of }f\text{ at }0\text{)}\\
&  =\left(  1-q^{2}\right)  \frac{1}{1-q^{2}}f\left(  0\right)  \\
&  =f\left(  0\right)  \,.
\end{align*}
Finally, since for $x \in I$,
\begin{align*}
\int_{0}^{x}\tilde{D}_{q}\left[  f\right]  \left(  t\right)  \tilde{d}_{q}t &
=\left(  1-q^{2}\right)  x\sum_{n=0}^{+\infty}q^{2n}\tilde{D}_{q}\left[
f\right]  \left(  q^{2n+1}x\right)  \\
&  =\left(  1-q^{2}\right)  x\sum_{n=0}^{+\infty}q^{2n}\frac{f\left(
qq^{2n+1}x\right)  -f\left(  q^{-1}q^{2n+1}x\right)  }{\left(  q-q^{-1}
\right)  \left(  q^{2n+1}x\right)  }\\
&  =\sum_{n=0}^{+\infty}\left[  f\left(  q^{2n}x\right)  -f\left(  q^{2\left(
n+1\right)  }x\right)  \right]  \\
&  =f\left(  x\right)  -f\left(  0\right),
\end{align*}
we have
\begin{align*}
\int_{a}^{b}\tilde{D}_{q}\left[  f\right]  \left(  t\right)  \tilde{d}_{q}t &
=\int_{0}^{b}\tilde{D}_{q}\left[  f\right]  \left(  t\right)  \tilde{d}_{q}t
-\int_{0}^{a}\tilde{D}_{q}\left[  f\right]  \left(  t\right)  \tilde{d}_{q}t\\
&  =f\left(  b\right)  -f\left(  a\right) .
\end{align*}
\end{proof}

The $q$-symmetric integral has the following properties.

\begin{theorem}
Let $f,g$ $:I\rightarrow\mathbb{R}$ be $q$-symmetric integrable on $I$,
$a,b,c\in I$ and $\alpha,\beta\in\mathbb{R}$. Then

\begin{enumerate}
\item $\displaystyle\int_{a}^{a}f\left(  t\right)  \tilde{d}_{q}t=0$;

\item $\displaystyle\int_{a}^{b}f\left(  t\right)  \tilde{d}_{q}t
=-\int_{b}^{a}f\left(t\right)  \tilde{d}_{q}t$;

\item $\displaystyle\int_{a}^{b}f\left(  t\right)\tilde{d}_{q}t
=\int_{a}^{c}f\left(t\right)  \tilde{d}_{q}t
+\int_{c}^{b}f\left(  t\right)  \tilde{d}_{q}t$;

\item $\displaystyle\int_{a}^{b}\left(  \alpha f
+\beta g\right)  \left(  t\right)
\tilde{d}_{q}t=\alpha\int_{a}^{b}f\left(  t\right)\tilde{d}_{q}t
+\beta \int_{a}^{b}g\left(  t\right)  \tilde{d}_{q}t$;

\item If $\displaystyle\tilde{D}_{q}\left[  f\right]$
and $\tilde{D}_{q}\left[  g\right]$ are continuous at $0$, then
\[
\int_{a}^{b}f\left(  q^{-1}t\right)  \tilde{D}_{q}\left[  g\right]  \left(
t\right)  \tilde{d}_{q}t=f\left(  t\right)  g\left(  t\right)  \bigg|_{a}^{b}
-\int_{a}^{b}\tilde{D}_{q}\left[  f\right]  \left(  t\right)  g\left(
qt\right)  \tilde{d}_{q}t .
\]
We call this formula \emph{$q$-symmetric
integration by parts}\index{$q$-symmetric integration by parts}.
\end{enumerate}
\end{theorem}

\begin{proof}
Properties 1-4 are trivial. Property 5 follows from
Theorem~\ref{qs:props derivada} and Theorem~\ref{qs:Fundamental}:
\begin{align*}
&  \tilde{D}_{q}\left[  fg\right]  \left(  t\right)  =\tilde{D}_{q}\left[
f\right]  \left(  t\right)  g\left(  qt\right)  +f\left(  q^{-1}t\right)
\tilde{D}_{q}\left[  g\right]  \left(  t\right)  \\
&  \Leftrightarrow f\left(  q^{-1}t\right)  \tilde{D}_{q}\left[  g\right]  \left(
t\right)  =\tilde{D}_{q}\left[  fg\right]  \left(  t\right)
- \tilde{D}_{q}\left[  f\right]  \left(  t\right)  g\left(  qt\right)  \\
&  \Rightarrow\int_{a}^{b}f\left(  q^{-1}t\right)  \tilde{D}_{q}\left[
g\right]  \left(  t\right)  \tilde{d}_{q}t=f\left(  t\right)  g\left(
t\right)  \bigg|_{a}^{b}-\int_{a}^{b}\tilde{D}_{q}\left[  f\right]  \left(
t\right)  g\left(  qt\right)  \tilde{d}_{q}t.
\end{align*}
\end{proof}

\begin{remark}
Let us consider the function $\sigma$ defined by $\sigma\left(  t\right)  :=qt$
and use the notation $f^{\sigma}\left(  t\right)  :=f\left(  qt\right)  $.
Since, for each $t \in I^q\setminus\{0\}$,
$$
\tilde{D}_{q}\left[  f^{\sigma}\right]  \left(  t\right)     =\frac{f^{\sigma}\left(
qt\right)  -f^{\sigma}\left(  q^{-1}t\right)  }{\left(  q-q^{-1}\right)  t}
 =\frac{f\left(  q^{2}t\right)  -f\left(  t\right)  }{\left(  q-q^{-1}
\right)  t}
$$
and
$$
\tilde{D}_{q}\left[  f\right]  \left(  qt\right)  =\frac{f\left(  q^{2}t\right)  -f\left(
t\right)  }{\left(  q-q^{-1}\right)  qt}
$$
then we may conclude that
$$
\tilde{D}_{q}\left[  f^{\sigma}\right]  \left(  t\right)  =q\tilde{D}_{q}\left[  f\right]
\left(  qt\right), \quad \quad \forall t \in I^q\setminus\{0\}.
$$
Since
$\tilde{D}_{q}\left[  f^{\sigma}\right]  \left(  0\right)  =q\tilde{D}_{q}\left[  f\right]
\left(  0\right)$, we may conclude that
\begin{equation}
\label{qs:comp}
\tilde{D}_{q}\left[  f^{\sigma}\right]  \left(  t\right)  =q\tilde{D}_{q}\left[  f\right]
\left(  qt\right), \quad \quad \forall t \in I^q.
\end{equation}
Hence, an analogue formula for \emph{$q$-symmetric integration by parts} is given by
\begin{equation}
\int_{a}^{b}f\left(  t\right)  \tilde{D}_{q}\left[  g\right]  \left(
t\right)  \tilde{d}_{q}t=f\left(  qt\right)  g\left(  t\right)  \bigg|_{a}
^{b}-q\int_{a}^{b}\tilde{D}_{q}\left[  f\right]  \left(  qt\right)
g\left(  qt\right)  \tilde{d}_{q}t.\label{qs:partes}
\end{equation}
\end{remark}

\begin{proposition}
\label{qs:desigualdade}
Let $c\in I$ and let $f$ and $g$ be $q$-symmetric
integrable on  $I$. Suppose that
\[
\left\vert f\left( t\right) \right\vert \leqslant g\left( t\right)
\]
for all $t\in \left\{ q^{2n+1}c:n\in \mathbb{\mathbb{N}}_{0}\right\} \cup \left\{ 0\right\}$.

\begin{enumerate}
\item If $c\geqslant 0$, then
\[
\left\vert \int_{0}^{c}f\left( t\right) \tilde{d}_{q}t\right\vert
\leqslant\int_{0}^{c}g\left( t\right) \tilde{d}_{q}t\text{.}
\]

\item If $c<0$, then
\[
\left\vert \int_{c}^{0}f\left( t\right) \tilde{d}_{q}t\right\vert
\leqslant\int_{c}^{0}g\left( t\right) \tilde{d}_{q}t\text{.}
\]
\end{enumerate}
\end{proposition}

\begin{proof}
If $c\geqslant 0$, then
\begin{eqnarray*}
\left\vert \int_{0}^{c}f\left( t\right) \tilde{d}_{q}t\right\vert
&=&\left\vert \left( 1-q^{2}\right) c\sum_{n=0}^{+\infty }q^{2n}f\left(
q^{2n+1}c\right) \right\vert
\leqslant \left( 1-q^{2}\right) c\sum_{n=0}^{+\infty }q^{2n}\left\vert f\left(
q^{2n+1}c\right) \right\vert  \\
&\leqslant &\left( 1-q^{2}\right) c\sum_{n=0}^{+\infty }q^{2n}g\left(
q^{2n+1}c\right)  = \int_{0}^{c}g\left( t\right) \tilde{d}_{q}t.
\end{eqnarray*}
If $c<0$, then
\begin{eqnarray*}
\left\vert \int_{c}^{0}f\left( t\right) \tilde{d}_{q}t\right\vert
&=&\left\vert \left( 1-q^{2}\right) c\sum_{n=0}^{+\infty }q^{2n}f\left(
q^{2n+1}c\right) \right\vert \leqslant
-\left( 1-q^{2}\right) c\sum_{n=0}^{+\infty }q^{2n}\left\vert
f\left( q^{2n+1}c\right) \right\vert  \\
&\leqslant &-\left( 1-q^{2}\right) c\sum_{n=0}^{+\infty }q^{2n}g\left(
q^{2n+1}c\right)  = -\int_{0}^{c}g\left( t\right) \tilde{d}_{q}t
=\int_{c}^{0}g\left( t\right) \tilde{d}_{q}t
\end{eqnarray*}
proving the desired result.
\end{proof}

As an immediate consequence, we have the following result.

\begin{corollary}
Let $c\in I$ and $f$ be $q$-symmetric integrable on $I$. Suppose that
\[
f\left(  t\right)  \geqslant0\text{, }
\forall t\in\left\{  q^{2n+1}c:n\in\mathbb{N}_{0}\right\}  \cup\left\{  0\right\}.
\]
\begin{enumerate}
\item If $c\geqslant0$, then
\[
\int_{0}^{c}f\left(  t\right)  \tilde{d}_{q}t\geqslant 0.
\]

\item If $c<0$, then
\[
\int_{c}^{0}f\left(  t\right)  \tilde{d}_{q}t\geqslant 0.
\]
\end{enumerate}
\end{corollary}

\begin{remark}
In general it is not true that if $f$
is a nonnegative function on $\left[a,b\right]$, then
\[
\int_{a}^{b}f\left(  t\right)  \tilde{d}_{q}t \geqslant 0.
\]
For example, consider the function $f$ defined in $\left[  -1,1\right]$ by
\[
f\left(  t\right)  =\left\{
\begin{array}[c]{ccc}
1 & \text{\ if \ } & t=\frac{1}{2}\\
\\
6 & \text{\ \ if \ } & t=\frac{1}{6}\\
\\
0 & \text{\ \ if \ } & t\in\left[  -1,1\right]  \backslash\left\{  \frac{1}
{6},\frac{1}{2}\right\}  .
\end{array}
\right.
\]
For $q=\frac{1}{2}$ this function is $q$-symmetric integrable because is
continuous at $t=0$ and
\begin{align*}
\int_{\frac{1}{3}}^{1}f\left(  t\right)  \tilde{d}_{q}t &
=\int_{0}^{1} f\left(  t\right)  \tilde{d}_{q}t
-\int_{0}^{\frac{1}{3}}f\left(  t\right)
\tilde{d}_{q}t\\
&= \frac{3}{4}\sum_{n=0}^{+\infty}\left(  \frac{1}{2}\right)  ^{2n}f\left(
\left(  \frac{1}{2}\right)  ^{2n+1}\right)  -\frac{3}{4}\left(
\frac{1}{3}\right) \sum_{n=0}^{+\infty}\left(  \frac{1}{2}\right)  ^{2n}f\left(\frac{1}{3}\cdot
\left( \frac{1}{2}\right)  ^{2n+1}\right)  \\
&  =\frac{3}{4}\times1-\frac{1}{4}\times6=-\frac{3}{4}.
\end{align*}
This example also proves that in general it is not true that, for any $a,b\in I$,
\[
\left\vert \int_{a}^{b}f\left(  t\right)  \tilde{d}_{q}t\right\vert
\leqslant\int_{a}^{b}\left\vert f\left(  t\right)
\right\vert \tilde{d}_{q}t\text{.}
\]
\end{remark}

\noindent
We conclude this section with some useful definitions and notations.
For $s\in I$ we set
\[
\left[  s\right]  _{q}:=\left\{  q^{2n+1}s:n\in\mathbb{N}_{0}\right\}  \cup\left\{  0\right\}  .
\]
Let $a,b\in I$ with $a<b$. We define the $q$-symmetric interval from $a$ to $b$ by
\[
\left[  a,b\right]  _{q}:=\left\{  q^{2n+1}a:n\in\mathbb{N}_{0}\right\}  \cup\left\{  q^{2n+1}b:n\in
\mathbb{N}
_{0}\right\}  \cup\left\{  0\right\}  ,
\]
that is,
\[
\left[  a,b\right]  _{q}=\left[  a\right]_{q}\cup \left[  b\right]_{q}.
\]

We introduce the linear spaces
\begin{align*}
\mathcal{Y}^{0}\left(  \left[
a,b\right]  _{q},
\mathbb{R}\right):=&
\left\{  y:I\rightarrow
\mathbb{R}
|\text{ } y  \text{ is bounded on
}\left[  a,b\right]  _{q}\text{ and continuous at }0\right\} \\
\mathcal{Y}^{1}\left(  \left[
a,b\right]  _{q},
\mathbb{R}\right):=&
\left\{  y\in\mathcal{Y}^{0}\left(  \left[
a,b\right]  _{q},
\mathbb{R}
\right)
|\text{ } \tilde{D}_{q}\left[  y\right]  \text{ is bounded on
}\left[  a,b\right]  _{q}\text{ and continuous at }0\right\}
\end{align*}
and, for  $r=0,1$,  we endow these spaces with the norm
\[
\left\Vert y\right\Vert _{r}=\sum_{i=0}^{r}\sup_{t\in\left[  a,b\right]_{q}  }\left\vert
\tilde{D}_{q}^{i}\left[  y\right]  \left(  t\right)  \right\vert
\]
where $\tilde{D}_{q}^{0}\left[  y\right]\equiv y$.

\begin{definition}
We say that $y$ is an admissible function to problem \eqref{qs:P} if
$y\in \mathcal{Y}^{1}\left(  \left[  a,b\right]_{q},\mathbb{R}\right)$
and $y$ satisfies the boundary conditions $y\left(  a\right)
=\alpha$ and $y\left(  b\right) =\beta$.
\end{definition}

\begin{definition}
We say that $y_{\ast}$ is a local minimizer (resp. local maximizer) to
problem \eqref{qs:P} if $y_{\ast}$ is an admissible function and there exists
$\delta>0$ such that
\[
\mathcal{L}\left[  y_{\ast}\right]  \leqslant\mathcal{L}\left[  y\right]
\text{ (resp. }\mathcal{L}\left[  y_{\ast}\right]  \geqslant\mathcal{L}\left[
y\right]  \text{)}
\]
for all admissible $y$ with $\left\Vert y_{\ast}-y\right\Vert _{1}<\delta$.
\end{definition}

\begin{definition}
We say that $\eta\in\mathcal{Y}^{1}\left(\left[a,b\right]_{q},\mathbb{R}\right)$
is an admissible variation to problem \eqref{qs:P}
if $\eta\left(  a\right)  =0=\eta\left(  b\right)$.
\end{definition}


\section{Main results}
\label{qs:sec3}

In this section we apply similar techniques than those used in Hahn's quantum
variational calculus \cite{Brito:da:Cruz,Malinowska:3}.


\subsection{Basic lemmas}

We now present some lemmas which are useful to achieve
our main results.\index{Fundamental lemma of the $q$-symmetric variational calculus}

\begin{lemma}[Fundamental Lemma of the $q$-symmetric variational calculus]
\label{qs:fundamental}
Let $f\in\mathcal{Y}^{0}\left(\left[a,b\right]_q,\mathbb{R}\right)$. One has
\[
\int_{a}^{b}f\left(  t\right)  h\left(  qt\right)  \tilde{d}_{q}t=0
\]
for all functions $h\in\mathcal{Y}^{0}\left(\left[a,b\right]_q,\mathbb{R}\right)$ with
\[
h\left(  a\right)  =h\left(  b\right)  =0
\]
if, and only if,
\[
f\left(  t\right) = 0
\]
for all $t\in\left[  a,b\right]_{q}$.
\end{lemma}

\begin{proof}
The implication \textquotedblleft$\Leftarrow$\textquotedblright\ is obvious.
Let us prove the implication \textquotedblleft$\Rightarrow$\textquotedblright.
Suppose, by contradiction, that there exists $p\in\left[  a,b\right]  _{q}$ such
that $f\left(  p\right)  \neq0$.

\begin{enumerate}
\item If $p\neq0$, then $p=q^{2k+1}a$ or $p=q^{2k+1}b$ for some $k\in\mathbb{N}_{0}$.

\begin{enumerate}
\item Suppose that $a\neq0$ and $b\neq0$. In this case we can assume, without
loss of generality that $p=q^{2k+1}a$. Define
\[
h\left(  t\right)  =\left\{
\begin{array}[c]{ccc}
f\left(  q^{2k+1}a\right)   & \text{\ \ if \ } & t=q^{2k+2}a\\
&  & \\
0 &  & \text{otherwise.}
\end{array}
\right.
\]
Then
\begin{align*}
\int_{a}^{b}f\left(  t\right)  h\left(  qt\right)  \tilde{d}_{q}t &  =\int
_{0}^{b}f\left(  t\right)  h\left(  qt\right)  \tilde{d}_{q}t
-\int_{0}^{a} f\left(  t\right)  h\left(  qt\right)  \tilde{d}_{q}t\\
&  =0-\left(  1-q^{2}\right)  a\sum_{n=0}^{+\infty}q^{2n}f\left(
q^{2n+1}a\right)  h\left(  q^{2n+2}a\right)  \\
&  =-\left(  1-q^{2}\right)  aq^{2k}\left[  f\left(  q^{2k+1}a\right)
\right]  ^{2}\neq0,
\end{align*}
which is a contradiction.

\item If $a=0$ and $b\neq0$, then $p=q^{2k+1}b$ for some $k\in\mathbb{N}_{0}$. Define
\[
h\left(  t\right)  =\left\{
\begin{array}[c]{ccc}
f\left(  q^{2k+1}b\right)   & \text{\ \ if \ } & t=q^{2k+2}b\\
&  & \\
0 &  & \text{otherwise}
\end{array}
\right.
\]
and with a similar proof to (a) we obtain a contradiction.

\item The case $b=0$ and $a\neq0$ is similar to the previous one.
\end{enumerate}

\item If $p=0$, without loss of generality, we can assume $f\left(  p\right) >0$. Since
\[
\lim_{n\rightarrow+\infty}q^{2n+1}a=\lim_{n\rightarrow+\infty}q^{2n+1}b=0
\]
and $f$ is continuous at $0$ we have
\[
\lim_{n\rightarrow+\infty}f\left(  q^{2n+1}a\right)
=\lim_{n\rightarrow +\infty}
f\left(  q^{2n+1}b\right)  =f\left(  0\right) .
\]
Therefore, there exists an order $n_{0}\in\mathbb{N}$ such that for all $n>n_{0}$ the inequalities
\[
f\left(  q^{2n+1}a\right)  >0\text{ and }f\left(  q^{2n+1}b\right)  >0
\]
hold.

\begin{enumerate}
\item If $a,b\neq0$, then for some $k>n_{0}$ we define
\[
h\left(  t\right)  =\left\{
\begin{array}[c]{ccc}
f\left(  q^{2k+1}b\right)   & \text{\ \ if \ } & t=q^{2k+2}a\\
&  & \\
f\left(  q^{2k+1}a\right)   & \text{\ \ if \ } & t=q^{2k+2}b\\
&  & \\
0 &  & \text{otherwise.}
\end{array}
\right.
\]
Hence
\[
\int_{a}^{b}f\left(  t\right)  h\left(  qt\right)  \tilde{d}_{q}t=\left(
1-q^{2}\right)  \left(  b-a\right)  q^{2k}\left[  f\left(  q^{2k+1}a\right)
f\left(  q^{2k+1}b\right)  \right]  \neq0.
\]

\item If $a=0$, then we define
\[
h\left(  t\right)  =\left\{
\begin{array}[c]{ccc}
f\left(  q^{2k+1}b\right) & \text{\ \ if \ } & t=q^{2k+2}b\\
&  & \\
0 &  & \text{otherwise.}
\end{array}
\right.
\]
Therefore,
\[
\int_{0}^{b}f\left(  t\right)  h\left(  qt\right)  \tilde{d}_{q}t=\left(
1-q^{2}\right)  bq^{2k}\left[  f\left(  q^{2k+1}b\right)  \right]  ^{2}\neq 0.
\]

\item If $b=0$, the proof is similar to the previous case.
\end{enumerate}
\end{enumerate}
\end{proof}

\begin{definition}[\cite{Malinowska:3}]
Let $s\in I$ and $g:I\times\left]  -\bar{\theta},\bar{\theta}\right[
\rightarrow\mathbb{R}$. We say that $g\left(  t,\cdot\right)$
is differentiable at $\theta_{0}$ uniformly in $\left[  s\right]_{q}$
if, for every $\varepsilon>0$, there exists $\delta>0$ such that
\[
0<\left\vert \theta-\theta_{0}\right\vert <\delta\Rightarrow\left\vert
\frac{g\left(  t,\theta\right)  -g\left(  t,\theta_{0}\right)  }{\theta
-\theta_{0}}-\partial_{2}g\left(  t,\theta_{0}\right)  \right\vert
<\varepsilon
\]
for all $t\in\left[  s\right]_{q}$,
where $\partial_{2}g=\frac{\partial g}{\partial\theta}$.
\end{definition}

\begin{lemma}[cf. \cite{Malinowska:3}]
\label{qs:tecnico}
Let $s\in I$ and assume that $g:I\times\left]-\bar{\theta}
,\bar{\theta}\right[  \rightarrow\mathbb{R}$
is differentiable at $\theta_{0}$ uniformly in $\left[s\right]_{q}$.
If
$\displaystyle \int_{0}^{s}g\left(  t,\theta_{0}\right)  \tilde{d}_{q}t$
exists, then
$\displaystyle G\left(  \theta\right)
:=\int_{0}^{s}g\left(  t,\theta\right) \tilde{d}_{q}t$
for $\theta$\ near $\theta_{0}$, is differentiable at $\theta_{0}$ and
\[
G^{\prime}\left(  \theta_{0}\right)  =\int_{0}^{s}\partial_{2}g\left(
t,\theta_{0}\right)  \tilde{d}_{q}t.
\]
\end{lemma}

\begin{proof}
For $s>0$ the proof is similar to the proof given in Lemma~3.2 of \cite{Malinowska:3}.
The result is trivial for $s=0$. For $s<0$, let $\varepsilon >0$ be
arbitrary. Since $g\left( t,\cdot \right) $ is differentiable at $\theta _{0}
$ uniformly in $\left[ s\right] _{q}$, then there exists $\delta >0$ such that, for all
$t\in \left[ s\right] _{q}$ and for $0<\left\vert \theta -\theta_{0}\right\vert <\delta$,
\[
\left\vert \frac{g\left( t,\theta \right) -g\left( t,\theta _{0}\right) }{\theta -\theta_{0}}
-\partial _{2}g\left( t,\theta _{0}\right) \right\vert < - \frac{\varepsilon }{2s}.
\]
Applying Proposition~\ref{qs:desigualdade}, for $0<\left\vert \theta -\theta
_{0}\right\vert <\delta $, we have
\begin{align*}
\left\vert \frac{G\left( \theta \right) -G\left( \theta _{0}\right) }{\theta
-\theta _{0}}-\int_{0}^{s}\partial
_{2}g\left( t,\theta _{0}\right) \tilde{d}_{q}t \right\vert & =\left\vert
\frac{\int_{0}^{s}g\left( t,\theta \right) \tilde{d}_{q}t-\int_{0}^{s}g\left(
t,\theta _{0}\right) \tilde{d}_{q}t}{\theta -\theta _{0}}-\int_{0}^{s}\partial
_{2}g\left( t,\theta _{0}\right) \tilde{d}_{q}t\right\vert  \\
& =\left\vert \int_{0}^{s}\left[ \frac{g\left( t,\theta \right) -g\left(
t,\theta _{0}\right) }{\theta -\theta _{0}}-\partial _{2}g\left( t,\theta
_{0}\right) \right] \tilde{d}_{q}t\right\vert  \\
& < \int_{s}^{0}-\frac{\varepsilon }{2s}\tilde{d}_{q}t
= -\frac{\varepsilon }{2s}\int_{s}^{0}1\tilde{d}_{q}t = \frac{\varepsilon}{2}<\varepsilon
\end{align*}
proving that
$$
G'\left(\theta_0\right)=
\int_{0}^{s}\partial_{2}
g\left( t,\theta _{0}\right) \tilde{d}_{q}t.
$$
\end{proof}

For an admissible variation $\eta$ and
an admissible function $y$, we define the real function
$\phi$ by
\[
\phi\left(  \epsilon\right)
:=\mathcal{L}\left[  y+\epsilon\eta\right]  .
\]
The first variation of the functional $\mathcal{L}$
of the problem \eqref{qs:P} is defined by
\[
\delta\mathcal{L}\left[  y,\eta\right]  :=\phi^{\prime}\left(  0\right).
\]
Note that
\begin{align*}
\mathcal{L}\left[  y+\epsilon\eta\right]
&= \int_{a}^{b}L\left(  t,y\left(
qt\right)+\epsilon\eta\left(
qt\right), \tilde{D}_{q}\left[  y\right]\left(  t\right)
+\epsilon\tilde{D}_{q}\left[\eta
\right]  \left(  t\right)  \right)
\tilde{d}_{q}t\\
&= \mathcal{L}_{b}\left[  y+\epsilon\eta\right]
-\mathcal{L}_{a}\left[y+\epsilon\eta\right],
\end{align*}
where
\[
\mathcal{L}_{\xi}\left[  y+\epsilon\eta\right]
=\int_{0}^{\xi}L\left(t,y\left(  qt\right)
+\epsilon\eta\left(qt\right), \tilde{D}_{q}\left[  y\right]  \left(  t\right)
+\epsilon\tilde{D}_{q}\left[\eta\right]  \left(  t\right)\right)\tilde{d}_{q}t
\]
with $\xi\in\left\{  a,b\right\}  $. Therefore,
\[
\delta\mathcal{L}\left[  y,\eta\right]  =\delta\mathcal{L}_{b}\left[
y,\eta\right]  -\delta\mathcal{L}_{a}\left[  y,\eta\right]  .
\]
The following lemma is a direct consequence of Lemma~\ref{qs:tecnico}.

\begin{lemma}
\label{qs:tecnico 2}
For an admissible variation $\eta$ and an admissible function $y$, let
\[
g\left(  t,\epsilon\right)  :=L\left(  t,y\left(  qt\right)  +\epsilon
\eta\left(  qt\right)  ,\tilde{D}_{q}\left[  y\right]  \left(  t\right)
+\epsilon\tilde{D}_{q}\left[\eta
\right]  \left(  t\right)  \right)  .
\]
Assume that
\begin{enumerate}
\item $g\left(  t,\cdot\right)$ is differentiable at $0$ uniformly in
$\left[  a,b\right]_{q}$;

\item $\mathcal{L}_{a}\left[  y+\epsilon\eta\right]  =\int_{0}^{a}g\left(
t,\epsilon\right)  \tilde{d}_{q}t$ and $\mathcal{L}_{b}\left[  y+\epsilon
\eta\right]  =\int_{0}^{b}g\left(  t,\epsilon\right)  \tilde{d}_{q}t$ exist
for $\epsilon\approx 0$;

\item $\int_{0}^{a}\partial_{2}g\left(  t,0\right)  \tilde{d}_{q}t$
and $\int_{0}^{b}\partial_{2}g\left(  t,0\right)  \tilde{d}_{q}t$ exist.
\end{enumerate}
Then
\begin{align*}
\phi^{\prime}\left(  0\right) & =\delta\mathcal{L}\left[  y,\eta\right] \\
& = \int_{a}^{b}\bigg(\partial_{2}L\left(  t,y\left(  qt\right),
\tilde{D}_{q}\left[  y\right]  \left(  t\right)  \right)  \eta\left(  qt\right)
+\partial_{3}L\left(  t,y\left(  qt\right)  ,\tilde{D}_{q}\left[  y\right]
\left(  t\right)  \right)  \tilde{D}_{q}\left[\eta\right]\left(  t\right)
\bigg)\tilde{d}_{q}t.
\end{align*}
\end{lemma}


\subsection{Optimality conditions}
\label{Optimality conditions}

In this section we present a necessary condition
(the $q$-symmetric Euler--Lagrange equation) and a sufficient
condition to our problem \eqref{qs:P}.\index{$q$-symmetric Euler--Lagrange equation}

\begin{theorem}[The $q$-symmetric Euler--Lagrange equation]
\label{qs:Euler}
Under hypotheses (H$_{q}$1)-(H$_{q}$3) and conditions 1-3 of Lemma~\ref{qs:tecnico 2}
on the Lagrangian $L$, if $y_{\ast}\in\mathcal{Y}^{1}\left(\left[a,b\right]_q,\mathbb{R}\right)$
is a local extremizer for problem \eqref{qs:P}, then $y_{\ast}$ satisfies
the $q$-symmetric Euler--Lagrange equation
\begin{equation}
\label{qs:eqEuler}
\partial_{2}L\left(  t,y\left(  qt\right)  ,\tilde{D}_{q}\left[  y\right]
\left(  t\right)  \right)  =\tilde{D}_{q}\left[\tau\rightarrow
\partial_{3}L\left(  q\tau,y\left(  q^{2}\tau\right)  ,\tilde{D}_{q}\left[  y\right]
\left(  q\tau\right)  \right)  \right]  \left(  t\right)
\end{equation}
for all $t\in\left[a,b\right]_{q}$.
\end{theorem}

\begin{proof}
Let $y_{\ast}$ be a local minimizer (resp. maximizer)
to problem \eqref{qs:P} and $\eta$ an admissible variation.
Define $\phi: \mathbb{R} \rightarrow\mathbb{R}$ by
\[
\phi\left(  \epsilon\right)  :=\mathcal{L}\left[  y_{\ast}+\epsilon
\eta\right]  .
\]
A necessary condition for $y_{\ast}$ to be an extremizer is given by
$\phi^{\prime}\left(  0\right)  =0$. By Lemma~\ref{qs:tecnico 2} we conclude that
\[
\int_{a}^{b}\bigg(\partial_{2}L\left(  t,y_{\ast}\left(  qt\right),
\tilde{D}_{q}\left[  y_{\ast}\right]  \left(  t\right)  \right)  \eta\left(  qt\right)
+\partial_{3}L\left(  t,y_{\ast}\left(  qt\right)  ,\tilde{D}_{q}\left[  y_{\ast}\right]
\left(  t\right)  \right)  \tilde{D}_{q}\left[  \eta\right]  \left(  t\right)
\bigg)\tilde{d}_{q}t=0.
\]
By integration by parts (equation \eqref{qs:partes}) we get
\begin{align*}
 & \int_{a}^{b}\partial_{3}L\left(  t,y_{\ast}\left(  qt\right)  ,\tilde{D}
_{q}\left[  y_{\ast}\right]  \left(  t\right)  \right)  \tilde{D}_{q}\left[
\eta\right]  \left(  t\right)  \tilde{d}_{q}t\\
=&\partial_{3}L\left(  qt,y_{\ast}\left(  q^{2}t\right)  ,\tilde{D}_{q}\left[
y_{\ast}\right]  \left(  qt\right)  \right)  \eta\left(  t\right)  \bigg|_{a}^{b}\\
 & -q\int_{a}^{b}\tilde{D}_{q}\left[\tau\rightarrow  \partial_{3}L\left(
\tau,y_{\ast}\left(  q\tau\right)  ,\tilde{D}_{q}\left[  y_{\ast}\right]  \left(
\tau\right)  \right)  \right]  \left(  qt\right)  \eta\left(  qt\right)
\tilde{d}_{q}t.
\end{align*}
Since $\eta\left(  a\right)  =\eta\left(  b\right)  =0$, then
\[
\int_{a}^{b}\bigg(\partial_{2}L\left(  t,y_{\ast}\left(  qt\right),
\tilde{D}_{q}\left[  y_{\ast}\right]  \left(  t\right)  \right)
-q\tilde{D}_{q}\left[\tau\rightarrow  \partial_{3}L\left(  \tau,y_{\ast}\left(  q\tau\right),
\tilde{D}_{q}\left[  y_{\ast}\right]  \left(  \tau\right)  \right)  \right]  \left(
qt\right)  \bigg)\eta\left(  qt\right)  \tilde{d}_{q}t=0.
\]
Finally, by Lemma~\ref{qs:fundamental}
and equation \eqref{qs:comp}, for all $t\in\left[ a,b\right]_{q}$
\begin{align*}
&\partial_{2}L\left(  t,y_{\ast}\left(  qt\right)  ,\tilde{D}_{q}\left[  y_{\ast}\right]
\left(  t\right)  \right)  =q\tilde{D}_{q}\left[\tau\rightarrow  \partial
_{3}L\left(  \tau,y_{\ast}\left(  q\tau\right)  ,\tilde{D}_{q}\left[  y_{\ast}\right]
\left(  \tau\right)  \right)  \right]  \left(  qt\right)\\
\Leftrightarrow & \partial_{2}L\left(  t,y_{\ast}\left(  qt\right)  ,\tilde{D}_{q}\left[  y_{\ast}\right]
\left(  t\right)  \right)  =\tilde{D}_{q}\left[\tau\rightarrow  \partial
_{3}L\left(  q\tau,y_{\ast}\left(  q^{2}\tau\right)  ,\tilde{D}_{q}\left[  y_{\ast}\right]
\left(  q\tau\right)  \right)  \right]  \left(  t\right) \text{.}
\end{align*}
\end{proof}

Observe that the $q$-symmetric Euler--Lagrange equation \eqref{qs:eqEuler}
is a second-order dynamic equation. To the best of our knowledge, there is no
general method to solve this type of equation.
We believe this is an interesting open problem.

To conclude this section we prove a sufficient
optimality condition for problem \eqref{qs:P}.

\begin{definition}
Given a function $L$, we say that $L\left(t,u,v\right)$
is \emph{jointly convex}\index{Function jointly convex}
(resp. \emph{jointly concave}\index{Function jointly concave})
in $\left(u,v\right)$, if, and only if, $\partial_{i}L$, $i=2,3$,
exist and are continuous and verify the following condition:
\[
L\left(  t,u+u_{1},v+v_{1}\right)  -L\left(  t,u,v\right)
\underset{\left(resp. \leqslant\right)}{\geqslant}\partial_{2}
L\left(  t,u,v\right) u_{1}+\partial_{3}L\left(  t,u,v\right)  v_{1}
\]
for all $\left(  t,u,v\right),\left(  t,u+u_{1},v
+v_{1}\right)  \in I\times\mathbb{R}^{2}$.
\end{definition}

\begin{theorem}
\label{qs:suff}
Suppose\index{$q$-symmetric optimality sufficient \\ condition}
that $a, b \in I$, $a<b$, and $a,b\in \left[ c \right]_{q}$ for some $c\in I$.
Also assume that  $L$ is a jointly convex (resp. concave) function
in $\left(u,v\right)$. If $y_{\ast}$ satisfies the $q$-symmetric
Euler--Lagrange equation \eqref{qs:eqEuler}, then $y_{\ast}$ is a global minimizer
(resp. maximizer) to the problem \eqref{qs:P}.
\end{theorem}

\begin{proof}
Let $L$ be a jointly convex function in $\left(u,v\right)$
(the concave case is similar). Then, for
any admissible variation $\eta$, we have
\begin{align*}
& \mathcal{L}\left[  y_{\ast}+\eta\right]  -\mathcal{L}\left[  y_{\ast}\right]  \\
& =\int_{a}^{b}\left[L\left(  t,y_{\ast}\left(  qt\right)  +\eta\left(  qt\right)
,\tilde{D}_{q}\left[  y_{\ast}\right]  \left(  t\right)  +\tilde{D}_{q}\left[
\eta\right]  \left(  t\right)  \right)  -L\left(  t,y_{\ast}\left(  qt\right)
,\tilde{D}_{q}\left[  y_{\ast}\right]  \left(  t\right)  \right)\right]  \tilde{d}_{q}t\\
& \geqslant\int_{a}^{b}\left[  \partial_{2}L\left(  t,y_{\ast}\left(  qt\right)
,\tilde{D}_{q}\left[  y_{\ast}\right]  \left(  t\right)  \right)\eta\left(qt\right)
+\partial_{3} L\left(  t,y_{\ast}\left(  qt\right)  ,\tilde{D}_{q}\left[  y_{\ast}\right]
\left(t\right)  \right)\tilde{D}_{q}\left[\eta\right]\left(t\right)  \right]  \tilde{d}_{q}t
\end{align*}
Using integrations by parts, formula \eqref{qs:partes}, we get
\begin{align*}
\mathcal{L}\left[  y_{\ast}+\eta\right] - \mathcal{L}\left[  y_{\ast}\right]
& \geqslant\partial_{3}L\left(  qt,y_{\ast}\left(  q^{2}t\right)  ,\tilde
{D}_{q}\left[  y_{\ast}\right]  \left(  qt\right)  \right)  \eta\left(  t\right)
\bigg|_{a}^{b}\\
& +\int_{a}^{b}\bigg(\partial_{2}L\left(  t,y_{\ast}\left(  qt\right)
,\tilde{D}_{q}\left[  y_{\ast}\right]  \left(  t\right)  \right)  \\
& -\tilde{D}_{q}\left[  \tau \rightarrow \partial_{3}
L\left( q \tau,y_{\ast}\left(
q^{2} \tau \right)  ,\tilde{D}_{q}\left[  y_{\ast}\right]  \left(  q \tau\right)  \right)
\right]  \left(  t\right)  \bigg)\eta\left(  qt\right)\tilde{d}_{q}t.
\end{align*}
Since $y_{\ast}$ satisfies \eqref{qs:eqEuler} and $\eta$ is an admissible
variation, we obtain
\[
\mathcal{L}\left[  y_{\ast}+\eta\right]  -\mathcal{L}\left[  y_{\ast}\right]
\geqslant0
\]
proving that $y_{\ast}$ is a minimizer of problem \eqref{qs:P}.
\end{proof}

\begin{ex}
Let $q\in\left]0,1\right[$ be a fixed real number and $I\subseteq\mathbb{R}$
be an interval such that $0,1\in I$. Consider the problem
\begin{align}
\label{problemaq}
\left\{
\begin{array}
[c]{l}
\mathcal{L}\left[  y\right]  =\displaystyle\int_{0}^{1}\left(1+\left(  \tilde{D}_{q}\left[
y\right]  \left(  t\right)  \right)  ^{2}\right)\tilde{d}_{q}t\rightarrow \text{minimize}\\
\\
y\in\mathcal{Y}^{1}\left(  \left[  0,1\right]_{q}  ,\mathbb{R}\right)\\
\\
y\left(  0\right)  =0\\
\\
y\left(  1\right)  =1.
\end{array}
\right.
\end{align}
If $y_{\ast}$ is a local minimizer of the problem,
then $y_{\ast}$ satisfies the $q$-symmetric Euler--Lagrange equation
\[
\tilde{D}_{q}\left[\tau\rightarrow  2\tilde{D}_{q}\left[  y\right]  \left(  q\tau\right)
\right]  \left(  t\right)  =0\text{ for all }t\in\left[  0,1\right]_{q}.
\]
It's simple to see that the function
\[
y_{\ast}\left(  t\right)  =t
\]
is a candidate to the solution of problem \eqref{problemaq}.
Since the Lagrangian function is jointly convex in $\left(u,v\right)$,
then we may conclude that the function $y_{\ast}$ is a minimizer
of problem \eqref{problemaq}.
\end{ex}


\section{State of the Art}

The study of the $q$-quantum calculus of variations
and its applications is under current research \cite{Bangerezako,Bangerezako:2}.
In this chapter not only we developed it with new techniques,
but also we introduced the $q$-symmetric quantum calculus of variations.
These results were presented by the author at the international conference
Progress on Difference Equations 2011, Dublin City University, Ireland,
and are published in \cite{Brito:da:Cruz:2}.


\clearpage{\thispagestyle{empty}\cleardoublepage}


\chapter{Hahn's Symmetric Quantum Variational Calculus}
\label{Hahn's Symmetric Quantum Variational Calculus}

In this chapter we introduce and develop the Hahn symmetric quantum calculus
with applications to the calculus of variations. Namely, we obtain a necessary
optimality condition of Euler--Lagrange type and a sufficient optimality condition
for variational problems within the context of Hahn's symmetric calculus.
Moreover, we show the effectiveness of Leitmann's direct method when applied
to Hahn's symmetric variational calculus. Illustrative examples are provided.

It should be noted that the $q$-symmetric derivative is a particular case
of Hahn's symmetric derivative, and so this chapter is a generalization
of Chapter~\ref{The $q$-Symmetric Variational Calculus}.


\section{Introduction}

Due to its many applications, quantum operators are recently subject
to an increase number of investigations \cite{Malinowska:3,Martins:2,Martins:3}.
The use of quantum differential operators, instead of classical derivatives,
is useful because they allow to deal with sets
of nondifferentiable functions \cite{Almeida:5,Cresson}.
Applications include several fields of physics, such as cosmic strings
and black holes \cite{Strominger}, quantum mechanics \cite{Feynman,Youm},
nuclear and high energy physics \cite{Lavagno}, just to mention a few.
In particular, the $q$-symmetric quantum calculus
has applications in quantum mechanics \cite{Lavagno2}.

In 1949, Hahn introduced his quantum difference operator \cite{Hahn},
which is a generalization of the quantum $q$-difference operator defined
by Jackson \cite{Jackson}. However, only in 2009, Aldwoah \cite{Aldwoah}
defined the inverse of Hahn's difference operator, and short after,
Malinowska and Torres \cite{Malinowska:3} introduced and investigated
the Hahn quantum variational calculus. For a deep understanding
of quantum calculus, we refer the reader to
\cite{Aldwoah:3,Boole,Brito:da:Cruz,Ernst:2,Kac,Koekoek} and references therein.

For a fixed $q\in \left] 0,1\right[$ and an $\omega \geqslant 0$, we introduce here
the Hahn symmetric difference operator of function $f$ at point $t\neq \displaystyle
\frac{\omega }{1-q}$ by
\begin{equation*}
\tilde{D}_{q,\omega }\left[ y\right] \left( t\right)
=\frac{f\left( qt+\omega \right) -f\left( q^{-1}\left( t-\omega \right)
\right) }{\left( q-q^{-1}\right) t+\left( 1+q^{-1}\right) \omega }.
\end{equation*}
Our main aim is to establish a necessary optimality condition
and a sufficient optimality condition for the Hahn symmetric
variational problem\index{Hahn's symmetric variational problem}
\begin{equation}
\label{qhs:P}
\left\{
\begin{array}[c]{l}
\mathcal{L}(y) =\displaystyle\int_{a}^{b}L\left( t,y^{\sigma }\left( t\right)
,\tilde{D}_{q,\omega }\left[ y\right] \left( t\right) \right)
\tilde{d}_{q,\omega }t \longrightarrow \textrm{extremize} \\
\\
y\in \mathcal{Y}^{1}\left( \left[ a,b\right]_{q,\omega} ,\mathbb{R}
\right)  \\
\\
y\left( a\right) =\alpha,
\quad y\left( b\right) =\beta ,
\end{array}
\right.
\end{equation}
where $\alpha$ and $\beta $ are fixed real numbers,
and extremize means maximize or minimize. Problem \eqref{qhs:P}
will be clear and precise after definitions of Section~\ref{qhs:Pre}.
We assume that the Lagrangian $L$ satisfies the following hypotheses:
\begin{enumerate}
\item[(H$_{q,\omega}$1)] $\left( u,v\right) \rightarrow L\left( t,u,v\right)$
is a $C^{1}\left(\mathbb{R}^{2}, \mathbb{R}\right)$ function for any $t\in I$;

\item[(H$_{q,\omega}$2)] $t\rightarrow L\left( t,y^{\sigma }\left( t\right),
\tilde{D}_{q,\omega }\left[ y\right]\left(t\right) \right)$
is continuous at $\omega_{0}$ for any admissible function $y$;

\item[(H$_{q,\omega}$3)] functions $t\rightarrow \partial_{i+2}L\left(t,
y^{\sigma}\left( t\right),
\tilde{D}_{q,\omega }\left[ y\right]\left(t\right) \right)$
belong to $\mathcal{Y}^{1}\left( \left[a,b\right]_{q,\omega},
\mathbb{R}\right)$ for all admissible $y$, $i=0,1$;
\end{enumerate}
where $I$ is an interval of $\mathbb{R}$ containing
$\omega _{0}:=\displaystyle\frac{\omega }{1-q}$,
$a,b\in I$, $a<b$, and $\partial_{j}L$ denotes
the partial derivative of $L$ with respect to its $j$th argument.

In Section~\ref{qhs:Pre} we introduce the necessary
definitions and prove some basic results for the Hahn symmetric calculus.
In Section~\ref{qhs:M} we formulate and prove our main results for the Hahn symmetric
variational calculus. New results include a necessary optimality
condition (Theorem~\ref{qhs:Euler}) and a sufficient optimality condition
(Theorem~\ref{qhs:Suficiente}) to problem \eqref{qhs:P}. In Section~\ref{qhs:L}
we show that Leitmann's direct method can also be applied to variational
problems within Hahn's symmetric variational calculus.
Leitmann introduced his direct method in the sixties of the 20th century
\cite{Leitmann3}, and the approach has recently proven to be universal: see, \textrm{e.g.},
\cite{Almeida:7,Leitmann1,Leitmann2,Leitmann4,Leitmann5,Leitmann6,Malinowska:4,Torres:2}.


\section{Hahn's symmetric calculus}
\label{qhs:Pre}

Let $q\in \left] 0,1\right[ $ and $\omega \geqslant 0$ be real fixed numbers.
Throughout this chapter, we make the assumption that $I$ is an interval
(bounded or unbounded) of $\mathbb{R}$ containing
$\omega _{0}:=\displaystyle\frac{\omega }{1-q}$. We denote by
$I^{q,\omega }$ the set $I^{q,\omega }:=qI+\omega :=\left\{ qt+\omega: t\in I\right\}$.
Note that $I^{q,\omega }\subseteq I$ and, for all $t\in I^{q,\omega }$, one has
$q^{-1}\left(t-\omega\right)\in I$ . For $k\in \mathbb{N}_{0}$,
$$
\left[ k\right] _{q}:=\displaystyle\frac{1-q^{k}}{1-q}.
$$

\begin{definition}
\label{qhs:def:dhsd}
Let $f$ be a real function defined on $I$. The \emph{Hahn symmetric difference
operator}\index{Hahn's symmetric difference operator} of $f$ at a point
$t\in I^{q,\omega }\backslash \left\{ \omega _{0}\right\}$ is defined by
\begin{equation*}
\tilde{D}_{q,\omega }\left[ f\right] \left( t\right) =\frac{f\left(
qt+\omega \right) -f\left( q^{-1}\left( t-\omega \right) \right) }{\left(
q-q^{-1}\right) t+\left( 1+q^{-1}\right) \omega },
\end{equation*}
while $\tilde{D}_{q,\omega }\left[ f\right] \left( \omega _{0}\right)
:=f^{\prime }\left( \omega _{0}\right)$, provided $f$ is differentiable at
$\omega _{0}$ (in the classical sense). We call to
$\tilde{D}_{q,\omega }\left[ f\right]$
the \emph{Hahn symmetric derivative}\index{Hahn's symmetric derivative} of $f$.
\end{definition}

\begin{remark}
If $\omega =0$, then the Hahn symmetric difference operator $\tilde{D}_{q,\omega }$
coincides with the $q$-symmetric difference operator $\tilde{D}_{q}$
(see Definition~\ref{q:symmetric:derivative}): if $t\neq 0$, then
\begin{equation*}
\tilde{D}_{q,0}\left[ f\right](t)
=\frac{f\left( qt\right)
-f\left( q^{-1}t\right) }{\left( q-q^{-1}\right) t}
=: \tilde{D}_{q}\left[ f\right] \left( t\right);
\end{equation*}
for $t = 0$ and $f$ differentiable at $0$,
$\tilde{D}_{q,0}\left[ f\right](0) = f^{\prime }\left( 0\right)
=: \tilde{D}_{q}\left[ f\right] \left( 0\right)$.
\end{remark}

\begin{remark}
If $\omega> 0$ and we let $q\rightarrow 1$ in Definition~\ref{qhs:def:dhsd},
then we obtain the well known symmetric difference operator $\tilde{D}_{\omega}$:
\begin{equation*}
\tilde{D}_{\omega}\left[ f\right] \left( t\right) :=\frac{f\left( t+\omega\right)
-f\left( t-\omega\right) }{2\omega}.
\end{equation*}
\end{remark}

\begin{remark}
If $f$ is differentiable at $t\in I^{q,\omega }$ in the classical sense, then
\begin{equation*}
\lim_{(q,\omega)\rightarrow (1,0)}\tilde{D}_{q,\omega }\left[ f
\right] \left( t\right) =f^{\prime }\left( t\right) .
\end{equation*}
\end{remark}

In what follows we make use of the operator
$\sigma$ defined by $\sigma \left( t\right) :=qt+\omega $, $t\in I$.
Note that the inverse operator of $\sigma$, $\sigma^{-1}$,
is defined by $\sigma^{-1}\left(t\right):=q^{-1}\left(t-\omega\right)$.
Moreover, Aldwoah \cite[Lemma~6.1.1]{Aldwoah} proved the following useful result.

\begin{lemma}[\cite{Aldwoah}]
\label{qhs:lemma:2.2}
Let $k\in \mathbb{N}$ and $t\in I$. Then,
\begin{enumerate}
\item $\sigma^{k}\left( t\right)
=\underset{k\text{ times}}{\underbrace{\sigma \circ \sigma \circ
\cdots \circ \sigma }}\left( t\right) =q^{k}t+\omega\left[ k\right]_{q}$;

\item $\left( \sigma^{k}\left( t\right) \right)^{-1}=\sigma^{-k}\left(
t\right) =\displaystyle q^{-k}\left( t-\omega \left[ k\right] _{q}\right)$.
\end{enumerate}
Furthermore, $\{\sigma^{k}\left( t\right)\}_{k=1}^{\infty}$
is a decreasing (resp. an increasing) sequence in $k$ when $t > \omega_0$
(resp. $t < \omega_0$) with
$$
\omega_0 = \inf_{k \in \mathbb{N}} \sigma^{k}(t)
\quad \left(resp. \ \ \omega_0 = \sup_{k \in \mathbb{N}} \sigma^{k}(t)\right).
$$
The sequence $\{\sigma^{-k}(t)\}_{k=1}^{\infty}$
is increasing (resp. decreasing) when $t > \omega_0$ (resp. $t < \omega_0$) with
$$
+\infty = \sup_{k \in \mathbb{N}} \sigma^{-k}(t)
\quad \left(resp. \ \ -\infty = \inf_{k \in \mathbb{N}} \sigma^{-k}(t)\right).
$$
\end{lemma}

For simplicity of notation, we write
$f\left( \sigma \left( t\right) \right) :=f^{\sigma}\left( t\right)$.

\begin{remark}
With above notations, if $t\in I^{q,\omega }\backslash \left\{ \omega _{0}\right\}$,
then the Hahn symmetric difference operator of $f$ at point $t$ can be written as
\begin{equation*}
\tilde{D}_{q,\omega }\left[ f\right] \left( t\right) =\frac{f^{\sigma}\left( t\right)
-f^{\sigma ^{-1}}\left( t\right) }{\sigma \left( t\right)-\sigma ^{-1}\left( t\right)}.
\end{equation*}
\end{remark}

\begin{lemma}
\label{qhs:q^n}
Let $n \in \mathbb{N}_{0}$ and $t\in I$. Then,
\begin{equation*}
\sigma^{n+1}\left( t\right) -\sigma ^{n-1}\left( t\right) =q^{n}\left(
\sigma\left( t\right) - \sigma ^{-1}\left( t\right) \right),
\end{equation*}
where $\sigma^{0}\equiv id$ is the identity function.
\end{lemma}

\begin{proof}
The equality follows by direct calculations:
\begin{equation*}
\begin{split}
\sigma^{n+1}(t) -\sigma ^{n-1}\left( t\right)&=q^{n+1}t
+\omega \left[ n+1\right] _{q}-q^{n-1}t-\omega \left[ n-1\right]_{q}\\
&= q^{n}\left( q-q^{-1}\right) t+\omega \left( q^{n}+q^{n-1}\right)\\
&=q^{n}\left( qt+\omega -q^{-1}t+q^{-1}\omega \right)\\
&=q^{n}\left( \sigma \left( t\right) -\sigma ^{-1}\left( t\right) \right).
\end{split}
\end{equation*}
\end{proof}

The Hahn symmetric difference operator has the following properties.

\begin{theorem}
\label{qhs:props derivada}
Let $\alpha, \beta \in \mathbb{R}$ and $t\in I^{q,\omega }$.
If $f$ and $g$ are Hahn symmetric differentiable on $I$, then
\begin{enumerate}
\item $\tilde{D}_{q,\omega}\left[ \alpha f+\beta g\right] \left( t\right)
=\alpha\tilde{D}_{q,\omega}\left[ f\right] \left( t\right)
+\beta\tilde {D}_{q,\omega}\left[ g\right] \left( t\right)$;

\item $\tilde{D}_{q,\omega}\left[ fg\right] \left( t\right)
=\tilde {D}_{q,\omega}\left[ f\right] \left( t\right) g^{\sigma}\left( t\right)
+f^{\sigma^{-1}}\left(t\right) \tilde{D}_{q,\omega}\left[ g\right] \left(t\right)$;

\item $\tilde{D}_{q,\omega }\left[ \displaystyle\frac{f}{g}\right] \left(
t\right) =\displaystyle\frac{\tilde{D}_{q,\omega }\left[ f\right] \left(
t\right) g^{\sigma ^{-1}}\left( t\right) -f^{\sigma ^{-1}}\left( t\right)
\tilde{D}_{q,\omega }\left[ g\right] \left( t\right) }{g^{\sigma }\left(
t\right) g^{\sigma ^{-1}}\left( t\right) }$ if $g^{\sigma }\left( t\right)
g^{\sigma ^{-1}}\left( t\right) \neq 0$;

\item $\tilde{D}_{q,\omega }\left[ f\right]
\equiv 0$ if, and only if, $f$ is constant on $I$.
\end{enumerate}
\end{theorem}

\begin{proof}
For $t=\omega_{0}$ the equalities are trivial
(note that $\sigma(\omega_0)=\omega_0=\sigma^{-1}(\omega_0)$).
We do the proof for $t\neq\omega_{0}$:
\begin{enumerate}
\item
\begin{align*}
\tilde{D}_{q,\omega}\left[ \alpha f+\beta g\right] \left( t\right)
&= \frac{\left( \alpha f+\beta g\right)^{\sigma}\left( t\right)
-\left(\alpha f +\beta g\right)^{\sigma^{-1}}\left( t\right)}{\sigma\left( t\right)
-\sigma^{-1}\left( t\right)}\\
&= \alpha\frac{f^{\sigma}\left( t\right)
-f^{\sigma^{-1}}\left( t\right)}{\sigma\left( t\right)
-\sigma^{-1}\left( t\right) }+\beta \frac{g^{\sigma}\left( t\right)
-g^{\sigma^{-1}}\left( t\right)}{\sigma\left(t\right)
-\sigma^{-1}\left( t\right)}\\
&= \alpha\tilde{D}_{q,\omega}\left[ f\right] \left( t\right)
+\beta \tilde{D}_{q,\omega}\left[ g\right] \left( t\right) .
\end{align*}

\item
\begin{align*}
\tilde{D}_{q,\omega}\left[ fg\right] \left( t\right) & =\frac{\left(
fg\right) ^{\sigma}\left( t\right) -\left( fg\right) ^{\sigma^{-1}}\left(
t\right) }{\sigma\left( t\right) -\sigma^{-1}\left( t\right) } \\
&= \frac{f^{\sigma}\left( t\right) -f^{\sigma^{-1}}\left( t\right)}{\sigma
\left( t\right) -\sigma^{-1}\left( t\right) }g^{\sigma}\left( t\right)
+f^{\sigma^{-1}}\left( t\right) \frac{g^{\sigma}\left( t\right)
-g^{\sigma^{-1}}\left( t\right) }{\sigma\left( t\right)
-\sigma ^{-1}\left(t\right) } \\
&= \tilde{D}_{q,\omega}\left[ f\right] \left( t\right) g^{\sigma}\left(
t\right) +f^{\sigma^{-1}}\left( t\right) \tilde{D}_{q,\omega}\left[ g\right]
\left( t\right) .
\end{align*}

\item Because
\begin{align*}
\tilde{D}_{q,\omega }\left[ \frac{1}{g}\right] \left( t\right) & =\frac{
\frac{1}{g^{\sigma }\left( t\right) }-\frac{1}{g^{\sigma ^{-1}}\left(
t\right) }}{\sigma \left( t\right) -\sigma ^{-1}\left( t\right) } \\
& =-\frac{1}{g^{\sigma }\left( t\right) g^{\sigma ^{-1}}\left( t\right) }
\frac{g^{\sigma }\left( t\right) -g^{\sigma ^{-1}}\left( t\right) }{\sigma
\left( t\right) -\sigma ^{-1}\left( t\right) } \\
& =-\frac{\tilde{D}_{q,\omega }\left[ g\right]\left( t\right)}{g^{\sigma}\left(t\right)
g^{\sigma^{-1}}\left( t\right)},
\end{align*}
one has
\begin{align*}
\tilde{D}_{q,\omega }\left[ \frac{f}{g}\right] \left( t\right)
&= \tilde{D}_{q,\omega }\left[ f\frac{1}{g}\right] \left( t\right)\\
&= \tilde{D}_{q,\omega }\left[ f\right] \left( t\right) \frac{1}{g^{\sigma
}\left( t\right) }+f^{\sigma ^{-1}}\left( t\right) \tilde{D}_{q,\omega }
\left[ \frac{1}{g}\right] \left( t\right)  \\
& =\frac{\tilde{D}_{q,\omega }\left[ f\right] \left( t\right) }{g^{\sigma
}\left( t\right) }-f^{\sigma ^{-1}}\left( t\right) \frac{\tilde{D}_{q,\omega
}\left[ g\right] \left( t\right) }{g^{\sigma }\left( t\right) g^{\sigma
^{-1}}\left( t\right) } \\
& =\frac{\tilde{D}_{q,\omega }\left[ f\right] \left( t\right) g^{\sigma
^{-1}}\left( t\right) -f^{\sigma ^{-1}}\left( t\right) \tilde{D}_{q,\omega }
\left[ g\right] \left( t\right) }{g^{\sigma }\left( t\right) g^{\sigma
^{-1}}\left( t\right) }.
\end{align*}

\item If $f$ is constant on $I$,
then it is clear that $\tilde{D}_{q,\omega }\left[f\right] \equiv 0$.
Suppose now that $\tilde{D}_{q,\omega }\left[f\right] \equiv 0$.
Then, for each $t\in I$,
$\left( \tilde{D}_{q,\omega }\left[ f\right] \right)^{\sigma}(t) =0$
and, therefore, $f\left( t\right) =f^{\sigma ^{2}}\left( t\right)$.
Hence, $$f\left( t\right) =f^{\sigma ^{2}}\left( t\right)
=\cdots =f^{\sigma^{2n}}\left( t\right)$$
for each $n\in\mathbb{N}$ and $t\in I$.
Because
$$
\lim_{n\rightarrow +\infty }f\left( t\right)
=\lim_{n\rightarrow +\infty}f^{\sigma ^{2n}}\left( t\right),
$$
$$
\lim_{n\rightarrow + \infty}\sigma^{2n}\left(t\right)
=\omega_{0} \text{\  \  \ (by Lemma~\ref{qhs:lemma:2.2})},
$$
and $f$ is continuous at $\omega_{0}$, then
$$f\left( t\right) =f\left( \omega _{0}\right)$$
for all $t \in I$.
\end{enumerate}
\end{proof}

\begin{lemma}
\label{qhs:composta}
For $t\in I$ one has $\tilde{D}_{q,\omega }\left[ f^{\sigma }\right]\left(t\right)
=q\tilde{D}_{q,\omega }\left[f\right] \left( \sigma \left( t\right) \right)$.
\end{lemma}

\begin{proof}
For each $t\in I\backslash \{\omega_{0}\}$ we have
\begin{equation*}
\tilde{D}_{q,\omega }\left[ f^{\sigma }\right] \left( t\right) =\frac{
f^{\sigma ^{2}}\left( t\right) -f\left( t\right) }{\sigma \left( t\right)
-\sigma ^{-1}\left( t\right) }
\end{equation*}
and
\begin{eqnarray*}
\tilde{D}_{q,\omega }\left[ f\right] \left( \sigma \left( t\right) \right)
&=&\frac{f^{\sigma ^{2}}\left( t\right) -f\left( t\right) }{\sigma
^{2}\left( t\right) -t} \\
&=&\frac{f^{\sigma ^{2}}\left( t\right) -f\left( t\right) }{q\left( \sigma
\left( t\right) -\sigma ^{-1}\left( t\right) \right)} \ \ \text{ (see Lemma~\ref{qhs:q^n})}.
\end{eqnarray*}
We conclude that
$$\tilde{D}_{q,\omega }\left[ f^{\sigma }\right]\left( t\right)=q\tilde{D}_{q,\omega }\left[
f\right] \left( \sigma \left( t\right) \right).$$
Finally, the intended result follows from the fact that
$$\tilde{D}_{q,\omega }\left[ f^{\sigma }\right] \left( \omega_{0}\right)
=q\tilde{D}_{q,\omega }\left[ f\right] \left( \omega_{0}\right).$$
\end{proof}

\begin{definition}
Let $a,b\in I$ and $a<b$. For $f:I\rightarrow \mathbb{R}$
the \emph{Hahn symmetric integral}\index{Hahn's symmetric integral}
of $f$ from $a$ to $b$ is given by
\begin{equation*}
\int_{a}^{b}f\left( t\right) \tilde{d}_{q,\omega }t
=\int_{\omega_{0}}^{b}f\left( t\right) \tilde{d}_{q,\omega }t
-\int_{\omega_{0}}^{a}f\left( t\right) \tilde{d}_{q,\omega }t,
\end{equation*}
where
\begin{equation*}
\int_{\omega _{0}}^{x}f\left( t\right) \tilde{d}_{q,\omega }t
=\left( \sigma^{-1}\left( x\right)
-\sigma \left( x\right) \right) \sum_{n=0}^{+\infty}
q^{2n+1}f^{\sigma ^{2n+1}}\left( x\right), \quad x\in I\text{,}
\end{equation*}
provided the series converges at $x=a$ and $x=b$. In that case, $f$ is said
to be Hahn symmetric integrable on $[a,b]$. We say that $f$ is Hahn symmetric
integrable on $I$ if it is Hahn symmetric integrable over $[a,b]$ for all $a,b\in I$.
\end{definition}

\begin{remark}
If $\omega =0$, then the Hahn symmetric integral of $f$ from $a$ to $b$
coincides with the $q$-symmetric integral of $f$ from $a$ to $b$
(see Definition~\ref{$q$-symmetric integral}) given by
\begin{equation*}
\int_{a}^{b}f\left(  t\right)  \tilde{d}_{q}t
:=\int_{a}^{b}f\left( t\right) \tilde{d}_{q,0 }t
=\int_{0}^{b}f\left( t\right) \tilde{d}_{q,0 }t
-\int_{0}^{a}f\left( t\right) \tilde{d}_{q,0 }t,
\end{equation*}
where
\begin{equation*}
\int_{0}^{x}f\left(  t\right)  \tilde{d}_{q}t:=\int_{0}^{x}f\left( t\right) \tilde{d}_{q,0 }t
=\left(  q^{-1}-q\right)
x\sum_{n=0}^{+\infty}q^{2n+1}f\left(  q^{2n+1}x\right), \quad x\in I\text{,}
\end{equation*}
provided the series converges at $x=a$ and $x=b$.
\end{remark}

We now present two technical results that are useful to prove the
fundamental theorem of Hahn's symmetric integral calculus
(Theorem~\ref{qhs:Fundamental}).

\begin{lemma}[cf. \cite{Aldwoah}]
\label{qhs:lema integral}
Let $a,b\in I$, $a<b$. If $f:I\rightarrow \mathbb{R}$
is continuous at $\omega_{0}$,
then, for $s\in \left[ a,b\right]$, the sequence
$\left( f^{\sigma ^{2n+1}}\left( s\right) \right)_{n\in\mathbb{N}}$
converges uniformly to $f\left( \omega _{0}\right)$ on $I$.
\end{lemma}

The next result tell us that if a function $f$ is continuous
at $\omega_{0}$, then $f$ is Hahn's symmetric integrable.

\begin{corollary}[cf. \cite{Aldwoah}]
\label{qhs:corolario integral}
Let $a,b\in I$, $a<b$, and $f:I\rightarrow \mathbb{R}$
be continuous at $\omega_{0}$. Then, for $s\in \left[ a,b\right]$, the
series $\sum_{n=0}^{+\infty }q^{2n+1}f^{\sigma ^{2n+1}}\left( s\right)$
is uniformly convergent on $I$.
\end{corollary}

\begin{theorem}[Fundamental theorem of the Hahn symmetric integral calculus]
\label{qhs:Fundamental}
Assume\index{Fundamental theorem of the Hahn symmetric integral calculus}
that $f:I\rightarrow \mathbb{R}$ is continuous at $\omega_{0}$ and,
for each $x\in I$, define
\begin{equation*}
F(x):=\int_{\omega _{0}}^{x}f\left( t\right) \tilde{d}_{q,\omega }t.
\end{equation*}
Then $F$ is continuous at $\omega _{0}$. Furthermore, $\tilde{D}_{q,\omega}[F](x)$
exists for every $x\in I^{q,\omega }$ with
$$
\tilde{D}_{q,\omega}[F](x)=f(x).
$$
Conversely,
\begin{equation*}
\int_{a}^{b}\tilde{D}_{q,\omega }\left[ f\right] \left( t\right)
\tilde{d}_{q,\omega }t=f\left( b\right) -f\left( a\right)
\end{equation*}
for all $a,b\in I$.
\end{theorem}

\begin{proof}
We note that function $F$ is continuous at $\omega _{0}$
by Corollary~\ref{qhs:corolario integral}.
Let us begin by considering $x\in I^{q,\omega}\backslash \{\omega _{0}\}$. Then,
\begin{align*}
\tilde{D}_{q,\omega }&\left[ \tau \mapsto \int_{0}^{\tau}
f\left( t\right) \tilde{d}_{q,\omega }t\right](x)\\
&= \frac{\int_{\omega _{0}}^{\sigma \left( x\right)}
f\left( t\right) \tilde{d}_{q,\omega }t-\int_{\omega _{0}}^{\sigma^{-1}\left( x\right)}
f\left( t\right) \tilde{d}_{q,\omega }t}{\sigma \left(x\right)
-\sigma ^{-1}\left( x\right) } \\
&= \frac{1}{\sigma \left( x\right) -\sigma ^{-1}\left( x\right) }\bigg\{\left[
\sigma ^{-1}\left( \sigma \left( x\right) \right) -\sigma \left( \sigma
\left( x\right) \right) \right] \sum_{n=0}^{+\infty }q^{2n+1}
f^{\sigma^{2n+1}}\left( \sigma \left( x\right) \right)  \\
& \qquad -\left[ \sigma ^{-1}\left( \sigma ^{-1}\left( x\right) \right) -\sigma
\left( \sigma ^{-1}\left( x\right) \right) \right] \sum_{n=0}^{+\infty}
q^{2n+1}f^{\sigma ^{2n+1}}\left( \sigma ^{-1}\left( x\right) \right) \bigg\}\\
& =\sum_{n=0}^{+\infty }q^{2n}f^{\sigma ^{2n}}\left( x\right)
-\sum_{n=0}^{+\infty }q^{2n+2}f^{\sigma ^{2n+2}}\left( x\right)\\
& =f\left( x\right) \text{.}
\end{align*}
If $x=\omega _{0}$, then
\begin{align*}
\tilde{D}&_{q,\omega }\left[F\right]\left( \omega _{0}\right)\\
&= \lim_{h\rightarrow 0}
\frac{F\left( \omega _{0}+h\right) -F\left( \omega _{0}\right) }{h} \\
& =\lim_{h\rightarrow 0}\frac{1}{h}\left[ \sigma ^{-1}\left( \omega
_{0}+h\right) -\sigma \left( \omega _{0}+h\right) \right] \sum_{n=0}^{+
\infty }q^{2n+1}f^{\sigma ^{2n+1}}\left( \omega _{0}+h\right)  \\
& =\lim_{h\rightarrow 0}\frac{1}{h}\left[q^{-1}\left(\omega_{0}
+h-\omega\right)-q\left(\omega_{0}+h\right)-\omega\right]\sum_{n=0}^{+
\infty }q^{2n+1}f^{\sigma ^{2n+1}}\left( \omega _{0}+h\right)  \\
& =\lim_{h\rightarrow 0}\frac{1}{h}\left[\left(q^{-1}-q\right)\omega_{0}
+\left(-q^{-1}-1\right)\omega+\left(q^{-1}-q\right)h\right]\sum_{n=0}^{+\infty }
q^{2n+1}f^{\sigma ^{2n+1}}\left( \omega _{0}+h\right)\\
& =\lim_{h\rightarrow 0}\frac{1}{h}\left[\frac{\left(q^{-1}-q\right)\omega}{1-q}
+\left(-q^{-1}-1\right)\omega+\left(q^{-1}-q\right)h\right]\sum_{n=0}^{+\infty}
q^{2n+1}f^{\sigma ^{2n+1}}\left( \omega _{0}+h\right)  \\
& =\lim_{h\rightarrow 0}\frac{1}{h}\left[\left(\frac{1+q}{q}
+\frac{-1-q}{q}\right)\omega+\left(q^{-1}-q\right)h\right]\sum_{n=0}^{+\infty}
q^{2n+1}f^{\sigma ^{2n+1}}\left( \omega _{0}+h\right)  \\
& =\lim_{h\rightarrow 0}\frac{1-q^{2}}{q}\sum_{n=0}^{+\infty}
q^{2n+1}f^{\sigma ^{2n+1}}\left( \omega _{0}+h\right)  \\
& =\left( 1-q^{2}\right) \sum_{n=0}^{+\infty }q^{2n}f\left( \omega_{0}\right)  \\
& =\left( 1-q^{2}\right) \frac{1}{1-q^{2}}f\left( \omega _{0}\right)  \\
& =f\left( \omega _{0}\right) .
\end{align*}
Finally, since for $x\in I\setminus \{\omega_{0}\}$ we have
\begin{align*}
\int_{\omega _{0}}^{x}  \tilde{D}_{q,\omega }\left[ f\right] \left( t\right)
\tilde{d}_{q,\omega }t&=\left[ \sigma ^{-1}\left( x\right) -\sigma \left(
x\right) \right] \sum_{n=0}^{+\infty }q^{2n+1}\tilde{D}_{q,\omega }\left[ f
\right] ^{\sigma ^{2n+1}}\left( x\right)  \\
& =\left[ \sigma ^{-1}\left( x\right) -\sigma \left( x\right) \right]
\sum_{n=0}^{+\infty }q^{2n+1}\frac{f^{\sigma }\left( \sigma ^{2n+1}\left(
x\right) \right) -f^{\sigma ^{-1}}\left( \sigma ^{2n+1}\left( x\right)
\right) }{\sigma \left( \sigma ^{2n+1}\left( x\right) \right) -\sigma
^{-1}\left( \sigma ^{2n+1}\left( x\right) \right) } \\
& =\left[ \sigma ^{-1}\left( x\right) -\sigma \left( x\right) \right]
\sum_{n=0}^{+\infty }q^{2n+1}\frac{f^{\sigma }\left( \sigma ^{2n+1}\left(
x\right) \right) -f^{\sigma ^{-1}}\left( \sigma ^{2n+1}\left( x\right)
\right) }{q^{2n+1}\left(\sigma \left(x\right)- \sigma ^{-1}\left( x\right)\right) }\\
& =\sum_{n=0}^{+\infty }\left[ f^{\sigma ^{2n}}\left( x\right)
-f^{\sigma^{2\left( n+1\right) }}\left( x\right) \right]  \\
& =f\left( x\right) -f\left( \omega _{0}\right),
\end{align*}
where in the third equality we use Lemma~\ref{qhs:q^n}, then
\begin{align*}
\int_{a}^{b}\tilde{D}_{q,\omega }\left[ f\right] \left( t\right) \tilde{d}
_{q,\omega }t& =\int_{\omega_0}^{b}\tilde{D}_{q,\omega }\left[ f\right] \left(
t\right) \tilde{d}_{q,\omega }t-\int_{\omega_0}^{a}\tilde{D}_{q,\omega }\left[ f
\right] \left( t\right) \tilde{d}_{q,\omega }t \\
& =f\left( b\right) -f\left( a\right) \text{.}
\end{align*}
\end{proof}

The Hahn symmetric integral has the following properties.

\begin{theorem}
Let $f,g$ $:I\rightarrow \mathbb{R}$ be Hahn's symmetric integrable
on $I$, $a,b,c\in I$, and $\alpha ,\beta \in \mathbb{R}$. Then,

\begin{enumerate}
\item $\displaystyle\int_{a}^{a}f\left( t\right) \tilde{d}_{q,\omega}t=0$;

\item $\displaystyle\int_{a}^{b}f\left( t\right) \tilde{d}_{q,\omega}t=-\int_{b}^{a}f
\left( t\right) \tilde{d}_{q,\omega}t$;

\item $\displaystyle\int_{a}^{b}f\left( t\right) \tilde{d}_{q,\omega}t=\int_{a}^{c}
f\left(t\right) \tilde{d}_{q,\omega}t+\int_{c}^{b}f\left( t\right)
\tilde{d}_{q,\omega}t$;

\item $\displaystyle\int_{a}^{b}\left( \alpha f+\beta g\right) \left( t\right) \tilde{d}
_{q,\omega}t=\alpha\int_{a}^{b}f\left( t\right) \tilde{d}_{q,\omega}t
+\beta \int_{a}^{b}g\left( t\right) \tilde{d}_{q,\omega}t$;

\item if $\displaystyle\tilde{D}_{q,\omega }\left[ f\right] $ and $\tilde{D}_{q,\omega }
\left[ g\right] $ are continuous at $\omega _{0}$, then
\begin{equation}
\label{qhs:eq:int:parts}
\int_{a}^{b}f^{\sigma ^{-1}}\left( t\right) \tilde{D}_{q,\omega }\left[ g
\right] \left( t\right) \tilde{d}_{q,\omega }t=f\left( t\right) g\left(
t\right) \bigg|_{a}^{b}-\int_{a}^{b}\tilde{D}_{q,\omega }\left[ f\right]
\left( t\right) g^{\sigma }\left( t\right) \tilde{d}_{q,\omega }t.
\end{equation}
\end{enumerate}
\end{theorem}

\begin{proof}
Properties~1 to 4 are trivial. Property~5 follows
from Theorem~\ref{qhs:props derivada}
and Theorem~\ref{qhs:Fundamental}: since
$$
\tilde{D}_{q,\omega }\left[ fg\right] \left( t\right)
=\tilde{D}_{q,\omega}\left[ f\right] \left( t\right) g^{\sigma }\left( t\right)
+f^{\sigma^{-1}}\left( t\right) \tilde{D}_{q,\omega }\left[ g\right] \left( t\right),
$$
then
$$
f^{\sigma ^{-1}}\left( t\right) \tilde{D}_{q,\omega }\left[ g
\right] \left( t\right) =\tilde{D}_{q,\omega }\left[ fg\right] \left(
t\right) -\tilde{D}_{q,\omega }\left[ f\right] \left( t\right)
g^{\sigma}\left( t\right)
$$
and hence,
$$
\int_{a}^{b}f^{\sigma ^{-1}}\left( t\right)
\tilde{D}_{q,\omega }\left[ g\right] \left( t\right) \tilde{d}_{q,\omega }t=f\left(
t\right) g\left( t\right) \bigg|_{a}^{b}-\int_{a}^{b}\tilde{D}_{q,\omega }
\left[ f\right] \left( t\right) g^{\sigma }\left( t\right) \tilde{d}_{q,\omega }t.
$$
\end{proof}

\begin{remark}
Relation \eqref{qhs:eq:int:parts} gives  a \emph{Hahn's symmetric integration
by parts}\index{Hahn's symmetric integration by parts} formula.
\end{remark}

\begin{remark}
Using Lemma~\ref{qhs:composta} and the Hahn symmetric
integration by parts formula \eqref{qhs:eq:int:parts},
we conclude that
\begin{equation}
\label{qhs:partes}
\int_{a}^{b}f\left( t\right) \tilde{D}_{q,\omega }\left[ g\right] \left(
t\right) \tilde{d}_{q,\omega }t=f^{\sigma }\left( t\right) g\left( t\right)
\bigg|_{a}^{b}-q\int_{a}^{b}\left( \tilde{D}_{q,\omega }\left[ f\right]
\right) ^{\sigma }\left( t\right) g^{\sigma }\left( t\right)
\tilde{d}_{q,\omega }t.
\end{equation}
\end{remark}

\begin{proposition}
\label{qhs:desigualdade}
Let $c\in I$, $f$ and $g$ be Hahn's symmetric
integrable on $I$. Suppose that
$$
\left\vert f\left( t\right) \right\vert \leqslant g\left( t\right)
$$
for all $t\in \left\{ \sigma ^{2n+1}\left( c\right) :n\in
\mathbb{N}_{0}\right\} \cup \left\{ \omega _{0}\right\}$.
\begin{enumerate}
\item If $c\geqslant \omega _{0}$, then
\begin{equation*}
\left\vert \int_{\omega _{0}}^{c}f\left( t\right)
\tilde{d}_{q,\omega}t\right\vert \leqslant \int_{\omega_{0}}^{c}
g\left( t\right) \tilde{d}_{q,\omega}t.
\end{equation*}

\item If $c<\omega _{0}$, then
\begin{equation*}
\left\vert \int_{c}^{\omega _{0}}f\left( t\right) \tilde{d}_{q,\omega}
t\right\vert \leqslant \int_{c}^{\omega _{0}}g\left( t\right)
\tilde{d}_{q,\omega }t.
\end{equation*}
\end{enumerate}
\end{proposition}

\begin{proof}
If $c\geqslant \omega_{0}$, then
\begin{equation*}
\begin{split}
\left\vert \int_{\omega _{0}}^{c}f\left( t\right) \tilde{d}_{q,\omega}
t\right\vert  &=\left\vert \left[ \sigma ^{-1}\left( c\right) -\sigma
\left( c\right) \right] \sum_{n=0}^{+\infty }q^{2n+1}
f^{\sigma^{2n+1}}\left( c\right) \right\vert\\
&\leqslant \left[ \sigma ^{-1}\left( c\right)
-\sigma \left( c\right) \right] \sum_{n=0}^{+\infty }q^{2n+1}\left\vert
f^{\sigma ^{2n+1}}\left( c\right) \right\vert  \\
&\leqslant \left[ \sigma ^{-1}\left( c\right) -\sigma \left( c\right) \right]
\sum_{n=0}^{+\infty }q^{2n+1}g^{\sigma ^{2n+1}}\left( c\right)\\
&=\int_{\omega_{0}}^{c}g\left( t\right) \tilde{d}_{q,\omega }t.
\end{split}
\end{equation*}
If $c<\omega_{0}$, then
\begin{equation*}
\begin{split}
\left\vert \int_{c}^{\omega _{0}}
f\left( t\right) \tilde{d}_{q,\omega}t\right\vert
&= \left\vert -\left[ \sigma ^{-1}\left( c\right) -\sigma
\left( c\right) \right] \sum_{n=0}^{+\infty }q^{2n+1}
f^{\sigma^{2n+1}}\left( c\right) \right\vert\\
&\leqslant \left\vert \sigma ^{-1}\left( c\right)
-\sigma \left( c\right) \right\vert \sum_{n=0}^{+\infty }q^{2n+1}\left\vert
f^{\sigma ^{2n+1}}\left( c\right) \right\vert\\
&=-\left[ \sigma ^{-1}\left( c\right)
-\sigma \left( c\right) \right] \sum_{n=0}^{+\infty }q^{2n+1}\left\vert
f^{\sigma ^{2n+1}}\left( c\right) \right\vert\\
&\leqslant -\left[ \sigma ^{-1}\left( c\right) -\sigma \left( c\right) \right]
\sum_{n=0}^{+\infty }q^{2n+1}g^{\sigma ^{2n+1}}\left( c\right)\\
&=-\int_{\omega _{0}}^{c}g\left( t\right) \tilde{d}_{q,\omega }t\\
&=\int_{c}^{\omega _{0}}g\left( t\right) \tilde{d}_{q,\omega }t
\end{split}
\end{equation*}
providing the desired equality.
\end{proof}

As an immediate consequence, we have the following result.

\begin{corollary}
Let $c\in I$ and $f$ be Hahn's symmetric integrable on $I$. Suppose that
$$f\left( t\right) \geqslant 0$$ for all $t\in \left\{ \sigma ^{2n+1}\left(
c\right) :n\in \mathbb{N}_{0}\right\} \cup \left\{ \omega _{0}\right\}$.
\begin{enumerate}
\item If $c\geqslant \omega_{0}$, then
\begin{equation*}
\int_{\omega _{0}}^{c}f\left( t\right) \tilde{d}_{q,\omega }t\geqslant 0\text{;}
\end{equation*}
\item If $c<\omega_{0}$, then
\begin{equation*}
\int_{c}^{\omega _{0}}f\left( t\right) \tilde{d}_{q,\omega }t\geqslant 0\text{.}
\end{equation*}
\end{enumerate}
\end{corollary}

\begin{remark}
In general it is not true that if $f$ is a nonnegative function
on $\left[ a,b\right]$, then
\begin{equation*}
\int_{a}^{b}f\left( t\right) \tilde{d}_{q,\omega }t\geqslant 0\text{.}
\end{equation*}
As an example, consider the function $f$ defined in $\left[ -5,5\right]$ by
\begin{equation*}
f\left( t\right) =\left\{
\begin{array}{ccc}
6 & \text{\ \ if \ } & t=3 \\
\\
1 & \text{\ \ if \ } & t=4 \\
\\
0 & \text{\ \ if \ } & t\in \left[ -5,5\right]
\backslash \left\{ 3, 4\right\}.
\end{array}
\right.
\end{equation*}
For $q=\frac{1}{2}$ and $\omega =1$, this function is Hahn's symmetric
integrable because is continuous at $\omega_0=2$. However,
\begin{align*}
\int_{4}^{6} & f\left( t\right) \tilde{d}_{q,\omega }t
=\int_{2}^{6}f\left( t\right) \tilde{d}_{q,\omega }t-\int_{2}^{4}
f\left( t\right) \tilde{d}_{q,\omega }t \\
& =\left(10-4\right)\sum_{n=0}^{+\infty }\left( \frac{1}{2}\right)
^{2n+1}f^{\sigma ^{2n+1}}\left( 6\right)
-\left(6-3\right) \sum_{n=0}^{+\infty }\left( \frac{1}{2}\right)^{2n+1}
f^{\sigma^{2n+1}}\left( 4\right)\\
& =6\left( \frac{1}{2}\right) \times 1-3\left( \frac{1}{2}\right)
\times 6\\
&=-6.
\end{align*}
This example also proves that, in general, it is not true that
\begin{equation*}
\left\vert \int_{a}^{b}f\left( t\right) \tilde{d}_{q,\omega }t\right\vert
\leqslant \int_{a}^{b}\left\vert f\left( t\right) \right\vert \tilde{d}_{q,\omega}t
\end{equation*}
for any $a,b\in I$.
\end{remark}


\section{Hahn's symmetric variational calculus}
\label{qhs:M}

We begin this section with some useful definitions and notations.
For $s\in I$ we set
\begin{equation*}
\left[ s\right] _{q,\omega }:=\left\{ \sigma ^{2n+1}\left( s\right) :
n \in \mathbb{N}_{0}\right\} \cup \left\{ \omega _{0}\right\} .
\end{equation*}
Let $a,b\in I$ with $a<b$. We define the Hahn symmetric interval from $a$
to $b$ by
\begin{equation*}
\left[ a,b\right] _{q,\omega }:=\left\{ \sigma ^{2n+1}\left( a\right) :n\in
\mathbb{N}_{0}\right\} \cup \left\{ \sigma ^{2n+1}\left( b\right) :n\in
\mathbb{N}_{0}\right\} \cup \left\{ \omega _{0}\right\} ,
\end{equation*}
that is,
\begin{equation*}
\left[ a,b\right] _{q,\omega }=\left[ a\right] _{q,\omega }\cup \left[ b
\right]_{q,\omega }.
\end{equation*}
Let $r \in \{0,1\}$. We denote the linear space
\begin{equation*}
\left\{ y:I\rightarrow
\mathbb{R}\  |\  \tilde{D}_{q,\omega }^{i}\left[ y\right] ,
i=0,r, \text{ are bounded on }\left[a,b\right]_{q,\omega }
\text{ and continuous at }\omega _{0}\right\}
\end{equation*}
endowed with the norm
\begin{equation*}
\left\Vert y\right\Vert _{r}=\sum_{i=0}^{r}\sup_{t\in \left[ a,b\right]
_{q,\omega }}\left\vert \tilde{D}_{q,\omega }^{i}\left[ y\right] \left(
t\right) \right\vert,
\end{equation*}
where $\tilde{D}_{q,\omega }^{0}\left[ y\right]=y$, by
$\mathcal{Y}^{r}\left( \left[ a,b\right]_{q,\omega }, \mathbb{R}\right)$.

\begin{definition}
We say that $y$ is an admissible function to problem \eqref{qhs:P}
if  $\ y\in \ \mathcal{Y}^{1}\left( \left[ a,b\right] _{q,\omega }, \mathbb{R}\right)$
and $y$ satisfies the boundary conditions
$y\left( a\right)=\alpha $ and $y\left( b\right) =\beta$.
\end{definition}

\begin{definition}
We say that $y_{\ast }$ is a local minimizer (resp. local maximizer) to
problem \eqref{qhs:P} if $y_{\ast }$ is an admissible function and there exists
$\delta >0$ such that
\begin{equation*}
\mathcal{L}\left( y_{\ast }\right) \leqslant \mathcal{L}\left( y\right)
\quad \text{ (resp. }\mathcal{L}\left( y_{\ast }\right)
\geqslant \mathcal{L}\left( y\right) \text{)}
\end{equation*}
for all admissible $y$ with $\left\Vert y_{\ast }-y\right\Vert _{1}<\delta $.
\end{definition}

\begin{definition}
We say that $\eta \in \mathcal{Y}^{1}\left(\left[ a,b\right]_{q,\omega},
\mathbb{R}\right)$ is an admissible variation to problem \eqref{qhs:P} if $\eta \left(
a\right) =0=\eta \left( b\right) $.
\end{definition}


Before proving our main results,
we begin with three basic lemmas.


\subsection{Basic lemmas}

The following results are useful to prove Theorem~\ref{qhs:Euler}.

\begin{lemma}[Fundamental lemma of the Hahn symmetric variational calculus]
\label{qhs:fundamental}
Let\index{Fundamental lemma of the Hahn symmetric variational calculus}
$f\in \mathcal{Y}^{0}\left(\left[ a,b\right]_{q,\omega},
\mathbb{R}\right)$. One has
\begin{equation*}
\int_{a}^{b}f\left( t\right) h^{\sigma }\left( t\right) \tilde{d}_{q,\omega}t=0
\end{equation*}
for all $h\in \mathcal{Y}^{0}\left(\left[ a,b\right]_{q,\omega},
\mathbb{R}\right)$ with
$h\left( a\right) =h\left( b\right) =0$
if, and only if,
$f\left( t\right) =0$
for all $t\in \left[ a,b\right]_{q,\omega }$.
\end{lemma}

\begin{proof}
The implication \textquotedblleft$\Leftarrow$\textquotedblright\ is obvious.
Let us prove the implication \textquotedblleft$\Rightarrow$\textquotedblright.
Suppose, by contradiction, that exists $p\in\left[ a,b\right]_{q,\omega}$
such that $f\left( p\right) \neq 0$.
\begin{enumerate}
\item If $p\neq\omega_{0}$, then $p=\sigma^{2k+1}\left( a\right)$ or
$p=\sigma^{2k+1}\left( b\right)$
for some $k \in \mathbb{N}_{0}$.
\begin{enumerate}
\item Suppose that $a\neq \omega_{0}$ and $b\neq \omega_{0}$. In this case
we can assume, without loss of generality, that $p=\sigma^{2k+1}\left(a\right)$.
Define
\begin{equation*}
h\left( t\right)
=
\begin{cases}
f^{\sigma^{2k+1}}\left( a\right)  & \text{if } t=\sigma^{2k+2}\left(a\right)\\
\\
0 & \text{otherwise.}
\end{cases}
\end{equation*}
Then,
\begin{align*}
\int_{a}^{b} & f\left( t\right) h^{\sigma }\left( t\right) \tilde{d}_{q,\omega}t \\
&=\left[ \sigma ^{-1}\left( b\right) -\sigma \left( b\right) \right]
\sum_{n=0}^{+\infty }q^{2n+1}f^{\sigma ^{2n+1}}\left( b\right) h^{\sigma^{2n+2}}\left( b\right)\\
&\qquad -\left[ \sigma ^{-1}\left( a\right) -\sigma \left( a\right) \right]
\sum_{n=0}^{+\infty }q^{2n+1}f^{\sigma ^{2n+1}}\left( a\right)
h^{\sigma^{2n+2}}\left( a\right)  \\
& =-\left[ \sigma ^{-1}\left( a\right) -\sigma \left( a\right) \right]
q^{2k+1}\left[ f^{\sigma ^{2k+1}}\left( a\right) \right] ^{2}\neq 0,
\end{align*}
which is a contradiction.

\item Suppose that $a\neq\omega_{0}$ and $b=\omega_{0}$. Therefore,
$p=\sigma^{2k+1}\left( a\right) $ for some $k\in \mathbb{N}_{0}$. Define
\begin{equation*}
h\left( t\right) =
\begin{cases}
f^{\sigma^{2k+1}}\left( a\right) & \text{if } t=\sigma^{2k+2}\left(a\right)\\
\\
0 & \text{otherwise.}
\end{cases}
\end{equation*}
We obtain a contradiction
with a similar proof as in case (a).

\item The case $a=\omega_{0}$ and $b\neq\omega_{0}$ is similar to (b).
\end{enumerate}

\item If $p=\omega_{0}$, we assume,
without loss of generality, that $f\left(p\right) >0$. Since
\begin{equation*}
\lim_{n\rightarrow +\infty }\sigma ^{2k+2}\left( a\right)
=\lim_{n\rightarrow +\infty }\sigma ^{2k+2}\left( b\right)
=\omega _{0}
\end{equation*}
and $f$ is continuous at $\omega_{0}$,
\begin{equation*}
\lim_{n\rightarrow +\infty }f^{\sigma ^{2k+1}}\left( a\right)
=\lim_{n\rightarrow +\infty }f^{\sigma ^{2k+1}}\left( b\right)
=f\left(\omega_{0}\right) .
\end{equation*}
Therefore, there exists an order $n_{0}\in \mathbb{N}$
such for all $n>n_{0}$ the inequalities
\begin{equation*}
f^{\sigma ^{2k+1}}\left( a\right) >0 \ \text{ and }\
f^{\sigma^{2k+1}}\left(b\right) >0
\end{equation*}
hold.

\begin{enumerate}
\item If $a,b\neq \omega_{0}$, then for some $k>n_{0}$ we define
\begin{equation*}
h\left( t\right)
=
\begin{cases}
-\frac{f^{\sigma ^{2k+1}}\left( b\right)}{\sigma ^{-1}\left( a\right)
-\sigma \left( a\right) } & \text{if } t=\sigma ^{2k+2}\left( a\right)\\
\\
\frac{f^{\sigma ^{2k+1}}\left( a\right) }{\sigma ^{-1}\left( b\right)
-\sigma \left( b\right) } & \text{if } t=\sigma ^{2k+2}\left( b\right)\\
\\
0 & \text{otherwise.}
\end{cases}
\end{equation*}
Hence,
\begin{equation*}
\int_{a}^{b}f\left( t\right) h^{\sigma }\left( t\right) \tilde{d}_{q,\omega}t
=2q^{2k+1}f^{\sigma ^{2k+1}}\left( a\right) f^{\sigma ^{2k+1}}\left(b\right)
> 0.
\end{equation*}

\item If $a=\omega_{0}$, then we define
\begin{equation*}
h\left( t\right) =
\begin{cases}
f^{\sigma^{2k+1}}\left( b\right)  & \text{if } t=\sigma ^{2k+2}\left(b\right)\\
\\
0 & \text{otherwise.}
\end{cases}
\end{equation*}
Therefore,
\begin{equation*}
\int_{\omega _{0}}^{b}f\left( t\right) h^{\sigma }\left( t\right) \tilde{d}_{q,\omega}t
=\left[ \sigma ^{-1}\left( b\right) -\sigma \left( b\right) \right]
q^{2k+1}\left[ f^{\sigma ^{2k+1}}\left( b\right) \right] ^{2}\neq 0.
\end{equation*}
\item If $b=\omega_{0}$, the proof is similar to the previous case.
\end{enumerate}
\end{enumerate}
\end{proof}

\begin{definition}[\cite{Malinowska:3}]
Let $s\in I$ and $g:I\times \left] -\bar{\theta},\bar{\theta}\right[
\rightarrow \mathbb{R}$. We say that $g\left( t,\cdot \right)$
is differentiable at $\theta_{0}$ uniformly in $\left[ s\right]_{q,\omega}$
if, for every $\varepsilon >0$, there exists $\delta >0$ such that
\begin{equation*}
0<\left\vert \theta -\theta _{0}\right\vert <\delta \Rightarrow
\left\vert \frac{g\left( t,\theta \right) -g\left( t,\theta _{0}\right) }{
\theta -\theta _{0}}-\partial _{2}g\left( t,\theta _{0}\right) \right\vert
<\varepsilon
\end{equation*}
for all $t\in \left[ s\right] _{q,\omega }$, where $\partial_{2}g
=\frac{\partial g}{\partial \theta }$.
\end{definition}

\begin{lemma}[\textrm{cf.} \cite{Malinowska:3}]
\label{qhs:tecnico}
Let $s\in I$ and assume that
$g:I\times \left] -\bar{\theta},\bar{\theta}\right[ \rightarrow \mathbb{R}$
is differentiable at $\theta _{0}$ uniformly in $\left[ s\right]_{q,\omega }$.
If $\int_{\omega _{0}}^{s}g\left( t,\theta _{0}\right) \tilde{d}_{q,\omega }t$
exists, then $G\left( \theta \right) :=\int_{\omega_{0}}^{s}g\left( t,\theta \right)
\tilde{d}_{q,\omega }t$ for $\theta$ near $\theta_{0}$,
is differentiable at $\theta_{0}$ with
\begin{equation*}
G^{\prime }\left( \theta _{0}\right) =\int_{\omega _{0}}^{s}\partial_{2}
g\left( t,\theta _{0}\right) \tilde{d}_{q,\omega }t.
\end{equation*}
\end{lemma}

\begin{proof}
For $s>\omega_{0}$ the proof is similar to the proof given in Lemma~3.2 of
\cite{Malinowska:3}. The result is trivial for $s = \omega_{0}$.
Suppose that  $s<\omega _{0}$ and  let $\varepsilon >0$ be arbitrary.
Since $g\left( t,\cdot \right)$ is differentiable at $\theta_{0}$
uniformly in $\left[ s\right] _{q,\omega }$, then there exists $\delta >0$
such that for all $t\in \left[ s\right] _{q,\omega }$ and for
$0<\left\vert \theta -\theta _{0}\right\vert <\delta$
the following inequality holds:
\begin{equation}
\label{qhs:treta}
\left\vert \frac{g\left( t,\theta \right) -g\left( t,\theta _{0}\right) }{
\theta -\theta _{0}}-\partial _{2}g\left( t,\theta_{0}\right) \right\vert
< \frac{\varepsilon}{2\left( \omega _{0}-s\right)}.
\end{equation}
Since, for $0<\left\vert \theta -\theta_{0}\right\vert <\delta $, we have
\begin{align*}
\Biggl\vert &\frac{G\left( \theta \right)-G\left( \theta_{0}\right) }{\theta
-\theta _{0}}-\int_{\omega _{0}}^{s}\partial _{2}g\left( t,\theta_{0}\right)
\tilde{d}_{q,\omega }t \Biggr\vert \\
&=\left\vert
\frac{\int_{\omega _{0}}^{s}g\left( t,\theta \right) \tilde{d}_{q,\omega}t
-\int_{\omega _{0}}^{s}g\left( t,\theta _{0}\right) \tilde{d}_{q,\omega }t}{\theta
-\theta _{0}}-\int_{\omega _{0}}^{s}\partial _{2}g\left( t,\theta_{0}\right)
\tilde{d}_{q,\omega }t\right\vert  \\
&=\left\vert \int_{\omega _{0}}^{s}\left[ \frac{g\left( t,\theta \right)
-g\left( t,\theta _{0}\right) }{\theta -\theta _{0}}
-\partial_{2}g\left(t,\theta _{0}\right) \right] \tilde{d}_{q,\omega }t\right\vert\\
&< \int_{s}^{\omega _{0}}\frac{\varepsilon }{2\left(\omega_{0}- s\right)}
\tilde{d}_{q,\omega }t    \text{\ \ \ \ (using Proposition~\ref{qhs:desigualdade}
and inequality \eqref{qhs:treta})}                              \\
&= \frac{\varepsilon }{2\left(
\omega _{0}-s\right) }\int_{s}^{\omega _{0}}1\tilde{d}_{q,\omega }t\\
&=\frac{\varepsilon }{2}\\
&<\varepsilon \ \
\end{align*}
then we can conclude that
$$
G^{\prime }\left( \theta \right)
=\int_{\omega _{0}}^{s}\partial _{2}g\left( t,\theta_{0}\right) \tilde{d}_{q,\omega}t.
$$
\end{proof}

For an admissible variation $\eta$ and an admissible function $y$,
we define $\phi :\left] -\bar{\epsilon},\bar{\epsilon}\right[
\rightarrow \mathbb{R}$ by $\phi \left( \epsilon \right)
:= \mathcal{L}\left[ y+\epsilon \eta \right]$.
The first variation of functional $\mathcal{L}$
of problem \eqref{qhs:P} is defined by
$\delta \mathcal{L}\left[ y,\eta \right] :=\phi ^{\prime }\left( 0\right)$.
Note that
\begin{align*}
\mathcal{L}\left[ y+\epsilon \eta \right]
&= \int_{a}^{b}L\left( t,y^{\sigma}\left( t\right)
+\epsilon \eta ^{\sigma }\left( t\right),
\tilde{D}_{q,\omega }\left[ y\right] \left( t\right)
+\epsilon \tilde{D}_{q,\omega}
\left[ \eta \right] \left( t\right) \right) \tilde{d}_{q,\omega }t \\
& =\mathcal{L}_{b}\left[ y+\epsilon \eta \right] -\mathcal{L}_{a}\left[
y + \epsilon \eta \right],
\end{align*}
where
\begin{equation*}
\mathcal{L}_{\xi }\left[ y+\epsilon \eta \right]
=\int_{\omega _{0}}^{\xi}L\left( t,y^{\sigma }\left( t\right)
+\epsilon \eta ^{\sigma }\left(t\right), \tilde{D}_{q,\omega }\left[ y\right]\left( t\right)
+\epsilon \tilde{D}_{q,\omega }\left[ \eta \right]\left(t\right)\right)\tilde{d}_{q,\omega }t
\end{equation*}
with $\xi \in \left\{ a,b\right\}$. Therefore,
$\delta \mathcal{L}\left[ y,\eta \right]
=\delta \mathcal{L}_{b}\left[ y,\eta\right]
-\delta \mathcal{L}_{a}\left[ y,\eta \right]$.

The following lemma is a direct consequence of Lemma~\ref{qhs:tecnico}.

\begin{lemma}
\label{qhs:tecnico 2}
For an admissible variation $\eta $ and an admissible function $y$, let
\begin{equation*}
g\left( t,\epsilon \right) :=L\left( t,y^{\sigma }\left( t\right) +\epsilon
\eta ^{\sigma }\left( t\right) ,\tilde{D}_{q,\omega }\left[ y\right] \left(
t\right) +\epsilon \tilde{D}_{q,\omega }\left[\eta\right]\left( t\right) \right) .
\end{equation*}
Assume that
\begin{enumerate}
\item $g\left( t,\cdot \right) $ is differentiable at $\omega_0$
uniformly in $\left[ a,b\right]_{q,\omega }$;

\item $\mathcal{L}_{\xi}\left[ y+\epsilon \eta \right]
=\int_{\omega_{0}}^{\xi}g\left( t,\epsilon \right) \tilde{d}_{q,\omega }t$,
$\xi \in \left\{ a,b\right\}$, exist for $\epsilon \approx 0$;

\item $\int_{\omega _{0}}^{a}\partial_{2}g\left( t,0\right) \tilde{d}_{q,\omega }t$
and $\int_{\omega_{0}}^{b}\partial_{2}g\left( t,0\right)
\tilde{d}_{q,\omega }t$ exist.\\
\end{enumerate}
Then,
\begin{multline*}
\phi^{\prime }\left( 0\right) :=\delta \mathcal{L}\left[ y,\eta \right]
= \int_{a}^{b}\Biggl[\partial _{2}L\left( t,y^{\sigma }\left( t\right) ,
\tilde{D}_{q,\omega }\left[ y\right] \left( t\right) \right)\eta^{\sigma }\left(t\right)\\
+\partial _{3}L\left( t,y^{\sigma }\left( t\right) ,\tilde{D}_{q,\omega }
\left[ y\right] \left( t\right) \right) \tilde{D}_{q,\omega }\left[\eta\right]\left(t\right)
\Biggr]\tilde{d}_{q,\omega }t.
\end{multline*}
\end{lemma}


\subsection{Optimality conditions}
\label{qhs:sec:o:c}

In this section we present a necessary optimality condition
(the Hanh symmetric Euler--Lagrange equation)
and a sufficient optimality condition to problem \eqref{qhs:P}.

\begin{theorem}[The Hahn symmetric Euler--Lagrange equation]
\label{qhs:Euler}
Under\index{Hahn's symmetric Euler--Lagrange equation}
hypotheses (H$_{q,\omega}$1)-(H$_{q,\omega}$3) and conditions 1 to 3
of Lemma~\ref{qhs:tecnico 2} on the Lagrangian $L$,
if $y_{\ast}\in \mathcal{Y}^{1}\left(\left[ a,b\right]_{q,\omega},
\mathbb{R}\right)$ is a local extremizer
to problem \eqref{qhs:P}, then $y_{\ast }$ satisfies the
Hahn symmetric Euler--Lagrange equation
\begin{equation}
\label{qhs:EqEuler}
\partial_{2}L\left( t,y^{\sigma }\left( t\right) ,\tilde{D}_{q,\omega }
\left[ y\right] \left( t\right) \right)
=\tilde{D}_{q,\omega }\left[\tau \mapsto
\partial_{3}L\left( \sigma\left(\tau \right),
y^{\sigma ^{2}}\left( \tau \right),
\left(\tilde{D}_{q,\omega }\left[ y\right] \right) ^{\sigma }\left(
\tau\right) \right) \right] \left( t\right)
\end{equation}
for all $t\in \left[ a,b\right]_{q,\omega}$.
\end{theorem}

\begin{proof}
Let $y_{\ast}$ be a local minimizer (resp. maximizer) to problem \eqref{qhs:P}
and $\eta $ an admissible variation. Define
$\phi : \mathbb{R} \rightarrow \mathbb{R} $ by
$\phi \left( \epsilon \right) :=\mathcal{L}\left[ y_{\ast}
+\epsilon \eta \right]$.
A necessary condition for $y_{\ast}$ to be an extremizer is given by
$\phi^{\prime }\left( 0\right) =0$. By Lemma~\ref{qhs:tecnico 2},
\begin{multline*}
\int_{a}^{b}\Biggl[ \partial_{2}L\left( t,y_{\ast}^{\sigma }\left( t\right),
\tilde{D}_{q,\omega }\left[ y_{\ast}\right] \left( t\right) \right)
\eta^{\sigma }\left(t\right) \\
+\partial_{3}L\left( t,y_{\ast}^{\sigma }\left( t\right),
\tilde{D}_{q,\omega }\left[ y_{\ast}\right] \left( t\right) \right)
\tilde{D}_{q,\omega}\left[\eta\right] \left(
t\right) \Biggr] \tilde{d}_{q,\omega }t =0.
\end{multline*}
Using the integration by parts formula \eqref{qhs:partes}, we get
\begin{align*}
\int_{a}^{b} & \partial _{3}L\left( t,y_{\ast}^{\sigma }\left( t\right),
\tilde{D}_{q,\omega }\left[ y_{\ast}\right] \left( t\right) \right)
\tilde{D}_{q,\omega}\left[\eta\right] \left( t\right) \tilde{d}_{q,\omega }t \\
& =\partial _{3}L\left( \sigma \left( t\right), y_{\ast}^{\sigma ^{2}}\left(
t\right), \left( \tilde{D}_{q,\omega }\left[
y_{\ast}\right] \right)^{\sigma}\left( t\right) \right)
\eta \left( t\right) \bigg|_{a}^{b}\\
&\qquad - q\int_{a}^{b}\left( \tilde{D}_{q,\omega }\left[\tau
\mapsto \partial_{3}L\left(\tau ,y_{\ast}^{\sigma }\left( \tau \right),
\left( \tilde{D}_{q,\omega }\left[ y_{\ast}\right] \right) \left(
\tau \right) \right) \right] \right)^{\sigma}\left( t\right)
\eta^{\sigma }\left( t\right) \tilde{d}_{q,\omega }t.
\end{align*}
Since $\eta \left( a\right) =\eta \left( b\right) =0$, then
\begin{multline*}
\int_{a}^{b}\bigg[\partial _{2}L\left( t,y_{\ast}^{\sigma }\left( t\right) ,
\tilde{D}_{q,\omega }\left[ y_{\ast}\right] \left( t\right) \right)\\
-q\left( \tilde{D}_{q,\omega }\left[ \tau \mapsto \partial_{3}L\left(\tau,
y_{\ast}^{\sigma}\left(\tau\right), \left( \tilde{D}_{q,\omega }\left[ y_{\ast}\right] \right)
\left( \tau \right) \right) \right] \right) ^{\sigma }\left( t\right) \bigg]
\eta^{\sigma }\left( t\right) \tilde{d}_{q,\omega }t = 0
\end{multline*}
and by Lemma~\ref{qhs:fundamental} we get
\begin{eqnarray*}
&\partial _{2}L\left( t,y_{\ast}^{\sigma }\left( t\right) ,\tilde{D}_{q,\omega }
\left[ y_{\ast}\right] \left( t\right) \right) =q\left( \tilde{D}_{q,\omega }\left[
\tau \mapsto \partial _{3}L\left(\tau, y_{\ast}^{\sigma }\left(\tau\right),
\tilde{D}_{q,\omega}\left[ y_{\ast}\right]\left(\tau\right)\right)
\right]\right)^{\sigma }\left( t\right)
\end{eqnarray*}
for all $t\in \left[ a,b\right]_{q,\omega}$.
Finally, using Lemma~\ref{qhs:composta}, we conclude that
\begin{align*}
\partial _{2}L\left( t,y_{\ast}^{\sigma }\left( t\right) ,
\tilde{D}_{q,\omega }\left[ y_{\ast}\right] \left( t\right) \right)
=\tilde{D}_{q,\omega}\left[ \tau \mapsto
\partial_{3}L\left( \sigma\left(\tau \right),
y_{\ast}^{\sigma ^{2}}\left(\tau\right),\left(\tilde{D}_{q,
\omega}\left[y_{\ast}\right]\right)^{\sigma}\left(
\tau\right)\right)\right] \left( t\right).
\end{align*}
\end{proof}

The particular case $\omega =0$ gives the
$q$-symmetric Euler--Lagrange equation (see Theorem~\ref{qs:Euler}).

\begin{corollary}[The $q$-symmetric Euler--Lagrange equation \cite{Brito:da:Cruz:2}]
\label{qhs:q-Euler}
Let $\omega =0$. Under hypotheses
(H$_{q,\omega}$1)--(H$_{q,\omega}$3) and conditions 1 to 3 of
Lemma~\ref{qhs:tecnico 2} on the Lagrangian $L$,
if $y_{\ast }\in \mathcal{Y}^{1}\left(\left[ a,b\right]_{q,0},
\mathbb{R}\right)$ is a local extremizer to problem \eqref{qhs:P}
(with $\omega =0$), then $y_{\ast }$ satisfies
the $q$-symmetric Euler--Lagrange equation
\begin{equation*}
\partial_{2}L\left( t,y\left( qt\right) ,\tilde{D}_{q}\left[ y\right]
\left( t\right) \right) =\tilde{D}_{q}\left[ \tau \mapsto
\partial_{3}L\left( q\tau, y\left( q^{2}\tau \right),
\tilde{D}_{q}\left[ y\right] \left( q\tau\right)
\right) \right] \left( t\right)
\end{equation*}
for all $t\in \left[ a,b\right]_{q}$.
\end{corollary}

To conclude this section, we prove
a sufficient optimality condition to \eqref{qhs:P}.

\begin{theorem}
\label{qhs:Suficiente}
Suppose that $a<b$ and $a,b\in \lbrack c]_{q,\omega}$ for some $c\in I$.
Also, assume that $L$ is a jointly convex (resp. concave) function in
$\left(u,v\right) $. If $y_{\ast }$ satisfies the Hahn symmetric Euler--Lagrange
equation \eqref{qhs:EqEuler}, then $y_{\ast }$ is a global minimizer (resp.
maximizer) to problem \eqref{qhs:P}.
\end{theorem}

\begin{proof}
Let $L$ be a jointly convex function in $\left( u,v\right) $
(the concave case is similar). Then, for any admissible variation $\eta$, we have
\begin{align*}
\mathcal{L}[ y_{\ast} &+ \eta ]
-\mathcal{L}\left[ y_{\ast}\right]  \\
&= \int_{a}^{b}\Biggl(L\left( t,y_{\ast }^{\sigma }\left( t\right) +\eta ^{\sigma }\left(
t\right) ,\tilde{D}_{q,\omega }\left[ y_{\ast }\right] \left( t\right)
+\tilde{D}_{q,\omega }\left[ \eta \right] \left( t\right) \right)\\
&\qquad\qquad -L\left(t,y_{\ast }^{\sigma}\left( t\right),
\tilde{D}_{q,\omega }\left[ y_{\ast }\right] \left( t\right)\right)\Biggr)\tilde{d}_{q,\omega }t\\
& \geqslant \int_{a}^{b}\Biggl( \partial _{2}L\left( t,y_{\ast }^{\sigma }\left( t\right) ,
\tilde{D}_{q,\omega }\left[ y_{\ast }\right] \left( t\right) \right) \eta^{\sigma}\left( t\right)\\
&\qquad\qquad +\partial_{3}L\left( t,y_{\ast }^{\sigma }\left( t\right),
\tilde{D}_{q,\omega }\left[ y_{\ast }\right] \left( t\right) \right) \tilde{D}_{q,\omega}
\left[ \eta \right] \left( t\right) \Biggr) \tilde{d}_{q,\omega }t.
\end{align*}
Using the integration by parts formula \eqref{qhs:partes} and Lemma~\ref{qhs:composta}, we get
\begin{align*}
\mathcal{L}\left[ y_{\ast }+\eta \right] &-\mathcal{L}\left[ y_{\ast }\right]
\geqslant \partial_{3}L\left( \sigma \left( t\right) ,y_{\ast }^{\sigma^{2}}\left( t\right)
,\left( \tilde{D}_{q,\omega }\left[ y_{\ast }\right] \right)^{\sigma }\left( t\right) \right)
\eta \left( t\right) \bigg|_{a}^{b} \\
& +\int_{a}^{b}\bigg[\partial _{2}L\left( t,y_{\ast }^{\sigma }\left(t\right)
,\tilde{D}_{q,\omega }\left[ y_{\ast }\right] \left( t\right)\right)\\
& -\tilde{D}_{q,\omega }\left[\tau \mapsto \partial_{3}L\left( \sigma\left(\tau \right),
y_{\ast }^{\sigma ^{2}}\left(\tau\right), \left( \tilde{D}_{q,\omega }\left[ y_{\ast }
\right] \right) ^{\sigma }\left(\tau\right) \right) \left( t\right)
\right] \bigg]\eta ^{\sigma }\left( t\right) \tilde{d}_{q,\omega }t.
\end{align*}
Since $y_{\ast }$ satisfies \eqref{qhs:EqEuler} and $\eta $ is an admissible variation, we obtain
\begin{equation*}
\mathcal{L}\left[ y_{\ast }+\eta \right]
-\mathcal{L}\left[ y_{\ast }\right] \geqslant 0,
\end{equation*}
proving that $y_{*}$ is a minimizer to problem \eqref{qhs:P}.
\end{proof}

\begin{ex}
Let $q\in \left] 0,1\right[ $ and $\omega \geqslant 0$ be  fixed real numbers.
Also, let $I \subseteq \mathbb{R}$ be an interval such that
$a:=\omega_{0},b\in I$ and $a<b$. Consider the problem
\begin{equation}
\label{qhs:eq:prb:ex1}
\left\{
\begin{array}[c]{l}
\mathcal{L}\left[ y\right] =\displaystyle\int_{a}^{b} \sqrt{1+\left( \tilde{D}_{q,\omega }
\left[ y\right] \left( t\right) \right) ^{2}} \tilde{d}_{q,\omega}t
\longrightarrow \min \\
\\
y\in \mathcal{Y}^{1}\left( \left[ a,b\right]_{q,\omega },\mathbb{R}\right)\\
\\
y\left( a\right) =a,\quad y\left( b\right) =b.
\end{array}
\right.
\end{equation}

If $y_{\ast }$ is a local minimizer to the problem,
then $y_{\ast}$ satisfies the Hahn symmetric Euler--Lagrange equation
\begin{equation}
\label{qhs:eq:el:ex1}
\tilde{D}_{q,\omega }\left[ \tau \mapsto
\frac{\left(\tilde{D}_{q,\omega }\left[ y\right]
\right)^{\sigma }\left( \tau \right)}{\sqrt{1+\left(\left( \tilde{D}_{q,\omega }\left[ y\right]
\right)^{\sigma }\left( \tau \right)\right) ^{2}}} \right] \left( t\right)
=0 \ \text{ for all }\ t \in \left[ a,b\right]_{q,\omega}.
\end{equation}
It is simple to check that function $y_{\ast }\left( t\right) =t$
is a solution to \eqref{qhs:eq:el:ex1} satisfying the given boundary conditions.
Since the Lagrangian is jointly convex in $(u,v)$, then we conclude from
Theorem~\ref{qhs:Suficiente} that function $y_{\ast }\left( t\right) =t$
is indeed a minimizer to problem \eqref{qhs:eq:prb:ex1}.
\end{ex}


\subsection{Leitmann's direct method}
\label{qhs:L}

Similarly to Malinowska and Torres \cite{Malinowska:3},
we show that Leitmann's direct method \cite{Leitmann3}
has also applications in the Hahn symmetric variational calculus.
Consider the variational functional integral
\begin{equation*}
\mathcal{\bar{L}}\left[ \bar{y}\right] =\int_{a}^{b}\bar{L}\left(
t,\bar{y}^{\sigma }\left( t\right) ,\tilde{D}_{q,\omega }\left[
\bar{y}\right]\left(t\right) \right) \tilde{d}_{q,\omega }t.
\end{equation*}
As before, we assume that function $\bar{L}:I\times
\mathbb{R} \times \mathbb{R} \rightarrow \mathbb{R}$
satisfies the following hypotheses:
\begin{enumerate}
\item[($\overline{\text{H$_{q,\omega}$1}}$)] $\left( u,v\right) \rightarrow \bar{L}\left( t,u,v\right) $ is a
$C^{1}\left( \mathbb{R}^{2},\mathbb{R}\right)$ function for any $t\in I $;

\item[($\overline{\text{H$_{q,\omega}$2}}$)] $t\rightarrow \bar{L}\left( t,\bar{y}^{\sigma }\left( t\right),
\tilde{D}_{q,\omega }\left[ \bar{y}\right] \left( t\right) \right) $ is
continuous at $\omega _{0}$ for any admissible function $\bar{y}$;

\item[($\overline{\text{H$_{q,\omega}$3}}$)] functions
$t\rightarrow \partial _{i+2}\bar{L}\left(t,
\bar{y}^{\sigma }\left( t\right),
\tilde{D}_{q,\omega }\left[ \bar{y}\right] \left(
t\right) \right)$ belong to
$\mathcal{Y}^{1}\left(\left[a,b\right]_{q,\omega},\mathbb{R}\right)$
for all admissible $\bar{y}$, $i=0,1$.
\end{enumerate}

\begin{lemma}[Leitmann's fundamental lemma via Hahn's symmetric quantum operator]
Let $y=z\left( t,\bar{y}\right) $ be a transformation having a unique
inverse $\bar{y}=\bar{z}\left( t,y\right)$ for all
$t\in \left[ a,b\right]_{q,\omega }$, such that there
is a one-to-one correspondence
\begin{equation*}
y\left( t\right) \leftrightarrow\bar{y}\left( t\right)
\end{equation*}
for all functions $y\in \mathcal{Y}^{1}\left(\left[a,b\right]_{q,\omega},\mathbb{R}\right)$
satisfying the boundary conditions $y\left( a\right) =\alpha $ and $y\left( b\right) =\beta$
and all functions $\bar{y}\in \mathcal{Y}^{1}\left(\left[a,b\right]_{q,\omega},\mathbb{R}\right)$
satisfying
\begin{equation}
\label{qhs:leit}
\bar{y}\left(a\right)=\bar{z}\left( a,\alpha \right)
\text{ and }\bar{y}\left(b\right)=\bar{z}\left(
b,\beta \right) \text{.}
\end{equation}
If the transformation $y=z\left( t,\bar{y}\right)$ is such that there
exists a function $G:I\times \mathbb{R} \rightarrow \mathbb{R}$
satisfying the identity
\begin{equation*}
L\left( t,y^{\sigma }\left( t\right) ,\tilde{D}_{q,\omega }\left[ y\right]
\left( t\right) \right) -\bar{L}\left( t,\bar{y}^{\sigma }\left( t\right),
\tilde{D}_{q,\omega }\left[ \bar{y}\right] \left( t\right) \right)
=\tilde{D}_{q,\omega }\left[\tau \mapsto G\left( \tau,\bar{y}\left( \tau\right) \right)\right](t) ,
\end{equation*}
for all $t\in\left[a,b\right]_{q,\omega}$, then if $\bar{y}_{\ast}$
is a maximizer (resp. a minimizer) of $\mathcal{\bar{L}}$ with
$\bar{y}_{\ast }$ satisfying \eqref{qhs:leit}, $y_{\ast }=z\left( t,\bar{y}_{\ast}\right)$
is a maximizer (resp. a minimizer) of $\mathcal{L}$ for $y_{\ast }$ satisfying
$y_{\ast }\left( a\right) =\alpha$ and $y_{\ast }\left( b\right) =\beta$.
\end{lemma}

\begin{proof}
Suppose $y\in \mathcal{Y}^{1}\left(
\left[ a,b\right]_{q,\omega },\mathbb{R}\right)$
satisfies the boundary conditions
$y\left( a\right) =\alpha $ and $y\left( b\right) =\beta$.
Define function $\bar{y}\in \mathcal{Y}^{1}\left(
\left[ a,b\right]_{q,\omega },\mathbb{R}\right)$
through the formula $\bar{y}=\bar{z}\left(
t,y\right) $, $t\in \left[ a,b\right] _{q,\omega }$.
Then, $\bar{y}$ satisfies \eqref{qhs:leit} and
\begin{equation*}
\begin{split}
\mathcal{L}&\left[ y\right] -\mathcal{\bar{L}}\left[ \bar{y}\right]\\
&=\int_{a}^{b}L\left( t,y^{\sigma }\left( t\right),
\tilde{D}_{q,\omega}\left[ y\right]\left( t\right) \right)
\tilde{d}_{q,\omega }t -\int_{a}^{b}\bar{L}\left( t,\bar{y}^{\sigma }\left( t\right),
\tilde{D}_{q,\omega }\left[ \bar{y}\right] \left( t\right) \right)
\tilde{d}_{q,\omega }t \\
&=\int_{a}^{b}\tilde{D}_{q,\omega }\left[\tau \mapsto G\left(\tau,
\bar{y}\left(\tau\right) \right)\right](t) \, \tilde{d}_{q,\omega }t \\
&= G\left( b,\bar{y}\left( b\right) \right)
- G\left( a,\bar{y}\left(a\right) \right)\\
&= G\left( b,\bar{z}\left( b,\beta \right) \right)
- G\left( a,\bar{z}\left(a,\alpha \right) \right) .
\end{split}
\end{equation*}
The desired result follows immediately because the right-hand side of the
above equality is a constant, depending only on the fixed-endpoint
conditions $y\left( a\right) =\alpha $ and $y\left( b\right) =\beta$.
\end{proof}

\begin{ex}
Let $q\in\left]0,1\right[$, $\omega \geqslant 0$, and $a:=\omega_0,b$
with $a<b$ be fixed real numbers. Also, let $I$ be
an interval of $\mathbb{R}$ such that $a,b\in I$. We consider the problem
\begin{equation}
\label{qhs:Lproblem}
\left\{
\begin{array}[c]{l}
\mathcal{L}\left[ y\right] =\displaystyle\int_{a}^{b}\left( \left(
\tilde{D}_{q,\omega}\left[ y\right]\left( t\right) \right)^{2}
+qy^{\sigma }\left( t\right)
+t \tilde{D}_{q,\omega }\left[ y\right] \left( t\right) \right)
\tilde{d}_{q,\omega }t \longrightarrow \min \\
\\
y\in \mathcal{Y}^{1}\left( \left[ a,b\right] _{q,\omega },\mathbb{R}\right)\\
\\
y\left( a\right) =\alpha, \quad y\left( b\right) =\beta,
\end{array}
\right.
\end{equation}
where $\alpha ,\beta \in \mathbb{R}$ and $\alpha \neq \beta$.
We transform problem \eqref{qhs:Lproblem} into the trivial problem
\begin{equation*}
\left\{
\begin{array}[c]{l}
\mathcal{\bar{L}}\left[ \bar{y}\right]
=\displaystyle\int_{a}^{b} \left(
\tilde{D}_{q,\omega }\left[ \bar{y}\right] \left( t\right) \right) ^{2}
\tilde{d}_{q,\omega }t\longrightarrow \min \\
\\
\bar{y}\in \mathcal{Y}^{1}\left( \left[ a,b\right] _{q,\omega },\mathbb{R}\right)\\
\\
\bar{y}\left( a\right) =0, \quad \bar{y}\left( b\right) =0,
\end{array}
\right.
\end{equation*}
which has solution $\bar{y}\equiv 0$.
For that we consider the transformation
\begin{equation*}
y\left( t\right) =\bar{y}\left( t\right) +ct+d,
\end{equation*}
where $c,d$ are real constants that will be chosen later. Since
$$
y^{\sigma }\left( t\right)
=\bar{y}^{\sigma }\left( t\right)
+c\sigma \left(t\right) + d
$$
and
$$
\tilde{D}_{q,\omega }\left[ y\right] \left( t\right) =\tilde{D}_{q,\omega}
\left[ \bar{y}\right] \left( t\right) +c,
$$
we have
\begin{equation*}
\begin{split}
&\left( \tilde{D}_{q,\omega }\left[ y\right] \left( t\right) \right)^{2}
+qy^{\sigma }\left( t\right) +t\tilde{D}_{q,\omega }\left[ y\right]
\left( t\right) \\
&=\left( \tilde{D}_{q,\omega }\left[ \bar{y}\right] \left( t\right) \right)
^{2}+2c\tilde{D}_{q,\omega }\left[ \bar{y}\right] \left( t\right)
+ c^{2}+q d + q\bar{y}^{\sigma}\left(t\right) + t\tilde{D}_{q,\omega }\left[
\bar{y}\right] \left( t\right) + c\left( q\sigma \left( t\right)
+t\right).
\end{split}
\end{equation*}
Therefore,
\begin{equation*}
\begin{split}
&\left[ \left( \tilde{D}_{q,\omega }\left[ y\right] \left( t\right) \right)^{2}
+qy^{\sigma }\left( t\right) +t\tilde{D}_{q,\omega }\left[ y\right]
\left( t\right) \right] -\left( \tilde{D}_{q,\omega }\left[ \bar{y}\right]
\left( t\right) \right)^{2}\\
&=\tilde{D}_{q,\omega }\left[ 2c\bar{y}\right] \left( t\right)
+\tilde{D}_{q,\omega }\left[ \left( c^{2}+qd\right) id\right] \left( t\right)
+ \tilde{D}_{q,\omega }\left[ \sigma \cdot \bar{y}\right] \left( t\right)
+c \tilde{D}_{q,\omega }\left[ \sigma \cdot id\right] \left( t\right) \\
&=\tilde{D}_{q,\omega }\left[ 2c\bar{y}+\left( c^{2}+qd\right) id+\sigma
\cdot \bar{y}+c\left( \sigma \cdot id\right) \right] \left( t\right),
\end{split}
\end{equation*}
where $id$ represents the identity function. In order to obtain the
solution to the original problem, it suffices to choose $c$ and $d$ such that
\begin{equation}
\label{qhs:sol}
\left\{
\begin{array}{c}
ca+d=\alpha \\
\\
cb+d=\beta .
\end{array}
\right.
\end{equation}
Solving the system of equations \eqref{qhs:sol}, we obtain
$$
c = \frac{\alpha -\beta}{a-b}
$$
and
$$
d = \frac{a\beta -b\alpha }{a-b}.
$$
Hence, the global minimizer to problem \eqref{qhs:Lproblem} is
\begin{equation*}
y\left( t\right) =\frac{\alpha -\beta }{a-b}t+\frac{a\beta -b\alpha }{a-b}.
\end{equation*}
\end{ex}


\section{State of the Art}

Since the recent construction of the inverse operator
of Hahn's derivative \cite{Aldwoah,Aldwoah:2},
the calculus of variations within Hahn's quantum calculus
is under current research. We provide here some references
within this subject: \cite{Brito:da:Cruz,Malinowska:5,Malinowska:3}.

The results of this chapter were presented by the author at the
International Conference on Differential \& Difference Equations
and Applications 2011, University of Azores, Portugal and are published in \cite{Brito:da:Cruz:5}.


\clearpage{\thispagestyle{empty}\cleardoublepage}


\chapter{The Symmetric Calculus on Time Scales}
\label{The Symmetric Calculus on Time Scales}

In this chapter we define a symmetric derivative on time scales and derive some of its properties.
A diamond integral, which is a refined version of the diamond-$\alpha$ integral,
is also introduced. A mean value theorem is proved for the diamond integral
as well as versions of Holder's, Cauchy-Schwarz's and Minkowski's inequalities.


\section{Introduction}

Symmetric properties of functions are very useful in a large number of
problems. In particular, in the theory of trigonometric series, applications
of these properties are well known \cite{Zygmund}. Differentiability of a
function is one of the most important properties in the theory of functions
of real variables. However, even simple functions such as
\begin{equation}
\begin{array}{l}
f\left( x\right) =\left \vert x\right \vert  \\
\\
g\left( x\right) =\left \{
\begin{array}{ccc}
x\sin \frac{1}{x} & \text{ , } & x\neq 0 \\
&  &  \\
0 & \text{ , } & x=0
\end{array}
\right.  \\
\\
h\left( x\right) =\frac{1}{x^{2}} \text{ , \ \ }  x \neq 0
\end{array}
\label{ts:1}
\end{equation}
do not have (classical) derivative at $x=0$.
Authors like Riemann, Schwarz, Peano, Dini
and de la Vall\'{e}e-Poussin extended the notion of the classical derivative
in different ways, depending on the purpose \cite{Zygmund}. One of such
generalizations is the symmetric derivative defined by
\begin{equation*}
f^{s}\left( x\right)
=\lim_{h\rightarrow 0}\frac{f\left( x+h\right) - f\left(x-h\right) }{2h}.
\end{equation*}
As notice before, the functions in \eqref{ts:1} do not have ordinary derivative at $x=0$,
however they have symmetric derivative at $x=0$
and $f^{s}\left( 0\right) =0$, $g^{s}\left( 0\right) =0$
and $h^{s}\left( 0\right) =0$. For a deeper understanding of the symmetric
derivative and its properties, we refer the reader to the book \cite{Thomson}.

On the other hand, the symmetric quotient
\begin{equation*}
\frac{f\left( x+h\right) -f\left( x-h\right) }{2h}
\end{equation*}
has, in general, better convergence properties than the ordinary difference
quotient \cite{Serafin}. The study of the symmetric quotient naturally
lead us to the quantum calculus. Kac and Cheung \cite{Kac} defined the
$h$-symmetric difference and the $q$-symmetric difference operators, respectively, by
\begin{equation*}
\tilde{D}_{h}\left( x\right) =\frac{f\left( x+h\right) -f\left( x-h\right) }{2h}
\end{equation*}
and
\begin{equation*}
\tilde{D}_{q}\left( x\right) =\frac{f\left( qx\right) -f\left(
q^{-1}x\right) }{\left( q-q^{-1}\right) x}
\end{equation*}
where $h \neq 0$ and $q \neq 1$.

In 1988, Hilger introduced the theory of time scales,
which is a theory that was created in order to unify, and extend, discrete
and continuous analysis into a single theory \cite{Hilger}.
In Chapter~\ref{ts:sec:pre} we reviewed the necessary
definitions and results on time scales \cite{Bohner,Bohner:2}.
Namely, we presented the nabla and the delta calculus. We also presented the
diamond-$\alpha$ dynamic derivative, which was introduced by Sheng, Fadag,
Henderson and Davis, as a linear combination of the delta and nabla derivatives \cite{Sheng}.

Our goal in this chapter is to define the symmetric
derivative on time scales and to develop the symmetric
time scale calculus. This chapter is organized as follows.
In Section~\ref{ts:sec:dif} we define the time scale
symmetric derivative and derive some of its properties.
In some special cases we will show that the diamond-$\alpha$
derivative coincides with the symmetric derivative on time scales, however
in general this is not true. In Section~\ref{ts:sec:int} we
introduce the diamond integral which is an attempt to obtain a symmetric
integral and whose construction is similar to the diamond-$\alpha$ integral.
Finally, we prove some inequalities for the diamond integral.

In what follows, $\mathbb{T}$ denotes a time scale with operators
$\sigma$, $\rho$, $\Delta$ and $\nabla$ (see Chapter~\ref{ts:sec:pre}).
We recall that $\mathbb{T}$ has the topology inherited
from $\mathbb{R}$ with the standard topology.


\section{Symmetric differentiation}
\label{ts:sec:dif}

\begin{definition}
We say that a function $f:\mathbb{T} \rightarrow \mathbb{R}$ is
\emph{symmetric continuous}\index{Symmetric continuity} at $t\in $ $\mathbb{T}$ if,
for any $\varepsilon >0$, there exists a neighborhood $U_{t}$ of $t$ such that for all
$s\in U_{t}$ for which  $2t-s\in U_{t}$ one has
\begin{equation*}
\left \vert f\left( s\right) -f\left( 2t-s\right) \right \vert \leqslant \varepsilon .
\end{equation*}
\end{definition}

It is easy to see that continuity implies symmetric continuity.

\begin{proposition}
Let $\mathbb{T}$ be a time scale. If $f:\mathbb{T\rightarrow\mathbb{R}}$ is a continuous function, then $f$ is
symmetric continuous.
\end{proposition}

\begin{proof}
Since $f$ is continuous at $t\in \mathbb{T}$, then, for any $\varepsilon >0$,
there exists an neighborhood $U_{t}$ such that
\[
\left\vert f\left( s\right) -f\left( t\right) \right\vert <\frac{\varepsilon}{2}
\]
and
\[
\left\vert f\left( 2t-s\right) -f\left( t\right) \right\vert <\frac{
\varepsilon }{2}
\]
for all $s\in U_{t}$ for which $2t-s\in U_t$. Thus,
\begin{eqnarray*}
\left\vert f\left( s\right) -f\left( 2t-s\right) \right\vert  &\leqslant
&\left\vert f\left( s\right) -f\left( t\right) \right\vert +\left\vert
f\left( t\right) -f\left( 2t-s\right) \right\vert  \\
&<&\varepsilon.
\end{eqnarray*}
\end{proof}

The next example shows that symmetric continuity does not imply continuity.

\begin{example}
Consider the function $f:\mathbb{
\mathbb{R}
\rightarrow
\mathbb{R}
}$ defined by
\[
f\left( t\right) =\left\{
\begin{array}{ccc}
0 &  if  & t\neq 0 \\
&  &  \\
1 &  \ if \  & t=0.
\end{array}
\right.
\]
Function $f$ is symmetric continuous at $0$: for any $\varepsilon> 0$, there exists a neighborhood $U_t$ of $t=0$
such that
\[
\vert f(s)-f(-s)\vert = 0 < \varepsilon
\]
for all $s\in U_t$ for which $-s\in U_t$. However, $f$ is not continuous
at $0$.
\end{example}

Now we define the symmetric derivative on time scales.

\begin{definition}
Let $f:\mathbb{T} \rightarrow \mathbb{R}$ be a function and let
$t\in \mathbb{T}_{\kappa}^{\kappa}$. The \emph{symmetric derivative}\index{Symmetric derivative}
of $f$ at $t$, denoted by $f^{\diamondsuit }\left( t\right) $, is the real number
(provided that exists) with the property that, for any $\varepsilon >0$,
there exists a neighborhood $U$ of $t$ such that,
for all $s\in U$ for which  $2t-s\in U$, we have
\begin{multline}
\label{ts:derivada}
\left \vert \left[ f^{\sigma }\left( t\right) -f\left( s\right) +f\left(
2t-s\right) -f^{\rho }\left( t\right) \right] -f^{\diamondsuit }\left(
t\right) \left[ \sigma \left( t\right) +2t-2s-\rho \left( t\right) \right]
\right \vert \\
\leqslant \varepsilon \left \vert \sigma \left( t\right) +2t-2s-\rho \left(
t\right) \right \vert  .
\end{multline}
A function $f$ is said to be symmetric differentiable provided $f^{\diamondsuit }\left(t\right)$
exists for all $t\in \mathbb{T}_{\kappa}^{\kappa}$.
\end{definition}

Some useful properties are given next.

\begin{theorem}
\label{ts:propriedade}
Let $f:\mathbb{T}\rightarrow \mathbb{R}$
be a function and let $t\in \mathbb{T}_{\kappa}^{\kappa}$.

\begin{enumerate}
\item Then $f$ has at most one symmetric
derivative at $t$;

\item If $f$ is symmetric differentiable at $t$, then $f$ is symmetric
continuous at $t$;

\item If $f$ is continuous at $t$ and if $t$ is not dense,
then $f$ is symmetric differentiable at $t$ and
\begin{equation*}
f^{\diamondsuit }\left( t\right) =\frac{f^ \sigma \left( t\right)
- f^\rho\left( t\right)  }{\sigma \left( t\right)
- \rho\left( t\right) };
\end{equation*}

\item If $t$ is dense, then $f$ is symmetric differentiable at $t$ if and
only if the limit
\begin{equation*}
\lim_{s\rightarrow t}\frac{-f\left( s\right) +f\left( 2t-s\right) }{2t-2s}
\end{equation*}
exists as a finite number. In this case
\begin{eqnarray*}
f^{\diamondsuit }\left( t\right)  &=&\lim_{s\rightarrow t}\frac{-f\left(
s\right) +f\left( 2t-s\right) }{2t-2s} \\
&=&\lim_{h\rightarrow 0}\frac{f\left( t+h\right) -f\left( t-h\right) }{2h};
\end{eqnarray*}

\item If $f$ is symmetric differentiable and continuous at $t$, then
\begin{equation*}
f^{\sigma }\left( t\right) =f^{\rho }\left( t\right) +f^{\diamondsuit
}\left( t\right) \left[ \sigma \left( t\right) -\rho \left( t\right) \right].
\end{equation*}
\end{enumerate}
\end{theorem}

\begin{proof}
\begin{enumerate}
\item Let us suppose that $f$ has two symmetric derivatives at $t$,
$f_{1}^{\diamondsuit }\left( t\right)$ and
$f_{2}^{\diamondsuit }\left(t\right)$.
Then there exists a neighborhood $U_{1}$ of $t$ such that
\begin{eqnarray*}
&&\left \vert \left[ f^{\sigma }\left( t\right) -f\left( s\right) +f\left(
2t-s\right) -f^{\rho }\left( t\right) \right] -f_{1}^{\diamondsuit }\left(
t\right) \left[ \sigma \left( t\right) +2t-2s-\rho \left( t\right) \right]
\right \vert\\
\leqslant &&\frac{\varepsilon }{2}\left \vert \sigma \left( t\right)
+2t-2s-\rho \left( t\right) \right \vert
\end{eqnarray*}
for all $s\in U_{1}$ for which $2t-s\in U_{1}$,
and there exists a neighborhood $U_{2}$ of $t$ such that
\begin{eqnarray*}
&&\left \vert \left[ f^{\sigma }\left( t\right) -f\left( s\right) +f\left(
2t-s\right) -f^{\rho }\left( t\right) \right] -f_{2}^{\diamondsuit }\left(
t\right) \left[ \sigma \left( t\right) +2t-2s-\rho \left( t\right) \right]
\right \vert  \\
\leqslant &&\frac{\varepsilon }{2}\left \vert \sigma \left( t\right)
+2t-2s-\rho \left( t\right) \right \vert
\end{eqnarray*}
for all $s\in U_{2}$ for which $2t-s \in U_{2}$.
Therefore, for all $s\in U_{1}\cap U_{2}$ for which $2t-s\in U_{1}\cap U_{2}$,
\begin{eqnarray*}
&&\bigg{|}f_{1}^{\diamondsuit }\left( t\right) -f_{2}^{\diamondsuit }\left(
t\right) \bigg{|}
=\left \vert \left[ f_{1}^{\diamondsuit }\left( t\right)
-f_{2}^{\diamondsuit }\left( t\right) \right] \frac{\sigma \left( t\right)
+2t-2s-\rho \left( t\right) }{\sigma \left( t\right) +2t-2s-\rho \left(
t\right) }\right \vert  \\
=&&\bigg{|}\left[ f^{\sigma }\left( t\right) -f\left( s\right) +f\left(
2t-s\right) -f^{\rho }\left( t\right) \right] -f_{2}^{\diamondsuit }\left(
t\right) \left[ \sigma \left( t\right) +2t-2s-\rho \left( t\right) \right]\\
&&-\left[ f^{\sigma }\left( t\right) -f\left( s\right) +f\left( 2t-s\right)
-f^{\rho }\left( t\right) \right] +f_{1}^{\diamondsuit }\left( t\right)
\left[ \sigma \left( t\right) +2t-2s-\rho \left( t\right) \right] \bigg{|} \\
&&\times \frac{1}{\left \vert \sigma \left( t\right) +2t-2s-\rho \left(
t\right) \right \vert } \\
\leqslant &&\bigg{(}\left \vert \left[ f^{\sigma }\left( t\right) -f\left(
s\right) +f\left( 2t-s\right) -f^{\rho }\left( t\right) \right]
-f_{2}^{\diamondsuit }\left( t\right) \left[ \sigma \left( t\right)
+2t-2s-\rho \left( t\right) \right] \right \vert  \\
&&+\left \vert \left[ f^{\sigma }\left( t\right) -f\left( s\right) +f\left(
2t-s\right) -f^{\rho }\left( t\right) \right] -f_{1}^{\diamondsuit }\left(
t\right) \left[ \sigma \left( t\right) +2t-2s-\rho \left( t\right) \right]
\right \vert \bigg{)} \\
&&\times \frac{1}{\left \vert \sigma \left( t\right) +2t-2s-\rho \left(
t\right) \right \vert } \\
\leqslant &&\varepsilon
\end{eqnarray*}
proving the desired result.

\item By definition of symmetric derivative, for any $\epsilon^{*}>0$
there exists a neighborhood $U$ of $t$ such that
\begin{multline*}
\left \vert \left[ f^{\sigma }\left( t\right) -f\left( s\right) +f\left(
2t-s\right) -f^{\rho }\left( t\right) \right] -f^{\diamondsuit }\left(
t\right) \left[ \sigma \left( t\right) +2t-2s-\rho \left( t\right) \right]
\right \vert  \\
\leqslant \varepsilon ^{\ast }\left \vert \sigma \left( t\right)
+2t-2s-\rho \left( t\right) \right \vert
\end{multline*}
for all $s\in U$ for which $2t-s \in U$.
Therefore we have for all $s\in U \cap \left] t
-\varepsilon^{\ast },t+\varepsilon ^{\ast }\right[$
\begin{eqnarray*}
&&\left \vert -f\left( s\right) +f\left( 2t-s\right) \right \vert \\
\leqslant&& \left \vert \left[ f^{\sigma }\left( t\right) -f\left( s\right)
+f\left( 2t-s\right) -f^{\rho }\left( t\right) \right] -f^{\diamondsuit
}\left( t\right) \left[ \sigma \left( t\right) +2t-2s-\rho \left( t\right)
\right] \right \vert  \\
&&+\left \vert \left[ f^{\sigma }\left( t\right) -f^{\rho }\left( t\right)
\right] -f^{\diamondsuit }\left( t\right) \left[ \sigma \left( t\right)
+2t-2s-\rho \left( t\right) \right] \right \vert  \\
\leqslant& &\left \vert \left[ f^{\sigma }\left( t\right) -f\left( s\right)
+f\left( 2t-s\right) -f^{\rho }\left( t\right) \right] -f^{\diamondsuit
}\left( t\right) \left[ \sigma \left( t\right) +2t-2s-\rho \left( t\right)
\right] \right \vert  \\
&&+\left \vert \left[ f^{\sigma }\left( t\right) -f\left( t\right) +f\left(
t\right) -f^{\rho }\left( t\right) \right] -f^{\diamondsuit }\left( t\right)
\left[ \sigma \left( t\right) +2t-2t-\rho \left( t\right) \right]
\right \vert  \\
&&+2\left \vert f^{\diamondsuit }\left( t\right) \right \vert \left \vert
t-s\right \vert\\
\leqslant& &\varepsilon ^{\ast }\left \vert \sigma \left( t\right)
+2t-2s-\rho \left( t\right) \right \vert +\varepsilon ^{\ast }\left \vert
\sigma \left( t\right) +2t-2t-\rho \left( t\right) \right \vert +2\left \vert
f^{\diamondsuit }\left( t\right) \right \vert \left \vert t-s\right \vert  \\
\leqslant&& \varepsilon ^{\ast }\left \vert \sigma \left( t\right) -\rho
\left( t\right) \right \vert +2\varepsilon ^{\ast }\left \vert t-s\right \vert
+\varepsilon ^{\ast }\left \vert \sigma \left( t\right) -\rho \left( t\right)
\right \vert +2\left \vert f^{\diamondsuit }\left( t\right) \right \vert
\left \vert t-s\right \vert  \\
= &&2\varepsilon ^{\ast }\left \vert \sigma \left( t\right) -\rho
\left( t\right) \right \vert +2\left( \varepsilon ^{\ast }+\left \vert
f^{\diamondsuit }\left( t\right) \right \vert \right) \left \vert
t-s\right \vert  \\
\leqslant &&2\varepsilon ^{\ast }\left \vert \sigma \left( t\right) -\rho \left(
t\right) \right \vert +2\left( \varepsilon ^{\ast }+\left \vert
f^{\diamondsuit }\left( t\right) \right \vert \right) \varepsilon ^{\ast } \\
= &&2\varepsilon ^{\ast }\left[ \left \vert \sigma \left( t\right)
-\rho \left( t\right) \right \vert +\varepsilon ^{\ast }+\left \vert
f^{\diamondsuit }\left( t\right) \right \vert \right]  \\
\end{eqnarray*}
proving that $f$ is symmetric continuous at $t$.

\item Suppose that $t\in\mathbb{T}_{\kappa}^{\kappa}$
is not dense and $f$ is continuous at $t$. Then,
\begin{equation*}
\lim_{s\rightarrow t}\frac{f^{\sigma }\left( t\right) -f\left( s\right)
+f\left( 2t-s\right) -f^{\rho }\left( t\right) }{\sigma \left( t\right)
+2t-2s-\rho \left( t\right) }=\frac{f^{\sigma }\left( t\right)
-f^{\rho}\left( t\right)}{\sigma \left( t\right) -\rho \left( t\right)}.
\end{equation*}
Hence, for any $\varepsilon >0$, there exists a neighborhood $U$ of $t$ such that
\begin{equation*}
\left \vert \frac{f^{\sigma }\left( t\right) -f\left( s\right) +f\left(
2t-s\right) -f^{\rho }\left( t\right) }{\sigma \left( t\right) +2t-2s-\rho
\left( t\right) }-\frac{f^{\sigma }\left( t\right) -f^{\rho }\left(
t\right)}{\sigma \left( t\right) -\rho \left( t\right) }\right \vert \leqslant \varepsilon
\end{equation*}
for all $s\in U$ for which  $2t-s \in U$. It follows that
\begin{multline*}
\left \vert \left[ f^{\sigma }\left( t\right) -f\left( s\right) +f\left(
2t-s\right) -f^{\rho }\left( t\right) \right] -\frac{f^{\sigma }\left(
t\right) -f^{\rho }\left( t\right) }{\sigma \left( t\right) -\rho \left(
t\right) }\left[ \sigma \left( t\right) +2t-2s-\rho \left( t\right) \right]
\right \vert  \\
\leqslant \varepsilon \left[ \sigma \left( t\right) +2t-2s
-\rho \left(t\right) \right],
\end{multline*}
which proves that
\begin{equation*}
f^{\diamondsuit }\left( t\right) =\frac{f^{\sigma }\left( t\right)
-f^{\rho}\left( t\right) }{\sigma \left( t\right) - \rho\left( t\right)}.
\end{equation*}

\item Assume that $f$ is symmetric differentiable at $t$ and $t$ is dense.
Let $\varepsilon >0$ be given.
Hence there exists a neighborhood $U$ of $t$ such that
\begin{multline*}
\left \vert \left[ f^{\sigma }\left( t\right) -f\left( s\right) +f\left(
2t-s\right) -f^{\rho }\left( t\right) \right] -f^{\diamondsuit }\left(
t\right) \left[ \sigma \left( t\right) +2t-2s-\rho \left( t\right) \right]
\right \vert \\
\leqslant\varepsilon \left \vert \sigma \left( t\right) +2t-2s-\rho \left(
t\right) \right \vert
\end{multline*}
for all $s\in U$ for which $2t-s \in U$.
Since $t$ is dense we have that
\begin{equation*}
\left \vert \left[ -f\left( s\right) +f\left( 2t-s\right) \right]
-f^{\diamondsuit }\left( t\right) \left[ 2t-2s\right] \right \vert \leqslant
\varepsilon \left \vert 2t-2s\right \vert
\end{equation*}
for all $s\in U$ for which  $2t-s\in U$. It follows that
\begin{equation*}
\left \vert \frac{-f\left( s\right) +f\left( 2t-s\right) }{2t-2s}
-f^{\diamondsuit }\left( t\right) \right \vert \leqslant \varepsilon
\end{equation*}
for all $s\in U$ with $s\neq t$. Therefore we get the desired result:
\begin{equation*}
f^{\diamondsuit }\left( t\right) =\lim_{s\rightarrow t}\frac{-f\left(
s\right) +f\left( 2t-s\right) }{2t-2s}.
\end{equation*}
Conversely, let us suppose that $t$ is dense and the limit
\begin{equation*}
\lim_{s\rightarrow t}\frac{-f\left( s\right) +f\left( 2t-s\right) }{2t-2s}
\end{equation*}
exists. Let $L:=\displaystyle \lim_{s\rightarrow t}\frac{-f\left( s\right)
+f\left( 2t-s\right) }{2t-2s}$. Hence, there exists a neighborhood $U$ of $t$
such that
\begin{equation*}
\left \vert \frac{-f\left( s\right) +f\left( 2t-s\right) }{2t-2s}
-L\right \vert \leqslant \varepsilon
\end{equation*}
for all  $s\in U$ for which  $2t-s\in U$,
and since $t$ is dense, we have
\begin{eqnarray*}
&&\left \vert \frac{f^{\sigma }\left( t\right) -f\left( s\right) +f\left(
2t-s\right) -f^{\rho }\left( t\right) }{\sigma \left( t\right) +2t-2s-\rho
\left( t\right) }-L\right \vert \leqslant \varepsilon
\end{eqnarray*}
and hence
\begin{eqnarray*}
\left \vert \left[ f^{\sigma }\left( t\right) -f\left( s\right)
+f\left( 2t-s\right) -f^{\rho }\left( t\right) \right] -L\left[ \sigma
\left( t\right) +2t-2s-\rho \left( t\right) \right] \right \vert
\leqslant \varepsilon \left \vert \sigma \left( t\right) +2t-2s-\rho \left(
t\right) \right \vert
\end{eqnarray*}
which lead us to the conclusion that $f$ is symmetric differentiable and
\begin{equation*}
f^{\diamondsuit }\left( t\right) =\lim_{s\rightarrow t}\frac{-f\left(
s\right) +f\left( 2t-s\right) }{2t-2s}\text{.}
\end{equation*}
Note that if we use the substitution $s=t+h$, then
\begin{eqnarray*}
f^{\diamondsuit }\left( t\right)  &=&\lim_{s\rightarrow t}\frac{-f\left(
s\right) +f\left( 2t-s\right) }{2t-2s} \\
&=&\lim_{h\rightarrow 0}\frac{-f\left( t+h\right) +f\left( 2t-\left(
t+h\right) \right) }{2t-2\left( t+h\right) } \\
&=&\lim_{h\rightarrow 0}\frac{f\left( t+h\right) -f\left( t-h\right)}{2h}.
\end{eqnarray*}

\item If $t$ is a dense point, then
$\sigma \left( t\right) =\rho \left(t\right)$ and
\begin{eqnarray*}
&&f^{\sigma }\left( t\right) -f^{\rho }\left( t\right) =f^{\diamondsuit}\left(
t\right) \left[ \sigma \left( t\right) -\rho \left( t\right) \right]\\
&\Leftrightarrow &f^{\sigma }\left( t\right) =f^{\rho }\left( t\right)
+f^{\diamondsuit }\left( t\right) \left[ \sigma \left( t\right)
-\rho \left(t\right) \right] .
\end{eqnarray*}
If $t$ is not dense and since $f$ is continuous, then
\begin{eqnarray*}
&&f^{\diamondsuit }\left( t\right) =\frac{f^{\sigma }\left( t\right)
-f^{\rho }\left( t\right) }{\sigma \left( t\right) -\rho \left( t\right) } \\
&\Leftrightarrow &f^{\sigma }\left( t\right) =f^{\rho }\left( t\right)
+f^{\diamondsuit }\left( t\right) \left[ \sigma \left( t\right)
-\rho \left(t\right) \right] .
\end{eqnarray*}
\end{enumerate}
\end{proof}

\begin{ex}
\begin{enumerate}
\item If $\mathbb{T} = h \mathbb{Z}$ ($h>0$),
then the symmetric derivative is the symmetric
difference operator given by
\begin{equation*}
f^{\diamondsuit }\left( t\right) = \frac{f\left( t+h\right)
-f\left(t-h\right)}{2h}.
\end{equation*}

\item If $\mathbb{T} = \overline{q^{\mathbb{Z}}}$ and $0<q<1$,
then the symmetric derivative is the $q$-symmetric difference operator given by
\begin{equation*}
f^{\diamondsuit }\left( t\right) =\frac{f\left( qt\right)
-f\left(q^{-1}t\right)}{\left( q-q^{-1}\right) t}, \ \ \ \text{for } t\neq 0.
\end{equation*}

\item If $\mathbb{T} = \mathbb{R}$, then the symmetric derivative
is the classic symmetric derivative given by
\begin{equation*}
f^{\diamondsuit }\left( t\right) =\lim_{h\rightarrow 0}\frac{f\left(
t+h\right) -f\left( t-h\right) }{2h}.
\end{equation*}
\end{enumerate}
\end{ex}

\begin{remark}
It is clear that the symmetric derivative of a constant function is zero
and the symmetric derivative of the identity functions is one.
\end{remark}

\begin{remark}
An alternative way to define the symmetric derivative of $f$ at
$t\in\mathbb{T}_{\kappa}^{\kappa}$ is saying
that the following limit exists:
\begin{eqnarray*}
f^{\diamondsuit }\left( t\right) &=&\lim_{s\rightarrow t}\frac{f^{\sigma
}\left( t\right) -f\left( s\right) +f\left( 2t-s\right) -f^{\rho }\left(
t\right) }{\sigma \left( t\right) +2t-2s-\rho \left( t\right) } \\
&=&\lim_{h\rightarrow 0}\frac{f^{\sigma }\left( t\right) -f\left( t+h\right)
+f\left( t-h\right) -f^{\rho }\left( t\right) }{\sigma \left( t\right)
-2h-\rho \left( t\right) }.
\end{eqnarray*}
\end{remark}

\begin{ex}
Let $f:\mathbb{T} \rightarrow \mathbb{R}$ be defined by
$f\left( t\right) =\left \vert t\right \vert $. Then
\begin{eqnarray*}
f^{\diamondsuit }\left( 0\right)  &=&\lim_{h\rightarrow 0}\frac{f^{\sigma
}\left( 0\right) -f\left( 0+h\right) +f\left( 0-h\right) -f^{\rho }\left(
0\right) }{\sigma \left( 0\right) -2h-\rho \left( 0\right) } \\
&=&\lim_{h\rightarrow 0}\frac{\sigma \left( 0\right) -\left \vert
h\right \vert +\left \vert -h\right \vert +\rho \left( 0\right) }{\sigma \left(
0\right) -2h-\rho \left( 0\right) } \\
&=&\lim_{h\rightarrow 0}\frac{\sigma \left( 0\right)
+\rho \left( 0\right) }{\sigma \left( 0\right) -2h-\rho \left( 0\right)}.
\end{eqnarray*}
In the particular case $\mathbb{T} = \mathbb{R}$, we have
\begin{equation*}
f^{\diamondsuit }\left( 0\right) =0.
\end{equation*}
\end{ex}

\begin{proposition}
\label{ts:delta nabla}
If $f$ is delta and nabla differentiable, then $f$ is symmetric
differentiable and, for each $t\in\mathbb{T}_{\kappa}^{\kappa}$,
\begin{eqnarray*}
f^{\diamondsuit }\left( t\right)
&=&\gamma _{1}\left( t\right) f^{\Delta}\left( t\right)
+\bigg{(}1-\gamma _{1}\left( t\right)\bigg{)} f^{\nabla }\left( t\right),
\end{eqnarray*}
where
\begin{eqnarray*}
\gamma _{1}\left( t\right)  &=&\lim_{s\rightarrow t}\frac{\sigma \left(
t\right) -s}{\sigma \left( t\right) +2t-2s-\rho \left( t\right)}.
\end{eqnarray*}
\end{proposition}

\begin{proof}
Note that
\begin{eqnarray*}
f^{\diamondsuit }\left( t\right)  &=&\lim_{s\rightarrow t}\frac{f^{\sigma}\left(
t\right) -f\left( s\right) +f\left( 2t-s\right) -f^{\rho }\left(
t\right) }{\sigma \left( t\right) +2t-2s-\rho \left( t\right) } \\
&=&\lim_{s\rightarrow t}\bigg{(}\frac{\sigma \left( t\right) -s}{\sigma
\left( t\right) +2t-2s-\rho \left( t\right) }\frac{f^{\sigma }\left(
t\right) -f\left( s\right) }{\sigma \left( t\right) -s} \\
&&+\frac{\left( 2t-s\right) -\rho \left( t\right) }{\sigma \left( t\right)
+2t-2s-\rho \left( t\right) }\frac{f\left( 2t-s\right) -f^{\rho }\left(
t\right) }{\left( 2t-s\right) -\rho \left( t\right) }\bigg{)} \\
&=&\lim_{s\rightarrow t}\left( \frac{\sigma \left( t\right) -s}{\sigma
\left( t\right) +2t-2s-\rho \left( t\right) }f^{\Delta }\left( t\right)
+ \frac{\left( 2t-s\right) -\rho \left( t\right) }{\sigma \left( t\right)
+2t-2s-\rho \left( t\right) }f^{\nabla }\left( t\right) \right).
\end{eqnarray*}
Let
\begin{eqnarray*}
\gamma _{2}\left( t\right)  &:=&\lim_{s\rightarrow t}\frac{\left(
2t-s\right) -\rho \left( t\right) }{\sigma \left( t\right) +2t-2s
-\rho\left( t\right)}.
\end{eqnarray*}
It is clear that
\begin{equation*}
\gamma _{1}\left( t\right) +\gamma _{2}\left( t\right) =1.
\end{equation*}
Note that if $t\in \mathbb{T}$ is dense, then
\begin{eqnarray*}
\gamma _{1}\left( t\right)  &=&\lim_{s\rightarrow t}\frac{\sigma \left(
t\right) -s}{\sigma \left( t\right) +2t-2s-\rho \left( t\right) } \\
&=&\lim_{s\rightarrow t}\frac{t-s}{2t-2s} \\
&=&\frac{1}{2}
\end{eqnarray*}
and therefore
\begin{equation*}
\gamma _{2}\left( t\right) =\frac{1}{2}.
\end{equation*}
On the other hand, if $t\in \mathbb{T}$ is not dense, then
\begin{eqnarray*}
\gamma _{1}\left( t\right)  &=&\lim_{s\rightarrow t}\frac{\sigma \left(
t\right) -s}{\sigma \left( t\right) +2t-2s-\rho \left( t\right) } \\
&=&\frac{\sigma \left( t\right) -t}{\sigma \left( t\right) -\rho \left(
t\right) }
\end{eqnarray*}
and
\begin{equation*}
\gamma _{2}\left( t\right) =\frac{t-\rho \left( t\right) }{\sigma \left(
t\right) -\rho \left( t\right) }.
\end{equation*}
Hence, the functions $\gamma_{1}$, $\gamma_{2}:\mathbb{T}\rightarrow \mathbb{R}$
are well defined and if $f$ is delta and nabla differentiable we have
\begin{eqnarray*}
f^{\diamondsuit }\left( t\right)  &=&\gamma _{1}\left( t\right) f^{\Delta}\left(
t\right) +\gamma _{2}\left( t\right) f^{\nabla }\left( t\right)  \\
&=&\gamma _{1}\left( t\right) f^{\Delta }\left( t\right) +\bigg{(} 1
-\gamma_{1}\left( t\right) \bigg{)} f^{\nabla }\left( t\right).
\end{eqnarray*}
\end{proof}

\begin{remark}
\begin{enumerate}
\item Suppose that $f$ is delta and nabla differentiable.
When a point $t\in\mathbb{T}_{\kappa}^{\kappa}$ is right-scattered
and left-dense, then its symmetric derivative is equal to its delta
derivative and when $t$ is left-scattered and right-dense,
its symmetric derivative is equal to its nabla derivative.

\item Note that if all the points in $\mathbb{T}_{\kappa}^{\kappa}$
are dense and $f$ is delta (or nabla) differentiable, then
\begin{equation*}
f^{\diamondsuit}\left ( t \right)
= \frac{1}{2}f^{\Delta}\left ( t \right )
+ \frac{1}{2}f^{\nabla}\left ( t \right )\\
=f'\left( t \right), \text{   } t \in \mathbb{T}_{\kappa}^{\kappa}.
\end{equation*}

\item If $f$ is delta and nabla differentiable and if $\gamma_{1}$ is
constant, then the symmetric derivative coincides with the diamond-$\alpha$
derivative. The symmetric derivative tell us exactly the weight of the delta
and nabla derivative at each point. In the definition of the
diamond-$\alpha$ derivative we choose the influence of the nabla and delta
derivative. Moreover, this influence of the nabla and delta derivative does
not depend of the point we choose in the diamond-$\alpha$ derivative.
\end{enumerate}
\end{remark}

\begin{proposition}
The functions $\gamma_{1}$, $\gamma_{2}:\mathbb{T}\rightarrow \mathbb{R}$
are bounded and nonnegative. Moreover,
\begin{equation*}
0\leqslant \gamma _{i}\leqslant 1\text{ ,  \  } i=1,2.
\end{equation*}
\end{proposition}

\begin{proof}
The result follows from the fact that
\begin{equation*}
\rho \left( t\right)\leqslant t \leqslant\sigma \left( t\right)
\end{equation*}
for every $t\in \mathbb{T}$.
\end{proof}

\begin{theorem}
Let $f,g:\mathbb{T}\rightarrow \mathbb{R}$
be functions symmetric differentiable at
$t\in \mathbb{T}_{\kappa}^{\kappa}$
and let $\lambda \in \mathbb{R}$. Then:
\begin{enumerate}
\item $f+g$ is symmetric differentiable at $t$ with
\begin{equation*}
\left( f+g\right) ^{\diamondsuit }\left( t\right) =f^{\diamondsuit }\left(
t\right) +g^{\diamondsuit }\left( t\right) \text{;}
\end{equation*}

\item $\lambda f$ is symmetric differentiable at $t$ with
\begin{equation*}
\left( \lambda f\right) ^{\diamondsuit }\left( t\right) =\lambda
f^{\diamondsuit }\left( t\right) \text{;}
\end{equation*}

\item $fg$ is symmetric differentiable at $t$ with
\begin{equation*}
\left( fg\right) ^{\diamondsuit }\left( t\right) =f^{\diamondsuit }\left(
t\right) g^{\sigma }\left( t\right) +f^{\rho }\left( t\right)
g^{\diamondsuit }\left( t\right) \text{,}
\end{equation*}
provided that $f$ and $g$ are continuous;

\item $\displaystyle \frac{1}{f}$ is symmetric differentiable at $t$ with
\begin{equation*}
\left( \frac{1}{f}\right) ^{\diamondsuit }\left( t\right)
= -\frac{f^{\diamondsuit }\left( t\right) }{f^{\sigma }\left( t\right)
f^{\rho}\left( t\right)},
\end{equation*}
provided that $f$ is continuous and $f^{\sigma }\left( t\right)
f^{\rho}\left( t\right) \neq 0$;

\item $\displaystyle \frac{f}{g}$ is symmetric differentiable at $t$ with
\begin{equation*}
\left( \frac{f}{g}\right) ^{\diamondsuit }\left( t\right)
= \frac{f^{\diamondsuit }\left( t\right) g^{\rho }\left( t\right) -f^{\rho }\left(
t\right) g^{\diamondsuit }\left( t\right)}{g^{\sigma }\left( t\right)
g^{\rho }\left( t\right)}
\end{equation*}
provided that $f$ and $g$ are continuous and $g^{\sigma }\left( t\right)
g^{\rho }\left( t\right) \neq 0$.
\end{enumerate}
\end{theorem}

\begin{proof}
\begin{enumerate}
\item For $t\in \mathbb{T}_{\kappa}^{\kappa}$ we have
\begin{eqnarray*}
\left( f+g\right) ^{\diamondsuit }\left( t\right)
&=&\lim_{s\rightarrow t}\frac{\left( f+g\right) ^{\sigma }\left( t\right)
-\left( f+g\right) \left( s\right) +\left( f+g\right) \left( 2t-s\right)
-\left( f+g\right) ^{\rho }\left( t\right) }{\sigma \left( t\right)
+2t-2s-\rho \left( t\right) } \\
&=&\lim_{s\rightarrow t}\frac{f^{\sigma }\left( t\right) -f\left( s\right)
+f\left( 2t-s\right) -f^{\rho }\left( t\right) }{\sigma \left( t\right)
+2t-2s-\rho \left( t\right) }\\
&&+\lim_{s\rightarrow t}\frac{g^\sigma\left( t\right)
-g\left( s\right) +g\left( 2t-s\right) -g^{\rho }\left( t\right) }{\sigma
\left( t\right) +2t-2s-\rho \left( t\right) } \\
&=&f^{\diamondsuit }\left( t\right) +g^{\diamondsuit }\left( t\right) .
\end{eqnarray*}

\item Let $t\in \mathbb{T}_{\kappa}^{\kappa}$ and
$\lambda \in \mathbb{R}$, then
\begin{eqnarray*}
\left( \lambda f\right) ^{\diamondsuit }\left( t\right)
&=&\lim_{s\rightarrow t}\frac{\left( \lambda f\right) ^{\sigma }\left(
t\right) -\left( \lambda f\right) \left( s\right) +\left( \lambda f\right)
\left( 2t-s\right) -\left( \lambda f\right) ^{\rho }\left( t\right) }{\sigma
\left( t\right) +2t-2s-\rho \left( t\right) } \\
&=&\lambda \lim_{s\rightarrow t}\frac{f^{\sigma }\left( t\right) -f\left(
s\right) +f\left( 2t-s\right) -f^{\rho }\left( t\right) }{\sigma \left(
t\right) +2t-2s-\rho \left( t\right) } \\
&=&\lambda f^{\diamondsuit }\left( t\right) .
\end{eqnarray*}

\item Let us assume that $t\in \mathbb{T}_{\kappa}^{\kappa}$ and $f$ and $g$ are
continuous. If $t$ is dense, then
\begin{eqnarray*}
\left( fg\right) ^{\diamondsuit }\left( t\right)
&=&\lim_{h\rightarrow 0}\frac{\left( fg\right) \left( t+h\right) -\left(
fg\right) \left( t-h\right) }{2h} \\
&=&\lim_{h\rightarrow 0}\left(\frac{f\left( t+h\right) -f\left( t-h\right) }{2h}
g\left( t+h\right)\right) +\lim_{h\rightarrow 0}\left(\frac{g\left( t+h\right)
-g\left(t-h\right) }{2h}f\left( t-h\right)\right)  \\
&=&f^{\diamondsuit }\left( t\right) g\left( t\right) +f\left( t\right)
g^{\diamondsuit }\left( t\right)  \\
&=&f^{\diamondsuit }\left( t\right) g^{\sigma }\left( t\right)
+f^{\rho}\left( t\right) g^{\diamondsuit }\left( t\right) .
\end{eqnarray*}
If $t$ is not dense, then
\begin{eqnarray*}
\left( fg\right) ^{\diamondsuit }\left( t\right)
&=&\frac{\left( fg\right)^{\sigma}\left( t\right)
-\left( fg\right) ^{\rho }\left( t\right) }{\sigma
\left( t\right) -\rho \left( t\right) } \\
&=&\frac{f^{\sigma }\left( t\right) -f^{\rho }\left( t\right) }{\sigma
\left( t\right) -\rho \left( t\right) }g^{\sigma }\left( t\right)
+\frac{g^{\sigma }\left( t\right) -g^{\rho }\left( t\right)}{\sigma \left(
t\right) -\rho \left( t\right) }f^{\rho }\left( t\right)\\
&=&f^{\diamondsuit }\left( t\right) g^{\sigma }\left( t\right)
+ f^{\rho}\left( t\right) g^{\diamondsuit }\left( t\right)
\end{eqnarray*}
proving the intended result.

\item Because
\begin{equation*}
\left( \frac{1}{f}\times f\right) \left( t\right) =1
\end{equation*}
one has
\begin{eqnarray*}
0 &=&\left( \frac{1}{f}\times f\right) ^{\diamondsuit }\left( t\right) \\
&=&f^{\diamondsuit }\left( t\right) \left( \frac{1}{f}\right)^{\sigma}\left(
t\right) +f^{\rho }\left( t\right) \left( \frac{1}{f}\right)^{\diamondsuit }\left( t\right).
\end{eqnarray*}
Therefore,
\begin{equation*}
\left( \frac{1}{f}\right) ^{\diamondsuit }\left( t\right)
= -\frac{f^{\diamondsuit}\left( t\right)}{f^{\sigma }\left( t\right)
f^{\rho}\left( t\right)}.
\end{equation*}

\item Let $t\in \mathbb{T}_{\kappa}^{\kappa}$, then
\begin{eqnarray*}
\left( \frac{f}{g}\right) ^{\diamondsuit }\left( t\right) &=&\left( f\times
\frac{1}{g}\right) ^{\diamondsuit }\left( t\right) \\
&=&f^{\diamondsuit }\left( t\right) \left( \frac{1}{g}\right) ^{\sigma
}\left( t\right) +f^{\rho }\left( t\right) \left( \frac{1}{g}\right)
^{\diamondsuit }\left( t\right) \\
&=&\frac{f^{\diamondsuit }\left( t\right) }{g^{\sigma }\left( t\right)}
+f^{\rho }\left( t\right) \left( -\frac{g^{\diamondsuit }\left(
t\right)}{g^{\sigma }\left( t\right) g^{\rho }\left( t\right) }\right) \\
&=&\frac{f^{\diamondsuit }\left( t\right) g^{\rho }\left( t\right) -f^{\rho
}\left( t\right) g^{\diamondsuit }\left( t\right) }{g^{\sigma }\left(
t\right) g^{\rho }\left( t\right) }.
\end{eqnarray*}
\end{enumerate}
\end{proof}

\begin{ex}
\begin{enumerate}
\item The symmetric derivative of $f\left( t\right) =t^{2}$ is
\begin{equation*}
f^{\diamondsuit }\left( t\right) =\sigma \left( t\right) +\rho \left(
t\right) .
\end{equation*}

\item The symmetric derivative of $f\left( t\right) =1/t$ is
\begin{equation*}
f^{\diamondsuit }\left( t\right) =-\frac{1}{\sigma \left( t\right) \rho
\left( t\right) }.
\end{equation*}
\end{enumerate}
\end{ex}

\begin{remark}
In the classical case, it can be proved that
``A continuous function is necessarily increasing in any interval in which its
symmetric derivative exists and is positive'' \cite{Thomson}.
However, it should be noted that this result is not valid for the symmetric
derivative on time scales. For instance, consider the time scale
$\mathbb{T}=\mathbb{N}$ and the function
\[
f\left( n\right) =\left \{
\begin{array}{ccc}
n & \text{ \ if \ } & n\text{ is odd } \\
&  &  \\
10n & \text{ \ if \ } & n\text{ is even.}
\end{array}
\right.
\]
The symmetric derivative of $f$ for $n$ odd is given by
\begin{eqnarray*}
f^{\diamondsuit }\left( n\right)  &=&\frac{f^{\sigma }\left( n\right)
-f^{\rho }\left( n\right) }{\sigma \left( n\right) -\rho \left( n\right) } \\
&=&\frac{10\left( n+1\right) -10\left( n-1\right) }{\left( n+1\right)
-\left( n-1\right) } \\
&=&10
\end{eqnarray*}
and for $n$ even is given by
\begin{eqnarray*}
f^{\diamondsuit }\left( n\right)  &=&\frac{f^{\sigma }\left( n\right)
-f^{\rho }\left( n\right) }{\sigma \left( n\right) -\rho \left( n\right) } \\
&=&\frac{\left( n+1\right) -\left( n-1\right) }{\left( n+1\right) -\left(
n-1\right) } \\
&=&1.
\end{eqnarray*}
Clearly, the function is not increasing although its symmetric derivative is always positive.

In this example the symmetric derivative coincides with the diamond-$\alpha$ derivative
with $\alpha=\frac{1}{2}$, and this shows that there is an inconsistency in \cite{Ozkan}.
Indeed, Corollary 2.1. of \cite{Ozkan} is not valid.
\end{remark}


\section{The diamond integral}
\label{ts:sec:int}

In the classical calculus, there are some attempts to define a symmetric integral:
see, for example, \cite{Thomson:2}. However, those integrals invert
only ``approximately'' the symmetric derivatives. In discrete time we have some examples
of symmetric integrals, namely the $q$-symmetric integral: see, for example,
\cite{Brito:da:Cruz:2,Kac}. In the context of quantum calculus,
in \cite{Brito:da:Cruz:5} the authors introduced the Hahn symmetric integral that inverts
the Hahn symmetric derivative. In time scale calculus,
the problem of determining a symmetric integral is an interesting open question. In our opinion,
the diamond-$\alpha $ integral is by now the nearest idea to construct a symmetric
integral. Regarding the similarities and the advantages of the symmetric
derivative towards the diamond-$\alpha $ derivative, we think that the
construction of the diamond integral bring us closer to the construction of
a genuine symmetric integral on time scales.

\begin{definition}
Let $f:\mathbb{T} \rightarrow \mathbb{R}$ and $a,b\in \mathbb{T}$, $a<b$.
We define the \emph{ diamond integral}\index{$\diamondsuit$-integral}
(or \emph{ $\diamondsuit$-integral}\index{Diamond integral})
of $f$ from $a$ to $b$ (or on $[a,b]_\mathbb{T}$) by
\begin{equation*}
\int_{a}^{b}f\left( t\right) \diamondsuit t=\int_{a}^{b}\gamma _{1}\left(
t\right) f\left( t\right) \Delta t+\int_{a}^{b}\gamma _{2}\left( t\right)
f\left( t\right) \nabla t,
\end{equation*}
provided $\gamma _{1}f$ is delta integrable and $\gamma_{2}f$ is nabla
integrable on $[a,b]_{\mathbb{T}}$. We say that the function $f$ is diamond integrable
(or $\diamondsuit$-integrable) on $\mathbb{T}$
if it is $\diamondsuit$-integrable for all $a,b \in \mathbb{T}$.
\end{definition}

\begin{ex}
\begin{enumerate}
\item Let $f: \mathbb{Z} \rightarrow \mathbb{R}$ be defined by
\begin{equation*}
f\left( t\right) =t^{2}.
\end{equation*}
Then,
\begin{eqnarray*}
\int_{0}^{2}f\left( t\right) \diamondsuit t &=&\frac{1}{2}
\int_{0}^{2}f\left( t\right) \Delta t+\frac{1}{2}\int_{0}^{2}f\left(
t\right) \nabla t \\
&=&\frac{1}{2}\sum_{t=0}^{1}f\left( t\right) +\frac{1}{2}\sum_{t=1}^{2}f\left( t\right) \\
&=&\frac{1}{2}\left( 0+1\right) +\frac{1}{2}\left( 1+4\right) \\
&=&3
\end{eqnarray*}

\item Let $f:\left[ 0,1\right] \cup \left \{ 2,4\right \}
\rightarrow \mathbb{R}$ be defined by
\begin{equation*}
f\left( t\right) =1.
\end{equation*}
Then
\begin{eqnarray*}
\int_{0}^{4}f\left( t\right) \diamondsuit t &=&\int_{0}^{1}f\left( t\right)
\diamondsuit t+\int_{1}^{4}f\left( t\right) \diamondsuit t \\
&=&\int_{0}^{1}1dt+\int_{1}^{4}\gamma _{1}\left( t\right) \Delta
t+\int_{1}^{4}\gamma _{2}\left( t\right) \nabla t \\
&=&1+\gamma _{1}\left( 1\right) \left( \sigma \left( 1\right) -1\right)
+\gamma _{1}\left( 2\right) \left( \sigma \left( 2\right) -2\right) \\
&&+\gamma _{2}\left( 2\right) \left( 2-\rho \left( 2\right) \right)
+\gamma_{2}\left( 4\right) \left( 4-\rho \left( 4\right) \right) \\
&=&1+1+\frac{4}{3}+\frac{1}{3}+2 \\
&=&\frac{17}{3}.
\end{eqnarray*}
Note that the diamond-$\alpha$ integral of the same function is
\begin{eqnarray*}
\int_{0}^{4}f\left( t\right) \diamondsuit _{\alpha }t &=&\int_{0}^{1}f\left(
t\right) \diamondsuit t+\alpha \int_{1}^{4}1\Delta t+\left( 1-\alpha \right)
\int_{1}^{4}1\nabla t \\
&=&1+3\alpha +\left( 1-\alpha \right) 3 \\
&=&4.
\end{eqnarray*}
\end{enumerate}
\end{ex}

The $\diamondsuit$-integral has the following properties.

\begin{theorem}
\label{ts:propriedades integral}
Let $f,g:\mathbb{T}\rightarrow \mathbb{R}$ be $\diamondsuit$-integrable on
$\left[ a,b\right]_{\mathbb{T}} $. Let $c\in \left[a,b\right]_{\mathbb{T}}$
and $\lambda \in \mathbb{R}$. Then,
\begin{enumerate}
\item $\displaystyle \int_{a}^{a}f\left( t\right) \diamondsuit t=0$;

\item $\displaystyle \int_{a}^{b}f\left( t\right) \diamondsuit
t=\int_{a}^{c}f\left( t\right) \diamondsuit t+\int_{c}^{b}f\left( t\right)
\diamondsuit t$;

\item $\displaystyle \int_{a}^{b}f\left( t\right) \diamondsuit t
=-\displaystyle \int_{b}^{a}f\left( t\right) \diamondsuit t$;

\item $f+g$ is $\diamondsuit$-integrable on $\left[ a,b\right]_{\mathbb{T}}$ and
\begin{equation*}
\int_{a}^{b}\left( f+g\right) \left( t\right) \diamondsuit
t=\int_{a}^{b}f\left( t\right) \diamondsuit t+\int_{a}^{b}g\left( t\right)
\diamondsuit t\text{;}
\end{equation*}

\item $\lambda f$ is $\diamondsuit$-integrable on $\left[ a,b\right]_{\mathbb{T}}$ and
\begin{equation*}
\int_{a}^{b}\lambda f\left( t\right) \diamondsuit t=\lambda
\int_{a}^{b}f\left( t\right) \diamondsuit t\text{;}
\end{equation*}

\item $fg$ is $\diamondsuit$-integrable on $\left[ a,b\right]_{\mathbb{T}}$;

\item for $p>0$, $|f|^{p}$ is $\diamondsuit$-integrable on $\left[ a,b\right]_{\mathbb{T}}$;

\item if $f\left(t\right) \leqslant g\left( t\right)$
for all $t\in \left[ a,b\right]_{\mathbb{T}}$, then
\begin{equation*}
\int_{a}^{b}f\left( t\right) \diamondsuit t
\leqslant \int_{a}^{b}g\left(t\right) \diamondsuit t;
\end{equation*}

\item $\left \vert f\right \vert$ is
$\diamondsuit$-integrable on $\left[ a,b\right]_{\mathbb{T}}$ and
\begin{equation*}
\left \vert \int_{a}^{b}f\left( t\right) \diamondsuit t\right \vert
\leqslant \int_{a}^{b}\left \vert f\left( t\right) \right \vert \diamondsuit t.
\end{equation*}
\end{enumerate}
\end{theorem}

\begin{proof}
The results follow straightforwardly from
the properties of the nabla and delta integrals.
\end{proof}

\begin{remark}
It is clear that the $\diamondsuit$-integral coincides with the $\diamondsuit_{\alpha}$-integral
when the functions $\gamma_{1}$ and $\gamma_{2}$ are constant and we choose
a proper $\alpha$. There are several and important time scales where this happens.
For instance, when $\mathbb{T} = \mathbb{R}$ or $\mathbb{T} = h\mathbb{Z}$,
$h>0$, we have that the $\diamondsuit$-integral is equal to the $\diamondsuit_{\frac{1}{2}}$-integral.
Since the Fundamental Theorem of Calculus is not valid for the $\diamondsuit_{\alpha}$-integral \cite{Sheng},
then it is clear that the Fundamental Theorem of Calculus is also not valid for the
$\diamondsuit$-integral.
\end{remark}

Next we prove some integral inequalities which are similar to the
$\diamondsuit_{\alpha}$-integral counterparts found in \cite{Ammi,Malinowska:2}.

\begin{theorem}[Mean value theorem for the diamond integral]
Let\index{Mean value theorem for the diamond integral}
$f,g:\mathbb{T}\rightarrow \mathbb{R}$ be bounded and $\diamondsuit$-integrable
functions on $[a,b]_\mathbb{T}$, and let $g$ be nonnegative or nonpositive on
$\left[ a,b\right]_{\mathbb{T}}$. Let $m$ and $M$ be the infimum and supremum,
respectively, of function $f$. Then, there exists
a real number $K$ satisfying the inequalities
\begin{equation*}
m\leqslant K\leqslant M
\end{equation*}
such that
\begin{equation*}
\int_{a}^{b}\left( fg\right) \left( t\right) \diamondsuit t
=K\int_{a}^{b}g\left( t\right) \diamondsuit t.
\end{equation*}
\end{theorem}

\begin{proof}
Without loss of generality, we suppose that $g$ is nonnegative
on $\left[ a,b\right]_{\mathbb{T}}$.
Since, for all $t\in\left[ a,b\right]_{\mathbb{T}}$
\begin{equation*}
m\leqslant f\left( t\right) \leqslant M
\end{equation*}
and\thinspace $g\left( t\right) \geqslant 0$, then
\begin{equation*}
mg\left( t\right) \leqslant f\left( t\right) g\left( t\right) \leqslant
Mg\left( t\right)
\end{equation*}
for all $t\in \left[ a,b\right]_{\mathbb{T}}$. Each of the functions $mg$, $fg$ and $Mg$ is
$\diamondsuit$-integrable from $a$ to $b$ and,
by Theorem~\ref{ts:propriedades integral}, one has
\begin{equation*}
m\int_{a}^{b}g\left( t\right) \diamondsuit t\leqslant \int_{a}^{b}f\left(
t\right) g\left( t\right) \diamondsuit t\leqslant M\int_{a}^{b}g\left(
t\right) \diamondsuit t.
\end{equation*}
If $\displaystyle \int_{a}^{b}g\left( t\right) \diamondsuit t=0$, then
$\displaystyle \int_{a}^{b}f\left( t\right) g\left( t\right)
\diamondsuit t=0$ and can choose any $K\in[m,M]$.
If \\$\displaystyle \int_{a}^{b}g\left( t\right) \diamondsuit t>0$, then
\begin{equation*}
m\leqslant \frac{\int_{a}^{b}f\left( t\right) g\left( t\right) \diamondsuit t}{
\int_{a}^{b}g\left( t\right) \diamondsuit t}\leqslant M
\end{equation*}
and we choose $K$ as
\begin{equation*}
K:=\frac{\int_{a}^{b}f\left( t\right) g\left( t\right)
\diamondsuit t}{\int_{a}^{b}g\left( t\right) \diamondsuit t}.
\end{equation*}
\end{proof}

We now present  $\diamondsuit $-versions
of H\"{o}lder's, Cauchy-Schwarz's
and Minkowski's inequalities.

\begin{theorem}[H\"{o}lder's inequality for the diamond integral]
\label{ts:Holders Inequality TS}
Let\index{H\"{o}lder's inequality for the diamond integral}
$f,g:\mathbb{T}\rightarrow \mathbb{R}$ be functions $\diamondsuit$-integrable on $\left[ a,b\right]_{\mathbb{T}}$.
Then
\begin{equation*}
\int_{a}^{b}\left\vert f\left( t\right) g\left( t\right) \right \vert
\diamondsuit t\leqslant \left( \int_{a}^{b}\left \vert f\left( t\right)
\right \vert ^{p}\diamondsuit t\right) ^{\frac{1}{p}}\left(
\int_{a}^{b}\left \vert g\left( t\right)
\right\vert^{q}\diamondsuit t\right)^{\frac{1}{q}},
\end{equation*}
where $p>1$ and $q=\displaystyle \frac{p}{p-1}$.
\end{theorem}

\begin{proof}
For $\lambda,\gamma\in \mathbb{R}_{0}^{+}$ and
$p,q$ such that $p>1$ and $\frac{1}{p}+\frac{1}{q}=1$, the
following inequality holds (Young's inequality):
\begin{equation*}
\lambda ^{\frac{1}{p}}\gamma ^{\frac{1}{q}}\leqslant \frac{\lambda }{p}
+ \frac{\gamma}{q}.
\end{equation*}
Without loss of generality, let us suppose that
\begin{equation*}
\left( \int_{a}^{b}\left \vert f\left( t\right) \right \vert^{p}\diamondsuit
t\right) \left( \int_{a}^{b}\left \vert g\left(
t\right) \right \vert^{q}\diamondsuit t\right) \neq 0
\end{equation*}
(note that both integrals exist by Theorem~\ref{ts:propriedades integral}). Define
\begin{equation*}
\lambda \left( t\right) :=\frac{\left \vert f\left( t\right)
\right \vert ^{p}}{\int_{a}^{b}\left \vert
f\left( \tau \right) \right \vert ^{p}\diamondsuit \tau}
\text{ \ and \ }\gamma \left( t\right) :=\frac{\left \vert g\left( t\right)
\right \vert ^{q}}{\int_{a}^{b}\left \vert g\left( \tau \right) \right \vert
^{q}\diamondsuit \tau}.
\end{equation*}
Since both functions $\lambda $ and $\gamma $ are $\diamondsuit$-integrable
on $\left[ a,b\right]_\mathbb{T} $, then
\begin{align*}
& \int_{a}^{b}\frac{\left \vert f\left( t\right) \right \vert }{\left(
\int_{a}^{b}\left \vert f\left( \tau \right) \right \vert^{p}\diamondsuit\tau \right) ^{\frac{1}{p}}}\frac{\left \vert g\left( t\right)
\right\vert}{\left( \int_{a}^{b}\left \vert g\left( \tau \right) \right \vert ^{q}\diamondsuit\tau \right) ^{\frac{1}{q}}}\diamondsuit t \\
& =\int_{a}^{b}\left(\lambda \left( t\right) ^{\frac{1}{p}}\right)\left(\gamma \left( t\right)
^{\frac{1}{q}}\right)\diamondsuit t \\
& \leqslant \int_{a}^{b}\left( \frac{\lambda \left( t\right) }{p}
+\frac{\gamma \left( t\right)}{q}\right) \diamondsuit t \\
& =\int_{a}^{b}\left( \frac{1}{p}\frac{\left \vert f\left( t\right)
\right \vert ^{p}}{\int_{a}^{b}\left \vert f\left( \tau \right) \right \vert
^{p}\diamondsuit \tau}+\frac{1}{q}\frac{\left \vert g\left( t\right) \right \vert
^{q}}{\int_{a}^{b}\left \vert g\left( \tau \right) \right \vert
^{q}\diamondsuit \tau}\right) \diamondsuit t \\
& =\frac{1}{p}\int_{a}^{b}\left( \frac{\left \vert f\left( t\right)
\right \vert ^{p}}{\int_{a}^{b}\left \vert f\left( \tau \right) \right \vert
^{p}\diamondsuit \tau}\right) \diamondsuit t+\frac{1}{q}\int_{a}^{b}\left(
\frac{\left \vert g\left( t\right) \right \vert ^{q}}{\int_{a}^{b}\left \vert
g\left( \tau \right) \right \vert ^{q}\diamondsuit \tau}\right) \diamondsuit t \\
& =\frac{1}{p}+\frac{1}{q} \\
& =1.
\end{align*}
Hence
\begin{equation*}
\int_{a}^{b}\left \vert f\left( t\right) g\left( t\right) \right \vert
\diamondsuit t\leqslant \left( \int_{a}^{b}\left \vert f\left( t\right)
\right \vert ^{p}\diamondsuit t\right) ^{\frac{1}{p}}\left(
\int_{a}^{b}\left \vert g\left( t\right) \right \vert ^{q}\diamondsuit
t\right)^{\frac{1}{q}}.
\end{equation*}
\end{proof}

\begin{corollary}[Cauchy-Schwarz inequality for the diamond integral]
If\index{Cauchy-Schwarz inequality for the diamond integral}
$f,g:\mathbb{T}\rightarrow \mathbb{R}$ are $\diamondsuit$-integrable
on $\left[ a,b\right]_{\mathbb{T}}$, then
\begin{equation*}
\int_{a}^{b}\left \vert f\left( t\right) g\left( t\right) \right \vert
\diamondsuit t\leqslant \sqrt{\left( \int_{a}^{b}\left \vert f\left( t\right)
\right \vert ^{2}\diamondsuit t\right) \left( \int_{a}^{b}\left \vert g\left(
t\right) \right \vert ^{2}\diamondsuit t\right) }\text{.}
\end{equation*}
\end{corollary}

\begin{proof}
This is a particular case of
Theorem~\ref{ts:Holders Inequality TS} where $p=2=q$.
\end{proof}

\begin{theorem}[Minkowski's inequality for diamond integral]\index{Minkowski's inequality for diamond integral}
If $f,g:\mathbb{T}\rightarrow \mathbb{R}$ are
$\diamondsuit$-integrable on $\left[ a,b\right]_{\mathbb{T}} $ and $p>1$, then
\begin{equation*}
\left( \int_{a}^{b}\left \vert f\left( t\right) +g\left( t\right) \right \vert^{p}
\diamondsuit t\right) ^{\frac{1}{p}}\leqslant \left(
\int_{a}^{b}\left \vert f\left( t\right) \right \vert ^{p}\diamondsuit
t\right) ^{\frac{1}{p}}+\left( \int_{a}^{b}\left \vert g\left( t\right)
\right \vert ^{p}\diamondsuit t\right) ^{\frac{1}{p}}.
\end{equation*}
\end{theorem}

\begin{proof}
If $\displaystyle \int_{a}^{b}\left \vert f\left( t\right)
+g\left( t\right) \right \vert^{p}\diamondsuit t=0$ the result is trivial.
Suppose that $\displaystyle \int_{a}^{b}\left \vert f\left( t\right)
+g\left( t\right) \right \vert^{p}\diamondsuit t\neq0$.
Since
\begin{align*}
& \int_{a}^{b}\left \vert f\left( t\right) +g\left( t\right) \right \vert^{p}
\diamondsuit t=\int_{a}^{b}\left \vert f\left( t\right) +g\left( t\right)
\right \vert ^{p-1}\left \vert f\left( t\right) +g\left( t\right) \right \vert
\diamondsuit t \\
& \leqslant \int_{a}^{b}\left \vert f\left( t\right) \right \vert \left \vert
f\left( t\right) +g\left( t\right) \right \vert ^{p-1}\diamondsuit
t+\int_{a}^{b}\left \vert g\left( t\right) \right \vert \left \vert f\left(
t\right) +g\left( t\right) \right \vert ^{p-1}\diamondsuit t,
\end{align*}
then by H\"{o}lder's inequality (Theorem~\ref{ts:Holders Inequality TS}) with
$q=\displaystyle \frac{p}{p-1}$, we obtain
\begin{align*}
& \int_{a}^{b}\left \vert f\left( t\right) +g\left( t\right) \right \vert
^{p}\diamondsuit t  \\
\leqslant &\left( \int_{a}^{b}\left \vert f\left( t\right) \right \vert
^{p}\diamondsuit t\right) ^{\frac{1}{p}}\left( \int_{a}^{b}\left \vert
f\left( t\right) +g\left( t\right) \right \vert ^{\left( p-1\right)
q}\diamondsuit t\right) ^{\frac{1}{q}}  \\
&+\left( \int_{a}^{b}\left \vert g\left( t\right) \right \vert
^{p}\diamondsuit t\right) ^{\frac{1}{p}}\left( \int_{a}^{b}\left \vert
f\left( t\right) +g\left( t\right) \right \vert ^{\left( p-1\right)
q}\diamondsuit t\right) ^{\frac{1}{q}} \\
=& \left[ \left( \int_{a}^{b}\left \vert f\left( t\right) \right \vert
^{p}\diamondsuit t\right) ^{\frac{1}{p}}+\left( \int_{a}^{b}\left \vert
g\left( t\right) \right \vert ^{p}\diamondsuit t\right) ^{\frac{1}{p}}\right]
 \left( \int_{a}^{b}\left \vert f\left( t\right) +g\left( t\right)
\right \vert ^{\left( p-1\right) q}\diamondsuit t\right) ^{\frac{1}{q}}.
\end{align*}
Hence,
\begin{align*}
\int_{a}^{b}\left \vert f\left( t\right) +g\left( t\right) \right \vert^{p}\diamondsuit t
\leqslant  \left[ \left( \int_{a}^{b}\left \vert f\left( t\right) \right \vert^{p}
\diamondsuit t\right)^{\frac{1}{p}}+\left( \int_{a}^{b}\left \vert
g\left( t\right) \right \vert ^{p}\diamondsuit t\right) ^{\frac{1}{p}}\right]
\left( \int_{a}^{b}\left \vert f\left( t\right) +g\left( t\right)
\right \vert ^{ p}\diamondsuit t\right) ^{\frac{1}{q}}
\end{align*}
and, dividing both sides by
\begin{equation*}
\left(
\int_{a}^{b}\left \vert f\left( t\right) +g\left( t\right) \right \vert^{p}
\diamondsuit t\right) ^{\frac{1}{q}},
\end{equation*}
we obtain
\begin{equation*}
\left( \int_{a}^{b}\left \vert f\left( t\right) +g\left( t\right) \right \vert^{p}
\diamondsuit t\right) ^{\frac{1}{p}}\leqslant \left(
\int_{a}^{b}\left \vert f\left( t\right) \right \vert ^{p}\diamondsuit
t\right) ^{\frac{1}{p}}+\left( \int_{a}^{b}\left \vert g\left( t\right)
\right \vert ^{p}\diamondsuit t\right) ^{\frac{1}{p}}.
\end{equation*}
\end{proof}


\section{State of the Art}

The time scale theory is quite new and is under strong current research.
We now summarize the references already given in the Chapter~\ref{ts:sec:pre}.
For the time scale calculus we refer to \cite{Agarwal:2,Agarwal,Bohner,Bohner:2,Hilger},
for the calculus of variations within the time scale setting we refer to
\cite{Almeida,Bartosiewicz05,Bohner:3,Ferreira:3,Ferreira,Ferreira:2,Hilscher:2,Malinowska35,Malinowska:4,Martins,Martins15,Martins:4,Torres},
while for the diamond-$\alpha$ integral we suggest \cite{Ammi,Malinowska:2,Sheng:2,Sheng:3,Sheng}.

The results of this chapter were presented by the author at
The International Meeting on Applied Mathematics in Errachidia,
Morocco, April, 2012 and are published in \cite{Brito:da:Cruz:6}.


\clearpage{\thispagestyle{empty}\cleardoublepage}


\chapter{Conclusions and Future Work}
\label{Conclusions}

The goal of this thesis was the development of a symmetric variational calculus.
We studied some symmetric quantum calculus and the symmetric time scale calculus and,
whenever possible, we introduced the calculus of variations within the set of study.

For a first experience on quantum calculus, we began our work on Hahn's (nonsymmetric)
quantum calculus (see Chapter~\ref{Higher-order Hahn's Quantum Variational Calculus}).
We contributed with a necessary optimality condition of Euler--Lagrange type involving
Hahn's derivatives of higher-order (Theorem~\ref{hoHigher order E-L}).

For the quantum symmetric calculus, we established and proved results for the
$\alpha,\beta$-calculus, $q$-calculus and Hahn's calculus.

In the $\alpha,\beta$-symmetric calculus, Chapter~\ref{A Symmetric Quantum Calculus},
we defined and proved some properties of the $\alpha,\beta$-symmetric derivative
and the $\alpha,\beta$-symmetric N\"{o}rlund sum. In Section~\ref{hs:sec:3.3},
we derived some mean value theorems for the $\alpha,\beta$-symmetric derivative
and we presented $\alpha,\beta$-symmetric versions of Fermat's theorem for stationary points,
Rolle's, Lagrange's, and Cauchy's mean value theorems.
In Section~\ref{hi:sec:ineq} we obtained some $\alpha,\beta$-symmetric N\"{o}rlund sum inequalities,
namely $\alpha,\beta$-symmetric versions of H\"{o}lder's, Cauchy-Scharwz's
and Minkowski's integral inequalities.

In the symmetric $q$-calculus, Chapter~\ref{The $q$-Symmetric Variational Calculus},
we were able not only to develop this calculus for any real interval,
instead the quantum lattice $\{a,aq,aq^{2},aq^{2},\ldots\}$,
but also to introduce the $q$-symmetric variational calculus.
We proved a necessary optimality condition of Euler--Lagrange type
(Theorem~\ref{qs:Euler}) and a sufficient optimality condition (Theorem~\ref{qs:suff})
for $q$-symmetric variational problems. It should be noted that in the $q$-symmetric calculus
we were able to introduce the calculus of variations because we proved
the fundamental theorem for the $q$-symmetric calculus (Theorem~\ref{qs:Fundamental}).

The Hahn symmetric calculus, Chapter~\ref{Hahn's Symmetric Quantum Variational Calculus},
is a generalization of the $q$-symmetric calculus. This calculus has good properties,
like the fundamental theorem of Hahn's symmetric integral calculus
(Theorem~\ref{qhs:Fundamental}) and hence we were able to introduce
the Hahn symmetric variational calculus. In Section~\ref{qhs:sec:o:c} we presented necessary
(Euler--Lagrange type equation) and sufficient optimality conditions for the Hahn symmetric calculus.
Moreover, we were able to apply Leitmann's direct method in the Hahn symmetric variational calculus
(see Section~\ref{qhs:L}).

Right from the start, we wanted to define the symmetric calculus on time scales.
In Chapter~\ref{The Symmetric Calculus on Time Scales} we successfully defined
a symmetric derivative on time scales and derived some of its properties.
Although we did not define the symmetric integral, we defined the diamond integral,
which is a refined version of the diamond-$\alpha$ integral. We proved a mean value theorem
for the diamond integral and we proved versions of H\"{o}lder's, Cauchy-Schwarz's,
and Minkowski's inequalities (see Section~\ref{ts:sec:int}).

In this thesis we defined several new quantum calculus and, as in every new type of calculus,
there are more questions than answers. Some possible directions for future work are:
\begin{itemize}
\item to develop and derive new properties for symmetric quantum derivatives and integrals,
for example, to study first- and second-order equations, linear systems,
and higher-order differential equations;

\item to explore what good properties one gets if we study the limits
of symmetric quantum calculus;

\item to study optimality conditions for more general variable endpoint
variational problems and isoperimetric problems;

\item to extend the results on symmetric quantum variational problems
for higher-order problems of the calculus of variations;

\item to obtain a Legendre's necessary condition for Hahn's variational calculus
and for the Hahn symmetric variational calculus.
\end{itemize}

We would like to be able to construct the symmetric integral for an arbitrary time scale.
We trust that after this work we are several steps closer to a solution.
However, such question remains an open problem, and it will be one direction
for our future research.

This Ph.D. thesis comes to an end with the list of the author publications during his Ph.D.:
\cite{MyID:178,Brito:da:Cruz:2,Brito:da:Cruz,Brito:da:Cruz:3,Brito:da:Cruz:4,Brito:da:Cruz:5,Brito:da:Cruz:6,Brito:da:Cruz:BMMSS}.


\clearpage{\thispagestyle{empty}\cleardoublepage}



\clearpage{\thispagestyle{empty}\cleardoublepage}


\printindex


\clearpage{\thispagestyle{empty}\cleardoublepage}


\end{document}